\documentclass[10pt]{article}

\usepackage[english]{babel}
\usepackage{epsfig}

\usepackage[T1]{fontenc}
\usepackage{lmodern}

\usepackage{stmaryrd}
\usepackage{amsmath,amsfonts,amssymb,mathrsfs,amsthm}
\usepackage[bb=stix]{mathalpha}
\usepackage{mathtools}
\let\voidset = \emptyset 
\let\capsymb = \cap
\let\cupsymb = \cup
\let\negsymb = \neg
\usepackage{mathabx}
\let\emptyset = \voidset 
\let\cap = \capsymb
\let\cup = \cupsymb
\let\neg = \negsymb
\usepackage{xcolor}
\usepackage{dsfont}
\usepackage{graphicx}
\usepackage{euscript}
\usepackage{makeidx}
\usepackage{verbatim}
\usepackage{graphics,graphicx}
\usepackage{tikz}
\usetikzlibrary{patterns,positioning,arrows.meta}
\usepackage{textcomp}
\usepackage{float}
\usepackage{url}
\usepackage[colorlinks=true]{hyperref}
\usepackage{multicol}
\usepackage{pgfplots}
\pgfplotsset{compat=1.18}

\newcommand{\CC}{\mathbb{C}}\newcommand{\MM}{\mathbb{M}}\newcommand{\NN}{\mathbb{N}}\newcommand{\QQ}{\mathbb{Q}}\newcommand{\RR}{\mathbb{R}}\renewcommand{\SS}{\mathbb{S}}\newcommand{\TT}{\mathbb{T}}\newcommand{\ZZ}{\mathbb{Z}}

\newcommand{\indicatrice}{\mathds{1}}
\newcommand{\Cc}{\mathcal{C}}\newcommand{\Ec}{\mathcal{E}}\newcommand{\Lc}{\mathcal{L}}\newcommand{\Mc}{\mathcal{M}}\newcommand{\Sc}{\mathcal{S}}\newcommand{\Uc}{\mathcal{U}}

\newcommand{\Ak}{\mathfrak{A}}\newcommand{\Ck}{\mathfrak{C}}\newcommand{\Fk}{\mathfrak{F}}

\newcommand{\Cscr}{\mathscr{C}}\newcommand{\Escr}{\mathscr{E}}

\newcommand{\Ae}{\EuScript{A}}\newcommand{\Be}{\EuScript{B}}\newcommand{\De}{\EuScript{D}}\newcommand{\He}{\EuScript{H}}\newcommand{\Ke}{\EuScript{K}}\newcommand{\Le}{\EuScript{L}}\renewcommand{\Re}{\EuScript{R}}\newcommand{\Se}{\EuScript{S}}\newcommand{\Ue}{\EuScript{U}}\newcommand{\Ve}{\EuScript{V}}

\newcommand{\smooth}{\mathscr{C}^\infty}
\newcommand{\test}{\mathscr{C}^{\infty}_{c}}

\newcommand{\barre}[1]{\overline{#1}}

\newcommand{\pscal}[2]{\left\langle#1, #2\right\rangle}

\def\XXint#1#2#3{{\setbox0=\hbox{$#1{#2#3}{\int}$ }
\vcenter{\hbox{$#2#3$ }}\kern-.6\wd0}}

\DeclarePairedDelimiter\norm{\|}{\|}
\DeclarePairedDelimiter\abs{|}{|}
\DeclarePairedDelimiter\p{(}{)}
\DeclarePairedDelimiter\classe{[}{]}
\DeclarePairedDelimiter\jap{\langle}{\rangle}
\DeclarePairedDelimiter\set{\{}{\}}

\newcommand{\eqcases}[2]{\left\{\begin{array}{#1}#2\end{array}\right.}

\renewcommand{\leq}{\leqslant}
\renewcommand{\geq}{\geqslant}

\newcommand{\garding}{Gårding}

\newtheorem{lemma}{Lemma}[section]
\newtheorem{theorem}[lemma]{Theorem}
\newtheorem{proposition}[lemma]{Proposition}
\newtheorem{corollary}[lemma]{Corollary}

\newtheorem{assumption-notation}[lemma]{Assumption and notations}

\theoremstyle{definition}

\newtheorem{definition}[lemma]{Definition}
\newtheorem{remark}[lemma]{Remark}

\DeclareMathOperator{\phg}{hom}
\DeclareMathOperator{\tr}{tr}

\DeclareMathOperator{\loc}{loc}
\DeclareMathOperator{\comp}{comp}

\DeclareMathOperator{\Imaginary}{Im}

\DeclareMathOperator{\supp}{supp}
\DeclareMathOperator{\Char}{Char}

\DeclareMathOperator{\dist}{dist}

\renewcommand{\d}{\mathrm{d}}

\DeclareMathOperator{\Op}{Op}
\DeclareMathOperator{\WaveFront}{WF}

\newcommand{\ie}{{i.e.,}~}
\newcommand{\apriori}{{a priori}}

\newcommand{\resp}{\text{resp.~}}

\renewcommand{\ast}{\convolution}
\renewcommand{\tilde}{\widetilde}

\renewcommand{\check}{\widecheck}

\newcommand{\tang}{\mathrm{T}}
\newcommand{\cotang}{\mathrm{T}^{\star}\!}
\newcommand{\cotanginf}{\barre{\mathrm{T}}\hspace{-7pt}\phantom{\mathrm{T}}^{\star}\!}

\newcommand{\cosphere}{\mathrm{S}^{\star}\!}
\newcommand{\Cobs}{\Ck_{\mathrm{obs}}}
\newcommand{\leqsim}{\lesssim}

\newcommand{\ds}{\displaystyle}
\newcommand{\lmbd}{\lambda} 
\newcommand{\sqrtsymbol}{q}
\newcommand{\SqrtOp}{\mathcal{Q}}
\newcommand{\timespacesymbol}{\mathbb{p}}
\newcommand{\rest}{\mathcal{R}}
\newcommand{\symplectic}{\varsigma}
\newcommand{\linear}{\mathscr{L}}
\renewcommand{\O}[1]{\mathcal{O}\p{#1}}
\renewcommand{\o}[1]{o\p{#1}}
\newcommand{\holeperp}[1]{\Theta_{\xi}^{#1}}
\newcommand{\holepara}[1]{\Theta_{\tau}^{#1}}

\newcommand{\sdm}{\nu}
\newcommand{\mdm}{\mu}
\newcommand{\ml}{\mdm^{\mathsf{ml}}}

\newcommand{\signe}{\flat}

\newcommand{\HUM}{\mathrm{HUM}}
\newcommand{\pseudo}{\Psi}
\newcommand{\SchroOp}{\mathcal{P}}
\newcommand{\HWOp}{\mathcal{L}}

\newcommand\trans{{}^{\mathsf{t}} \!}

\numberwithin{equation}{section}

\makeatletter
\def\keywords{
\vspace{1ex}
\noindent
\if@twocolumn
\small{\bf  Keywords}\/---$\!$    \else
\begin{center}\small\ {\bf Keywords}\end{center}\quotation\small
\fi}
\def\endkeywords{\vspace{0.6em}\par\if@twocolumn\else\endquotation\fi
\normalsize\rm}
\makeatother
\title{Uniform controllability for the wave equation with large potential}
\author{Arthur Yax\footnote{Laboratoire de Mathématiques d'Orsay, Université Paris-Saclay, CNRS, Bâtiment 307, 91405 Orsay Cedex France, email: arthur.yax@universite-paris-saclay.fr.}} 
\date{\today}

\begin{document}

\maketitle

\begin{abstract}
This paper investigates the dependence of the control cost for a wave equation with respect to perturbation by a time-independent potential $\lmbd V$ scaled by a large parameter $\lmbd$ on a compact Riemannian manifold.
We introduce the geometric control condition~\eqref{GCC+}, a variant of the geometric control condition of Bardos--Lebeau--Rauch--Taylor, tailored to accommodate the influence of the potential $V$.
We show that~\eqref{GCC+} is necessary and sufficient for the existence of a uniform \emph{observability cost} with respect to the large parameter $\lmbd$.
We provide geometric examples satisfying~\eqref{GCC+} and estimate the blow-up rate of the \emph{observability cost} in situations where it fails.
The proofs rely on semiclassical and second microlocal defect measures.
\end{abstract}

\begin{keywords}
\noindent
Wave equation, Half-wave equations, Observability, Controllability, Semiclassical and microlocal analysis.
\medskip

\noindent
\textbf{2010 Mathematics Subject Classification:}
47F05, 
35L05, 
93B07, 
93B05, 
35F05. 
\end{keywords}

\begingroup
\tableofcontents
\endgroup
\setcounter{tocdepth}{1}
\section{Introduction, notations and results}
\subsection{Background and motivation}
Let $(M, g)$ be a $d$-dimensional smooth, compact, connected Riemannian manifold without boundary, and denote by $\Delta_g$ the associated Laplace-Beltrami operator (see~\eqref{eq:Laplace-Beltrami} for the definition).
Given a continuous nonnegative function $b_\omega\in\Cscr^0\p{M, \RR_+}$, we study the \emph{controllability in time $T>0$ from $\omega$} of the wave equation with a smooth positive potential $\mathsf{V}\in\smooth\p{M, \RR_+^*}$:
\begin{equation}\label{eq:Control:Wave}
\eqcases{rcll}{\partial_t^2 v - \Delta_g v+  \mathsf{V}v &=& b_\omega f &\text{ in } (0, T)\times M,\\
\p{v, \partial_t v}_{|t = 0} &=& (v_0, v_1)&\text{ in }M.}
\end{equation}
\begin{definition}
Let $T > 0$ and $b_\omega\in\Cscr^0(M, \RR_+)$.
\begin{itemize}
\item We say that we have \emph{exact controllability in time $T>0$ from $\omega$} of~\eqref{eq:Control:Wave} if for any initial conditions $(v_0, v_1)\in H^1(M)\times L^2(M)$ and any target $(v_0^T, v_1^T)\in H^1(M)\times L^2(M)$, there is a \emph{control function} $f \in L^2((0, T)\times M)$ such that $v$, the unique solution $\Cscr^0((0, T), H^1( M))\cap\Cscr^1((0, T), L^2( M))$ of~\eqref{eq:Control:Wave}, satisfies $(v, \partial_t v)_{t = T} = (v_0^T, v_1^T)$.
\item We say that we have \emph{null-controllability in time $T>0$ from $\omega$} of~\eqref{eq:Control:Wave} if for any initial conditions $(v_0, v_1)\in H^1(M)\times L^2(M)$, there is a \emph{control function} $f \in L^2((0, T)\times M)$ such that $v$, the unique solution $\Cscr^0((0, T), H^1( M))\cap\Cscr^1((0, T), L^2( M))$ of~\eqref{eq:Control:Wave}, satisfies $(v, \partial_t v)_{t = T} = (0, 0)$.
\end{itemize}
\end{definition}
\begin{remark}
In this context, by linearity and time reversibility of the operator $\partial_t^2 - \Delta_g + \mathsf{V}$, the \emph{null-controllability} of~\eqref{eq:Control:Wave} and the \emph{exact controllability} of~\eqref{eq:Control:Wave} are equivalent.
\end{remark}
\paragraph{Observability Inequality.}
By classical functional analysis arguments (see~\cite{Dolecki77observation,lions1988controlabilite}), proving \emph{exact controllability} of~\eqref{eq:Control:Wave} is equivalent to establishing an observability inequality for the dual equation
\begin{equation}\label{eq:Wave}\tag{WE}
\eqcases{rcll}{
\partial_t^2u - \Delta_gu + \mathsf{V}u &=& 0 &\text{in}~ (0, T)\times M,\\
\p{u, \partial_t u}_{|t = 0} &=& (u_0, u_1)&\text{in}~M,
}
\end{equation}
see Corollary~\ref{Cor:HUM}.
We say that~\eqref{eq:Wave} is \emph{observable in time $T$ from $\omega$} if there exists an \emph{observability cost} $\Cobs(\mathsf{V}) \in (0, \infty)$ such that for all $(u_0, u_1)\in H^1(M)\times L^2(M)$ we have
\begin{equation}\label{eq:Obs}\tag{Obs}
\Escr_{\mathsf{V}}(u_0, u_1) := \frac{1}{2}\p{\norm{u_0}_{H^1_{\mathsf{V}}(M)}^2 + \norm{u_1}^2_{L^2(M)}}\leq \Cobs(\mathsf{V})\int_{0}^{T}\norm{b_\omega \partial_t u}_{L^2(M)}^2\d t,
\end{equation}
where $u$ is the $\Cscr^0((0, T), H^1( M))\cap\Cscr^1((0, T), L^2( M))$-solution of~\eqref{eq:Wave}, the norm $\norm{\cdot}_{H^1_{\mathsf{V}}(M)}$ is defined as
\begin{equation}\label{df:H^1_V}
\norm{w}_{H^1_{\mathsf{V}}(M)} := \sqrt{\int_{M}\abs{\nabla_g w(x)}^2_g + \mathsf{V}(x)\abs{w(x)}^2 \d x},\quad w \in H^1_{\mathsf{V}}(M):=H^1(M),
\end{equation}
$\d x$ represents the Riemannian volume element (in local charts, $\d x$ is given by $(\sqrt{\det g})\d\Lc_d$, where $\Lc_d$ is the Lebesgue measure on $\RR^d$), and $\nabla_g$ is the Riemannian gradient (see~\eqref{eq:Gradient-Beltrami} for a definition).
Note that, for $\mathsf{V}$ fixed, the norms $\norm{\cdot}_{H^1(M)} = \norm{\cdot}_{H^1_1(M)}$ and $\norm{\cdot}_{H^1_{\mathsf{V}}(M)}$ are equivalent since we assume $\mathsf{V} > 0$.
Note furthermore that we have the conservation of the energy
\begin{equation}\label{eq:WP}
\forall t\in (0, T), \quad \Escr_{\mathsf{V}}\p{u(t), \partial_t u(t)} = \Escr_{\mathsf{V}}\p{u_0, u_1}.
\end{equation}
Moreover, as a result of~\cite{Dolecki77observation} and~\cite{lions1988controlabilite} (see Proposition~\ref{pr:dual:HUM}), if~\eqref{eq:Wave} satisfies the \emph{observability inequality}~\eqref{eq:Obs} with \emph{observability cost} $\Cobs(\mathsf{V})\in (0, \infty)$, then, for all initial data $(v_0, v_1)\in H^1(M)\times L^2(M)$, there is a \emph{control function} $f_{(v_0, v_1)}\in L^2\p{\p{0, T}\times M}$ of~\eqref{eq:Control:Wave} such that
\begin{equation*}
\p{v, \partial_tv}_{|t=T} = \p{0, 0}
\quad
\text{and}
\quad
\norm{f_{(v_0, v_1)}}_{L^2\p{\p{0, T}\times M}}\leq \sqrt{\Cobs\p{\mathsf{V}}}\norm{\p{v_0, v_1}}_{H^{1}_{\mathsf{V}}(M)\times L^2(M)},
\end{equation*}
where $v$ is the $\Cscr^0\p{\p{0, T}, H^1\p{M}}\cap\Cscr^1\p{\p{0, T}, L^2\p{M}}$-solution of~\eqref{eq:Control:Wave} with potential $\mathsf{V} = \lmbd V$ and \emph{control function} $f^\lmbd_{(v_0, v_1)}$.
Hence, $\Cobs\p{\mathsf{V}}$ also estimates the cost of the control.

According to~\cite{Rauch1974Exponential, bardos1992sharp}, if a pair $(\omega , T)$ satisfies the \emph{geometric control condition}~\eqref{GCC} (roughly speaking: all geodesics of length longer than $T$ intersect $\omega := \set{b_\omega > 0}$, see Subsection~\ref{subsection:main:results} for more details), there exists such an \emph{observability cost} $\Cobs(\mathsf{V})\in(0, \infty)$.
Using side-wise energy estimates in one-dimensional case, Zuazua proved in~\cite{zuazua1993exact} that one can take 
\begin{equation}\label{eq:Zuazua}
\Cobs(\mathsf{V}) \leq C\exp\p*{{C\norm{\mathsf{V}}_{L^\infty((0, T)\times M)}^{\frac{1}{2}}}}
\end{equation}
for any time-dependent bounded potential $\mathsf{V}\in L^{\infty}((0, T)\times M)$ with $C > 0$ independent of the choice of $\mathsf{V}$.
Then, in a subset $M\subset \RR^d$ and under a stronger multiplier-type condition on $(\omega, T)$, Duyckaerts--Zhang--Zuazua proved in~\cite[equation (1.12) and Theorem 2.2]{DUYCKAERTS20081} that one can take
\begin{equation*}
\Cobs(\mathsf{V})\leq C\exp\p{C\norm{\mathsf{V}}^{2/3}_{L^\infty((0, T)\times M)}}
\end{equation*}
for any time-dependent bounded potential $\mathsf{V}\in L^{\infty}((0, T)\times M)$ with $C > 0$ independent of the choice of $\mathsf{V}$.

Dehman--Ervedoza proved in~\cite{dehman2014dependence} that, under the sole~\eqref{GCC}, the dependence in $\mathsf{V}$ of the \emph{observability cost} can be replaced by a dependence in $\norm{\mathsf{V}}_{L^\infty(M)}$.
Then, still under the sole~\eqref{GCC}, for manifolds with or without boundary and time-independent real-valued potential $\mathsf{V}\in L^\infty(M)$ satisfying $-\Delta_g + \mathsf{V} \geq \delta$ for some $\delta > 0$: Laurent--L{\'e}autaud proved in~\cite[Theorem 1.5]{laurent2016uniform} (using estimates of~\cite{laurent15quantitative} in a key fashion), that one can take an \emph{observability cost}
\begin{equation}\label{eq:Laurent:Leautaud}
\Cobs (\mathsf{V})\leq C \exp\p{C\norm{\mathsf{V}}_{L^\infty(M)}},
\end{equation}
where $C > 0$ is independent of the choice of $\mathsf{V}$.

In this paper, given a fixed smooth and time-independent positive potential $V\in\smooth\p{M, (0, 1]}$, we study the observability of~\eqref{eq:Control:Wave} with $\mathsf{V} = \lmbd V$ and investigate the dependence of the cost $\Cobs(\lmbd V)$ on $\lmbd$.
We introduce a geometric control condition~\eqref{GCC+}, depending on the normalized potential $V$, and we prove that~\eqref{GCC+} is necessary and sufficient for the existence of a uniform \emph{observability cost} for~\eqref{eq:Wave} with potential $\lmbd V$, uniformly in $\lmbd\in [1, \infty)$.
We also provide a lower bound for $\Cobs(\lmbd V)$ in situations where~\eqref{GCC+} fails.
\subsection{Main results}\label{subsection:main:results}
In this article, we consider the situation in which $\mathsf{V} = \lmbd V$, where
\begin{equation}\label{df:potential:fixed:V}
V\in\smooth\p{M, \RR_+^*}~\text{is a fixed smooth positive potential with}~\norm{V}_{L^\infty(M)} = 1
\end{equation}
and $\lmbd\geq 1$.
We study the dependence of the \emph{observability cost} in~\eqref{eq:Obs} as a function of $\lmbd$ only.
We introduce a geometric control condition on $(\omega, T)$, depending only on $V$, that characterizes when a uniform \emph{observability inequality} holds with respect to $\lmbd$.
\subsubsection{Uniform Controllability and~\eqref{GCC+}}
Let $(\Mc, \mathsf{g})$ be a smooth and connected $\mathsf{d}$-dimensional Riemannian manifold, which can be $(M, g)$ or $(\RR\times M, {\d t^2 + g})$.
We denote by $\tang \Mc$ the tangent bundle, $\cotang \Mc$ the cotangent bundle of $\Mc$, and define the unit cosphere bundle as
\begin{equation}\label{df:cosphere}
\cosphere M := \set{(x, \xi)\in\cotang M, ~ \abs{\xi}_{g} = 1}~\text{or}~\cosphere\p{\RR\times M} := \set{(t, x, \tau, \xi)\in\cotang \p{\RR\times M}, ~\tau^2 + \abs{\xi}_{g}^2 = 2}.
\end{equation}
We also use the notion of \emph{fiber radially compactified} cotangent bundle, denoted $\cotanginf \Mc$ --- see~\cite[Section E.1.3]{dyatlov2019mathematical} for details.
\begin{remark}
Using~\eqref{df:cosphere}, for all $(t, x, \tau, \xi)\in\cotang \p{\RR\times M}$, we have the equivalence
\begin{equation*}
(t, x, \tau, \xi)\in\cosphere\p{\RR\times M}~\text{and}~\tau^2 = \abs{\xi}_g^2
\iff
\tau^2 = 1~\text{and}~(x, \xi)\in\cosphere M.
\end{equation*}
Hence, $\set{(t,x, \tau, \xi)\in\cosphere\p{\RR\times M},~\tau^2 = \abs{\xi}_g^2} \simeq \set{(t, \tau)\in\cotang\hspace{1 pt}\RR, ~ \tau^2 = 1}\times \cosphere M$.
\end{remark}
\paragraph{Geometric control conditions.}
Let $\pi_\Mc\colon (z, \zeta)\in\cotang \Mc\mapsto z\in \Mc$ be the canonical projection and, given an Hamiltonian $k\in \smooth(\cotang\Mc)$ or $\smooth(\cosphere \Mc)$, we denote by $H_k$ the Hamiltonian vector field of $k$ --- see~\eqref{eq:Hamiltonian:Vector:Field} for a definition and, if $k$ is a symbol (defined in Section~\ref{section:geometric:pseudo:preliminaries}), see ~\cite[Proposition E.5]{dyatlov2019mathematical} for its extension smooth vector field on $\cotanginf \Mc$ tangent to $\partial\cotanginf \Mc$.
We also denote by $\varphi^k$ the Hamiltonian flow associated to $H_k$, \ie
\begin{equation}\label{df:Hamiltonian:Flow}
\frac{\d}{\d t}\varphi^k_t(\rho) = H_k\p{\varphi^k_t(\rho)}, \quad t\in\RR,~ \varphi^k_0(\rho) = \rho\in\cotang \Mc.
\end{equation}
Recall that for all $(z_0, \zeta_0)\in\cosphere\Mc$ and all $t\in\RR$ we have $\varphi^{\abs{\zeta}_{\mathsf{g}}}_t(z_0, \zeta_0)\in\cosphere \Mc$ (see Remark~\ref{rk:preserved:along:flow}) and that $\varphi^{\abs{\zeta}_{\mathsf{g}}}$ is called the \emph{homogeneous geodesic flow}.
To define geometric control conditions, we let
\begin{equation}\label{df:Sqrt:Sym}
p (x, \xi) := \abs{\xi}_g^2 + V(x)
\quad 
\text{and}
\quad
\sqrtsymbol(x, \xi) := \sqrt{\abs{\xi}_g^2 + V(x)}, \quad (x, \xi)\in\cotang M.
\end{equation}
Then, for a time $T > 0$ and an open subset $\omega$ of $M$, we say that:
\begin{itemize}
\item The pair $(\omega, T)$ satisfies the geometric control condition~\eqref{GCC}, if
\begin{equation}\tag{GCC}\label{GCC}
\forall \rho\in\cosphere M,~\exists t\in \p{0, T},\quad \pi_M\circ\varphi^{\abs{\xi}_g}_{t}(\rho)\in \omega.
\end{equation}
\item The pair $(\omega, T)$ satisfies the geometric control condition with respect to potential~\eqref{GCCV}, if
\begin{equation}\tag{GCC$\null_{V}$}\label{GCCV}
\forall \rho\in{\cotang M},~\exists t\in (0, T),\quad \pi_M\circ\varphi^{\sqrtsymbol}_{t}(\rho)\in \omega.
\end{equation}
\item The pair $(\omega, T)$ satisfies the geometric control condition with respect to potential~\eqref{GCC+}, if it satisfies~\eqref{GCC} and~\eqref{GCCV}.
In other words, if $(\omega, T)$ satisfies 
\begin{equation}\tag{GCC$\null_{V}^+$}\label{GCC+}
\forall \rho\in{\cotanginf M},~\exists t\in (0, T),\quad \pi_M\circ\varphi^{\sqrtsymbol}_{t}(\rho)\in \omega.
\end{equation}
\end{itemize}
As stated in the following lemma --- proved in Subsection~\ref{subsection:GCCV:GCC} ---~\eqref{GCCV} almost implies~\eqref{GCC}.
\begin{lemma}\label{lem:GCC+:implies:GCC}
If $(\omega, T)$ satisfies~\eqref{GCCV}, then $(\omega_\delta, T + \epsilon)$ satisfies~\eqref{GCC} for any $\epsilon, \delta > 0$ with $\omega_\delta := \set{x\in M,~ \dist_g(x, \omega) < \delta}$.
\end{lemma}
\paragraph{Uniform observability.}
The main result of this article is the equivalence between~\eqref{GCC+} and uniform observability~\eqref{eq:UO}.
The proof is given in Sections~\ref{section:CS} and~\ref{section:CN}.
In Section~\ref{section:CS} we prove that~\eqref{GCC+} is a sufficient condition to satisfy~\eqref{eq:UO} and in Section~\ref{section:CN} that~\eqref{GCC+} is a necessary condition to satisfy~\eqref{eq:UO}.
\begin{theorem}\label{Th:GCC:UO}
Assume $(M, g)$ is a smooth, compact, connected Riemannian manifold without boundary,
$T > 0$ is a time,
$V$ is a smooth potential satisfying~\eqref{df:potential:fixed:V},
$b_\omega$ is a continuous nonnegative function on $M$ and $\omega := \set{b_\omega > 0}$.
The following two assertions are equivalent.
\begin{enumerate}
\item The pair $(\omega, T)$ satisfies~\eqref{GCC+}.
\item There is an \emph{observability cost} $\Cobs \in(0, \infty)$ such that
\begin{equation}\label{eq:UO}\tag{UO}
\forall \lmbd\geq 1,~ \forall (u_0, u_1)\in H^1(M)\times L^2(M), \quad\Escr_{\lmbd V}{(u_0, u_1)}\leq \Cobs\int_0^T\norm{b_\omega \partial_tu_{\lmbd}(t)}^2_{L^2(M)}\d t,
\end{equation}
where $u_\lmbd$ is the unique $\Cscr^0(\RR, H^1( M))\cap\Cscr^1(\RR, L^2( M))$-solution of~\eqref{eq:Wave} with potential $\mathsf{V} = \lmbd V$ and initial data $(u_0, u_1)$.
\end{enumerate}
\end{theorem}
\paragraph{Uniform controllability.}
In turn, Theorem~\ref{Th:GCC:UO} implies the following uniform controllability statement: if~\eqref{GCC+} holds, then for all $\lmbd \geq 1$ and all initial data $(v_0, v_1)\in H^1(M)\times L^2(M)$ there exists a \emph{control function} $f^\lmbd_{(v_0, v_1)}\in L^2\p{(0, T)\times M}$ of~\eqref{eq:Control:Wave} such that
\begin{equation*}
\norm{f^\lmbd_{(v_0, v_1)}}_{L^2((0, T)\times \omega)} \leq \sqrt{\Cobs}\norm{\p{v_0, v_1}}_{H^1_{\mathsf{V}}(M)\times L^2(M)}
\quad
\text{and}
\quad
\p{v, \partial_t v}_{|t  = T} = (0, 0),
\end{equation*}
where $\Cobs$ is given by Theorem~\ref{Th:GCC:UO} and $v$ is the $\Cscr^0\p{\p{0, T}, H^1\p{M}}\cap\Cscr^1\p{\p{0, T}, L^2\p{M}}$-solution of~\eqref{eq:Control:Wave} with potential $\mathsf{V} = \lmbd V$ and \emph{control function} $f^\lmbd_{(v_0, v_1)}$.
The (classical) duality argument is recalled in Appendix~\ref{section:HUM}.
\subsubsection{Examples}
Here, we provide examples of non-trivial pairs that satisfy~\eqref{GCC+}, compare~\eqref{GCC} and~\eqref{GCCV}, and conclude by giving a blow-up rate for the \emph{optimal observability cost} (defined later on) in some examples where~\eqref{GCC} holds but~\eqref{GCCV} does not.
\paragraph{Examples of pairs satisfying~\eqref{GCC+}.}
According to the following Propositions
---
\ref{pr:1D:circle} (see Figure~\ref{fig:1D:circle} for an illustration),~\ref{pr:2D:torus} (see Figure~\ref{eg:2D:torus} for a concrete example) and~\ref{pr:2D:sphere} (see Figure~\ref{eg:2D:sphere} for an illustration), see Section~\ref{section:Some examples} for proofs
---
we provide examples of Riemannian manifolds $(M, g)$ and strictly positive bounded smooth potentials $V$ such that there are non-trivial pairs $(\omega, T)$ satisfying~\eqref{GCC+}.
\begin{proposition}\label{pr:1D:circle}
Let $M$ be the circle $\TT^1 = \RR/\ZZ$ equipped with the Euclidean metric and $V$ as~\eqref{df:potential:fixed:V}.
For an open set $\omega\subset \TT^1$, the following two assertions are equivalent:
\begin{enumerate}
\item There is $T>0$ such that~\ref{GCC+}$(\omega, T)$ is satisfied.
\item $\omega$ contains all critical points of $V$.
\end{enumerate}
\end{proposition}
\begin{figure}
\begin{center}
\begin{tikzpicture}
\def\e{0.025}
\begin{axis}[
width=12cm, height=7cm,
axis lines=middle,
ylabel={$V(x)$},
xmin=0, xmax=1.05,
ymin=0, ymax=1.05,
xtick={0, 1},
ytick={0, 1},
domain=0:1.01,
samples=400,
every axis y label/.style={at={(ticklabel* cs:1.05)}},
clip=false
]
\node[above] at (axis cs:1.1,0) {$x\in\TT^1 \simeq [0, 1)$} ;
\node[below left] at (axis cs:0,0) {$0$} ;
\node[below] at (axis cs:0.220382,-0.02) {$x_2$} ;
\fill[gray!30] (axis cs:{0.220382-2*\e}, 0) -- (axis cs:{0.220382 - 2*\e},1) -- (axis cs:{0.220382 + 2*\e},1) -- (axis cs:{0.220382 + 2*\e},0) -- cycle;
\draw[dashed] (axis cs:0.220382, 0) -- (axis cs:0.220382,1) ;
\node[below] at (axis cs:0.769477,-0.02) {$x_4$} ;
\fill[gray!30] (axis cs:{0.769477-3*\e}, 0) -- (axis cs:{0.769477 - 3*\e},1) -- (axis cs:{0.769477 + 3*\e},1) -- (axis cs:{0.769477 + 3*\e},0) -- cycle;
\draw[dashed] (axis cs:0.769477, 0) -- (axis cs:0.769477,1) ;
\node[below] at (axis cs:0.0382971,-0.02) {$x_1$} ;
\fill[gray!30] (axis cs:{0.0382971-\e}, 0) -- (axis cs:{0.0382971 - \e},1) -- (axis cs:{0.0382971 + \e},1) -- (axis cs:{0.0382971 + \e},0) -- cycle;
\draw[dashed] (axis cs:0.0382971, 0) -- (axis cs:0.0382971,1) ;
\node[below] at (axis cs:0.471845,-0.02) {$x_3$} ;
\fill[gray!30] (axis cs:{0.471845-0.5*\e}, 0) -- (axis cs:{0.471845 - 0.5*\e},1) -- (axis cs:{0.471845 + 0.5*\e},1) -- (axis cs:{0.471845 + 0.5*\e},0) -- cycle;
\draw[dashed] (axis cs:0.471845, 0) -- (axis cs:0.471845,1) ;
\draw (axis cs:0, 0) -- (axis cs:1,0) ;
\addplot [thick] { (2.1 + sin(deg(2*pi*(x - 0.1))) + cos(deg(4*pi*(x + 0.5)))) / 3.759045};
\node[left] at (axis cs:0.769477, 0.5) {$\omega$};
\draw[dashed, gray!50] (axis cs:1, 0) -- (axis cs:1,1) ;
\end{axis}
\end{tikzpicture}
\end{center}
\caption{Illustration, in the one-dimensional case $M = \TT^1$ equipped with the Euclidean metric, of a pair $(\omega, V)$ such that there is a time $T > 0$ for which~\ref{GCC+}$(\omega, T)$ holds.
This case is already treated for Schrödinger quasimodes in~\cite{laurent23uniform1D} with a boundary and less regularity assumptions on $V$.}
\label{fig:1D:circle}
\end{figure}
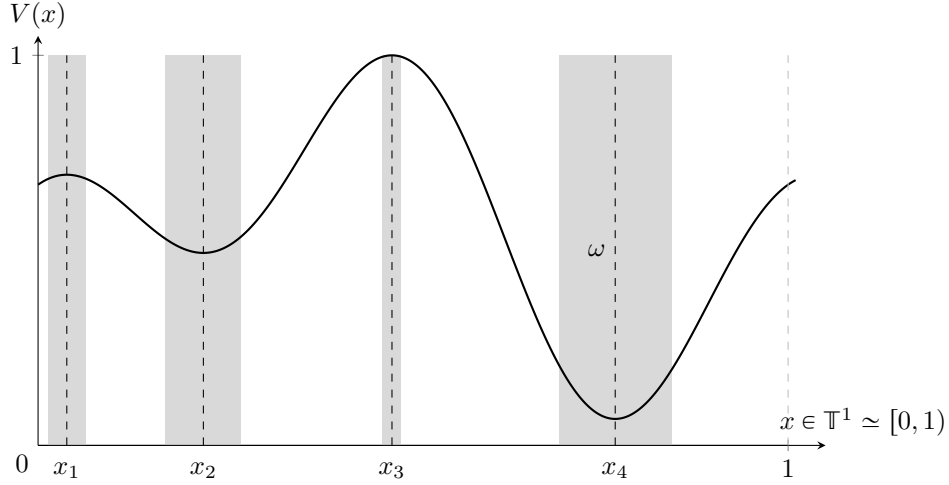
\begin{proposition}\label{pr:2D:torus}
Let $M$ be the two-dimensional torus $\TT^2 = \RR^2/\ZZ^2\simeq [-\frac{1}{2}, \frac{1}{2})$ equipped with the Euclidean metric.
For a radius $r\in (0, \frac{1}{2})$, let $\Ve \in\smooth(\RR, (0, 1])$ be an increasing function on $[0, r^2]$ and 1 near $[r^2, 1]$.
Assume that $V\in\smooth\p{\TT^2}$ satisfies
\begin{equation*}
V(x) = \Ve(\abs{x}^2),\quad x\in\TT^2
\end{equation*}
and $\omega$ is an open subset of $\TT^2$ containing $\set{x\in\TT^2,~ \abs{x}^2 > r}\cup [0, r]\times \set{0}$.
Then, there is a time $T > 0$ satisfying~\ref{GCC+}$(\omega, T)$.
\end{proposition}
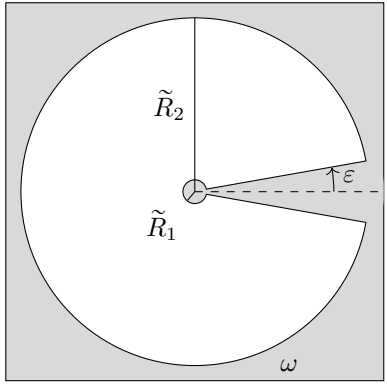
\begin{figure}
\begin{center}
\begin{tikzpicture}
    \def\w{2.5}
    \def\R{2.3}
    \def\r{0.15}
    \def\a{10}
    \draw[fill = gray!30] (-\w, -\w) -- (\w, -\w) -- (\w, \w) -- (-\w, \w) -- cycle;
    \draw[fill=white] (0,0) circle (\R);
    \draw[fill=gray!30, thick] (0,0) circle (\r);
    \draw[thick]
    (0,0) -- ({\R*cos(-\a)}, {\R*sin(-\a)});
    \draw[thick]
    (0,0) -- ({\R*cos(\a)}, {\R*sin(\a)});
    \fill[gray!30] (0,0) -- ({\w*cos(-\a)}, {\w*sin(-\a)}) arc[start angle=-\a, end angle=\a, radius=\w] -- cycle;
    \fill[gray!30] (0,0) circle (\r);
    \draw[dashed] (0,0) -- (\w, 0);
    \draw[->] ({\R *4/5}, 0) arc[start angle=0, end angle=\a, radius={\R*4/5}];
    \draw ({\R * 4/5}, 0) node[above right] {$\varepsilon$};
    \draw (0, {\R/2}) node[left] {$\tilde{R}_2$};
    \draw (0, 0) -- (0, \R);
    \draw ({\r*cos(230)}, {\r*sin(230)}) node[below left] {$\tilde{R}_1$};
    \draw (0, 0) -- ({\r*cos(230)}, {\r*sin(230)});
    \draw ({\w/2}, -\w) node[above] {$\omega$};
\end{tikzpicture}
\end{center}
\caption{Illustration of Proposition~\ref{pr:2D:torus} on $M = \TT^2$, equipped with the Euclidean metric, of an open subset $\omega$, such that there is $T$ for which the pair $(\omega, T)$ satisfies~\eqref{GCC+}.
In this example, we set two radii $0 < \tilde{R}_1< \tilde{R}_2 < \frac{1}{2}$ and an angle $\varepsilon \in (0, \pi)$.}\label{eg:2D:torus}
\end{figure}
\begin{proposition}\label{pr:2D:sphere}
Let $M$ be the two-dimensional sphere $\SS^2$ equipped with its canonical metric induced by the Euclidean metric of $\RR^3$, $S$ and $N$ be its south and north pole, and $\Ve \in\smooth([0, \pi^2], (0, 1])$ be an increasing function.
Then, for all potential
\begin{equation*}
V\colon x\in\SS^2\mapsto \Ve(\dist_g(x, S)^2)
\end{equation*}
and all connected open subset of $\SS^2$ containing $\set{S, N}$, denoted $\omega$, there exists a time $T>0$ such that $(\omega, T)$ satisfies~\eqref{GCC+}.
\end{proposition}
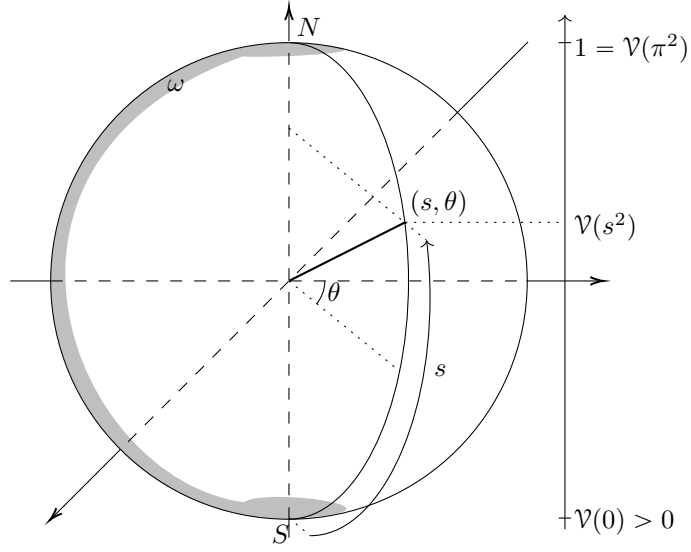
\begin{figure}
\begin{center}
\begin{tikzpicture}[x=0.75pt,y=0.75pt,yscale=-0.6,xscale=0.6]
\draw  [draw opacity=0][fill=gray!50] (266.43,58) .. controls (381.72,64.98) and (248.72,71.48) .. (227.72,69.48) .. controls (-17.78,157.98) and (83.72,422.48) .. (228.72,442.48) .. controls (246.22,428.98) and (386.22,448.98) .. (266.36,456.54) .. controls (139.58,455.11) and (66.58,351.44) .. (67.12,257.31) .. controls (67.66,163.17) and (136.91,60.91) .. (266.43,58) -- cycle ;
\draw    (266.43,56.72) -- (266.43,30.06) ;
\draw [shift={(266.43,28.06)}, rotate = 90] [line width=0.75]    (10.93,-3.29) .. controls (6.95,-1.4) and (3.31,-0.3) .. (0,0) .. controls (3.31,0.3) and (6.95,1.4) .. (10.93,3.29)   ;
\draw    (465.09,257.31) -- (527.43,257.31) ;
\draw [shift={(529.43,257.31)}, rotate = 180] [line width=0.75]    (10.93,-3.29) .. controls (6.95,-1.4) and (3.31,-0.3) .. (0,0) .. controls (3.31,0.3) and (6.95,1.4) .. (10.93,3.29)   ;
\draw    (33.39,257.31) -- (67.12,257.31) ;
\draw  [dash pattern={on 4.5pt off 4.5pt}]  (67.12,257.31) -- (463.39,257.31) ;
\draw    (266.43,472.06) -- (266.43,456.62) ;
\draw  [dash pattern={on 4.5pt off 4.5pt}]  (266.43,456.62) -- (266.43,56.72) ;
\draw  [dash pattern={on 0.84pt off 2.51pt}]  (266.43,257.31) -- (359.36,332.96) ;
\draw  [dash pattern={on 4.5pt off 4.5pt}]  (407.54,116.2) -- (125.31,398.42) ;
\draw    (466.43,57.31) -- (407.54,116.2) ;
\draw    (125.31,398.42) -- (67.84,455.89) ;
\draw [shift={(66.43,457.31)}, rotate = 315] [line width=0.75]    (10.93,-3.29) .. controls (6.95,-1.4) and (3.31,-0.3) .. (0,0) .. controls (3.31,0.3) and (6.95,1.4) .. (10.93,3.29)   ;
\draw  [draw opacity=0] (284.36,470.4) .. controls (284.38,470.4) and (284.4,470.4) .. (284.43,470.4) .. controls (339.65,470.4) and (384.43,381.2) .. (384.43,271.17) .. controls (384.43,254.64) and (383.41,238.57) .. (381.51,223.21) -- (284.43,271.17) -- cycle ;
\draw[->]   (284.36,470.4) .. controls (284.38,470.4) and (284.4,470.4) .. (284.43,470.4) .. controls (339.65,470.4) and (384.43,381.2) .. (384.43,271.17) .. controls (384.43,254.64) and (383.41,238.57) .. (381.51,223.21) ;  
\draw  [dash pattern={on 0.84pt off 2.51pt}]  (266.43,455.64) -- (284.36,470.4) ;
\draw  [draw opacity=0] (296.42,256.93) .. controls (296.42,257.05) and (296.43,257.18) .. (296.43,257.31) .. controls (296.43,264.65) and (293.79,271.37) .. (289.41,276.58) -- (266.43,257.31) -- cycle ; \draw    (296.42,256.93) .. controls (296.42,257.05) and (296.43,257.18) .. (296.43,257.31) .. controls (296.43,264.65) and (293.79,271.37) .. (289.41,276.58) ;
\draw   (67.12,257.31) .. controls (67.12,147.23) and (156.35,58) .. (266.43,58) .. controls (376.5,58) and (465.73,147.23) .. (465.73,257.31) .. controls (465.73,367.38) and (376.5,456.62) .. (266.43,456.62) .. controls (156.35,456.62) and (67.12,367.38) .. (67.12,257.31) -- cycle ;
\draw  [draw opacity=0] (266.36,456.54) .. controls (266.38,456.54) and (266.4,456.54) .. (266.43,456.54) .. controls (321.65,456.54) and (366.43,367.34) .. (366.43,257.31) .. controls (366.43,148.25) and (322.44,59.66) .. (267.89,58.1) -- (266.43,257.31) -- cycle ; \draw   (266.36,456.54) .. controls (266.38,456.54) and (266.4,456.54) .. (266.43,456.54) .. controls (321.65,456.54) and (366.43,367.34) .. (366.43,257.31) .. controls (366.43,148.25) and (322.44,59.66) .. (267.89,58.1) ;  
\draw[->] (497.26,464.98) -- (497.26,34.81);
\draw    (491.9,58) -- (502.4,58) ;
\draw    (491.9,455.64) -- (502.4,455.64) ;
\draw  [dash pattern={on 0.84pt off 2.51pt}]  (363.63,208.28) -- (497.26,208.28) ;
\draw  [dash pattern={on 0.84pt off 2.51pt}]  (363.18,208.28) -- (381.11,223.04) ;
\draw[dash pattern={on 0.84pt off 2.51pt}] (266.43,129.52) -- (363.18,208.28);
\draw[thick] (266.43,257.31) -- (363.63,208.28);
\draw (385.83,325.23) node [anchor=north west][inner sep=0.75pt]    {$s$};
\draw (296.49,256.23) node [anchor=north west][inner sep=0.75pt]    {$\theta $};
\draw (162.67,87) node [anchor=north west][inner sep=0.75pt]    {$\omega $};
\draw (365.33,179.4) node [anchor=north west][inner sep=0.75pt]    {$( s, \theta )$};
\draw (503, 445.54) node [anchor=north west][inner sep=0.75pt]    {$\Ve(0) > 0$};
\draw (503, 48.9) node [anchor=north west][inner sep=0.75pt]    {$1=\Ve(\pi^2)$};
\draw (503, 198) node [anchor=north west][inner sep=0.75pt]    {$\Ve(s^2)$};
\draw (271,35.53) node [anchor=north west][inner sep=0.75pt]    {$N$};
\draw (250,459.53) node [anchor=north west][inner sep=0.75pt]    {$S$};
\end{tikzpicture}
\end{center}
\caption{Illustration of Proposition~\ref{pr:2D:sphere} on $M = \SS^2$, equipped with the metric induced by $\RR^3$, of an open subset $\omega$, such that there is $T$  for which the pair $(\omega, T)$ satisfies~\eqref{GCC+}.
}\label{eg:2D:sphere}
\end{figure}
\paragraph{Comparing~\eqref{GCC} and~\eqref{GCCV}.}
On the one hand, even if Lemma~\ref{lem:GCC+:implies:GCC} states (for any pair $(\omega, T)$ and any $\delta, \epsilon > 0$) that~\ref{GCCV}$(\omega, T)$ implies~\ref{GCC}$(\omega_\delta, T + \varepsilon)$, a pair can satisfy~\eqref{GCCV} but not~\eqref{GCC}, see Lemma~\ref{lem:GCCV-and-non-GCC} and Figure~\ref{fig:GCCV-and-non-GCC} for a concrete example.
On the other hand, with the following Lemma~\ref{lem:GCC:and:no:GCCV}, we now provide examples where~\eqref{GCC} holds and~\eqref{GCCV} does not.
The proofs are provided in Subsection~\ref{subsection:GCCV:GCC}.
\begin{lemma}\label{lem:GCCV-and-non-GCC}
Let $M$ be the two-dimensional torus $\TT^2 = \TT^1_{x_1}\times \TT^1_{x_2}\simeq [-\frac{1}{2}, \frac{1}{2})^2$ equipped with the Euclidean metric.
Assume $V$ satisfies~\eqref{df:potential:fixed:V} and that there is $c\in\TT^1$ such that
\begin{equation}\label{eq:3}
\forall x_2\in \TT^1, \quad \partial_{x_1}V (c, x_2) < 0,
\end{equation}
and let  $\omega := \set{(x_1, x_2)\in\TT^2,~x_1\neq c}$.
Then, for all $T > 0$, \ref{GCCV}$(\omega , T)$ holds but~\ref{GCC}$(\omega , T)$ does not.
\end{lemma}
\begin{figure}
\begin{center}
\tikzset{
pattern size/.store in=\mcSize, 
pattern size = 5pt,
pattern thickness/.store in=\mcThickness, 
pattern thickness = 0.3pt,
pattern radius/.store in=\mcRadius, 
pattern radius = 1pt}
\makeatletter
\pgfutil@ifundefined{pgf@pattern@name@_98xmizcpf}{
\pgfdeclarepatternformonly[\mcThickness,\mcSize]{_98xmizcpf}
{\pgfqpoint{0pt}{0pt}}
{\pgfpoint{\mcSize+\mcThickness}{\mcSize+\mcThickness}}
{\pgfpoint{\mcSize}{\mcSize}}
{
\pgfsetcolor{\tikz@pattern@color}
\pgfsetlinewidth{\mcThickness}
\pgfpathmoveto{\pgfqpoint{0pt}{0pt}}
\pgfpathlineto{\pgfpoint{\mcSize+\mcThickness}{\mcSize+\mcThickness}}
\pgfusepath{stroke}
}}
\makeatother
\begin{tikzpicture}
  \begin{axis}[
    width=8.5cm, height=7cm,
    axis lines=middle,
    xlabel={$x_1$},
    ylabel={$V(x_1)$},
    xmin=0, xmax=1.05,
    ymin=0, ymax=1.05,
    xtick={0, 1},
    ytick={0, 1},
    domain=0:1.05,
    samples=400,
    grid=both,
    grid style={dashed,gray!30},
    every axis y label/.style={at={(ticklabel* cs:1.05)}},
    clip=false
  ]
    \addplot [thick] { (1.1 + sin(deg(2*pi*(x - 0.1)))) / 2.1 };
    \draw[dashed] (axis cs:0.4,0) -- (axis cs:0.4,0.9766935791881683) ;
    \draw[dotted] (axis cs:0,0.047619047619047616) -- (axis cs:0.85,0.047619047619047616) ;
    \draw[dotted] (axis cs:0,0.2439117846226318) -- (axis cs:1,0.2439117846226318) ;
    \node[below] at (axis cs:0.4,-0.02) {$c$} ;
    \node[below] at (axis cs:-0.1, 0.1) {$\min V$} ;
    \node[below] at (axis cs:-0.1, 0.3) {$V(0)$} ;
    \node[below] at (axis cs:0,0) {$0$} ;
  \end{axis}
\end{tikzpicture}
\hspace{1cm}
\begin{tikzpicture}[x=0.75pt,y=0.75pt,yscale=-0.6,xscale=0.6]
\draw (161.67,3.17) -- (509.39,3.17) -- (509.39,350.83) -- (161.67,350.83) -- cycle ;
\draw [color={rgb, 255:red, 0; green, 0; blue, 0 }, draw opacity = 1]   (279.89,3.17) -- (279.89,350.83) ;
\draw[->] [line width=1.5] (279.89,336) -- (279.89,32);
\draw [shift={(279.89,336)}, rotate = 90] [color={rgb, 255:red, 0; green, 0; blue, 0 }  ][fill={rgb, 255:red, 0; green, 0; blue, 0 }  ] (0, 0) circle [x radius= 4.36, y radius= 4.36]   ;
\draw[->][line width=1.5] (279.89,336) .. controls (279.89,210) and (295,100) .. (320,32);
\draw[->] (450,330) -- (480,330);
\draw (480,330) node [anchor=west][inner sep=0.75pt] {$x_1$};
\draw[->] (450,330) -- (450,300);
\draw (450,300) node [anchor=south][inner sep=0.75pt] {$x_2$};
\draw (190, 320) node [anchor=north west][inner sep=0.75pt]    {$\pi_{\TT^2}(\eta_0)$};
\draw (512.9,324.4) node [anchor=north west][inner sep=0.75pt]    {$\TT^{2}$};
\draw (230,360) node [anchor=north west][inner sep=0.75pt]  [color={rgb, 255:red, 0; green, 0; blue, 0 }  ,opacity=1 ]  {$\TT^2\setminus \omega = \set{x_1 = c}$};
\draw (180, 10) node [anchor=north west][inner sep=0.75pt]    {$\varphi_{t}^{\abs{\xi}}(\eta_0)$};
\draw (330,10) node [anchor=north west][inner sep=0.75pt]    {$\varphi_{t}^{\sqrtsymbol}(\eta_0)$};
\end{tikzpicture}
\end{center}
\caption{
Illustration of Lemma~\ref{lem:GCCV-and-non-GCC} on $M = \TT^2$, equipped with the Euclidean metric, of a \emph{control open} subset $\omega = \set{x_1 \neq c}$ and a potential $V$ only depending on $x_1$, with $\partial_{x_1}V(c, x_2) < 0$, such that, for any time $T > 0$,~\ref{GCCV}$(\omega , T)$ holds but~\ref{GCC}$(\omega , T)$ does not.}
\label{fig:GCCV-and-non-GCC}
\end{figure}
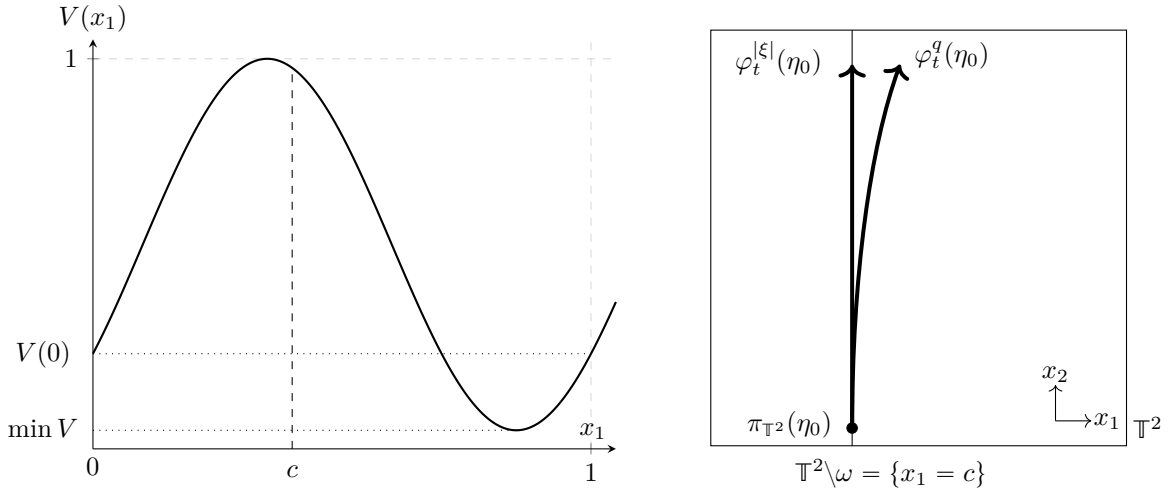
\begin{lemma}\label{lem:GCC:and:no:GCCV}
Let $\omega$ be an open subset of $M$ and $V$ be as~\eqref{df:potential:fixed:V}.
Assume that there is $x_0\in M$, a critical point of $V$, such that $x_0\notin \omega$.
Then, for any time $T > 0$,~\ref{GCCV}$(\omega, T)$ does not hold.
\end{lemma}
Consequently, for every set $\omega$ with $\barre{\omega}\neq M$, there exists a potential $V$ such that, for all $T>0$, condition~\ref{GCCV}$(\omega, T)$ does not hold.
For instance $(\TT^1, \mathrm{Eucl})$ with $\omega$ an interval such that $\barre{\omega}\subsetneq \TT^1$ or $(\TT^2, \mathrm{Eucl})$ with $\omega \supset \set{x_1 = \frac{1}{2}\text{ or }x_2 = \frac{1}{2}}$ such that $\barre{\omega}\subsetneq \TT^2$.
\paragraph{Blow-up estimates.}
In some situations where~\eqref{GCCV} does not hold, we are able to provide an explicit blow-up rate of the \emph{optimal observability cost} for~\eqref{eq:Obs} defined as
\begin{equation}\label{df:optimal:Cobs}
\Cobs^{\mathrm{opt}}(\lmbd, V, T) := 
\sup_{
\begin{subarray}{c}
(u_0, u_1)\in H^1(M)\times L^2(M)\\
(u_0, u_1)\neq (0, 0)
\end{subarray}
}
\p*{\frac{\Escr_{\lmbd V}(u_0, u_1)}{\int_0^T\norm{b_\omega \partial_t u_{\lmbd}(t)}_{L^2(M)}^2\d t}}
\in (0, \infty],
\end{equation}
where $u_{\lmbd}$ is the $\Cscr^0\p{\p{0, T}, H^1\p{M}}\cap\Cscr^1\p{\p{0, T}, L^2\p{M}}$-solution of~\eqref{eq:Wave} with potential $\mathsf{V} = \lmbd V$ and initial data $(u_0, u_1)\in H^1(M)\times L^2(M)$.
\begin{proposition}\label{pr:Blow:Up:Estimates}
Assume $(M, g)$ is a smooth, compact, connected Riemannian manifold without boundary and $b_\omega$ is a continuous nonnegative function on $M$ such that $\omega = \set{b_\omega > 0}$ has its closure $\barre{\omega}$ different from $M$.
There is a potential $V$ satisfying~\eqref{df:potential:fixed:V} such that 
\begin{equation}\label{eq:blow:up:estimate}
\exists \mathsf{C}_\star > 0,~\forall T > 0,~\forall \lmbd \geq 1, \quad
\Cobs^{\mathrm{opt}}(\lmbd,  V, T) \geq \mathsf{C}_\star T^{-1}e^{\mathsf{C}_\star \sqrt{\lmbd}},
\end{equation}
where $\Cobs^{\mathrm{opt}}$ is the \emph{optimal observability cost} defined by~\eqref{df:optimal:Cobs}.
\end{proposition}
The proof of Proposition~\ref{pr:Blow:Up:Estimates} is given in Subsection~\ref{subsection:blow-up}.
Furthermore, in the one-dimensional setting, the factor $\sqrt{\lmbd}$ coincides with $\norm{\mathsf{V}}_{L^2((0, T)\times M)}^{1/2}$ appearing in~\eqref{eq:Zuazua}, as established in~\cite{zuazua1993exact}, thereby confirming its optimal character.
Please note that, in the context of Proposition~\ref{pr:Blow:Up:Estimates}, if we assume furthermore that $(\omega, T)$ satisfy~\eqref{GCC}, we have the upper bound
\begin{equation*}
\exists T_0, \mathsf{C}^\star> 0,~\forall T > T_0,~\forall \lmbd > 0, \quad
\Cobs^{\mathrm{opt}}(\lmbd,  V, T) \leq \mathsf{C}^\star e^{\mathsf{C}^\star \lmbd},
\end{equation*}
given by~\eqref{eq:Laurent:Leautaud} from~\cite[Theorem 1.5]{laurent2016uniform}.
Thus, in this context, we have
\begin{equation*}
\exists T_0, \mathsf{C}^\star, \mathsf{C}_\star> 0,~\forall T > T_0,~\forall \lmbd \geq 1, \quad
\mathsf{C}_\star T^{-1}e^{\mathsf{C}_\star \sqrt{\lmbd}} 
\leq \Cobs^{\mathrm{opt}}(\lmbd,  V, T)
\leq \mathsf{C}^\star e^{\mathsf{C}^\star {\lmbd}},
\end{equation*}
\subsection{Sketch of proofs and plan of the paper}
The main goal of this article is to prove Theorem~\ref{Th:GCC:UO}, 
\ie to characterize when the \emph{observability inequality}~\eqref{eq:Obs} holds with an \emph{observability cost} uniform in $\lmbd\geq 1$.
\begin{itemize}
    \item On the one hand, in Section~\ref{section:CS}, we prove by \emph{reductio ad absurdum} that~\eqref{GCC+} implies the uniform observability inequality~\eqref{eq:UO}.
    To this end, we assume the existence of a sequence of scale parameters $(\lmbd_n)_n$ for which~\eqref{eq:UO} fails.
    Then two scenarios arise: either the sequence $(\lmbd_n)_n$ remains bounded, in which case the classical~\eqref{GCC} leads to a contradiction; or $(\lmbd_n)_n$ is unbounded, and we reformulate the problem as a semiclassical one with a parameter $h := 1/\sqrt{\lmbd}$.
    Then, in this second case, following the strategy of Lebeau~\cite{lebeau1996equation}, we use the time-slice of semiclassical and microlocal defect measures (defined in Section~\ref{section:microlocalization}) to obtain a contradiction.
    \item On the other hand, in Section~\ref{section:CN}, we prove that~\eqref{eq:UO} implies~\eqref{GCC+}.
    To do this, we assume the existence of an Hamiltonian curve $\varGamma$ that does not enter the control region $\omega$.
    Then we construct a sequence of solutions of~\eqref{eq:Wave} which concentrate along $\varGamma$ (thus avoiding $\omega$) while having an energy bounded from below.
    This sequence of functions thus transgresses the~\eqref{eq:UO} inequality.
\end{itemize}

To carry out this strategy, Sections~\ref{section:geometric:pseudo:preliminaries},~\ref{section:microlocalization} and~\ref{section:slice:defect:measure} collect the geometric, semiclassical and microlocal tools (and their main properties) used in the proof.
\begin{itemize}
    \item In Section~\ref{section:geometric:pseudo:preliminaries}, we introduce the geometric notations and semiclassical pseudodifferential calculus used throughout the paper: we recall Hamiltonian dynamics on the cotangent space, define symbol classes and quantization procedures, and review the basic properties of microlocal and semiclassical defect measures underlying our high-frequency analysis.
    \item In Section~\ref{section:microlocalization}, we provide a procedure to microlocalize bounded $L^2$-sequences (indexed by $h$) on the cosphere at infinity in the semiclassical limit $h\to 0$, following the idea of the proof of~\cite[Theorem 1]{gerard1991defectmeasure} for standard microlocalization of sequences weakly converging toward $0$.
    This result, given in Theorem~\ref{Th:Microlocalization}, plays a key role in the study of high-frequency oscillations $\zeta\gg h^{-1}$ and will be essential to prove Theorem~\ref{Th:GCC:UO}.
    \item In Section~\ref{section:slice:defect:measure}, we introduce the half-wave equation~\eqref{eq:HWE}. Then, we give some properties of its solutions and state a theorem (Theorem~\ref{Th:Slices}, proved in Appendix~\ref{proof:Th:Slices}) inspired by~\cite[Proposition 4.4]{LecturePG}, which allows us to disintegrate (and thus describe) their associated semiclassical and second microlocal defect measures by their time-slices.
    In other words, in this particular case, it allows us to study measures on $\cotang \hspace{1 pt}(\RR\times M)$ as families of measures on $\cotang M$ indexed by a time parameter $t\in \RR$.
\end{itemize}

In Section~\ref{section:Some examples}, we give examples of tuples $(\omega, T, V)$ that satisfy~\eqref{GCCV}.
Firstly, we give examples that satisfy it on tori or spheres of dimension 1 and 2, proving Propositions~\ref{pr:1D:circle},~\ref{pr:2D:torus} and~\ref{pr:2D:sphere}.
We also provide examples satisfying~\eqref{GCC} but not~\eqref{GCCV}, proving Lemma~\ref{lem:GCC:and:no:GCCV}, and examples satisfying~\eqref{GCCV} but not~\eqref{GCC}, proving Lemma~\ref{lem:GCCV-and-non-GCC}.
We also provide a lower bound of $\Cobs^{\mathrm{opt}}(\lmbd V)$ --- defined in~\eqref{df:optimal:Cobs} --- proving Proposition~\ref{pr:Blow:Up:Estimates}, the proof relies on Agmon estimates.

The paper ends with two appendices.
In Appendix~\ref{section:HUM} we prove Proposition~\ref{Cor:HUM}, which gives the equivalence between~\eqref{eq:Obs} and the \emph{null-controllability of~\eqref{eq:Control:Wave} at time $T$ from $\omega$} of~\eqref{eq:Control:Wave}.
In Appendix~\ref{proof:Th:Slices}, we prove Theorem~\ref{Th:Slices}.

The dependencies between sections are summarized with the following graph.
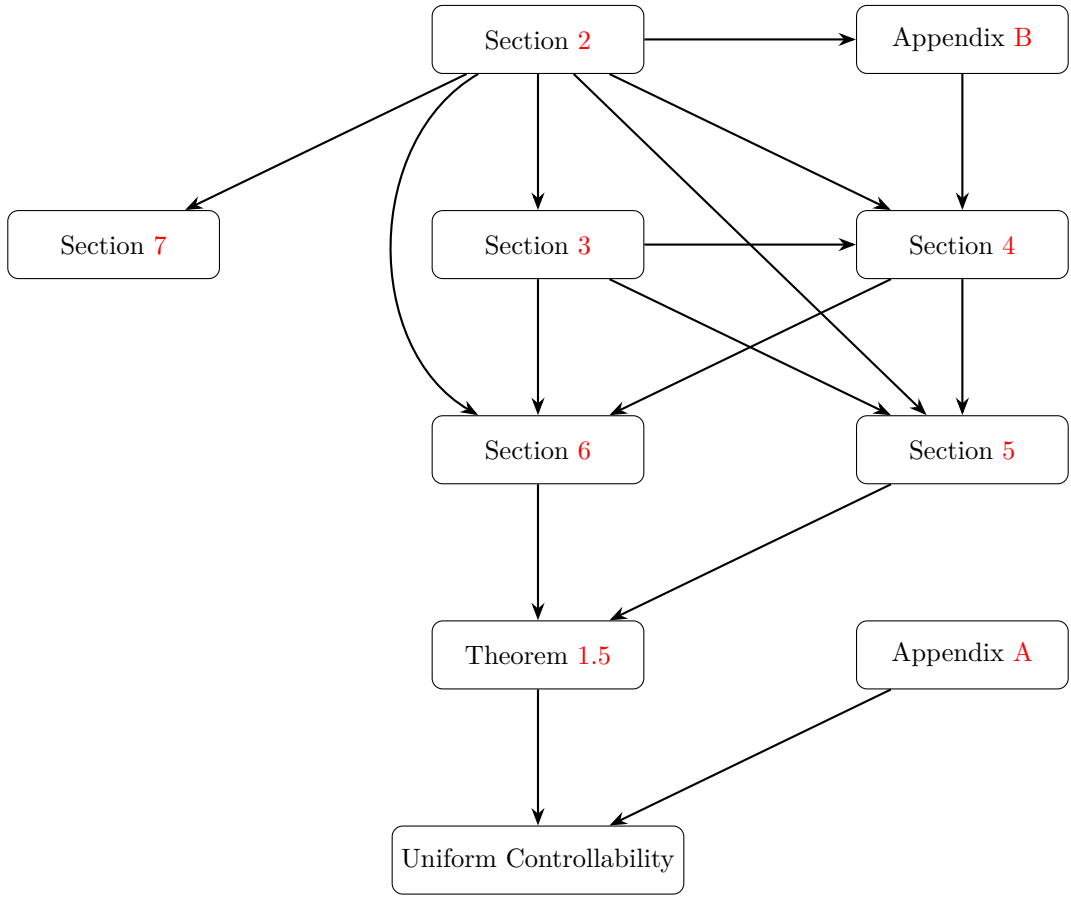
\begin{center}
\begin{tikzpicture}[
    node distance=1.8cm and 2.8cm,
    every node/.style={
        draw,
        rectangle,
        rounded corners,
        align=center,
        minimum width=2.8cm,
        minimum height=0.9cm
    },
    arrow/.style={-Stealth, thick}
]
\node (sec2) {Section~\ref{section:geometric:pseudo:preliminaries}};
\node (sec3) [below=of sec2] {Section~\ref{section:microlocalization}};
\node (sec4) [right=of sec3] {Section~\ref{section:slice:defect:measure}};
\node (sec6) [below=of sec3] {Section~\ref{section:CN}};
\node (sec5) [right=of sec6] {Section~\ref{section:CS}};
\node (sec7) [left=of sec3] {Section~\ref{section:Some examples}};

\node (appB) [right=of sec2] {Appendix~\ref{proof:Th:Slices}};
\node (thm) [below=of sec6] {Theorem~\ref{Th:GCC:UO}};
\node (appA) [right=of thm] {Appendix~\ref{section:HUM}};
\node (uniform) [below=of thm] {Uniform Controllability};
\draw[arrow] (sec2) -- (sec3);
\draw[arrow] (sec2) -- (appB);
\draw[arrow] (sec2) -- (sec4);
\draw[arrow] (sec2) -- (sec5);
\draw[arrow] (sec2) -- (sec7);
\draw[arrow, bend right = 60] (sec2) to (sec6);
\draw[arrow] (sec3) -- (sec4);
\draw[arrow] (sec3) -- (sec5);
\draw[arrow] (sec3) -- (sec6);
\draw[arrow] (sec4) -- (sec5);
\draw[arrow] (sec4) -- (sec6);
\draw[arrow] (appB) -- (sec4);
\draw[arrow] (sec5) -- (thm);
\draw[arrow] (sec6) -- (thm);
\draw[arrow] (thm) -- (uniform);
\draw[arrow] (appA) -- (uniform);
\end{tikzpicture}
\end{center}

\medskip\noindent
\textbf{\emph{Acknowledgments.}} I am very grateful to Matthieu Léautaud for many discussions on the topics of this article, for thoughtful advice, and for a careful reading of a preliminary version of this article.
\section{Geometric and pseudodifferential calculus preliminaries}\label{section:geometric:pseudo:preliminaries}
To study the high-frequency regimes relevant to our problem, we collect here the geometric and microlocal notation used throughout the article.
We first recall the basic Riemannian structures and associated Hamiltonian dynamics on $\cotang\Mc$, where $(\Mc,\mathsf{g})$ is a smooth, connected Riemannian manifold of dimension $\mathsf{d}$, either $(M,g)$ or $(I\times M,\d t^2+g)$ with $I\subset\RR$ a nonempty open interval.
We then introduce the semiclassical pseudodifferential calculus (symbol classes, quantization procedures, and boundedness properties), and review microlocal and semiclassical defect measures, which are central to the proof of Theorem~\ref{Th:GCC:UO}.
\subsection{Geometric tools and classical definitions}\label{subsection:Geometric:Tools}
In this subsection, we recall the classical notation used for Riemannian geometry (following the example of~\cite[Chapter 2]{lafontaine2004riemannian}) on $(\Mc, \mathsf{g})$.

In local charts, the metric $\mathsf{g}$ on the tangent bundle is given by $\mathsf{G} = (\mathsf{g}_{ij})_{1\leq i, j\leq \mathsf{d}}$ and on the cotangent bundle it is given by $\mathsf{G}^{-1} = (\mathsf{g}^{ij})_{1\leq i, j\leq \mathsf{d}}$, see~\cite[Definition 2.1]{lafontaine2004riemannian}.
For all $z \in \Mc$, fixed until the end of this paragraph, let $\p{\frac{\partial}{\partial z_1}, \dots, \frac{\partial}{\partial z_{\mathsf{d}}}}$ be an orthonormal basis in a local chart around $z$.
Then, every $v\in\tang_z \Mc$ has a decomposition $v =  \sum_{j = 1}^{\mathsf{d}}v^j{\frac{\partial}{\partial z_j}}$, and
\begin{equation*}
\forall v, w\in\tang_z \Mc, \quad \pscal{v}{w}_{\mathsf{g}(z)} := \mathsf{g}(z)(v, w) = \sum_{1\leq i, j\leq \mathsf{d}} \mathsf{g}_{ij}(z)v^iw^j.
\end{equation*}
For every $v\in \tang_z\Mc$, we define $v^\ast\in\cotang_z \Mc$ uniquely by 
\begin{equation}\label{eq:rel:adjoint}
v^\ast(w) = \pscal{v}{w}_{\mathsf{g}(z)},\quad \forall w \in\tang_z \Mc.
\end{equation}
In local coordinates,~\eqref{eq:rel:adjoint} means $v^\ast = \mathsf{G}(z)v = ({\sum_{j}\mathsf{g}_{ij}(z)v^j})_{i}$.
Similarly, we define $\pscal{\cdot}{\cdot}_{\mathsf{g}(z)^\ast}$ by
\begin{equation*}
\forall \zeta, \zeta'\in \mathrm{T}^\star_z \Mc, \quad \pscal{\zeta}{\zeta'}_{\mathsf{g}(z)^\ast} := \sum_{1\leq i, j\leq \mathsf{d}} \mathsf{g}^{ij}(z)\zeta_i\zeta'_j.
\end{equation*}
For all $k\in\Cscr^1(\Mc)$, the Riemannian gradient of $k$ in $z$ is given by 
\begin{equation}\label{eq:Gradient-Beltrami}
\pscal{\nabla_gk}{X}_{\mathsf{g}} = \d k (X), \quad \forall X\in \tang \Mc,
\end{equation}
and, in local charts, by 
\begin{equation*}
\nabla_{\mathsf{g}} k (z) := \sum_{i = 1}^{\mathsf{d}}\p*{\sum_{j = 1}^\mathsf{d} \mathsf{g}^{ij}(z)\frac{\partial k}{\partial z_i}(z)}\frac{\partial}{\partial z_j}.
\end{equation*}
The associated Laplace-Beltrami operator is given by 
\begin{equation}\label{eq:Laplace-Beltrami}
-\Delta_{\mathsf{g}} = \frac{1}{\sqrt{\det \mathsf{g}}}\sum_{j, k}D_j\p{\mathsf{g}^{jk}\sqrt{\det \mathsf{g}}D_k}, \quad D_j = \frac{\partial_{z_j}}{i},
\end{equation}
and satisfies, for all functions $k_1, k_2 \in \Cscr^2_{c}(\Mc)$, 
\begin{equation*}
\pscal{\Delta_{\mathsf{g}} k_1}{k_2}_{L^2(\Mc)} = -\int_{\Mc} \langle\nabla_{\mathsf{g}} k_1(z), \nabla_{\mathsf{g}}k_2(z)\rangle_{\mathsf{g}(z)}\d z = \pscal{k_1}{\Delta_{\mathsf{g}} k_2}_{L^2(\Mc)}.
\end{equation*}
When there is no possible confusion, we will write for short $\pscal{\cdot}{\cdot}_{\mathsf{g}}$ for either $\pscal{\cdot}{\cdot}_{\mathsf{g}(z)}$ or $\pscal{\cdot}{\cdot}_{\mathsf{g}(z)^\ast}$.
For all $v\in\tang_z \Mc$ we finally define $\abs{v}_{\mathsf{g}} := \sqrt{\pscal{v}{v}_{\mathsf{g}}}$ and thus $\abs{v^\ast}_{\mathsf{g}} := \sqrt{\pscal{v^\ast}{v^\ast}_{\mathsf{g}}} = \sqrt{\pscal{v}{v}_{\mathsf{g}}} = \abs{v}_{\mathsf{g}}$.
\paragraph{Hamiltonian vector fields.}
Using the notation of~\cite{Hörmander1985AnalysisOfLinearOpVol3}, we recall the following definition and general remark.
For a smooth function $k\in\smooth(\cotang \Mc)$, we denote by $H_k\in \smooth\p{\cotang M, \tang\p{\cotang M}}$ its Hamiltonian vector field, which is the unique vector field on $\cotang \Mc$ satisfying
\begin{equation}\label{eq:Hamiltonian:Vector:Field}
\symplectic(X, H_k) = \d k(X), \quad \forall X\in \tang\p{\cotang \Mc}
\end{equation}
where $\symplectic$ is the symplectic 2-form on $\cotang\Mc$.
For all $k, a\in\smooth(\cotang \Mc)$,~\eqref{eq:Hamiltonian:Vector:Field} implies $H_ka = \set{k, a}$. Moreover, for all $\rho = (z, \zeta)\in\cotang \Mc$ and any coordinate vector field $\p{\frac{\partial}{\partial z_1}, \dots, \frac{\partial}{\partial z_{\mathsf{d}}}, \frac{\partial}{\partial \zeta_1}, \dots, \frac{\partial}{\partial \zeta_{\mathsf{d}}}}$ in local charts around $\rho$, we have
\begin{equation*}
H_k(\rho) = \sum_{j = 1}^{\mathsf{d}} \frac{\partial k}{\partial \zeta_j}(\rho)\frac{\partial }{\partial z_j} - \frac{\partial k}{\partial z_j}(\rho)\frac{\partial }{\partial \zeta_j}.
\end{equation*}
\begin{remark}\label{rk:preserved:along:flow}
A smooth function $k\in\smooth(\cotang \Mc)$ and its Hamiltonian flow $\varphi^k$ defined in~\eqref{df:Hamiltonian:Flow} satisfy $k\circ\varphi^k_t = k$ for all time $t\in\RR$.
In other words, $k$ is preserved along its Hamiltonian flow.
\end{remark}
\subsection{Notation for pseudodifferential calculus on Riemannian manifolds}\label{section:pseudo:calculus}
We will need a pseudodifferential calculus on $(\Mc, \mathsf{g})$ defined earlier.
Note that $M$ is compact, but $I\times M$ is not.
As in Subsection~\ref{subsection:Geometric:Tools}, we consider the general case of $(\Mc, \mathsf{g})$ a $\mathsf{d}$-dimensional smooth and connected Riemannian manifold, which can play the role of one or the other, and define symbols and pseudodifferential operators on a general Riemannian manifold as it is done in~\cite[Appendix E]{dyatlov2019mathematical}.
\paragraph{Symbols on Riemannian manifolds.}
Firstly, as~\cite[E.1.2]{dyatlov2019mathematical} or~\cite[Chapter 18]{Hörmander1985AnalysisOfLinearOpVol3}, given $k\geq 1$, we use the following three symbol classes.
\begin{enumerate}
\item For $m\in\RR$, the class $S^m(\cotang\Mc, \CC^{k\times k})$ is the set of matrix-valued symbols $(a^h)_{h\in (0, 1]}$ satisfying both $a^h \in\smooth\p{\cotang \Mc, \CC^{k\times k}}$ for all $h\in (0, 1]$, and for all $K$ relatively compact local chart of $\Mc$,
\begin{equation*}
\forall \alpha,\beta\in\NN_{0}^{\mathsf{d}},\quad \sup_{h\in (0, 1]}\sup_{\begin{subarray}{c} z \in K\\ \zeta \in \mathrm{T}^{\star}_{z}\Mc\end{subarray}}
\jap{\zeta}^{\abs{\beta} - m}_{\mathsf{g}}\abs{\partial_{z}^\alpha\partial_{\zeta}^\beta a^h (z, \zeta)}<\infty.
\end{equation*}
For brevity and a slight abuse of notation, when $h$ varies we will write $a^h$ instead of $(a^h)_h$.
\item For $m\in\RR$, we say that a symbol $a^h\in S^m(\cotang \Mc, \CC^{k\times k})$ is \emph{compactly supported in $\Mc$} if there is a compact subset $K_a\subset \Mc$ --- independent of the choice of $h$ --- such that $\pi_{\Mc}(\supp(a^h))\subset K_a$ for all $h\in (0, 1]$.
The set of symbols of order $m$ \emph{compactly supported} in $\Mc$ will be denoted by $S^{m}_{\comp}\p{\cotang \Mc, \CC^{k\times k}}$.
\item The class of homogeneous symbols of order $m$, $S^m_{\phg}(\cotang \Mc, \CC^{k\times k})$, is the set of $h$-independent symbols $a_{\phg}\in S^{m}_{\comp}(\cotang \Mc, \CC^{k\times k})$ that, furthermore, satisfy
\begin{equation*}
a_{\phg}(z, R\zeta) = R^m a_{\phg}(z, \zeta), \quad \forall R \geq 1,~ \forall (z,\zeta)\in\cotang \Mc~\text{such that}~\abs{\zeta}_{\mathsf{g}}\geq 1.
\end{equation*}
\end{enumerate}
In the case $k = 1$, we will write for short $S^m(\cotang\Mc)$ instead of $S^m(\cotang\Mc, \CC^{1\times 1})$, $S^m_{\comp}(\cotang\Mc, \CC^{1\times 1})$ instead of $S_{\comp}^m(\cotang\Mc, \CC^{1\times 1})$ and $S_{\phg}^m(\cotang\Mc)$ instead of $S_{\phg}^m(\cotang\Mc, \CC^{1\times 1})$.
\paragraph{Pseudodifferential operators on Riemannian manifolds.}
For $s\in\RR$, we define the semiclassical Sobolev space $H^s_{h, \comp}(\Mc)$ and its dual $H^{-s}_{h, \loc}(\Mc)$ as in~\cite[Definition E.20]{dyatlov2019mathematical}.
If $\Mc$ is compact, then $H^s_{h, \comp}(\Mc) = H^s_{h, \loc}(\Mc)$ and will be denoted by $H^s_{h}(\Mc)$.
We define the quantization on $\Mc$, denoted $\Op^{\Mc}_{h}$, as in~\cite[Appendix E]{dyatlov2019mathematical}, and, given $k\geq 1$, we use the following four pseudodifferential spaces.
\begin{enumerate}
\item We let $\pseudo_{h}^m(\Mc, \CC^{k\times k})$ be the space of semiclassical pseudodifferential operators of order $m$ on $\Mc$ with values in $\CC^{k\times k}$, see~\cite[Section 18.1]{Hörmander1985AnalysisOfLinearOpVol3} for details.
For all $A_{h}\in\pseudo_{h}^m(\Mc, \CC^{k\times k})$, a \emph{principal symbol of $A_{h}$} is a symbol $a^h\in S^m(\cotang\Mc, \CC^{k\times k})$ satisfying
\begin{equation*}
    A_{h} - \Op^{\Mc}_{h}\p{a^h} \in h\pseudo^{m-1}_{h}(\Mc).
\end{equation*}
For all $m\in\RR$, we define the surjective map
\begin{equation*}
\p{\sigma_{h}}_{|\pseudo^m_{h}(M)}\colon  \pseudo^m_{h}(\Mc, \CC^{k\times k}) \longrightarrow  S^m(\cotang \Mc, \CC^{k\times k})/hS^{m-1}(\cotang \Mc, \CC^{k\times k}).
\end{equation*}
By a slight abuse of notation, we use $\sigma_h(A_h)$ to denote a chosen representative in $S^m(\cotang \Mc, \CC^{k\times k})$, uniquely defined up to remainder terms in $hS^{m-1}(\cotang \Mc, \CC^{k\times k})$, or by~\cite[Proposition E.14]{dyatlov2019mathematical}.
\item  For $m\in\RR$, we say that $A_{h}\in\pseudo_{h}^m(\Mc, \CC^{k\times k})$ is \emph{compactly supported} if its \emph{Schwartz kernel} is compactly supported, \ie $A_{h}$ maps $\smooth(\Mc, \CC^{k\times k}) \to \mathcal{E}'(\Mc, \CC^{k\times k})$ (see \cite[Section A.7]{dyatlov2019mathematical}), and denote by $\pseudo^m_{h, \comp}(\Mc, \CC^{k\times k})$ the set of \emph{compactly supported} (in an $h$-independent compact of $\Mc$) semiclassical pseudodifferential operators of order $m$.
One can also say that $A_{h}\in\pseudo_{h, \comp}^{m}(\Mc, \CC^{k\times k})$ if and only if
\begin{equation*}
A_{h}\in\pseudo_{h}^{m}\p{\Mc, \CC^{k\times k}}
\quad\text{and}
\quad\exists \chi\in\test(\Mc), ~ \chi A_{h} \chi = A_{h} + \O{h^\infty}_{\pseudo^{-\infty}(\Mc, \CC^{k\times k})}.
\end{equation*}
As an example, for $m\in \RR$, if $a^h\in S^m_{\comp}(\cotang\Mc, \CC^{k\times k})$, then $\Op^{\Mc}_{h}(a^h)\in \pseudo_{h, \comp}^{m}(\Mc, \CC^{k\times k})$.
\item We let $\pseudo_{h}^{\comp}(\Mc, \CC^{k\times k})$ be the space of \emph{compactly supported} operators $A_{h}\in \pseudo^{-\infty}_{h, \comp}(\Mc, \CC^{k\times k})$ such that the semiclassical wavefront set $\WaveFront_{h}(A_{h})$ (see~\cite[Definitions E.27 and E.28]{dyatlov2019mathematical} or~\cite{zworski2022semiclassical}) is a compact subset of $\cotang \Mc$.

As an example, if $A_{h} = \Op^{\Mc}_{h}(a^h)$ with $a^h\in\bigcup_{m\in\RR}S^m\p{\cotang \Mc, \CC^{k\times k}}$ and there exists a compact subset $K_a\subset \cotang \Mc$ such that
\begin{equation*}
    \forall h \in (0, 1], \quad \supp\p{a^h}\subset K_a,
\end{equation*}
then $A_{h}\in \pseudo_{h}^{\comp}(\Mc, \CC^{k\times k})$.
\item We let $\pseudo_{h, \phg}^m(\Mc, \CC^{k\times k})$ be the space of homogeneous semiclassical pseudodifferential operators $A_{h} \in \pseudo_{h, \comp}^m(\Mc, \CC^{k\times k})$ in the following sense:
\begin{equation*}
\exists (a_{\phg}, b^h)\in S^{m}_{\phg}(\cotang\Mc, \CC^{k\times k})\times S^{m - 1}_{\comp}(\cotang \Mc, \CC^{k\times k}),\quad
\sigma_{h}(A_{h}) = a_{\phg} + b^h.
\end{equation*}
Furthermore, we define the \emph{homogeneous principal symbol} of operators on $\pseudo_{h,\phg}^m$ as
\begin{equation*}
\begin{array}{rrcl}
\sigma_{\phg}^m\colon &\pseudo_{h, \phg}^m\p{\Mc, \CC^{k\times k}}& \longrightarrow & S^{m}_{\phg}\p{\cotang \Mc, \CC^{k\times k}}/ S^{m-1}_{\comp}\p{\cotang \Mc, \CC^{k\times k}},\\
&A_{h}& \longmapsto & \sigma_{h}(A_{h})~\mathrm{mod}~S^{m-1}_{\comp}\p{\cotang \Mc, \CC^{k\times k}}.
\end{array}
\end{equation*}
By a slight abuse of notation, we use $\sigma_{\phg}^m(A_h)$ to denote a chosen representative in $S^{m}_{\phg}\p{\cotang \Mc, \CC^{k\times k}}$ (thus $h$-independent), uniquely defined up to remainder terms in $S^{m-1}_{\comp}\p{\cotang \Mc, \CC^{k\times k}}$.
\end{enumerate}
In the case $k = 1$, we will write for short $\pseudo_{h}^m(\Mc)$ instead of $\pseudo_{h}^m(\Mc, \CC^{1\times 1})$, $\pseudo_{h, \comp}^m(\Mc)$ instead of $\pseudo_{h, \comp}^m(\Mc, \CC^{1\times 1})$, $\pseudo_{h}^{\comp}(\Mc)$ instead of $\pseudo_{h}^{\comp}(\Mc, \CC^{1\times 1})$ and $\pseudo_{h, \phg}^m(\Mc)$ instead of $\pseudo_{h, \phg}^m(\Mc, \CC^{1\times 1})$ for any order $m\in\RR$.
\begin{remark}
For $m\in\RR$ and $A_{h}\in \pseudo^m_{h, \phg}$, one can compute $\sigma_{\phg}^m\p{A_{h}}$ as
\begin{equation*}
\sigma_{\phg}^m\p{A_{h}}\colon (z, \zeta)\mapsto \lim_{R\to +\infty}R^{-m}\sigma_{h}(A_{h})(z, R\zeta).
\end{equation*}
\end{remark}
\paragraph{Pseudodifferential calculus.}
In this section, we summarize the fundamental results of pseudodifferential calculus used in this article.
We first recall the following principal symbol statement from~\cite[Proposition E.17]{dyatlov2019mathematical}.
\begin{proposition}\label{pr:composition}
Given $k\geq 1$, for all $(m, m')\in\RR^2$, all $A_{h}\in \pseudo^m_{h, \comp}(\Mc, \CC^{k\times k})$ and all $B_{h}\in\pseudo^{m'}_{h}(\Mc, \CC^{k\times k})$ we have 
\begin{enumerate}
\item $A_{h}B_{h}\in \pseudo^{m+m'}_{h}(\Mc, \CC^{k\times k})$ with principal symbol $\sigma_{h}(A_{h}B_{h}) = \sigma_{h}(A_{h})\sigma_{h}(B_{h})$.
\item If $k = 1$, $\classe{A_{h}, B_{h}}: = A_{h}B_{h} - B_{h}A_{h}\in h\pseudo^{m+m'-1}_{h}(\Mc)$ with principal symbol 
\begin{equation*}
    \sigma_{h}\p{ih^{-1}\classe{A_{h} ,B_{h}}} = \set{\sigma_{h}(A_{h}), \sigma_{h}(B_{h})}.
\end{equation*}
\item $B_{h}^*\in\pseudo^{m'}_{h}(\Mc, \CC^{k\times k})$ with principal symbol $\sigma_{h}(B_{h}^*) = \sigma_{h}(B_{h})^*$.
\end{enumerate}
\end{proposition}
We now deduce a counterpart of Proposition~\ref{pr:composition} for homogeneous operators.
\begin{corollary}
Given $k\geq 1$, for all $(m, m')\in\RR^2$, all $A_{h}\in \pseudo^m_{h, \phg}(\Mc, \CC^{k\times k})$ and all $B_{h}\in\pseudo^{m'}_{h, \phg}(\Mc, \CC^{k\times k})$ we have 
\begin{enumerate}
\item $A_{h}B_{h} \in \pseudo_{h, \phg}^{m + m'}(\Mc, \CC^{k\times k})$ with \emph{homogeneous principal symbol} $\sigma_{\phg}^{m + m'}(A_{h}B_{h}) = \sigma_{\phg}^{m}(A_{h})\sigma_{\phg}^{m'}(B_{h})$.
\item If $k = 1$, $\classe{A_{h}, B_{h}} := A_{h}B_{h} - B_{h}A_{h}\in h\pseudo^{m+m'-1}_{h, \comp}(\Mc)$ with \emph{homogeneous principal symbol} 
\begin{equation*}
    \sigma_{\phg}^{m + m'-1}\p{ih^{-1}\classe{A_{h} ,B_{h}}} = \set{\sigma_{\phg}^{m}(A_{h}), \sigma_{\phg}^{m'}(B_{h})}.
\end{equation*}
\item $B_{h}^*\in\pseudo^{m'}_{h}(\Mc, \CC^{k\times k})$ with \emph{homogeneous principal symbol} $\sigma_{\phg}^{m'}(B_{h}^*) = \sigma_{\phg}^{m'}(B_{h})^*$.
\end{enumerate}
\end{corollary}\noindent
We recall boundedness statements from~\cite[Propositions E.22 and E.24]{dyatlov2019mathematical} and~\cite[Theorem 13.13]{zworski2022semiclassical}.
\begin{proposition}[Boundedness]\label{pr:boundedness:pseudo}
For all $m > 0$ and all $A_{h}\in \pseudo^m_{h}(\Mc)$, 
\begin{equation*}
\forall s\in\RR,~\forall \chi\in\test\p{M},~\exists M_{s, \chi} > 0,~ \forall h \in (0, 1],\quad \norm{\chi A_{h}\chi}_{\linear\p{H^{s}_{h}(\Mc), H^{s - m}_{h}(\Mc)}}\leq M_{s, \chi}.
\end{equation*}
\end{proposition}
\begin{proposition}[Calder{\'o}n-Vaillancourt]\label{pr:Calderon:Vaillancourt}
For all $A_{h}\in\pseudo_{h, \comp}^0(\Mc)$, there is $C_A > 0$ such that
\begin{equation*}
\forall h\in (0, 1], \quad \norm{A_{h}}_{\linear\p{L^2\p{\Mc}}}\leq \sup_{\cotang\Mc}\abs{\sigma_{h}\p{A_{h}}} + C_Ah^{\frac{1}{2}}.
\end{equation*}
\end{proposition}
\paragraph{Microlocal and semiclassical defect measures.}
This paragraph introduces the concept of defect measures introduced by Tartar in~\cite{Tartar1990Hmeasures} and Gérard in~\cite{gerard1991defectmeasure}.
From~\cite[Theorem 1]{gerard1991defectmeasure}, we have the following definition and proposition.
\begin{definition}[Microlocal defect measure]\label{df:1:MDM}
Given $k\geq 1$, let $(U^n)_n$ be a sequence of $L^2(\Mc)^k$ and $\ml\in\Mc_+(\cosphere \Mc, \CC^{k\times k})$ be a Radon measure with nonnegative hermitian matrix values.
We say that $(U^n)_n$ admits $\ml$ as unique microlocal defect measure if
\begin{equation}\tag{MDM}\label{eq:1:MDM}
\lim_{n\to\infty}\pscal{U^n}{AU^n}_{L^2\p{\Mc}^k} = \int_{\cosphere \Mc}\tr\classe{\sigma^0(A)(\rho)\d \ml(\rho)}, \quad \forall A\in\pseudo^{0}_{1, \phg}(\Mc, \CC^{k\times k}),
\end{equation}
where $\tr\colon \CC^{k\times k}\mapsto\RR$ is the classical trace map.
\end{definition}
\begin{proposition}
Given $k\geq 1$, let $(U^n)_n$ be a sequence of $L^2(\Mc)^k$. If $(U^n)_n$ weakly converges to $0$, then there exists a subsequence of $(U^n)_n$ (still denoted $(U^n)_n$ with a slight abuse of notation) which admits a unique microlocal defect measure $\ml\in\Mc_+(\cosphere \Mc, \CC^{k\times k})$ in the sense of~\eqref{eq:1:MDM}.
\end{proposition}
\begin{proposition}\label{pr:inter:supp:k:ml}
Given $k\geq 1$, let $(U^n)_n$ be a sequence of $L^2(\Mc)^k$ that admits a unique microlocal defect measure $\p{\ml_{jl}}_{1\leq j, l\leq n}\in\Mc_+(\cosphere \Mc, \CC^{k\times k})$ in the sense of~\eqref{eq:1:MDM}. Then
\begin{equation*}
\forall j, l\in\set{1,\dots, k}, \quad \supp\p{\ml_{jl}}\subset \supp\p{\ml_{jj}}\cap\supp\p{\ml_{ll}}.
\end{equation*}
\end{proposition}
\begin{proof}[Proof of Proposition~\ref{pr:inter:supp:k:ml}]
For all $j, l\in\set{1, \dots, k}$ and all $A\in\pseudo_{1, \phg}^0(\Mc)$ supported away from $\supp\p{\ml_{jj}}\cap\supp\p{\ml_{ll}}$, $A^* A$ is supported away from $\supp\p{\ml_{jj}}\cap\supp\p{\ml_{ll}}$, and thus
\begin{equation*}
    \min_{\bullet\in\set{j, l}}\norm{AU_{\bullet}^n}_{L^2} = \min_{\bullet\in\set{j, l}}\p*{\sqrt{\pscal{U_{\bullet}^n}{A^*A U_{\bullet}^n}}}\xrightarrow[n\to\infty]{}\min_{\bullet\in\set{j, l}}\sqrt{\pscal{\ml_{\bullet\bullet}}{\sigma(A)^*\sigma(A)}} = 0.
\end{equation*}
Then, using the Cauchy-Schwarz inequality, we have
\begin{equation*}
\abs{\pscal{U_j^n}{A U_l^n}}\leq \min\set{\norm{U_j^n}\norm{A U_l^n}, \norm{A^* U_j^n}\norm{U_l^n}}\xrightarrow[n\to\infty]{}0.
\end{equation*}
This concludes the proof of Proposition~\ref{pr:inter:supp:k:ml}.
\end{proof}
From~\cite[Definition E.41 and Proposition E.42]{dyatlov2019mathematical} we have the following definition and proposition.
The proofs of Propositions~\ref{pr:inter:supp:k:ml} and~\ref{pr:inter:supp:k} rely on the same argument.
\begin{definition}[Semiclassical defect measure]
Given $k\geq 1$, let $(h_n)_n$ be a sequence of $(0, 1]$ tending to $0$, $(U^n)_n$ be a sequence of $L^2(\Mc)^k$ and $\sdm\in\Mc_+(\cotang \Mc, \CC^{k\times k})$ be a Radon measure with nonnegative hermitian matrix values.
We say that $(U^n, h_n)_n$ admits $\sdm$ as unique semiclassical defect measure if, for all $A_{h}\in \pseudo^{\comp}_{h}(\Mc, \CC^{k\times k})$ with principal symbol $\sigma_h(A_h)$ independent of $h$, we have
\begin{equation}\tag{SDM}\label{eq:SDM}
\lim_{n\to\infty}\pscal{U^n}{A_{h_n}U^n}_{L^2(\Mc)^k} = \int_{\cotang \Mc}\tr\classe{\sigma_{h}(A_{h})(\rho)\d\sdm(\rho)},
\end{equation}
where $\tr\colon \CC^{k\times k}\mapsto\RR$ is the classical trace map.
\end{definition}
\begin{proposition}\label{pr:SDM}
Given $k\geq 1$, let $(h_n)_n$ be a sequence of $(0, 1]$ tending to $0$ and let $(U^n)_n$ be a bounded sequence of $L^2(\Mc)^k$.
Then, there exists a subsequence of $(U^n, h_n)_n$ (still denoted $(U^n, h_n)_n$ with a slight abuse of notation) which admits a unique semiclassical defect measure $\sdm\in\Mc_+(\cotang \Mc, \CC^{k\times k})$ in the sense of~\eqref{eq:SDM}.
\end{proposition}
\begin{proposition}\label{pr:inter:supp:k}
Given $k\geq 1$, let $(h_n)_n$ be a sequence of $(0, 1]$ tending to $0$ and $(U^n)_n$ be a sequence of $L^2(\Mc)^k$ that admits a unique semiclassical defect measure $\p{\sdm_{jl}}_{1\leq j, l\leq n}\in\Mc_+(\cotang \Mc, \CC^{k\times k})$ in the sense of~\eqref{eq:SDM}. Then
\begin{equation*}
\forall j, l\in\set{1,\dots, k}, \quad \supp\p{\sdm_{jl}}\subset \supp\p{\sdm_{jj}}\cap\supp\p{\sdm_{ll}}.
\end{equation*}
\end{proposition}
\section{Second microlocalization on the cosphere at infinity in the semiclassical limit}\label{section:microlocalization}
In this article, we study the equation~\eqref{eq:Wave} with potential $\mathsf{V} = \lmbd V$ --- for $V$ as~\eqref{df:potential:fixed:V} --- and, in particular, the case of large $\lmbd$.
In this case, setting the semiclassical parameter $h = \sqrt{\lmbd^{-1}}\in (0, 1]$, we use technical tools from semiclassical analysis, such as semiclassical defect measures, to study the behavior of solutions of~\eqref{eq:Wave} on the characteristic oscillation scale $\zeta \sim h^{-1}$.
However, this characteristic scale misses information contained on larger oscillation scales.
Following the ideas of~\cite[Section 2.6.2]{anantharaman2016wigner}, the goal of this technical section is to construct second microlocal defect measures
---
the notion was introduced by~\cite{miller1996propagation, Miller97Short, fermanian2000measures, Fermanian00Propagation, FKG02Mesures} (see also~\cite{Macia10HF, Anantharaman14Semiclassical, Anantharaman15longtime})
---
on the cosphere at infinity, namely $\partial\cotanginf\Mc$, to capture information in the regime of oscillations $\zeta\gg h^{-1}$.

We need second microlocalization both on $(M, g)$ and on $(I\times M, {\d t^2 + g})$, where $I$ is a nonempty open interval of $\RR$.
Accordingly, we consider a $\mathsf{d}$-dimensional, smooth, connected Riemannian manifold $(\Mc, \mathsf{g})$, which can play the role of either of these two manifolds.
The main result of this section is Theorem~\ref{Th:Microlocalization}, inspired by~\cite[Theorem 1]{gerard1991defectmeasure} and~\cite{fermanian2000measures}, which allows us to microlocalize bounded $L^2(\Mc)$-sequences on the cosphere at infinity in the semiclassical limit as $h$ tends to $0$.
This result will be crucial for the proof of Theorem~\ref{Th:GCC:UO}.
For simplicity, restricting ourselves to homogeneous pseudodifferential operators, we introduce our second microlocal defect measures on the unit cosphere $\cosphere\Mc$ and not directly on $\partial\cotanginf\Mc$.

In this section: first, we prove Theorem~\ref{Th:Microlocalization} in Subsection~\ref{subsection:Microlocalization:Proof}; second, we explain how to use it for the wave equation~\eqref{eq:Wave} in Subsection~\ref{subsection:Microlocalization:use:WE}.

\paragraph{Homogeneous symbols at infinity and 2-microlocal defect measure.}
In this paragraph we introduce and define the notions of 2-microlocal defect measure on the cosphere at infinity.
\begin{definition}[Homogeneous symbols of degree $m$ at infinity]\label{df:pseudo:phg:0}
Given $k\geq 1$, let $\tilde{S^m_{\phg}}\p{\cotang\Mc, \CC^{k\times k}}$ be the set of smooth functions $a^h\in S^m_{\comp}(\Mc, \CC^{k\times k})$ for which there is a symbol $a_{\phg}\in S^m_{\phg}(\cotang \Mc, \CC^{k\times k})$
---
see Subsection~\ref{section:pseudo:calculus} for definition 
---
and a smooth function $\chi^h\in S^m_{\comp}(\cotang \Mc, \CC^{k\times k})$ with   
\begin{equation*}
\lim_{R\to+\infty}\classe*{\sup_{h\in(0, 1]}\sup_{
\begin{subarray}{c}
(z, \zeta)\in\cotang \Mc,\\
\abs{\zeta}_{\mathsf{g}}\geq R
\end{subarray}
}\frac{\abs{\chi^h(z, \zeta)}}{\abs{\zeta}_{\mathsf{g}}^{m}}} = 0,
\end{equation*}
such that
\begin{equation}\label{eq:characterization:S^0_phg}
\forall (z,\zeta)\in\cotang \Mc,\quad a^h(z, \zeta) = a_{\phg}(z, \zeta) + \chi^h(z, \zeta).
\end{equation}
We say that $a_{\phg}$ is the homogeneous principal part of $a^h$.
\end{definition}
\begin{remark}
Note that, for $\epsilon > 0$, $S^{m-\epsilon}_{\phg}\p{\cotang \Mc, \CC^{k\times k}}\subset {S^{m}_{\phg}}\p{\cotang\Mc, \CC^{k\times k}}$.
Most of the time, we will have $\chi^h\in S_{\comp}^{m - \epsilon}(\cotang \Mc, \CC^{k\times k})$ with $\epsilon>0$.
\end{remark}
\begin{definition}\label{df:SFRC}
We call a \emph{smooth fiber-radial cut-off function} any smooth function 
\begin{equation*}
    \Theta\colon (R,\rho) \in\RR_+^*\times \cotang \Mc\longmapsto \Theta_R(\rho)\in [0,1]   
\end{equation*}
that satisfies the following two properties:
\begin{enumerate}
\item for all radius $R > 0$, $\Theta_R\in\test(\cotang\Mc, [0,1])$;
\item there are $\mathsf{C}^\star > \mathsf{C}_\star > 0$ such that for all radius $R \geq 1$,
\begin{enumerate}
\item $\Theta_R \equiv 1 \text{ near } \set{(z, \zeta)\in\cotang \Mc,~\jap{\zeta}_{\mathsf{g}} := \sqrt{1 + \abs{\zeta}_{\mathsf{g}}^2} \leq \mathsf{C}_\star R}$;
\item $\Theta_R \equiv 0 \text{ near } \set{(z, \zeta)\in\cotang \Mc,~\jap{\zeta}_{\mathsf{g}} \geq \mathsf{C}^\star R}$.
\end{enumerate}
\end{enumerate}
Using the notation of~\cite{Anantharaman14Semiclassical}, we also set $\Theta^R := 1-\Theta_R$.
\end{definition}
\begin{definition}[2-microlocal defect measure on the cosphere at infinity]\label{df:MDM}
Given $k\geq 1$,
let $(h_n)_n$ be a sequence in $(0, 1]$ tending to $0$, $(U^n)_n$ be a sequence in $L^2(\Mc)^k$ and 
$\mdm\in\Mc_+(\cosphere \Mc, \CC^{k\times k})$ be a Radon measure with nonnegative hermitian matrix values.
We say that $(U^n, h_n)_n$ admits $\mdm$ as unique 2-microlocal defect measure if:
for all \emph{smooth fiber-radial cut-off function} $\Theta$ (see Definition~\ref{df:SFRC}), and all $a^h\in \tilde{S^0_{\phg}}(\cotang \Mc, \CC^{k\times k})$ with homogeneous principal part $a_{\phg}\in S^0_{\phg}(\cotang \Mc, \CC^{k\times k})$, we have
\begin{equation}\tag{$2^{\mathrm{n d}}$-MDM}\label{eq:MDM}
\lim_{R\to\infty}\lim_{n\to\infty}\pscal{U^n}{\Op_{h_n}^{\Mc}\p{a^h\p{1 - \Theta_R}}U^n}_{L^2\p{\Mc}^k} = \int_{\cosphere \Mc}\tr\classe{{a_{\phg}(\rho)\d\mdm(\rho)}},
\end{equation}
where $\tr\colon \CC^{k\times k}\mapsto\RR$ is the classical trace map.
\end{definition}
\paragraph{Main result on 2-microlocal defect measures.}
In this paragraph, we state the main theorem of this section (proved in Subsection~\ref{subsection:Microlocalization:Proof}) and discuss a few related remarks.
\begin{theorem}\label{Th:Microlocalization}
Given $k\geq 1$, let $(h_n)_n$ be a sequence in $(0, 1]$ tending to $0$ and let $(U^n)_n$ be a bounded sequence in $L^2(\Mc)^k$.
Then, there exists a subsequence of $(U^n, h_n)_n$ (still denoted $(U^n, h_n)_n$ with a slight abuse of notation) that admits a unique 2-microlocal defect measure $\mdm\in\Mc_+(\cosphere \Mc, \CC^{k\times k})$ in the sense of~\eqref{eq:MDM}.
\end{theorem}
\begin{proposition}\label{pr:inter:supp:k:2mdm}
Given $k\geq 1$, let $(h_n)_n$ be a sequence of $(0, 1]$ tending to $0$ and $(U^n)_n$ be a sequence of $L^2(\Mc)^k$ that admits a unique 2-microlocal defect measure $\p{\mdm_{jl}}_{1\leq j, l\leq n}\in\Mc_+(\cosphere \Mc, \CC^{k\times k})$ in the sense of~\eqref{eq:MDM}. Then
\begin{equation*}
\forall j, l\in\set{1,\dots, k}, \quad \supp\p{\mdm_{jl}}\subset \supp\p{\mdm_{jj}}\cap\supp\p{\mdm_{ll}}.
\end{equation*}
\end{proposition}
\begin{remark}
For any cut-off function 
\begin{equation}\label{df:theta}
\theta\in\test(\RR_+, [0, 1])\text{ such that }\theta = 1\text{ near }[0, 1],
\end{equation}
the function $\Theta$ given by
\begin{equation}\label{df:Theta}
\Theta_R(z, \zeta) := \theta\p*{\frac{\jap{\zeta}_{\mathsf{g}}}{R}}, \quad R > 0, ~ (z, \zeta)\in\cotang \Mc
\end{equation}
satisfies the conditions of Definition~\ref{df:SFRC}.
In other words, if a sequence $(U^n, h_n)_n$ admits $\mdm$ as unique 2-microlocal defect measure in the sense of~\eqref{eq:MDM}, then: for all $a^h\in \tilde{S^0_{\phg}}(\cotang \Mc, \CC^{k\times k})$ with an homogeneous principal part $a_{\phg}$, we have
\begin{equation*}
\lim_{R\to\infty}\lim_{n\to\infty}
\pscal{U^n}{\Op_{h_n}^{\Mc}\p*{a^{h_n}\p*{1 - \theta\p*{{\jap{\zeta}_{\mathsf{g}}}/{R}}}}U^n}_{L^2\p{\Mc}^k}
=
\int_{\cosphere \Mc}\tr\classe{{a_{\phg}(\rho)\d\mdm(\rho)}}.
\end{equation*}
\end{remark}
\begin{remark}
Standard microlocal defect measures (as in Definition~\ref{df:1:MDM}) and 2-microlocal defect measures (as in Definition~\ref{df:MDM}) define nonnegative Radon measures on $\cosphere \Mc$ and capture the high-frequency regimes in phase space.
However, while a 2-microlocal defect measure only captures the regime of oscillations $\zeta\gg h^{-1}$, a microlocal defect measure captures all high-frequencies $\abs{\zeta}\to+\infty$.
\end{remark}
\paragraph{Link between microlocal, 2-microlocal and semiclassical defect measures.}
In the following lemma, we establish a link between these three defect measures.
\begin{lemma}\label{lem:ml:sdm:mdm}
Given $k\geq 1$, let $(h_n)_n$ be a sequence in $(0, 1]$ tending to $0$ and let $(U^n)_n$ be a sequence in $L^2(\Mc)^k$.
If $(U^n, h_n)_n$ admits: 
\begin{itemize}
\item a unique microlocal defect measure $\ml\in\Mc_+(\cosphere\Mc, \CC^{k\times k})$ in the sense of~\eqref{eq:1:MDM},
\item a unique semiclassical defect measure $\sdm\in\Mc_+(\cotang \Mc, \CC^{k\times k})$ in the sense of~\eqref{eq:SDM},
\item a unique 2-microlocal defect measure $\mdm\in\Mc_+(\cosphere\Mc, \CC^{k\times k})$ in the sense of~\eqref{eq:MDM},  
\end{itemize}
and is $h$-oscillating from below (see \cite[Definition 2.3]{anantharaman2016wigner}), \ie for all $\theta$ satisfying~\eqref{df:theta}, we have
\begin{equation*}
\lim_{\varepsilon \to 0^+}
\limsup_{n\to\infty}
\pscal{U^n}{\Op_{h_n}^{\Mc}\p*{\theta\p*{\frac{\jap{\zeta}_{\mathsf{g}}}{\varepsilon}}}U^n}_{L^2(\Mc)^k}
= 0.
\end{equation*}
Then
\begin{equation*}
\d\ml(z, \zeta) = \int_{\RR_+} r^{d-1} \d \sdm (z, r\zeta)\d r + \d \mdm(z, \zeta), \quad (z, \zeta)\in\cosphere\Mc.
\end{equation*}
In other words, for all $a\in \test\p{\cosphere\Mc}$, we have
\begin{equation}\label{eq:11}
\int_{\cosphere\Mc}\tr\classe{a\d\ml} =
\int_{\set{(z, \zeta)\in \cotang \Mc, ~\zeta \neq 0}}\tr\classe*{\textstyle a\p*{x, \frac{\zeta}{\abs{\zeta}_{\mathsf{g}}}}\d\sdm\p{z, \zeta}} + \int_{\cosphere\Mc}\tr\classe{a\d\mdm}.
\end{equation}
\end{lemma}
\begin{proof}[Proof of Lemma~\ref{lem:ml:sdm:mdm}]
Using~\cite[equation (2.19) from Section 2.6.2]{anantharaman2016wigner} and Definition~\ref{df:MDM}, we know that for all $a_{\phg}\in S^{0}_{\phg}\p{\cotang\Mc}$ and all $\theta$ satisfying~\eqref{df:theta}, we have
\begin{align*}
\int_{\cosphere\Mc}\tr\classe{a_{\phg}\d\ml}
=& \lim_{\varepsilon \to 0^+}\lim_{n\to\infty}\textstyle \pscal{U^n}{\Op_{h_n}^{\Mc}\p*{a_{\phg}\p*{1 - \theta\p*{\frac{\jap{\zeta}_{\mathsf{g}}}{\varepsilon}}}} U^n}_{L^2\p{\Mc}^k}\\
=& \lim_{
\begin{subarray}{ll}
R\to\infty\\
\varepsilon \to 0^+
\end{subarray}
}\lim_{n\to\infty}
\textstyle\pscal{U^n}{\Op_{h_n}^{\Mc}\p*{a_{\phg}\p*{\theta\p*{\frac{\jap{\zeta}_{\mathsf{g}}}{R}} - \theta\p*{\frac{\jap{\zeta}_{\mathsf{g}}}{\varepsilon}}}} U^n}_{L^2\p{\Mc}^k} + \int_{\cosphere\Mc}\tr\classe{a_{\phg}\d\mdm}.
\end{align*}
Then, since
\begin{equation*}
\lim_{
\begin{subarray}{ll}
R\to\infty\\
\varepsilon \to 0^+
\end{subarray}
}\lim_{n\to\infty}
{\textstyle\pscal{U^n}{\Op_{h_n}^{\Mc}\p*{a_{\phg}\p*{\theta\p*{\frac{\jap{\zeta}_{\mathsf{g}}}{R}} - \theta\p*{\frac{\jap{\zeta}_{\mathsf{g}}}{\varepsilon}}}} U^n}_{L^2\p{\Mc}^k}} 
= \int_{\cotang\Mc^*}\tr\classe*{\textstyle a_{\phg}\p*{x, \frac{\zeta}{\abs{\zeta}_{\mathsf{g}}}}\d\sdm\p{z, \zeta}},
\end{equation*}
with $\cotang\Mc^* := \set{(z, \zeta)\in \cotang \Mc, ~\zeta \neq 0}$, we conclude that~\eqref{eq:11} holds.
Moreover, since 
\begin{equation*}
(z, \zeta)\in \set{(z', \zeta')\in \cotang \Mc, ~\zeta' = 0}\mapsto \sqrt{\mathsf{g}_z\p{\zeta, \zeta}}\in\RR_+
\end{equation*}
is a smooth surjection and is a submersion, using coarea formula on Riemannian manifolds (see~\cite{chavel06riemannian} for details), we have
\begin{align*}
\int_{\cotang\Mc^*}\tr\classe*{\textstyle a_{\phg}\p*{x, \frac{\zeta}{\abs{\zeta}_{\mathsf{g}}}}\d\sdm\p{z, \zeta}}
&=\int_{r\in\RR_+}\int_{r\cosphere\Mc}\tr\classe*{a_{\phg}\d\sdm\p{z, \zeta}}\d r\\
&=\int_{r\in\RR_+}r^{d-1}\int_{\cosphere\Mc}\tr\classe*{a_{\phg}\d\sdm\p{z, r\zeta}}\d r
\end{align*}
with $r\cosphere\Mc = \set{(z, \zeta)\in\cotang \Mc,~\abs{\zeta}_{\mathsf{g}} = r}$ for all $r > 0$.
This, together with~\eqref{eq:11}, concludes the proof of Lemma~\ref{lem:ml:sdm:mdm}.
\end{proof}
\subsection{Proof of Theorem~\ref{Th:Microlocalization}}\label{subsection:Microlocalization:Proof}
In this subsection, we prove Theorem~\ref{Th:Microlocalization}.
We first prove an easier result, given by the following lemma with a fixed cut-off function and a discrete sequence of radii depending on the cut-off function.
\begin{lemma}\label{lem:principal:microlocalization}
Given $k\geq 1$, let us consider $\theta$ as in~\eqref{df:theta}, $(h_n)_n$ a sequence in $(0, 1]$ tending to $0$, and $({U}^n)_n$ a bounded sequence in $L^2(\Mc)^k$.
There is
\begin{enumerate}
\item a subsequence of $(U^n, h_n)_n$ (still denoted $(U^n, h_n)_n$), 
\item a measure $\mdm\in \Mc_+(\cosphere\Mc, \CC^{k\times k})$,
\item a sequence $(R_m)_m$ of radii tending to $+\infty$,
\end{enumerate}
such that for all $a^h\in\tilde{ S^0_{\phg}}\p{\cotang\Mc}$,
\begin{equation}\label{eq:pre:MDM}
\lim_{k\to\infty}\lim_{n\to\infty}
{\textstyle\pscal{{U}^n}{\Op_{h_n}^{\Mc}\p*{a^{h_n}\p*{1 - \theta\p*{\frac{\jap{\zeta}_{\mathsf{g}}}{R_k}}}}{U}^n}_{L^2(\Mc)^k}} = \int_{\cosphere \Mc}\tr\classe{{a_{\phg}(\rho)\d\mdm(\rho)}},
\end{equation}
where $\tr\colon \CC^{k\times k}\mapsto\RR$ is the classical trace map.
\end{lemma}
\subsubsection{Proof of Lemma~\ref{lem:principal:microlocalization}}
To prove Lemma~\ref{lem:principal:microlocalization}, we begin with an easier case where the homogeneous part (\ie $a_{\phg}$ in~\eqref{eq:characterization:S^0_phg}) vanishes identically.
\begin{lemma}\label{lem:microlocalization:remainder}
Given $k\geq 1$, let $\Theta$ be a \emph{smooth fiber-radial cut-off function} (see Definition~\ref{df:SFRC}), $(h_n)_n$ be a sequence in $(0, 1]$ tending to $0$, $({U}^n)_n$ be a bounded sequence in $L^2(\Mc)^k$ and $\chi^h\in S^0_{\comp}(\cotang \Mc, \CC^{k\times k})$ be such that
\begin{equation}\label{eq:chi:to:0}
\lim_{R\to+\infty}\classe*{\sup_{h\in(0, 1]}\p*{\sup_{(z, \zeta)\in\cotang \Mc,~\jap{\zeta}_{\mathsf{g}}\geq R}{\chi^h(z, \zeta)}}} = 0.
\end{equation}
Then, we have \({\displaystyle\lim_{R\to\infty}\lim_{n\to\infty}}\textstyle\pscal{{U}^n}{\Op_{h_n}^{\Mc}\p{\chi^{h_n}(1-\Theta_R)}{U}^n}_{L^2(\Mc)^k} = 0\).
\end{lemma}
\begin{proof}[Proof of Lemma~\ref{lem:microlocalization:remainder}]
Since the proof is the same for $k > 1$ and $k = 1$, we treat only the case $k = 1$.
By Proposition~\ref{pr:Calderon:Vaillancourt}, there is a constant $C_{\Theta, \chi, R} > 0$, depending on the choice of $\Theta$, $\chi^h$ and $R$, such that
\begin{align*}
\textstyle\norm{\Op_{h_n}^{\Mc}\p{\chi^{h_n}\Theta^R}}_{\linear(L^2(\Mc))}
&\textstyle\leq \norm{\chi^{h_n}\Theta^R}_{L^\infty\p{\cotang \Mc}} + C_{\Theta, \chi, R}h_n^{1/2}\\
&\leq \sup_{h\in(0, 1]}\norm{\indicatrice_{\set{\jap{\zeta}_{\mathsf{g}}\geq\mathsf{C}_{\star} R}}\chi^{h}}_{L^\infty\p{\cotang \Mc}} + C_{\Theta, \chi, R}h_n^{1/2}
\end{align*}
which concludes the proof when letting first $n$ tend to $+\infty$ since $h_n$ goes to $0$, and second using~\eqref{eq:chi:to:0} as $R$ tends to $\infty$.
\end{proof}
\begin{proof}[Proof of Lemma~\ref{lem:principal:microlocalization}]
Since the proof is the same for $k > 1$ and $k = 1$, we treat only the case $k = 1$.
The proof is largely inspired by~\cite[Theorem 1]{gerard1991defectmeasure}.
Thanks to Lemma~\ref{lem:microlocalization:remainder}, we can consider only functions $a_{\phg}$ in $S^0_{\phg}(\cotang \Mc)$.
Using the surjection
\begin{equation*}
a_{\phg} \in S^0_{\phg}\p{\cotang\Mc}\mapsto \p{a_{\phg}}_{|\cosphere\Mc}\in \test\p{\cosphere\Mc},
\end{equation*}
in what follows, we work with $\mathsf{a}\in \smooth\p{\cosphere \Mc}$ and extend $\mathsf{a}$ to $\cotang \Mc$ as an homogeneous symbol $a_{\phg}$ with
\begin{equation}\label{eq:Yannis}
\textstyle a_{\phg}(z, \zeta) = \mathsf{a} \p*{z, \frac{\zeta}{\abs{\zeta}_{\mathsf{g}}}}, \quad (z, \zeta)\in \set{(z', \zeta')\in\cotang \Mc, ~\abs{\zeta'}_{\mathsf{g}} \geq 1}.
\end{equation}
Because $\theta = 1$ near $[0, 1]$, then $(z, \zeta)\mapsto \theta\p*{\frac{\jap{\zeta}_{\mathsf{g}}}{R}}$ is equal to $1$ near $\set{(z, \zeta)\in\cotang \Mc, ~\abs{\zeta}_{\mathsf{g}} < 1}$ for any $R\geq \sqrt{2}$, thus the value of $a_{\phg}$ in $\set{(z, \zeta)\in\cotang \Mc, ~\abs{\zeta}_{\mathsf{g}} < 1}$ is not important.

For all $(\mathsf{a}, R)\in \Cscr^0_c(\cosphere\Mc)\times [\sqrt{2},+ \infty)$, with $a_{\phg}$ as defined in~\eqref{eq:Yannis}, we set
\begin{equation*}
\textstyle
F^{\mathsf{a}}_{R,n} := \pscal{{U}^n}{\Op_{h_n}^{\Mc}\p*{a_{\phg}\p*{1-\theta\p*{\frac{~\jap{\zeta}_{\mathsf{g}}}{R}}}}{U}^n}_{L^2(\Mc)}.
\end{equation*}
Let $\De$ be a countable subset of $\test\p{\cosphere\Mc}$ dense in $\Cscr_c^0\p{\cosphere\Mc}$.
By Proposition~\ref{pr:Calderon:Vaillancourt}, there is a constant $C>0$ such that: for all pairs $(\mathsf{a}, R)\in \De\times \p{\QQ_+^*\cap [\sqrt{2}, +\infty)}$, there is  $C_{\theta, \mathsf{a}, R} > 0$ such that
\begin{equation}\label{eq:bound:F:a:R:n}
\abs{F^{\mathsf{a}}_{R, n}}\leq C\norm{\mathsf{a}}_{L^\infty\p{\cosphere\Mc}} + C_{\theta, \mathsf{a}, R}h_n^{1/2}, \quad n\geq 1.
\end{equation}
This inequality allows us to perform two diagonal extraction arguments, the first one on $n$ and the second one on $R$.
\begin{enumerate}
\item After a first diagonal extraction argument on index $n$, $F^{\mathsf{a}}_{R, n}$ admits a limit as $n$ tends to infinity for each pair $(\mathsf{a}, R)\in \De\times\p{\QQ_+^*\cap[1, +\infty)}$, in other words
\begin{equation*}
\forall (\mathsf{a}, R)\in \De\times\p{\QQ_+^*\cap[1, +\infty)},~ \exists F^{\mathsf{a}}_{R, \infty}\in \CC, \quad \lim_{n\to\infty}F^{\mathsf{a}}_{R, n} = F^{\mathsf{a}}_{R, \infty}.
\end{equation*}
\item Using~\eqref{eq:bound:F:a:R:n}, the limit $F_{R, \infty}^{\mathsf{a}}$ is uniformly bounded with respect to $R$ by $C\norm{\mathsf{a}}_{L^\infty\p{\cosphere \Mc}}$.
Then, after a second diagonal extraction with respect to elements of $\De$, there is a sequence $(R_m)_m$ such that, for all $\mathsf{a}\in \De$, the limit
\begin{equation*}
\varLambda(\mathsf{a}) := \lim_{m\to+\infty}\p{F_{R_m, \infty}^{\mathsf{a}}}~\text{exists}.
\end{equation*}
\end{enumerate}
The map $\varLambda\colon \mathsf{a}\in \mathrm{Span}\p{\De} \to\varLambda(\mathsf{a}) \in \CC$ is a continuous
(for the topology of $\Cscr^0_c\p{\cosphere \Mc}$)
and nonnegative
(see \garding~inequality~\cite[Theorem 2.5.4]{lerner2011metrics} or~\cite[Theorem E.23]{dyatlov2019mathematical})
linear map.
We can therefore extend $\varLambda$ by density to $\Cscr^0_c\p{\cosphere \Mc}$ as $\underline{\varLambda}$.
The linear map $\underline{\varLambda}$ is continuous for the topology of $\Cscr^0_c\p{\cosphere \Mc}$ and nonnegative: it is thus a nonnegative Radon measure $\mdm$ on $\cosphere \Mc$, which satisfies
\begin{equation*}
\lim_{m\to+\infty}\p*{\lim_{n\to+\infty} F_{R_m, n}^{\mathsf{a}}} = \underline{\varLambda}\p{\mathsf{a}} = \pscal{\mdm}{\mathsf{a}}_{\Mc_+(\cosphere\Mc), \Cscr^0_c(\cosphere\Mc)}.
\end{equation*}
Thus, using $\mathsf{a} = \p{a_{\phg}}_{|\cosphere \Mc} $, we conclude the proof of Lemma~\ref{lem:principal:microlocalization}.
\end{proof}
\subsubsection{Proof of Theorem~\ref{Th:Microlocalization}}
In this part, we use Lemma~\ref{lem:principal:microlocalization} and generalize its statement to remove the \apriori~dependence of the 2-microlocal measure on the sequence of radii $(R_m)_m$ and on the cut-off function $\theta$, deducing the proof of Theorem~\ref{Th:Microlocalization}.
\begin{proof}[Proof of Theorem~\ref{Th:Microlocalization}]
Since the proof is the same for $k > 1$ and $k = 1$, we treat only the case $k = 1$.
Fix $\theta$ satisfying~\eqref{df:theta}, and let $\Theta$ be a \emph{smooth fiber-radial cut-off function} (see Definition~\ref{df:SFRC}).
Thanks to Lemma~\ref{lem:microlocalization:remainder}, we can consider only functions $a_{\phg}\in S^0_{\phg}(\cotang \Mc)$.
By~\cite[Theorem E.42]{dyatlov2019mathematical} and Lemma~\ref{lem:principal:microlocalization}, up to considering a subsequence in $({u}^n, h_n)_n$,
\begin{enumerate}
\item $(u^n, h_n)$ admits a unique semiclassical defect measure $\sdm\in\Mc_+(\cotang \Mc)$ in the sense of~\eqref{eq:SDM}.
The idea is to use the nonnegative character of $\sdm$ to be able to obtain the following upper bounds~\eqref{eq:upper:bound:proof:microlocalization} and~\eqref{eq:upper:bound:2:proof:microlocalization}.
\item $(u^n, h_n)$ admits a unique 2-microlocal defect measure $\mdm\in\Mc_+(\cosphere \Mc)$ in the sense of~\eqref{eq:pre:MDM} with the cut-off function $\theta$ and a radius sequence $(R_m)_m$.
\end{enumerate}
Since a symbol $a_{\hom}\in S^0_{\phg}(\cotang \Mc)$ is \emph{compactly supported} in $\Mc$, there is a cut-off function $\chi\in\test(\Mc, [0, 1])$ such that $a_{\hom} = \chi a_{\hom}$.
Furthermore, setting
\begin{equation*}
    \textstyle F^{\psi}_{m, n} := \pscal{u^n}{\Op_{h_n}^{\Mc}\p*{\psi\p*{1 - \theta\p*{\frac{\jap{\zeta}_{\mathsf{g}}}{R_m}}}}u^n}_{L^2\p{\Mc}},\quad \psi\in S^0_{\phg}\p{\cotang \Mc}
\end{equation*}
and according to Lemma~\ref{lem:principal:microlocalization}, $\p{\lim_{n\to\infty} F^{\chi}_{m, n}}_m$ tends to $\pscal{\mdm}{\chi}_{\Mc_+, \Cscr_c^0\p{\cosphere\Mc}}$ as $m$ tends to $\infty$.
Then, for all $\varepsilon > 0$ there is $N_\varepsilon\in\NN$ such that:
\begin{enumerate}
\item On the one hand, since $\p{\lim_{n\to\infty} F^{\chi}_{m, n}}_m$ is a Cauchy sequence, for all $m_1, m_2\geq N_\varepsilon$,
\begin{equation}
{\textstyle \abs*{\pscal{\sdm}{\chi\p*{\theta\p*{\frac{\jap{\zeta}_{\mathsf{g}}}{R_{m_1}}} - \theta\p*{\frac{\jap{\zeta}_{\mathsf{g}}}{R_{m_2}}}}}_{\Mc_+, \Cscr^0_c(\cotang \Mc)}}}
=\abs{\lim_{n\to\infty} F^{\chi}_{m_2, n} - \lim_{n\to\infty}F^{\chi}_{m_1, n}} \leq \varepsilon.\label{eq:sdm:cauchy:upper:bound}
\end{equation}
\item On the other hand, by~\eqref{eq:pre:MDM} from Lemma~\ref{lem:principal:microlocalization}, for all $m_3\geq N_\varepsilon$,
\begin{equation}\label{eq:Radon:1maj}
\lim_{n\to\infty}\textstyle
\abs*{F^{a_{\phg}}_{m_3, n}
-
\pscal{\mdm}{a_{\phg}}_{\Mc_+, \Cscr_c^0(\cosphere\Mc)}}\leq \varepsilon.
\end{equation}
\item Since $\Theta$ is a \emph{smooth fiber-radial cut-off function} (see Definition~\ref{df:SFRC}), for all $R$ large enough there are two indices $l > m\geq N_\varepsilon$
---
both depending on the choice of $\Theta$ and $R$
---
such that
\begin{equation}\label{eq:goulot:etranglement}
\textstyle\theta\p*{\frac{\jap{\zeta}_{\mathsf{g}}}{R_m}}\leq \Theta_R( z, \zeta)\leq \theta\p*{\frac{\jap{\zeta}_{\mathsf{g}}}{R_l}}.
\end{equation}
\end{enumerate}
Combining~\eqref{eq:sdm:cauchy:upper:bound},~\eqref{eq:Radon:1maj} and~\eqref{eq:goulot:etranglement}, we deduce that for all $\varepsilon > 0$, there are $N_{\varepsilon}, R_{\varepsilon} > 0$ such that for all $R > R_{\varepsilon}$ and for all $l, m\geq N_\varepsilon$ satisfying~\eqref{eq:goulot:etranglement} (and thus depending on the choice of $R$), using~\eqref{eq:SDM}, we have
\begin{equation*}
\lim_{n\to+\infty}\textstyle\abs*{\pscal{u^n}{\classe*{\Op_{h_n}^{\Mc}\p{a_{\phg}\p{1-\Theta_R}}}u^n}_{L^2(\Mc)} - F^{a_{\phg}}_{m, n}}=\abs*{\pscal{\sdm}{a_{\phg} \p*{\Theta_R - \theta\p*{\frac{\jap{\zeta}_{\mathsf{g}}}{R_m}}}}_{\Mc_+, \Cscr^0_c\p{\cosphere\Mc}}}.
\end{equation*}
Then, because $\sdm$ is a nonnegative Radon measure, we have
\begin{equation}
\textstyle\abs*{\pscal{\sdm}{a_{\phg}\p*{\Theta_R - \theta\p*{\frac{\jap{\zeta}_{\mathsf{g}}}{R_m}}}}_{\Mc_+, \Cscr_c^0(\cotang\Mc)}}
\leq \norm{a_{\phg}}_{L^\infty(\cosphere\Mc)}\pscal{\sdm}{\chi\p*{\Theta_R - \theta\p*{\frac{\jap{\zeta}_{\mathsf{g}}}{R_m}}}}_{\Mc_+, \Cscr_c^0(\cotang\Mc)}.
\label{eq:upper:bound:proof:microlocalization}
\end{equation}
Thus, because $l, m\geq N_\varepsilon$ and $(R, l, m)$ satisfies~\eqref{eq:goulot:etranglement}, we can use~\eqref{eq:sdm:cauchy:upper:bound} to obtain the upper bound
\begin{equation}
\textstyle\pscal{\sdm}{\chi\p*{\Theta_R - \theta\p*{\frac{\jap{\zeta}_{\mathsf{g}}}{R_m}}}}_{L^2\p{\Mc}}
\leq \pscal{\sdm}{\chi\p*{\theta\p*{\frac{\jap{\zeta}_{\mathsf{g}}}{R_l}} - \theta\p*{\frac{\jap{\zeta}_{\mathsf{g}}}{R_m}}}}_{L^2\p{\Mc}}
\leq \varepsilon.\label{eq:upper:bound:2:proof:microlocalization}
\end{equation}
Using now~\eqref{eq:upper:bound:2:proof:microlocalization} in~\eqref{eq:upper:bound:proof:microlocalization}, for all $R > R_{\varepsilon}$ and for all $l, m\geq N_\varepsilon$ satisfying~\eqref{eq:goulot:etranglement}, we obtain
\begin{equation}\label{eq:Radon:2maj}
\lim_{n\to+\infty}\textstyle\abs*{\pscal{u^n}{\classe*{\Op_{h_n}^{\Mc}\p{a_{\phg}\p{1-\Theta_R}}u^n}_{L^2\p{\Mc}} - F^{a_{\phg}}_{m, n}}}\leq \norm{a_{\phg}}_{L^\infty(\cosphere\Mc)}\varepsilon.
\end{equation}
Finally, using~\eqref{eq:Radon:1maj} and~\eqref{eq:Radon:2maj} we deduce that for all $\varepsilon > 0$ there is $R > 0$ large enough such that
\begin{equation*}
\lim_{n\to\infty}\abs*{\pscal{u^n}{\Op_{h_n}^{\Mc}\p{a_{\phg}\p{1-\Theta_R}}u^n}_{L^2\p{\Mc}} - \pscal{\mdm}{a_{\phg}}_{\Mc_+, \Cscr_c^0\p{\cosphere\Mc}}} \leq\p{\norm{a_{\phg}}_{L^\infty(\cosphere \Mc)} + 1}\varepsilon,
\end{equation*}
which concludes the proof of Theorem~\ref{Th:Microlocalization}.
\end{proof}
\subsection{Use of Theorem~\ref{Th:Microlocalization} in the context of the wave equation \texorpdfstring{\eqref{eq:Wave}}{}}\label{subsection:Microlocalization:use:WE}
In this subsection, we consider $(M, g)$ a smooth, connected and compact Riemannian manifold without boundary, a potential $V$ as~\eqref{df:potential:fixed:V}, a nonempty bounded open interval $I\subsetneq \RR$, and we define the Riemannian manifolds
\begin{equation}\label{df:MM}
\textstyle\MM := \RR\times M~\text{and}~\MM_I := I\times M~\text{equipped with the metric}~\d t^2 + g.
\end{equation}
\paragraph{Support of 2-microlocal defect measures.}
Like usual microlocal defect measure (given by~\eqref{eq:1:MDM}), 2-microlocal defect measures defined by~\eqref{eq:MDM} satisfy the following support property.
\begin{proposition}\label{pr:microlocal:properties}
Assume $(h_n)_n$ is a sequence in $(0, 1]$ tending to $0$, $(u^n)_n$ is a bounded sequence in $L^2(\Mc)$ such that the sequence $(u^n, h_n)_n$ admits $\sdm\in\Mc_{+}(\cotang \Mc)$ (\resp $\mdm\in\Mc_{+}(\cosphere \Mc)$) as unique semiclassical (\resp 2-microlocal) defect measure in the sense of~\eqref{eq:SDM} (\resp\eqref{eq:MDM}).
Let $A_{h}\in \pseudo_{h,\phg}^m(\Mc)$ with $m\geq 1$ (see Subsection~\ref{section:pseudo:calculus} for a definition) be selfadjoint with an $h$-independent principal homogeneous symbol, such that
\begin{equation}\label{eq:microlocal:properties}
A_{h_n}u^n = \o{1}_{H^{-m}_{h, \comp}\p{\Mc}},~\text{as}~ n\to +\infty.
\end{equation}
Then $\supp(\sdm)\subset \set{\rho\in\cotang \Mc, ~\sigma_{h}(A_{h})(\rho) = 0}$ (\resp$\supp(\mdm)\subset \set{\rho\in\cosphere \Mc, ~\sigma_{\phg}^{m}(A_{h})(\rho) = 0}$).
\end{proposition}
\begin{proof}[Proof of Proposition~\ref{pr:microlocal:properties}]
The proof is divided in two parts: the semiclassical part and the 2-microlocal part.

First, for the semiclassical defect measure, the result is given by \cite[Proposition 5.3]{zworski2022semiclassical}

Second, let $\Theta$ be defined by~\eqref{df:Theta} and for all $R > 0$, set $\Theta^R = 1- \Theta_R$.
For all $b \in {S^{-m}_{\phg}}(\Mc)$ and all temporarily fixed $R > 0$, we have
\begin{equation}\label{eq:proof:microlocal:properties}
\Op_{h}^{\Mc}\p*{b\sigma_{h}(A_{h})\Theta^R} = A_{h}\Op_{h}^{\Mc}\p*{b \Theta^R} + h\rest^{-1}_{R, h}
\end{equation}
with remainder $\rest^{-1}_{R, h}\in \pseudo_{h, \comp}^{-1}\p{\Mc}$.
Thus, using~\eqref{eq:proof:microlocal:properties},~\eqref{eq:microlocal:properties} and selfadjointness of $A_{h}$, we obtain
\begin{align*}
\pscal{u^n}{\Op_{h_n}^{\Mc}\p*{b\sigma_{h}(A_{h})\Theta^R} u^n}_{L^2\p{\Mc}}
&= \pscal{u^n}{A_{h_n}\Op_{h_n}^{\Mc}\p*{b\Theta^R} u^n}_{L^2\p{\Mc}} + \O{h_n}_{R}\\
&= \pscal{A_{h_n}u^n}{\Op_{h_n}^{\Mc}\p*{b\Theta^R} u^n}_{H^{-m}_{h_n, \comp}\p{\Mc}, H^{m}_{h_n, \loc}\p{\Mc}} + \O{h_n}_{R}\\
&= \O{h_n}_{R}.
\end{align*}
Letting $n$ tend to $\infty$ and then $R$ tend to $+\infty$, using~\eqref{eq:MDM}, we deduce
\begin{equation*}
\pscal{\mdm}{b\sigma_{\phg}^{m}(A_{h})}_{\Mc_+, \Cscr^0_c\p{\cosphere\Mc}}
= \pscal{\mdm}{b\sigma_{\phg}^{m}(A_{h})}_{\Mc_+, \Cscr^0_c\p{\cosphere\Mc}}
= 0.
\end{equation*}
Therefore, $\sigma_{\phg}^{m}(A_{h})\mdm$ vanishes identically on $\cosphere \Mc$ and hence $\supp(\mdm)\subset \set{\sigma_{\phg}^{m}(A_{h}) = 0}$.
\end{proof}
\paragraph{Application to the wave equation.}
We give a lemma which allows us to apply Theorem~\ref{Th:Microlocalization} more easily for sequences $(u^n, h_n)_n$ which only concentrate on the region $\abs{\tau}\sim\abs{\xi}_{g}$ of $\cotang\hspace{1 pt}\MM_I$ in the asymptotic regime $\tau^2 + \abs{\xi}_{g}^2\to\infty$.
This lemma is used in the proof of Theorem~\ref{Th:Slices} in Appendix~\ref{proof:Th:Slices}.
According to Proposition~\ref{pr:microlocal:properties} and~\cite[Theorems 5.2 and 5.3]{zworski2022semiclassical}, this will be the case for sequences $(u^n, h_n)_n$ of solutions to~\eqref{eq:Wave}, with potential $\mathsf{V} = h_n^{-2} V$, bounded in $L^2(\MM_I)$.
\begin{lemma}\label{lem:Wave:time:space:reduction}
Consider $\MM_I$ as in~\eqref{df:MM} with $I$ a nonempty bounded open interval of $\RR$.
Let $\theta$ be as in~\eqref{df:theta}, $(h_n)_n$ be a sequence tending to $0$ and $(u^n)_n$ be a bounded sequence in $L^2(\MM_I)$.
Assume $(u^n, h_n)_n$ admits $\mdm\in \Mc_+(\cosphere \hspace{1 pt}\MM_I)$ as unique 2-microlocal defect measure in the sense of~\eqref{eq:MDM} and that there are $c, C > 0$ such that
\begin{equation}
\supp(\mdm)\subset \set{(t, x, \tau, \xi)\in\cosphere\hspace{1 pt}\MM_I, ~c\abs{\tau} \leq \abs{\xi}_{g}\leq C\abs{\tau}}.\label{eq:supp:mdm}
\end{equation}
Then, for all $a\in \tilde{S^0_{\phg}}(\cotang \hspace{1 pt}\MM_I)$ associated by~\eqref{eq:characterization:S^0_phg} to an homogeneous symbol $a_{\phg}$, we have the following two limits:
\begin{align*}
\pscal{\mdm}{a_{\phg}}_{\Mc_+(\cosphere \hspace{1 pt} \MM_I), \Cscr_c^{0}(\cosphere \hspace{1 pt}\MM_I)}
&= \lim_{R\to\infty}\lim_{n\to\infty}\textstyle\pscal{u^n}{\Op_{h_n}^{\MM_I}\p{a\p{1 - \theta\p*{\frac{\jap{\xi}_{{g}}}{R}}}}u^n}_{L^2\p{\MM_I}}\\
&= \lim_{R\to\infty}\lim_{n\to\infty}\textstyle\pscal{u^n}{\Op_{h_n}^{\MM_I}\p{a\p{1 - \theta\p*{\frac{\tau}{R}}}}u^n}_{L^2\p{\MM_I}}.
\end{align*}
\end{lemma}
\begin{proof}[Proof of Lemma~\ref{lem:Wave:time:space:reduction}]
We consider $a^h\in\tilde{S^0_{\hom}}(\cotang\hspace{1 pt}\MM_I)$ real-valued, and, for $j\in\set{1, 2}$, we set
\begin{equation*}
\Cc_j := \set{(t, x, \tau, \xi)\in\cotang \hspace{1 pt}\MM_I, ~\frac{c}{j}\jap{\tau} \leq \jap{\xi}_{g}\leq jC\jap{\tau}}.
\end{equation*}
Let $\alpha\colon \cotang \hspace{1 pt}\MM_I\to\RR$ be either
\begin{equation*}
(t, x, \tau, \xi)\in\cotang \hspace{1 pt}\MM_I\mapsto \tau
\quad\text{or}\quad
(t, x, \tau, \xi)\in\cotang \hspace{1 pt}\MM_I\mapsto \jap{\xi}_g,
\end{equation*}
and, for $R > 0$, let $\Theta_{\alpha}^R := 1 - \theta\p{\frac{\alpha}{R}}$.
Then, let $\chi\in \tilde{S^0_{\phg}}\p{\cotang \hspace{1 pt}\MM_I}$ be an $h$-independent symbol such that: $\chi(\cdot, \cdot, \mathsf{R}\cdot, \mathsf{R}\cdot)\in S^0_{\phg}\p{\cotang \hspace{1 pt}\MM_I}$ for $\mathsf{R}$ large enough, and
\begin{equation*}
\textstyle 
\supp\p{\chi}\subset \Cc_2
\quad\text{and}\quad
\supp\p{1 - \chi}\cap \Cc_1 = \emptyset.
\end{equation*}

First, let us consider a \emph{smooth fiber-radial cut-off function} $\Theta$ (see Definition~\ref{df:SFRC}), with notation $\Theta^R = 1 - \Theta_R$.
By~\eqref{eq:supp:mdm} together with~\eqref{eq:MDM}, we have
\begin{equation}\label{eq:1}
\lim_{R\to\infty}\lim_{n\to\infty}
\pscal{u^n}{\Op_{h_n}^{\MM_I}\p*{a^{h_n} (1-\chi)\Theta^R}u^n}_{L^2\p{\MM_I}} = 0.    
\end{equation}
Then, still using~\eqref{eq:MDM}, we have
\begin{equation}\label{eq:2}
\pscal{\mdm}{a_{\phg}}_{\Mc_+, \Cscr^0_c\p{\cosphere\hspace{1 pt}\MM_I}}
=
\lim_{R\to\infty}\lim_{n\to\infty}
\pscal{u^n}{\Op_{h_n}^{\MM_I}\p*{a^{h_n} \chi\Theta^R}u^n}_{L^2\p{\MM_I}}.
\end{equation}

Second, since $\alpha(t, x, \tau, \xi)^{2}\leq {1 + \tau^2 + \abs{\xi}_{g}^2}$ for all $(t, x, \tau, \xi)\in \cotang \hspace{1 pt}\MM_I$, we have
\begin{equation*}
\forall R > 0,~
\exists R_1 \geq R,\quad
0\leq \Theta_\alpha^{R_1}\leq\Theta^{R}.
\end{equation*}
Thus, using sharp \garding~inequality from~\cite[Proposition E.23]{dyatlov2019mathematical}, we obtain that
\begin{equation*}
0
\leq 
\limsup_{n\to\infty}
\pscal{u^n}{\Op_{h_n}^{\MM_I}\p*{a^{h_n} (1 - \chi)\Theta_{\alpha}^{R_1}}u^n}_{L^2\p{\MM_I}}
\leq 
\lim_{n\to\infty}
\pscal{u^n}{\Op_{h_n}^{\MM_I}\p*{a^{h_n} (1 - \chi)\Theta^R}u^n}_{L^2\p{\MM_I}}.
\end{equation*}
Then, by~\eqref{eq:1}, we conclude 
\begin{equation}\label{eq:1:bis}
    \lim_{R\to\infty}\limsup_{n\to\infty}
    \pscal{u^n}{\Op_{h_n}^{\MM_I}\p*{a^{h_n} (1 - \chi)\Theta_{\alpha}^{R}}u^n}_{L^2\p{\MM_I}} = 0.
\end{equation}

Third, since $\abs{\tau}\sim_{\tau\to\pm\infty}\abs{\xi}_g$ on $\supp(\chi)$, we have
\begin{equation*}
    \forall R > 0,~
    \exists R_2 \geq R_1\geq R,\quad
    \eqcases{lll}{
    \chi\Theta^{R_2}
    &\leq \chi\Theta_{\alpha}^{R_1}
    &\leq \chi\Theta^{R}\\
    (1-\chi)\Theta^{R_2}
    &\leq (1-\chi)\Theta^{R_1}
    &\leq (1-\chi)\Theta^{R}
    }.
\end{equation*}
Thus, using that $\chi\Theta_{\alpha} + (1-\chi)\Theta$ is a \emph{smooth fiber-radial cut-off function} (see Definition~\ref{df:SFRC}) together with the sharp \garding~inequality from~\cite[Proposition E.23]{dyatlov2019mathematical}, we obtain that
\begin{align*}
\lim_{n\to\infty}
\pscal{u^n}{\Op_{h_n}^{\MM_I}\p*{a^{h_n}\Theta^{R_2}}u^n}_{L^2\p{\MM_I}}
&\leq 
\lim_{n\to\infty}
\pscal{u^n}{\Op_{h_n}^{\MM_I}\p*{a^{h_n} \classe{\chi\Theta_{\alpha}^{R_1} + (1-\chi)\Theta^{R_1}}}u^n}_{L^2\p{\MM_I}}\\
&\leq
\lim_{n\to\infty}
\pscal{u^n}{\Op_{h_n}^{\MM_I}\p*{a^{h_n} \Theta^{R}}u^n}_{L^2\p{\MM_I}}.
\end{align*}
Then, by~\eqref{eq:1} together with~\eqref{eq:2}, we conclude
\begin{equation}\label{eq:2:bis}
\pscal{\mdm}{a_{\phg}}_{\Mc_+, \Cscr^0_c\p{\cosphere\hspace{1 pt}\MM_I}}
=
\lim_{R\to\infty}\lim_{n\to\infty}
\pscal{u^n}{\Op_{h_n}^{\MM_I}\p*{a^{h_n} \chi\Theta_{\alpha}^R}u^n}_{L^2\p{\MM_I}}.    
\end{equation}

In conclusion, by~\eqref{eq:1:bis} and~\eqref{eq:2:bis}, we prove of Lemma~\ref{lem:Wave:time:space:reduction} for real-valued symbols.
Therefore, by linear combination, we conclude the proof of Lemma~\ref{lem:Wave:time:space:reduction} for any $a^{h}\in\tilde{S^{0}_{\phg}}(\cotang\hspace{1 pt}\MM_I)$.
\end{proof}
\section{Half-wave equation and slice defect measures}\label{section:slice:defect:measure}
In this section, we introduce the half-wave equation.
We study general properties of its solutions and then apply these results to solutions of the wave equation.
Subsection~\ref{subsection:half:wave:equation} introduces the half-wave equation.
Subsection~\ref{subsection:slice:defect:measure} then provides a key theorem (Theorem~\ref{Th:Slices}), inspired by~\cite[Proposition 4.4]{LecturePG}, which relates the semiclassical and 2-microlocal defect measures associated with solutions of the wave equation~\eqref{eq:Wave} to the corresponding ``slice'' semiclassical and 2-microlocal defect measures associated with solutions of the half-wave equation~\eqref{eq:HWE}.

Throughout this section, we consider $(M, g)$ a smooth, connected, compact Riemannian manifold without boundary, a potential $V$ as~\eqref{df:potential:fixed:V}, and fix a nonempty bounded open interval $I\subset \RR$. We recall that $\MM$ and $\MM_I$ are defined in~\eqref{df:MM}.
We use the notion of tangential symbol:
\begin{definition}[Tangential symbol]
The class of tangential symbols of order $m$, denoted $S_{\tang}^m(I\times \cotang M)$, is the set of smooth functions $a^h\in \smooth(I\times \cotang M)$ satisfying for any relatively compact local chart $K$ of $I\times M$ 
\begin{equation*}
\forall \alpha, \beta \in\NN_{0}^{d},~\forall \gamma\in\NN_{0},\quad \sup_{h\in(0, 1]}\sup_{\begin{subarray}{c} (t, x) \in K\\ \xi \in \mathrm{T}_x^{\star} M\end{subarray}}
\jap{\xi}^{\abs{\beta} - m}\abs{\partial_t^\gamma\partial_{x}^\alpha\partial_{\xi}^\beta a^h}<\infty.
\end{equation*}
$S^{m}_{\comp, \tang}(I\times \cotang M)$ denotes the class of tangential symbols of order $m$ \emph{compactly supported} in $I$ and $S_{\phg, \tang}^{m}(I\times \cotang M)$ denotes the class of homogeneous tangential symbols of order $m$, that is, the set of symbols $a_{\phg}\in S^{m}_{\comp, \tang}(I\times \cotang M)$ satisfying
\begin{equation*}
a_{\phg}(t, x,  R\xi) = R^m a_{\phg}(t, x, \xi), \quad \forall R \geq 1,~ \forall (t, x, \xi)\in I \times \cosphere M.
\end{equation*}
\end{definition}
\begin{definition}[Tangential quantization]\label{df:tangential:quantization}
For all $a^h\in S_{\tang}^m\p{I\times \cotang M}$, we define its tangential quantization by
\begin{equation*}
\classe{\Op_{h, \tang}^{M}\p{a^h} u} (t, x) = \classe{\Op_{h}^{M}\p{a^h(t, \cdot)} u(t, \cdot)}(x), \quad (t, x)\in\MM_I.
\end{equation*}
\end{definition}
Then, we state the following result from~\cite[Theorem 18.1.35]{Hörmander1985AnalysisOfLinearOpVol3}, which will be essential when using tangential symbols and usual symbols simultaneously.
\begin{theorem}[Hörmander 1985]\label{Th:hormander}
Let $m, m'\in\RR$, $a_1^h\in S^{m}\p{\cotang\hspace{1 pt} \MM_I}$ and $a_2^h\in S^{m'}_{\tang}\p{I\times \cotang M}$.
If there is $\varepsilon > 0$ such that 
\begin{equation*}
a_1^h(t, x, \tau, \xi) = 0 ~\text{near}~ \set{(t, x, \tau, \xi)\in\cotang\hspace{1 pt}\MM_I, ~\varepsilon\abs{\tau} > 1~\text{and}~\varepsilon\abs{\tau}\geq\abs{\xi}_g}
\end{equation*}
--- see Figure~\ref{fig:hormander} for an illustration.
Then, for $j, k\in\set{1, 2}$, $a_j^ha_k^h\in S^{m + m'}(\cotang \hspace{1 pt}\MM_I)$ and there is $c_{jk}^h\in S^{m + m'}\p{\cotang\hspace{1 pt} \MM_I}$ supported in $\supp\p{a_1^h}\cap\supp\p{a_2^h}$ such that 
\begin{equation*}
c_{jk}^h = a_j^ha_k^h\mod hS^{m + m'-1}(\cotang \hspace{1 pt}\MM_I)\quad\text{and}\quad\Op_{h}^{\MM}(c_{jk}^h) = \Op_{h}^{\MM}\p{a_j^h}\Op_{h}^{\MM}\p{a_k^h}\in\pseudo^{m + m'}_{h}(\MM_I).
\end{equation*}
\end{theorem}
\begin{figure}
\begin{center}
\begin{tikzpicture}[scale=1]
\def\eps{1.7}
\def\xmax{4}
\def\ymax{4}
\fill[gray!30]
  (-\ymax/\eps,\ymax) --
  (\ymax/\eps,\ymax) --
  (0,0) --
  cycle;
\fill[gray!30]
  (-\ymax/\eps,-\ymax) --
  (\ymax/\eps,-\ymax) --
  (0,0) --
  cycle;
\fill[white]
    (-\xmax, \eps) --
    (\xmax, \eps) --
    (\xmax, -\eps) --
    (-\xmax, -\eps) --
    cycle;
\def\sample{120}
\draw[thick] plot[domain=-\xmax:\xmax, samples=\sample] (\x, {sqrt(\x*\x + 0.5)});
\draw[thick] plot[domain=-\xmax:\xmax, samples=\sample] (\x, {-sqrt(\x*\x + 0.5)});
\draw[dashed] (-\xmax, \eps) -- (\xmax, \eps);
\draw[dashed] (-\xmax, -\eps) -- (\xmax, -\eps);
\draw[->] (-\xmax,0) -- (\xmax,0) node[right] {$\xi$};
\draw[->] (0,-\ymax) -- (0,\ymax) node[above] {$\tau$};
\node[above] at (\xmax, \eps) {$\varepsilon \abs{\tau} = 1$};
\node at (1, 2*\eps) {$\varepsilon\abs{\tau}> \abs{\xi}_{g}$};
\node[right] at (1.2, 1) {$\set{\tau^2 = \abs{\xi}_g^2 + V(x)}$};
\end{tikzpicture}
\end{center}
\caption{For any $(t, x)\in\MM_I$ and all $\varepsilon > 0$, here is an illustration of the problematic region $\set{(\tau, \xi)\in\mathrm{T}_{t, x}^\star\hspace{1 pt}\MM_I, ~\varepsilon\abs{\tau} > 1~\text{and}~\varepsilon\abs{\tau}\geq\abs{\xi}_g}$ in light gray.
We compare it, for $\varepsilon$ small enough, to the set $\set{(\tau, \xi) \in\mathrm{T}_{t, x}^\star\hspace{1 pt}\MM_I, ~ \tau^2 = \abs{\xi}_{g}^2 + V}$ in black.
One can choose $\varepsilon > 0$ uniformly small with respect to the choice of $(t, x)$.}
\label{fig:hormander}
\end{figure}
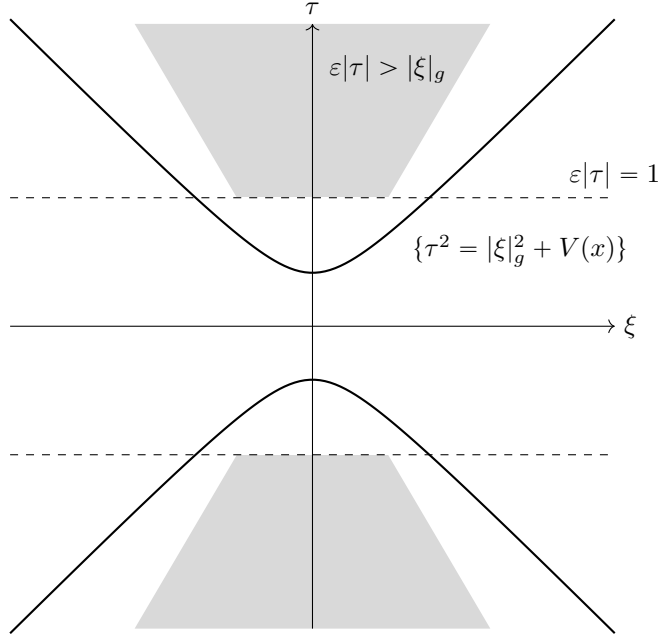
\subsection{Half-wave equation}\label{subsection:half:wave:equation}
In this subsection, we define the half-wave equation and give the first results on its solutions.
\paragraph{Half-Schrödinger operator.}
We begin by defining the Schrödinger operator 
\begin{equation*}
\SchroOp_{h} := - h^2\Delta_g + V\in \pseudo_{h}^2(M).
\end{equation*}
The principal symbol of $\SchroOp_{h}$ is $p=\sqrtsymbol^2 \in S^2(\cotang M)$ both defined in~\eqref{df:Sqrt:Sym}. Recall that $\sqrtsymbol = \sqrt{p} \in S^1(\cotang M)$ is real-valued and that we assume $V > 0$ on $M$.
Following~\cite[Remark 2.11]{laurent2016uniform}, we define an approximate square root of $\SchroOp_{h}$ by  
\begin{equation}\label{df:SqrtOp}
\SqrtOp_{h} := \frac{1}{2}(\Op^M_{h}(\sqrtsymbol) +\Op^M_{h}(\sqrtsymbol)^*)\in \pseudo_{h}^1(M).
\end{equation}
The operator $\SqrtOp_{h}$ is selfadjoint, its principal symbol is $\sqrtsymbol$ (see Proposition~\ref{pr:composition}) and is elliptic in the sense that
\begin{equation*}
\sqrtsymbol(x, \xi) \geq \p{\textstyle\min_MV}^{\frac{1}{2}}\jap{\xi}_g, ~ (x, \xi)\in \cotang M.
\end{equation*}
Moreover, using Proposition~\ref{pr:composition} and selfadjointness of $\SqrtOp_{h}$ and $\SchroOp_{h}$, there is a selfadjoint operator $\rest^{\SqrtOp, 1}_{h}\in \pseudo^1_{h}(M)$ such that
\begin{equation}\label{df:remainder:order:1}
\SqrtOp_{h}^2 = \SchroOp_{h} + h\rest^{\SqrtOp, 1}_{h}.
\end{equation}
\paragraph{Half-wave operator.}
We define the symbol $\timespacesymbol\in S^2\p{\cotang\hspace{1pt} \MM}$ as
\begin{equation}\label{df:Sym:time:times:space}
\timespacesymbol(t, x, \tau, \xi) = p(x, \xi) - \tau^2 = \abs{\xi}_g^2 + V(x) - \tau^2,\quad (t, x, \tau, \xi)\in \cotang\hspace{1.5pt}\MM.
\end{equation}
\begin{remark}\label{rk:sign}
Here and throughout, for any sign $\signe = \pm$, $-\signe$ denotes the opposite sign $\mp$.
\end{remark}
We define
\begin{equation*}
\ell^{\signe} = \tau + \signe\sqrtsymbol \in S^1\p{\cotang I} + S_{\tang}^1(I\times \cotang M),\quad \signe\in\set{-, +},
\end{equation*}
so that $-\timespacesymbol = \ell^{\signe}\ell^{-\signe}$.
We define the positive (\resp negative) half-wave operator as
\begin{equation}\label{eq:Half:Wave:Operator}
\HWOp^{\signe}_{h} := h D_t + \signe \SqrtOp_{h}, \quad \signe\in\set{-, +}.
\end{equation}
Note that $\ell^{\signe}\notin S^1(\cotang\hspace{1 pt}\MM_I)$ and $\HWOp^{\signe}_{h}\notin\pseudo_{h}^1(\MM_I)$.
Now, using $\classe{hD_t, \SqrtOp_{h}} = 0$ and taking the time-independent operator $\rest^{\SqrtOp, 1}_{h}\in\Psi^{1}_{h}(M)$ from~\eqref{df:remainder:order:1}, we have
\begin{equation*}
-\HWOp^{\signe}_{h}\HWOp^{-\signe}_{h} = -h^2D_t^2 + \SqrtOp_{h}^2 =  (h^2\partial_t^2 + \SchroOp_{h}) + h\rest^{\SqrtOp, 1}_{h}.
\end{equation*}
By~\cite[Theorem 10.2]{zworski2022semiclassical}, we know that the operator $\SqrtOp_{h}$ with domain $H^1_{h}(M)$ is essentially selfadjoint on $L^2(M)$ and $(e^{\frac{it}{h}\SqrtOp_{h}})_{t\in\RR}$ is a group of unitary operators on $L^2(M)$ such that, in a weak sense, we have:
\begin{equation}\label{eq:unitary:operators}
\forall u\in L^2(M), \quad \HWOp^{\signe}_{h}(e^{-\signe\frac{it}{h}\SqrtOp_{h}}u) = hD_t (e^{-\signe\frac{it}{h}\SqrtOp_{h}}u) + \signe\SqrtOp_{h}(e^{-\signe\frac{it}{h}\SqrtOp_{h}}u) = 0, \quad \signe\in\set{-, +}.
\end{equation}
\paragraph{Application to solutions of the wave equation.}
We use the notion of semiclassical Sobolev spaces defined in the introduction, see the paragraph on pseudodifferential operators on Riemannian manifolds in Subsection~\ref{section:pseudo:calculus}.
We define $u^h$ as the $\Cscr^0(I, H^1_{h}(M))\cap\Cscr^1(I, L^2(M))$-solution of
\begin{equation}\label{df:sequence}
    \eqcases{rll}{
    (h^2\partial_t^2 - h^2\Delta_g  + V)u^h&=0&\text{in}~\RR\times M\\
    (u^h, h\partial_t u^h)_{|t = 0} &=(u_0^h, hu_1^h)&\text{bounded in}~H^1_{h}(M)\times L^2(M).
    }
\end{equation}
For all $h \in (0, 1]$, we define the rescaled energy term as
\begin{equation}\label{df:NRJ:rescaled}
\Escr^{h}(\mathsf{u}, \mathsf{v}) := h^2\Escr_{h^{-2}V}(\mathsf{u}, \mathsf{v}) = \frac{1}{2}\int_{M}\abs{h\nabla_g \mathsf{u}}_g^2 + V\abs{\mathsf{u}}^2 + \abs{h\mathsf{v}}^2\d x, \quad (\mathsf{u}, \mathsf{v})\in H^1_h(M)\times L^2(M).
\end{equation}
As $\Escr_{h^{-2}V}$ defined in~\eqref{eq:Obs}, the rescaled energy $\Escr^{h}$ of $u^{h}$, namely $\Escr^{h}\p{u^{h}, \partial_t u^{h}}$, is preserved in time for any $h$ (see~\eqref{eq:WP}).
\begin{remark}
In the notation of~\eqref{eq:Wave}, $u^h$ is the (unique) solution with potential $\mathsf{V}=h^{-2}V$ and initial data $(u_0^h,u_1^h)$, where $(u_0^h,hu_1^h)_{h}$ is uniformly bounded in $H^1_{h}(M)\times L^2(M)$ or, equivalently, where 
\begin{equation*}
\sup_{h}\Escr^h(u_0^h, u_1^h) < \infty.
\end{equation*}
\end{remark}
Then, we also consider the intermediate and remainder terms
\begin{equation}\label{df:sequence:intermediate:terms}
\text{for a sign}~ \signe\in\set{-, +},~ h\in(0, 1[, \quad v^{\signe}_{h} := \HWOp^{\signe}_{h}u^h
\quad\text{and}\quad
f^h := h^{-1}\HWOp^{-\signe}_{h}v_{h}^{\signe} = -\rest^{\SqrtOp, 1}_{h} u^h,
\end{equation}
where $\rest^{\SqrtOp, 1}_{h}$ is the time-independent operator defined in~\eqref{df:remainder:order:1} and does not depend on the choice of sign $\signe$ (see Remark~\ref{rk:sign} for the definition of $-\signe$).
Note that for any $h\in (0, 1]$, we also have the following key equalities
\begin{equation}\label{eq:sum:of:split:terms}
v^+_{h} - v^-_{h} = 2\SqrtOp_{h}u^h
\quad\text{and}\quad
v^+_{h} + v^-_{h} = 2hD_tu^h.
\end{equation}
Then, as $\SqrtOp_{h}$ is an elliptic pseudodifferential operator in $\pseudo_{h}^1(M)$, using parametrix construction, there are remainder terms $\tilde{\rest}_{h}^{\SqrtOp, 0}$ and $\rest^{\SqrtOp, 0}_{h}$ in $\pseudo_{h}^{0}(M)$ (see for example~\cite[section E.7, exercise 11]{dyatlov2019mathematical}) with real-valued principal symbols such that
\begin{equation*}
f^h = -\rest^{\SqrtOp, 1}_{h}u^h = \tilde{\rest}^{\SqrtOp, 0}_{h}\SqrtOp_{h}u^h + h\rest^{\SqrtOp, 0}_{h}u^h.
\end{equation*}
As a consequence, by~\eqref{eq:sum:of:split:terms}, we have
\begin{equation}\label{eq:decomposition:f}
f^h = \frac{1}{2}\tilde{\rest}^{\SqrtOp, 0}_{h}\p{v^+_{h} - v^-_{h}} + h\rest^{\SqrtOp, 0}_{h}u^h.
\end{equation}   
These sequences satisfy the half-wave equation~\eqref{eq:HWE}, \ie for any $h\in (0, 1]$ we have
\begin{equation}\label{eq:HWE}\tag{HWE}
\eqcases{rcll}{\HWOp^{-\signe}_{h}v^{\signe}_{h} &=& hf^{h} &\text{in}~ I\times M,\\
\vspace{-8pt}\\
\p{v^{\signe}_{h}}_{|t = 0} &=& (-ihu_1^h + \signe \SqrtOp_{h} u_0^h) &\text{in}~M.}
\end{equation}
Note that, for any operator $A_{h}\in\pseudo^0_{h, \comp}(\MM_I)$ and any $h\in (0, 1]$, since the crossed terms cancel each other out, we have
\begin{align}
\notag\frac{1}{2}\sum_{\signe\in\set{-, +}}\pscal{v^{\signe}_{h}}{A_{h}v^{\signe}_{h}}_{L^2(\MM_I)}
&= \frac{1}{2}\sum_{\signe\in\set{-, +}}\pscal{\p{hD_t + \signe\SqrtOp_{h}}u^h}{A_{h}\p{hD_t + \signe\SqrtOp_{h}}u^h}_{L^2(\MM_I)}\\
&= \pscal{h\partial_t u^h}{A_{h} h\partial_t u^h}_{L^2(\MM_I)} + \pscal{\SqrtOp_{h} u^h}{A_{h} \SqrtOp_{h} u^h}_{L^2(\MM_I)}.\label{eq:identity:v:to:u}
\end{align}
Then, with the same calculation, for all $t\in I$, all $A_{h}\in \pseudo^0_{h}\p{M}$ and all $h\in (0, 1]$, we also have
\begin{equation}
\frac{1}{2}\sum_{\signe\in\set{-, +}}\pscal{v^{\signe}_{h}(t)}{A_{h}v^{\signe}_{h}(t)}_{L^2(M)} = \pscal{h\partial_t u^h(t)}{A_{h} h\partial_t u^h(t)}_{L^2(M)} + \pscal{\SqrtOp_{h} u^h(t)}{A_{h} \SqrtOp_{h} u^h(t)}_{L^2(M)}.\label{eq:identity:v:to:u:t}
\end{equation}
\paragraph{Boundedness for solutions to half-wave and wave equation.}
Using~\eqref{eq:WP} we have a first boundedness result.
\begin{lemma}\label{lem:SPLITTING}
The family $(u^h)_{h}$, defined in~\eqref{df:sequence},  is bounded in $L^{\infty}(I, H^1_{h}(M))$ and the families $\p{h D_t u^h}_{h}$, $(v_{h}^{\signe})_{h}$ and $\p{f^h}_{h}$ defined in~\eqref{df:sequence:intermediate:terms} are bounded in $L^{\infty}(I, L^2(M))$.
\end{lemma}
\begin{proof}[Proof of Lemma~\ref{lem:SPLITTING}]
First, using the boundedness of $(u_0^h, hu_1^h)_{h}$ in $H_{h}^1(M)\times L^2(M)$ and~\eqref{eq:WP}, we have that
\begin{equation*}
\max\set{\sup_{t\in I}\norm{u^h(t)}_{H^1_{h}(M)}^2,~\sup_{t\in I}\norm{hD_t u^h(t)}_{L^2(M)}^2} \leq \sup_{t\in I}\Escr^h(u^h(t), \partial_t u^h(t)) = \Escr^h(u_0^h, u^h_1) \leqsim 1,\quad h\in (0, 1].
\end{equation*}
Thus, $\p{\sup_{t\in I}\norm{u^h(t)}_{H^1_{h}(M)}}_{h}$ and $\p{\sup_{t\in I}\norm{hD_t u^h(t)}_{L^2(M)}}_{h}$ are bounded sequences.
Then, using $f^h = -\rest^{\SqrtOp, 1}_{h}u^h$ from~\eqref{df:sequence:intermediate:terms} and boundedness of $\rest^{\SqrtOp, 1}_{h}$ from Proposition~\ref{pr:boundedness:pseudo}, we have that
\begin{equation}\label{eq:5}
\sup_{t\in I}\norm{f^h(t)}_{L^2(M)}\leqsim \sup_{t\in I}\norm{u^h(t)}_{H^1_{h}(M)}\leqsim 1,\quad h\in (0, 1]. 
\end{equation}

Second, using~\eqref{eq:unitary:operators} and (classical) hyperbolic energy inequality on~\eqref{eq:HWE} (see for example~\cite[Chapter 23]{Hörmander1985AnalysisOfLinearOpVol3}), we obtain
\begin{equation*}
\notag\sup_{t\in I}\norm{v^{\signe}_{h}(t)}_{L^2(M)} \leq \norm{v^{\signe}_{h}(0)}_{L^2(M)} + h\norm{f^h}_{L^1(I, L^2(M))},\quad h\in (0, 1].
\end{equation*}
Then, by boundedness from Proposition~\ref{pr:boundedness:pseudo} in addition with $v^{\signe}_{h}(0) = -ihu_1^h + \signe \SqrtOp_{h} u_0^h$, we have
\begin{equation*}
\sup_{t\in I}\norm{v^{\signe}_{h}(t)}_{L^2(M)} \leqsim h\norm{u_1^h}_{L^2(M)} + \norm{u_0^{h}}_{H_{h}^1(M)} + h\norm{f^h}_{L^1(I, L^2(M))},\quad h\in (0, 1].
\end{equation*}
Finally, using $\Escr^h(u_0^h, u_1^h) \leqsim 1$ (see~\eqref{df:sequence}) and~\eqref{eq:5}, we have
\begin{equation*}
\sup_{t\in I}\norm{v^{\signe}_{h}(t)}_{L^2(M)} \leqsim 1 + h \leqsim 1,\quad h\in (0, 1].
\end{equation*}
This concludes the proof of Lemma~\ref{lem:SPLITTING}.
\end{proof}
\paragraph{First results on localization of semiclassical and 2-microlocal defect measures.}
We conclude this subsection by giving a result on the localization of semiclassical and 2-microlocal defect measures associated to a subsequence of $(u^h)_{h}$ defined in~\eqref{df:sequence}.
\begin{lemma}\label{lem:hDt:SqrtOp:sdm:mdm}
Let $(h_n)_n$ be a sequence of $(0, 1]$ tending to $0$, $\p{u^n}_n = \p{u^{h_n}}_n $ be the sequence defined in~\eqref{df:sequence}, $\sqrtsymbol$ as defined in~\eqref{df:Sqrt:Sym} and $\SqrtOp_{h}$ as defined in~\eqref{df:SqrtOp}.
There is a subsequence of $\p*{\p{u^n, h_n D_t u^n, \SqrtOp_{h_n}u^n}, h_n}_n$, still denoted $\p*{\p{u^n, h_n D_t u^n, \SqrtOp_{h_n}u^n}, h_n}_n$, which admits a Radon measure with hermitian matrix values 
\begin{equation*}
\begin{pmatrix}
\sdm_{(u, u)} &* & *\\
*&\sdm_{(\tau, \tau)} & *\\
* &*& \sdm_{(\sqrtsymbol, \sqrtsymbol)}
\end{pmatrix}\in\Mc_+\p{\cotang\hspace{1 pt}\MM_I, \CC^{3\times 3}}~\p*{\resp
\begin{pmatrix}
\mdm_{(u, u)} &* &*\\
*&\mdm_{(\tau, \tau)} & *\\
*&* & \mdm_{(\sqrtsymbol, \sqrtsymbol)}
\end{pmatrix}\in\Mc_+\p{\cosphere\hspace{1 pt}\MM_I, \CC^{3\times 3}}
}
\end{equation*}
as unique semiclassical (\resp 2-microlocal) defect measure in the sense of~\eqref{eq:SDM} (\resp \eqref{eq:MDM}).
Then, for all $\bullet\in\set{u, \tau, \sqrtsymbol}$, we have
\begin{equation}\label{eq:16}
\eqcases{l}{
\supp\p{\sdm_{(\bullet, \bullet)}}\subset \set{(t, x, \tau, \xi)\in\cotang\hspace{1 pt}\MM_I, ~\tau^2 = \sqrtsymbol(x, \xi)^2},\\
\supp\p{\mdm_{(\bullet, \bullet)}}\subset \set{(t, x, \tau, \xi)\in\cosphere\hspace{1 pt}\MM_I, ~\tau^2 = \abs{\xi}_g^2}.
}
\end{equation}
\end{lemma}
\begin{proof}[Proof of Lemma~\ref{lem:hDt:SqrtOp:sdm:mdm}]
By Lemma~\ref{lem:SPLITTING}, the sequence $\p{u^n, h_nD_tu^n, \SqrtOp_{h_n}u^n}_n$ is bounded in $L^\infty\p{I,  L^2(M)^3}$ and thus by Proposition~\ref{pr:SDM} (\resp Theorem~\ref{Th:Microlocalization}), up to considering a subsequence, we have existence and uniqueness of $\p{\sdm_{\bullet}}_{\bullet\in\set{u, \tau, \sqrtsymbol}^2}$ and $\p{\mdm_{\bullet}}_{\bullet\in\set{u, \tau, \sqrtsymbol}^2}$.

On the one hand, $u^h$ and --- since $\classe{\SchroOp_{h}, hD_t}=0$ --- $hD_t u^h$ are solutions to~\eqref{eq:Wave}.
Thus, by Proposition~\ref{pr:microlocal:properties},~\eqref{eq:16} holds for $\bullet\in\set{u, \tau}$.
On the other hand, using $\classe{h\partial_t, \SqrtOp_{h}} = 0$ and~\eqref{df:remainder:order:1}, we have
\begin{equation*}
\classe{h^2 \partial_t^2 + \SchroOp_{h}, \SqrtOp_{h}} =\classe{\SqrtOp_{h}^2 - h\rest^{Q, 1}_{h}, \SqrtOp_{h}} = -h\classe{\rest^{Q, 1}_{h}, \SqrtOp_{h}} \in h^2 \pseudo^{1}_{h}\p{M}\subset h^2 \pseudo^{2}_{h}\p{M}.
\end{equation*}
Thus, for any $\chi\in\test\p{I}$, $\chi\p{h^2 \partial_t^2 - \SchroOp_{h}}\SqrtOp_{h} u^h = \O{h^2}_{H^{-2}_{h, \comp}(\MM_I)}$.
Then, using Proposition~\ref{pr:microlocal:properties},~\eqref{eq:16} holds for $\bullet = \sqrtsymbol$.
This concludes the proof of Lemma~\ref{lem:hDt:SqrtOp:sdm:mdm}.
\end{proof}
\begin{remark}
    In previous proof, for the semiclassical defect measures, one can also say that $\sdm_{(\tau, \tau)} = \tau^2\sdm_{(u,u)}$ and $\sdm_{(\sqrtsymbol, \sqrtsymbol)} = \sqrtsymbol^2\sdm_{(u,u)}$ stating
    \begin{align*}
        \lim_{h\to 0}\pscal{h D_t u^h}{A_{h}h D_t u^h} &=
        \lim_{h\to 0} \pscal{u^h}{h D_t A_{h} h D_t u^h} =
        \pscal{\sdm_{(u, u)}}{\tau^2\sigma_{h}(A_{h})}\\
        \lim_{h\to 0}\pscal{\SqrtOp_{h} u^h}{A_{h}\SqrtOp_{h} u^h} &=
        \lim_{h\to 0} \pscal{u^h}{\SqrtOp_{h} A_{h} \SqrtOp_{h} u^h} =
        \pscal{\sdm_{(u, u)}}{\sqrtsymbol^2\sigma_{h}(A_{h})}
    \end{align*}
where $A_h\in\pseudo^{\comp}_{h}(\MM_I)$ has $h$-independent principal symbol. 
\end{remark}
\subsection{Slice defect measures}\label{subsection:slice:defect:measure}
This subsection is devoted to stating and proving the following theorem, inspired by~\cite[Proposition 4.4]{LecturePG}.
This theorem is useful because it disintegrates the semiclassical (\resp 2-microlocal) defect measures $\sdm$ (\resp $\mdm$) in space and time into tensor products of simpler measures, thereby separating the space variables from the time variables.
We recall that, in this section, $I$ is a nonempty bounded open interval of $\RR$ containing $0$.

The idea is to prove the existence of slice measures with respect to the time parameter.
We provide the proof of Theorem~\ref{Th:Slices} in Appendix~\ref{proof:Th:Slices}.
To do calculus with tangential symbols such as $\ell^{-\signe}$ (see Remark~\ref{rk:sign} for the definition of $-\signe$), we first introduce $\mathbf{X}(I, M)$, the set of "ghost" $0$-homogeneous cut-off functions (purely technical devices that are invisible when $h\to0$ for half-wave solutions), which will allow us to apply Theorem~\ref{Th:hormander}.
Then, we state the main theorem of this section.
\begin{definition}\label{df:WGF}
Let $\mathbf{X}(I, M)$ be the set of cut-off functions $\chi\in\smooth(\cotang\hspace{1pt}\MM_I;[0,1])$ such that (see Figure~\ref{fig:check:chi}):
\begin{enumerate}
\item $\supp\p{1 - \chi}\cap\set{(t,x,\tau,\xi)\in\cotang\hspace{1pt}\MM_I, ~\tau^2=\sqrtsymbol(x, \xi)^2} = \emptyset$, with $\sqrtsymbol$ defined in~\eqref{df:Sqrt:Sym}.
\item There are $\mathsf{C}_\star>\mathsf{C}^\star>0$ such that $\supp\chi\subset \set{(t,x,\tau,\xi)\in\cotang \hspace{1pt}\MM_I, ~\mathsf{C}_\star\jap{\tau}\le \jap{\xi}_g\le \mathsf{C}^\star\jap{\tau}}$.
\item\label{eq:chi:homogeneous} $\chi$ is $0$-homogeneous in the sense that there exists $\mathsf{R} = \mathsf{R}(\chi)>1$ such that
$\chi(\cdot,\cdot,\mathsf{R}\cdot,\mathsf{R}\cdot)\in S^0_{\phg}(\cotang\hspace{1pt}\MM_I)$.
\end{enumerate}
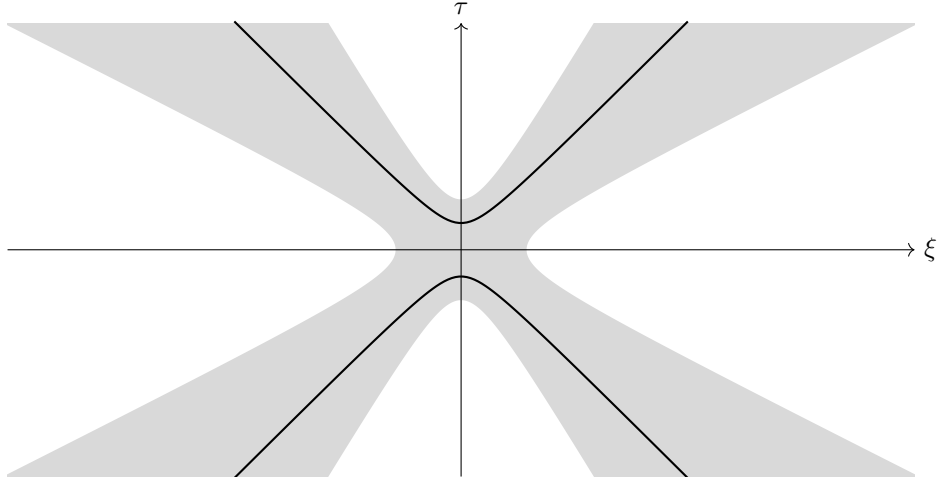
\begin{figure}
\begin{center}
\begin{tikzpicture}[scale=0.5]
\def\cOne{0.6}
\def\cTwo{2}
\def\ymax{6}
\def\sample{120}
\clip (-\cTwo*\ymax,-\ymax) rectangle (\cTwo*\ymax+1,\ymax+1);
\fill[gray!30]
    (- \cTwo*\ymax,-\ymax) -- (- \cTwo*\ymax, \ymax) -- ( \cTwo*\ymax,\ymax) -- (\cTwo*\ymax,-\ymax) -- cycle;
\fill[white] 
    plot[domain=-\ymax:0, samples=\sample] ({sqrt(\cTwo*\cTwo*(1+\x*\x)-1)},\x)
    --
    plot[domain= 0:\ymax, samples=\sample] ({sqrt(\cTwo*\cTwo*(1+\x*\x)-1)},\x) 
    -- cycle;
\fill[white] 
    plot[domain=-\ymax:0, samples=\sample] ({-sqrt(\cTwo*\cTwo*(1+\x*\x)-1)}, \x)
    --
    plot[domain= 0:\ymax, samples=\sample] ({-sqrt(\cTwo*\cTwo*(1+\x*\x)-1)}, \x) 
    -- cycle;
\fill[white] 
    plot[domain=-\ymax:0, samples=\sample] (\x,{sqrt((1+\x*\x)/(\cOne*\cOne) -1)})
    --
    plot[domain= 0:\ymax, samples=\sample] (\x,{sqrt((1+\x*\x)/(\cOne*\cOne) -1)}) 
    -- cycle;
\fill[white] 
    plot[domain=-\ymax:0, samples=\sample] (\x,{-sqrt((1+\x*\x)/(\cOne*\cOne) -1)})
    --
    plot[domain= 0:\ymax, samples=\sample] (\x,{-sqrt((1+\x*\x)/(\cOne*\cOne) -1)}) 
    -- cycle;
\draw[thick] plot[domain=-\ymax:\ymax, samples=\sample] (\x, {sqrt(\x*\x + 0.5)});
\draw[thick] plot[domain=-\ymax:\ymax, samples=\sample] (\x, {-sqrt(\x*\x + 0.5)});
\draw[->] (- \cTwo*\ymax,0) -- (\cTwo*\ymax,0) node[right] {$\xi$};
\draw[->] (0,-\ymax) -- (0,\ymax) node[above] {$\tau$};
\end{tikzpicture}
\end{center}
\caption{Illustration, for fixed $(t,x)\in\MM_I$, of the characteristic region $\set{\tau^2 = \abs{\xi}^2_g + V(x)}$ (black) and the region
$\set{\mathsf{C}_\star\jap{\tau}\le \jap{\xi}_g\le \mathsf{C}^\star\jap{\tau}}$ (light gray).
In the picture, for each $\chi\in\mathbf{X}(I, M)$, $\chi=1$ near the dark gray region and $\chi=0$ near the white region.}
\label{fig:check:chi}
\end{figure}
\end{definition}
\begin{theorem}\label{Th:Slices}
Let $\sqrtsymbol$ be defined in~\eqref{df:Sqrt:Sym}, $\SqrtOp_{h}$ in~\eqref{df:SqrtOp}, $\HWOp^{\signe}_{h}$ in~\eqref{eq:Half:Wave:Operator}, $(h_n)_n$ be a sequence of positive real numbers tending to $0$ and $\Lc_1$ be the Lebesgue measure on $\RR$.
Fix a sign $\signe\in\set{-, +}$.
Assume that $\p{v_{n}^{\signe}}_n$ and $\p{f^n}_n$ are bounded sequences in $L^\infty(I, L^2(M))$ such that
\begin{equation*}
\eqcases{rl}{
    \HWOp^{-\signe}_{h_n}v^{\signe}_{n} &= h_nf^n ~\text{in}~ \MM_I\\
    v^{\signe}_n &= v_n^{\signe}(0)\in L^2(M),
}
\end{equation*}
and
\begin{equation}\label{eq:wave-type-ghost}
\exists\chi\in\mathbf{X}(I, M)~\text{(see Definition~\ref{df:WGF})}, \quad \lim_{n\to\infty}\norm{\Op_{h_n}^{\MM}(1 - \chi)v_n^{\signe}}_{L^2(\MM_I)} = 0.
\end{equation}
\begin{enumerate}
\item Assume that the sequence $({(v^{\signe}_{n}, f^n)}, h_n)_n$ admits a unique semiclassical defect measure 
\begin{equation*}
\begin{pmatrix}
\sdm_{(\signe, \signe)} & \sdm_{(\signe, f)}\\
\sdm_{(f, \signe)} & \sdm_{(f, f)}
\end{pmatrix}
\in \Mc_+(\cotang\hspace{1 pt}\MM_I, \CC^{2\times 2})
\end{equation*}
in the sense of~\eqref{eq:SDM}.
Then the following statements hold:
\begin{enumerate}
\item There is a subsequence of $(v_n^{\signe}, h_n)_n$, still denoted $(v_n^{\signe}, h_n)_n$, and a continuous (for the weak-* topology of measures) map 
\begin{equation*}
\sdm^{\signe}\colon t\in I\mapsto \sdm^{\signe}_t\in \Mc_+(\cotang M),
\end{equation*}
such that: for all $t\in I$, the sequence $(v^{\signe}_{n}(t),h_n)_n$ admits $\sdm^{\signe}_t$ as unique semiclassical defect measure  in the sense of~\eqref{eq:SDM}.
\item We have
\begin{equation}\label{eq:PG:semiclassique:point2}
\sdm_{(\signe, \signe)}(t, x, \tau, \xi) = \sdm^{\signe}_t(x, \xi)\otimes \delta_{{\signe\sqrtsymbol(x, \xi)}}(\tau)\otimes \Lc_1(t),\quad (t, x, \tau, \xi)\in\cotang\hspace{1 pt}\MM_I,
\end{equation}
in the following sense: for all $a\in \test(\cotang\hspace{1 pt}\MM_I)$ we have 
\begin{equation*}
\int_{\cotang\hspace{1 pt}\MM_I}a(t, x, \tau, \xi)\d\sdm_{(\signe, \signe)}(t, x, \tau, \xi) = \int_{I\times\cotang M} a\p{t, x, \signe\sqrtsymbol(x,  \xi), \xi}\d\sdm^{\signe}_t(x, \xi)\d t.
\end{equation*}
\item There is a continuous (for the weak-* topology of measures) map $\tilde{\sdm}^{\signe}\colon t\in I\mapsto\tilde{\sdm}^{\signe}_{t}\in\Mc_+\p{\cotang M}$ such that for all $a\in \test(I\times \cotang M)$ we have
\begin{equation}\label{eq:PG:semiclassic:point3}
\int_{I \times \cotang M}(\partial_t a - H_{\signe\sqrtsymbol}a)\d\sdm^{\signe}_t(x, \xi)\d t = -2\int_{I\times \cotang M}a(t, x, \xi) \d\p{\Imaginary\tilde{\sdm}^{\signe}_{t}}(x, \xi)\d t,
\end{equation}
where $\sdm_{(\signe, f)}(t, x, \tau, \xi) = \tilde{\sdm}^{\signe}_{t}(x, \xi)\otimes \delta_{\signe\sqrtsymbol(x, \xi)}(\tau)\otimes \Lc_1(t)$.
\end{enumerate}
\item Assume that the sequence $({(v^{\signe}_{n}, f^n)}, h_n)_n$ admits a unique 2-microlocal defect measure 
\begin{equation*}
\begin{pmatrix}
\mdm_{(\signe, \signe)} & \mdm_{(\signe, f)}\\
\mdm_{(f, \signe)} & \mdm_{(f, f)}
\end{pmatrix}
\in \Mc_+(\cosphere\hspace{1 pt}\MM_I, \CC^{2\times 2})
\end{equation*}
in the sense of~\eqref{eq:MDM}  with $\supp\p{\mdm_{(\signe, \signe)}}\subset \set{(t, x, \tau, \xi)\in\cosphere \hspace{1 pt}\MM_I,~\tau^2 = \abs{\xi}_g^2}$.
Then  the following statements hold:
\begin{enumerate}
\item There is a subsequence of $(v_n^{\signe}, h_n)_n$ (still denoted $(v_n^{\signe}, h_n)_n$)  and a continuous (for the weak-* topology of measures) map 
\begin{equation*}
\mdm^{\signe}\colon t\in I\mapsto \mdm^{\signe}_t\in \Mc_+(\cosphere M)
\end{equation*}
such that: for all $t\in I$, the sequence $(v^{\signe}_{n}(t),h_n)_n$ admits $\mdm^{\signe}_t$ as unique 2-microlocal defect measure in the sense of~\eqref{eq:MDM}.
\item We have
\begin{equation}\label{eq:PG:microlocal:point2}
\mdm_{(\signe, \signe)}(t, x, \tau, \xi) = \mdm^{\signe}_t(x, \xi)\otimes \delta_{{\signe\abs{\xi}_g}}(\tau)\otimes \Lc_1(t),\quad (t, x, \tau, \xi)\in\cosphere\hspace{1 pt}\MM_I,
\end{equation}
in the following sense: for all $a\in S^0_{\phg}(\cotang\hspace{1 pt}\MM_I)$ we have 
\begin{equation*}
\int_{\cosphere\hspace{1 pt}\MM_I}a(t, x, \tau, \xi)\d\mdm_{(\signe, \signe)}(t, x, \tau, \xi) = \int_{I\times \cosphere M} a\p{t, x, \signe\abs{\xi}_{g}, \xi}\d\mdm^{\signe}_t(x, \xi)\d t.
\end{equation*}
\item There is a continuous map (for the weak-* topology of measures) $\tilde{\mdm}\colon t\in\RR\mapsto \tilde{\mdm}^{\signe}_{t}\in\Mc_+\p{\cosphere M}$ such that for all $a\in  S^0_{\phg, \tang}(I\times\cotang M)$ we have
\begin{equation*}
\int_{I \times \cosphere M}(\partial_t a - H_{\signe\abs{\xi}_g}a)\d\mdm^{\signe}_t(x, \xi)\d t = -2\int_{I\times \cosphere M}a(t, x, \xi) \d\p{\Imaginary\tilde{\mdm}^{\signe}_{t}}(x, \xi)\d t,
\end{equation*}
where $\mdm_{(\signe, f)}(t, x, \tau, \xi) = \tilde{\mdm}^{\signe}_{t}(x, \xi) \otimes \delta_{\signe\abs{\xi}_{g}}(\tau) \otimes \Lc_1(t)$.
\end{enumerate}
\end{enumerate}
\end{theorem}
\begin{remark}
Condition~\eqref{eq:wave-type-ghost} is only a technical assumption: for the half-wave equation sequences introduced later in~\eqref{df:sequence:indexed}, it is automatically satisfied.
This is proved later in Lemma~\ref{lem:intermediate:defect:measures}.
\end{remark}
\begin{corollary}[Propagation along time for defect measures]\label{cor:PG}
Under the notation and assumptions of Theorem~\ref{Th:Slices} and recalling the notation~\eqref{df:Hamiltonian:Flow} for the Hamiltonian flow, the following statements hold.
\begin{enumerate}
\item Assuming in addition that $\Imaginary(\sdm_{(\signe, f)})$ vanishes identically, then for all $t_0, t\in I$, $\sdm^\signe_t = \p{\varphi^{\signe\sqrtsymbol}_{t- t_0}}^\ast\sdm^{\signe}_{t_0}$.
Moreover,
\begin{equation*}
\forall a\in \test(\cotang\hspace{1 pt}\MM_I), \quad \pscal{\sdm_{(\signe, \signe)}}{a} = \int_I\pscal{\sdm^{\signe}_{t_0}}{\tilde{a}^{\signe}_t\circ\varphi^{\signe\sqrtsymbol}_{t-t_0}}_{\Mc_+, \Cscr^0_c(\cotang M)}\d t,
\end{equation*}
with $\tilde{a}^{\signe}_t(x, \xi) := a(t, x, \signe\sqrtsymbol(x, \xi), \xi)$ for $(x, \xi)\in\cotang M$ and $t\in I$.
\item Assuming in addition that $\Imaginary(\mdm_{(\signe, f)})$ vanishes identically, then for all $t_0, t\in I$, $\mdm^{\signe}_t = \p{\varphi^{\signe\abs{\xi}_g}_{t- t_0}}^\ast\mdm^{\signe}_{t_0}$.
Moreover,
\begin{equation*}
\forall a\in S^0_{\phg}(\cotang\hspace{1 pt}\MM_I), \quad \pscal{\mdm_{(\signe,\signe)}}{a} = \int_I\pscal{\mdm^{\signe}_{t_0}}{\check{a}^{\signe}_t\circ\varphi^{\signe\abs{\xi}_g}_{t-t_0}}_{\Mc_+, \Cscr^0_c(\cosphere M)}\d t,
\end{equation*}
with $\check{a}^{\signe}_t(x, \xi) := a(t, x, \signe\abs{\xi}_{g}, \xi)$ for $(x, \xi)\in\cotang M$ and $t\in I$.
\end{enumerate}
\end{corollary}
\begin{proof}[Proof of Corollary~\ref{cor:PG}]
Since the proof is the same in both cases, we only provide the proof for the first statement.
Since $\Imaginary\p{\sdm_{(\signe,f)}}$ vanishes identically, then, using~\eqref{eq:PG:semiclassic:point3} from Theorem~\ref{Th:Slices}, we deduce that
\begin{equation*}
\forall a \in \test\p{\cotang \hspace{1 pt}\MM_I}, \quad\int_{I\times \cotang M} \p{\partial_t a - H_{\signe\sqrtsymbol} a}\d \sdm_t^{\signe}(x, \xi)\d t = 0.
\end{equation*}
Thus, in the sense of distribution on $I\times \cotang M$, the measure $\sdm_t^{\signe}$ satisfies the transport equation
(since $H_k$ is divergence free, $^t{H_{\signe\sqrtsymbol}} = H_{\signe\sqrtsymbol}$)
\begin{equation}\label{eq:transport:semiclassic}
\frac{\partial}{\partial t}\sdm^{\signe}_t(x, \xi) - H_{\signe\sqrtsymbol}{\sdm^{\signe}_t(x, \xi)}  = 0 ,\quad (t, (x, \xi))\in I \times \cotang M.
\end{equation}
As a consequence, $\sdm_t^{\signe}$ is the unique solution to the transport equation~\eqref{eq:transport:semiclassic} with initial data $\sdm_{t_0}^{\signe}$, namely
\begin{equation*}
\sdm^{\signe}\colon t\in I \mapsto \p{\varphi_{t - t_0}^{\signe\sqrtsymbol}}^{\ast}\sdm^{\signe}_{t_0}\in\Mc_{+}\p{\cotang M}.
\end{equation*}
This concludes the proof of Corollary~\ref{cor:PG}.
\end{proof}
\subsection{Application of Theorem~\ref{Th:Slices} to the half-wave equation~\eqref{eq:HWE}}
In this subsection, we consider $(h_n)_n$ a sequence of positive real numbers that tends to $0$, and write 
\begin{equation}\label{df:sequence:indexed}
\p{u^n, v_n^{\signe}, f^n}_n := \p{u^{h_n}, v_{h_n}^{\signe}, f^{h_n}}_n
\end{equation}
where $\p{u^h}_{h}~\text{and}~(v_{h}^{\signe}, f^h)_{h}$ are defined in~\eqref{df:sequence} and~\eqref{df:sequence:intermediate:terms} respectively.
We use Theorem~\ref{Th:Slices} on the sequence $\p{\p{u_n, v_n^{\signe}, f^n}, h_n}_n$, to describe its semiclassical and 2-microlocal defect measures, and then describe sequences of energy functionals.
\paragraph{Support of semiclassical and 2-microlocal defect measures.}
In this paragraph, we begin by proving (with the following lemma) that the sequence (defined in~\eqref{df:sequence}), satisfies the condition to apply Theorem~\ref{Th:Slices} (especially~\eqref{eq:wave-type-ghost}).
Then, we provide a corollary of Theorem~\ref{Th:Slices} to describe semiclassical and 2-microlocal defect measures associated with the sequence~\eqref{df:sequence:indexed}.
\begin{lemma}\label{lem:intermediate:defect:measures}
Let $(h_n)_n$ be a sequence of $(0, 1]$ tending to $0$, $\p{u^n, v_{n}^{\signe}, f^{n}}_n$ be as~\eqref{df:sequence:indexed} and $\sqrtsymbol$ be as~\eqref{df:Sqrt:Sym}.
Then, there is a subsequence of $\p*{\p{v_{n}^{+}, v_{n}^{-}, f^{n}, u^n}, h_n}_n$, still denoted $\p*{\p{v_{n}^{+}, v_{n}^{-}, f^{n}, u^n}, h_n}_n$, which admits a Radon measure with hermitian matrix values $\sdm\in\Mc_+\p{\cotang\hspace{1 pt}\MM_I, \CC^{4\times4}}$ (\resp $\mdm\in\Mc_+\p{\cosphere\hspace{1 pt}\MM_I, \CC^{4\times4}}$) as unique semiclassical (\resp 2-microlocal) defect measure in the sense of~\eqref{eq:SDM} (\resp \eqref{eq:MDM}) and we set
\begin{equation}\label{eq:intermediate:defect:measure}
\begin{pmatrix}
\sdm_{(+,+)} & \sdm_{(+, -)} & \sdm_{(+,f)} & \sdm_{(+,u)}\\
*& \sdm_{(-,-)} & \sdm_{(-,f)} & \sdm_{(-,u)}\\
*& * & \sdm_{(f,f)} & \barre{\sdm_{(u,f)}}\\
* & * & * & \sdm_{(u,u)}
\end{pmatrix}
:=\sdm
~\p*{\resp
\begin{pmatrix}
\mdm_{(+,+)} & \mdm_{(+, -)} & \mdm_{(+,f)} & \mdm_{(+,u)}\\
* & \mdm_{(-,-)} & \mdm_{(-,f)} & \mdm_{(-,u)}\\
* & * & \mdm_{(f,f)} & \barre{\mdm_{(u,f)}}\\
* & * & * & \mdm_{(u,u)}
\end{pmatrix}:=\mdm}.
\end{equation}
Moreover, the following inclusions hold:
\begin{align}
    \supp(\sdm_{(\bullet, \bullet)})&\subset \set{(t, x, \tau, \xi)\in \cotang \hspace{1pt} \MM_I,~\tau^2 = \sqrtsymbol(x, \xi)^2},& \bullet \in\set{u, \signe}\label{eq:1:intermediate:defect:measure:support:22:sdm},\\
    \supp(\mdm_{(\bullet, \bullet)})&\subset \set{(t, x, \tau, \xi)\in \cosphere \hspace{1pt} \MM_I,~\tau^2 = \abs{\xi}_g^2}, &\bullet \in\set{u, \signe}.\label{eq:1:intermediate:defect:measure:support:22:mdm}
\end{align}
\end{lemma}
\begin{proof}[Proof of Lemma~\ref{lem:intermediate:defect:measures}]
By Lemma~\ref{lem:SPLITTING}, the sequence $\p{v_n^{+}, v_n^{-}, f^n, u^n}_n$ is bounded in $L^\infty\p{I,  L^2(M)^4}$ and thus by Proposition~\ref{pr:SDM} (\resp Theorem~\ref{Th:Microlocalization}), up to considering a subsequence, we have existence and uniqueness of $\sdm$ and $\mdm$.
Firstly, using  Lemma~\ref{lem:hDt:SqrtOp:sdm:mdm} and uniqueness of $\sdm_{(u, u)}$ and $\mdm_{(u, u)}$,~\eqref{eq:1:intermediate:defect:measure:support:22:sdm} and~\eqref{eq:1:intermediate:defect:measure:support:22:mdm} hold for $\bullet = u$.
Secondly, using identity~\eqref{eq:identity:v:to:u} and~\eqref{eq:16} from Lemma~\ref{lem:hDt:SqrtOp:sdm:mdm}, we have
\begin{align*}
\supp\p{\sdm_{(+,+)} + \sdm_{(-,-)}}&\subset\set{(t, x, \tau, \xi)\in\cotang\hspace{1 pt}\MM_I,~ \tau^2 = \sqrtsymbol(x, \xi)^{2}}\\
\supp\p{\mdm_{(+,+)} + \mdm_{(-,-)}}&\subset\set{(t, x, \tau, \xi)\in\cosphere\hspace{1 pt}\MM_I,~ \tau^2 = \abs{\xi}_g^2}.
\end{align*}
Since $\sdm_{(+, +)}$ and $\sdm_{(-, -)}$ (\resp$\mdm_{(+, +)}$ and $\mdm_{(-, -)}$) are nonnegative Radon measures, they both are supported in $\set{\tau^2 = \sqrtsymbol\p{x, \xi}^2}$ (\resp$\set{\tau^2 = \abs{\xi}_g^2}$).
This last point proves~\eqref{eq:1:intermediate:defect:measure:support:22:sdm} and~\eqref{eq:1:intermediate:defect:measure:support:22:mdm} hold for $\bullet = \signe$ and thus concludes the proof of lemma~\ref{lem:intermediate:defect:measures}.
\end{proof}
With the following corollary, using Theorem~\ref{Th:Slices} in addition with Lemma~\ref{lem:intermediate:defect:measures}, we describe the support of semiclassical and 2-microlocal defect measures associated with the sequence~\eqref{df:sequence:indexed}.
\begin{corollary}[Support of defect measures for wave and half-wave equations]\label{cor:intermediate:defect:measures}
Let $(h_n)_n$ be a sequence of $(0, 1]$ tending to $0$, $\p{u^n, v_{n}^{\signe}, f^{n}}_n$ be as~\eqref{df:sequence:indexed} and $\sqrtsymbol$ be as~\eqref{df:Sqrt:Sym}.
Assume that there is a subsequence of $\p*{\p{v_{n}^{+}, v_{n}^{-}, f^{n}, u^n}, h_n}_n$, still denoted $\p*{\p{v_{n}^{+}, v_{n}^{-}, f^{n}, u^n}, h_n}_n$, which admits a Radon measure with hermitian matrix values $\sdm\in\Mc_+\p{\cotang\hspace{1 pt}\MM_I, \CC^{4\times4}}$ (\resp $\mdm\in\Mc_+\p{\cosphere\hspace{1 pt}\MM_I, \CC^{4\times4}}$) as unique semiclassical (\resp 2-microlocal) defect measure in the sense of~\eqref{eq:SDM} (\resp \eqref{eq:MDM}) and, using  notation~\eqref{eq:intermediate:defect:measure} from Lemma~\ref{lem:intermediate:defect:measures}, we have that
\begin{equation}\label{eq:asymptotic:orthogonality}
\sdm_{(+, -)}~\text{and}~\mdm_{(+, -)}~\text{vanish identically}
\end{equation}
and that the following inclusions hold:
\begin{align}
\supp(\sdm_{(\signe, \bullet)})&\subset \set{(t, x, \tau, \xi)\in \cotang \hspace{1pt} \MM_I,~\tau = \signe \sqrtsymbol(x, \xi)},& \bullet \in \set{\signe, f, u},\label{eq:intermediate:defect:measure:support:sdm}\\
\supp(\mdm_{(\signe, \bullet)})&\subset \set{(t, x, \tau, \xi)\in \cosphere \hspace{1pt} \MM_I,~\tau = \signe \abs{\xi}_g},& \bullet \in\set{\signe, f, u},\label{eq:intermediate:defect:measure:support:mdm}\\
\supp(\sdm_{(\bullet, \bullet)})&\subset \set{(t, x, \tau, \xi)\in \cotang \hspace{1pt} \MM_I,~\tau^2 = \sqrtsymbol(x, \xi)^2},& \bullet \in\set{u, f}\label{eq:intermediate:defect:measure:support:22:sdm},\\
\supp(\mdm_{(\bullet, \bullet)})&\subset \set{(t, x, \tau, \xi)\in \cosphere \hspace{1pt} \MM_I,~\tau^2 = \abs{\xi}_g^2}, &\bullet \in\set{u, f}.\label{eq:intermediate:defect:measure:support:22:mdm}
\end{align}
\item
Furthermore, $\Imaginary\p{\sdm_{(\signe, f)}}$ (\resp $\Imaginary\p{\mdm_{(\signe, f)}}$) vanishes identically on $\cotang\hspace{1 pt}\MM_{I}$ (\resp $\cosphere\hspace{1 pt}\MM_{I}$).
\end{corollary}
\begin{proof}[Proof of Corollary~\ref{cor:intermediate:defect:measures}]
\begin{enumerate}
\item First, using Lemma~\ref{lem:intermediate:defect:measures},~\eqref{eq:1:intermediate:defect:measure:support:22:sdm} and~\eqref{eq:1:intermediate:defect:measure:support:22:mdm} hold for $\bullet \in\set{\signe, u}$.
By Lemma~\ref{lem:intermediate:defect:measures}, we can apply Theorem~\ref{Th:Slices} on the sequence $(v_n^{\signe}, f^n, h_n)_n$, then by~\eqref{eq:PG:semiclassique:point2} (\resp\eqref{eq:PG:microlocal:point2}) from Theorem~\ref{Th:Slices} together with Propositions~\ref{pr:inter:supp:k} (\resp\ref{pr:inter:supp:k:2mdm}), we have that~\eqref{eq:intermediate:defect:measure:support:sdm} (\resp\eqref{eq:intermediate:defect:measure:support:mdm}) holds.

Second, we prove~\eqref{eq:asymptotic:orthogonality}.
Using Propositions~\ref{pr:inter:supp:k} and~\ref{pr:inter:supp:k:2mdm} together with~\eqref{eq:intermediate:defect:measure:support:sdm} and~\eqref{eq:intermediate:defect:measure:support:mdm}, we have
\begin{equation*}
\supp\p{\sdm_{(+, -)}} \subset \p*{\bigcap_{\bullet \in\set{-, +}}\supp\p{\sdm_{(\bullet, \bullet)}}} = \emptyset
\quad\text{and}\quad
\supp\p{\mdm_{(+, -)}} \subset \p*{\bigcap_{\bullet \in\set{-, +}}\supp\p{\mdm_{(\bullet, \bullet)}}} = \emptyset.
\end{equation*}
Thus,~\eqref{eq:asymptotic:orthogonality} holds.

Finally, we already have~\eqref{eq:1:intermediate:defect:measure:support:22:sdm} and~\eqref{eq:1:intermediate:defect:measure:support:22:mdm} from Lemma~\ref{lem:intermediate:defect:measures}, thus we prove~\eqref{eq:intermediate:defect:measure:support:22:sdm} and~\eqref{eq:intermediate:defect:measure:support:22:mdm} for $\bullet = f$.
Using the decomposition
\begin{equation*}
f^n = \frac{1}{2}\tilde{\rest}^{\SqrtOp, 0}_{h_n}\p{v^{+}_n - v^{-}_n} + h\rest^{\SqrtOp, 0}_{h}u^h
\end{equation*}
from~\eqref{eq:decomposition:f}, we conclude
\begin{equation*}
\supp\p{\sdm_{(f, f)}}\subset \bigcup_{\bullet \in\set{+, -, u}}\supp\p{\sdm_{(\bullet, f)}}
\quad
\text{and}
\quad
\supp\p{\mdm_{(f, f)}}\subset \bigcup_{\bullet \in\set{+, -, u}}\supp\p{\mdm_{(\bullet, f)}}
\end{equation*}
Then, by Proposition~\ref{pr:inter:supp:k} and~\ref{pr:inter:supp:k:2mdm} together with~\eqref{eq:1:intermediate:defect:measure:support:22:sdm} and~\eqref{eq:1:intermediate:defect:measure:support:22:mdm}, we have
\begin{equation*}
\supp\p{\sdm_{(f, f)}} \subset\bigcup_{\bullet \in\set{+, -, u}}\supp\p{\sdm_{(\bullet, \bullet)}}\subset \set{(t, x, \tau, \xi)\in \cotang \hspace{1pt} \MM_I,~\tau^2 = \sqrtsymbol(x, \xi)^2}
\end{equation*}
and
\begin{equation*}
\supp\p{\mdm_{(f, f)}}\subset\bigcup_{\bullet \in\set{+, -, u}}\supp\p{\mdm_{(\bullet, \bullet)}}\subset \set{(t, x, \tau, \xi)\in \cosphere\hspace{1pt} \MM_I,~\tau^2 = \abs{\xi}_{g}^2}.
\end{equation*}
Then~\eqref{eq:intermediate:defect:measure:support:22:sdm} and~\eqref{eq:intermediate:defect:measure:support:22:mdm} hold for $\bullet = f$.
\item We now prove that $\Imaginary\p{\sdm_{(\signe, f)}}$ (\resp $\Imaginary\p{\mdm_{(\signe, f)}}$) vanishes identically on $\cotang\hspace{1 pt}\MM_{I}$ (\resp $\cosphere\hspace{1 pt}\MM_{I}$).
From~\eqref{eq:decomposition:f}, boundedness of $(u^n)_n$ and Proposition~\ref{pr:boundedness:pseudo}, we deduce that
\begin{equation*}
f^n = \frac{1}{2}\tilde{\rest}^{\SqrtOp, 0}_{h_n}\p{v^{+}_n - v^{-}_n} + \O{h_n}_{L^2(M)}.
\end{equation*}
Using boundedness of $(v_n^{\signe})_n$ in $L^\infty(I, L^2(M))$ from Lemma~\ref{lem:SPLITTING} and the asymptotic orthogonality~\eqref{eq:asymptotic:orthogonality}, we have
\begin{align*}
\pscal{v_n^{\signe}}{A_{h_n}f^n}_{L^2(\MM_I)} - \pscal{f^n}{A_{h_n}v_n^{\signe}}_{L^2(\MM_I)}
=& \frac{1}{2}\pscal{v_n^{\signe}}{A_{h_n}\tilde{\rest}^{\SqrtOp, 0}_{h_n}v^{\signe}_n}_{L^2(\MM_I)} -\frac{1}{2}\pscal{\tilde{\rest}^{\SqrtOp, 0}_{h_n}v^{\signe}_{n}}{A_{h_n}v^{\signe}_n}_{L^2(\MM_I)}\\
&+ \o{1}_n\\
=& \frac{1}{2}\pscal{v_n^{\signe}}{\p{A_{h_n}\tilde{\rest}^{\SqrtOp, 0}_{h_n} - \classe{\tilde{\rest}^{\SqrtOp, 0}_{h_n}}^*A_{h_n}}v^{\signe}_n}_{L^2(\MM_I)} + \o{1}_n,
\end{align*}
for all $A_{h}\in\pseudo^0_{h, \comp}(\MM_I)$.
\begin{enumerate}
\item 
For all $A_{h}\in \pseudo^{\comp}_{h}(\MM_I)$, since $\sigma_{h}\p{\tilde{\rest}^{\SqrtOp, 0}_{h}}\in\RR$, we have $\sigma_{h}({A_{h}\tilde{\rest}^{\SqrtOp, 0}_{h} - \classe{\tilde{\rest}^{\SqrtOp, 0}_{h}}^*A_{h}}) = 0$ by Proposition~\ref{pr:composition}, and thus
\begin{equation*}
\pscal{\Imaginary\sdm_{(\signe, f)}}{\sigma_{h}(A_{h})}_{\Mc_+, \Cscr^0_c(\cotang\hspace{1pt} \MM_I)} = \frac{1}{2i}\lim_{n\to\infty}\p*{\pscal{v_n^{\signe}}{A_{h_n}f^n}_{L^2(\MM_I)} - \pscal{f^n}{A_{h_n}v_n^{\signe}}_{L^2(\MM_I)}} = 0.
\end{equation*}
\item For all $a\in\tilde{S^0_{\phg}}(\cotang M)$ with homogeneous principal part $a_{\phg}\in{S^0_{\phg}}(\cotang M)$, and $A_{h}^R := \Op_{h}^{\MM}\p{a\p{1 - \theta\p{\frac{\jap{\xi}_g}{R}}}}\in \pseudo^{0}_{h, \comp}(\MM_I)$, since $\sigma_{h}\p{\tilde{\rest}^{\SqrtOp, 0}_{h}}\in\RR$, we have $A_{h}^R\tilde{\rest}^{\SqrtOp, 0}_{h} - \classe{\tilde{\rest}^{\SqrtOp, 0}_{h}}^*A_{h}^R \in h\pseudo_{h, \comp}^{-1}(\MM_I)$, and thus
\begin{equation*}
\pscal{\Imaginary\mdm_{(\signe, f)}}{a_{\phg}}_{\Mc_+, \Cscr^0_c\p{\cosphere\hspace{1pt} \MM_I}} = \frac{1}{2i}\lim_{R\to\infty}\lim_{n\to\infty}\p*{\pscal{v_n^{\signe}}{A_{h_n}^Rf^n}_{L^2(\MM_I)} - \pscal{f^n}{A_{h_n}^Rv_n^{\signe}}_{L^2(\MM_I)}} = 0.
\end{equation*}
\end{enumerate}
\end{enumerate}
This concludes the proof of Corollary~\ref{cor:intermediate:defect:measures}.
\end{proof}
\paragraph{Evolution of an energy functional.}
With the following corollary, using Lemma~\ref{lem:intermediate:defect:measures} and Theorem~\ref{Th:Slices}, we describe the evolution of a natural energy functional (see~\eqref{eq:NRJ} below) for solutions of \eqref{eq:Wave}.
\begin{corollary}\label{cor:use:HWE:for:WE}
Let $(h_n)_n$ be a sequence of $(0, 1]$ tending to $0$ and $(u^n, v_n^{\signe}, f^n)$ be the sequence defined in~\eqref{df:sequence:indexed}.
For all $A_{h}\in\pseudo^0_{h}(M)$ and all $t\in I$, set
\begin{equation}\label{eq:NRJ}
\Ec_{n}^A(t) = 2\classe{\pscal{h_n\partial_t u^n(t)}{A_{h_n}h_n\partial_tu^n(t)}_{L^2(M)} + \pscal{\SqrtOp_{h_n}u^n(t)}{A_{h_n}\SqrtOp_{h_n}u^n(t)}_{L^2(M)}}.    
\end{equation}
\begin{enumerate}
\item Assume that the pair $({\p{\SqrtOp_{h_n}u^n, h_n\partial_t u^n}}_{|t = t_0}, h_n)_n$ admits $\p{\tilde{\sdm}_{ij}^{t_0}}_{1\leq i, j\leq 2}\in\Mc_+(\cotang M, \CC^{2\times 2})$ as unique semiclassical defect measure in the sense of~\eqref{eq:SDM}. For any $t_0\in I$ and any sign $\signe \in\set{-, +}$, define the Radon measures
\begin{equation*}
\sdm^{\signe}_{t_0} := \tilde{\sdm}_{11}^{t_0} + \tilde{\sdm}_{22}^{t_0} + \signe 2\Imaginary\p{\tilde{\sdm}_{12}^{t_0}}\in \Mc_+(\cotang M).
\end{equation*}
Then, for all $t\in I$ and all $A_{h}\in\pseudo_{h}^{\comp}(M)$ with principal symbol $\sigma_h(A_h)$ independent of $h$ we have
\begin{equation*}
\lim_{n\to\infty}\Ec_{n}^A(t) = \pscal{{\p{\varphi^{\sqrtsymbol}_{t-t_0}}^\ast\sdm^+_{t_0}} + {\p{\varphi^{-\sqrtsymbol}_{t - t_0}}^\ast\sdm^-_{t_0}}}{\sigma_{h}(A_{h})}_{\Mc_+, \Cscr_c^0(\cotang M)}.
\end{equation*}
Furthermore, the sequence $(v_n^{\signe}(t_0))_n$ admits $\sdm_{t_0}^{\signe}$ as unique semiclassical defect measure in the sense of~\eqref{eq:SDM}.
\item Assume that the pair $({\p{\SqrtOp_{h_n}u^n, h_n\partial_t u^n}}_{|t = t_0}, h_n)_n$ admits $\p{\tilde{\mdm}_{ij}^{t_0}}_{1\leq i, j\leq 2}\in\Mc_+(\cosphere M, \CC^{2\times 2})$ as unique 2-microlocal defect measure in the sense of~\eqref{eq:MDM}. For any $\theta$ as~\eqref{df:theta}, any $t_0\in I$ and any sign $\signe \in\set{-, +}$, defining the Radon measure
\begin{equation*}
\mdm^{\signe}_{t_0} := \tilde{\mdm}_{11}^{t_0} + \tilde{\mdm}_{22}^{t_0} + \signe 2\Imaginary\p{\tilde{\mdm}_{12}^{t_0}}\in \Mc_+(\cosphere M),
\end{equation*}
for all $t\in I$ and all $a\in\tilde{S^0_{\phg}}(\cotang M)$ we have
\begin{equation*}
\lim_{R\to\infty}\lim_{n\to\infty}\Ec_{n}^{A^R}(t) = \pscal{{\p{\varphi^{\abs{\xi}_g}_{t- t_0}}^\ast\mdm^+_{t_0}} + {\p{\varphi^{-\abs{\xi}_g}_{t - t_0}}^\ast\mdm^-_{t_0}}}{\sigma_{\phg}^{0}(A_{h})}_{\Mc_+, \Cscr_c^0\p{\cosphere M}},
\end{equation*}
where we set
$A_{h}^R := \Op_{h}^{\MM}\p{a\p{1 - \theta\p{\frac{~\jap{\xi}_g}{R}}}}\in \pseudo^{0}_{h}(M)$.
Furthermore, the sequence $(v_n^{\signe}(t_0))_n$ admits $\mdm_{t_0}^{\signe}$ as unique 2-microlocal defect measure in the sense of~\eqref{eq:MDM}.
\end{enumerate}
\end{corollary}
\begin{proof}[Proof of Corollary~\ref{cor:use:HWE:for:WE}] 
Using Theorem~\ref{Th:Slices} (the assumptions are satisfied by~\eqref{eq:1:intermediate:defect:measure:support:22:sdm} and~\eqref{eq:1:intermediate:defect:measure:support:22:mdm} from Lemma~\ref{lem:intermediate:defect:measures}) and Corollary~\ref{cor:PG} (the assumptions are satisfied by Lemma~\ref{lem:intermediate:defect:measures} and the last point of Corollary~\ref{cor:intermediate:defect:measures}): there is a subsequence still denoted $(u^n, v_n^{\signe}, f^n, h_n)_n$, such that for all $t\in I$, $(v_{n}^{\signe}(t), h_n)_n$ admits
\begin{equation*}
\sdm^{\signe}_t = \p{\varphi_{t - t_0}^{\signe\sqrtsymbol}}^\ast\sdm^{\signe}_{t_0}\in \Mc_+\p{\cotang M}
\end{equation*}
as unique semiclassical defect measure in the sense of~\eqref{eq:SDM}.
Then, using identity~\eqref{eq:identity:v:to:u:t}, for any $A_{h}\in\pseudo_{h}^{\comp}(M)$ with principal symbol $\sigma_h(A_h)$ independent of $h$, we have
\begin{equation*}
\Ec^A_{n}(t) = \sum_{\signe\in\set{-, +}}\pscal{v^{\signe}_n(t)}{A_{h_n}v^{\signe}_n(t)}_{L^2(M)}.
\end{equation*}
Thus, taking the limit $n\to\infty$, we have
\begin{align*}
\lim_{n\to\infty}\Ec_{n}^A(t) &= \pscal{\sdm^+_{t} + \sdm^-_{t}}{\sigma_{h}(A_{h})}_{\Mc_+, \Cscr_c^0(\cotang M)}\\
&=\pscal{{\p{\varphi^{\sqrtsymbol}_{t-t_0}}^\ast\sdm^+_{t_0}} + {\p{\varphi^{-\sqrtsymbol}_{t - t_0}}^\ast\sdm^-_{t_0}}}{\sigma_{h}(A_{h})}_{\Mc_+, \Cscr_c^0(\cotang M)}.
\end{align*}
Finally, the equality
\begin{equation*}
v_n^{\signe}(t_0) = \p{\HWOp^{\signe}_{h_n}u^n}_{|t = t_0} = \p{h_n\partial_t u^n + \signe \SqrtOp_{h_n}u^n}_{|t = t_0},
\end{equation*}
implies that $\sdm^{\signe}_{t_0} := \tilde{\sdm}_{11}^{t_0} + \tilde{\sdm}_{22}^{t_0} + \signe 2\Imaginary\p{\tilde{\sdm}_{12}^{t_0}}$ and thus ends the proof for the first case.
Replacing $A_{h}\in\pseudo_{h}^{\comp}(M)$ by $\Op_{h}^{\MM}\p{a\p{1 - \theta\p{\frac{~\jap{\xi}_g}{R}}}}\in \pseudo^{0}_{h, \comp}(M)$, with $a\in\tilde{S^0_{\phg}}(M)$, the proof for the second case has exactly the same structure.
This concludes the proof of Corollary~\ref{cor:use:HWE:for:WE}.
\end{proof}
\section{\texorpdfstring{\eqref{GCC+}}{} implies uniform observability \texorpdfstring{\eqref{eq:UO}}{}}\label{section:CS}
In this section, relying on the technical Sections~\ref{section:microlocalization} and~\ref{section:slice:defect:measure}, we prove the sufficiency part of Theorem~\ref{Th:GCC:UO}.
We recall that $(M, g)$ is a smooth, compact, connected Riemannian manifold without boundary, $T > 0$ is a time, $b_\omega$ is a continuous nonnegative function, and $\omega = \set{b_\omega > 0}$ is an open subset of $M$.
We assume that $V$ satisfies~\eqref{df:potential:fixed:V} and that the pair $(\omega, T)$ satisfies~\eqref{GCC+}. We prove~\eqref{eq:UO}.
Proceeding by contradiction, we assume that for all $n\geq 1$,
\begin{equation}\label{eq:non:UO}\tag{$\neg$UO}
\exists(u_0^n, u_1^n)\in H^1(M)\times L^2(M), ~\exists\lmbd_n\geq 1, \quad \Escr_{\lmbd_n V}(u_0^n, u_1^n) = 1 \text{ and } \int_0^T\norm{b_\omega \partial_t u^n}^2_{L^2(M)}\d t < \frac{1}{n},
\end{equation}
where $u^n$ is the $\Cscr^0((0, T), H^1(M))\cap\Cscr^1((0, T), L^2(M))$-solution of~\eqref{eq:Wave} with initial condition $(u_0^n, u_1^n)$ and potential $\mathsf{V} = \lmbd_n V$ with $V$ satisfying~\eqref{df:potential:fixed:V}.

There are (\apriori) two  possible cases: either $(\lmbd_n)_n$ is bounded, or $\limsup_n\lmbd_n = +\infty$. 
First, in Subsection~\ref{subsection:compact:case}, we treat the case $(\lmbd_n)_n$ bounded and show a contradiction.
Second, if $(\lmbd_n)_n$ is unbounded, then up to considering a subsequence, we may assume that $\lim_n\lmbd_n = +\infty$.
In Subsection~\ref{subsection:semiclassical:rescaling}, we then introduce the semiclassical parameter $h_n = \lmbd_n^{-1/2}$ and use results on semiclassical and 2-microlocal (see Definition~\ref{df:MDM}) defect measures to finally obtain a contradiction.
The proofs rely on the general strategy of Lebeau~\cite{lebeau1996equation} (see also \cite{burq97Controlabilite, burq1997condition, dehman2014controllability}).
\subsection{Compact case}\label{subsection:compact:case}
In this subsection, we prove that~\eqref{GCC} (induced by~\eqref{GCC+}) implies that the sequence $(\lmbd_n)_n$ from assumption~\eqref{eq:non:UO} cannot be bounded.
\begin{lemma}\label{lem:lmbd:unbounded}
Let $V$ be as in~\eqref{df:potential:fixed:V}, $b_\omega \in\Cscr^{0}(M, \RR_+^*)$, $\omega = \set{b_\omega > 0}$, a time $T > 0$ and a sequence $(u^n, \lmbd_n)_n$ satisfying~\eqref{eq:non:UO}. 
If~\ref{GCC}$(\omega, T)$ holds, then $\ds\limsup_{n\to\infty}\lmbd_n = +\infty$.
\end{lemma}
\begin{proof}[Proof of Lemma~\ref{lem:lmbd:unbounded}]
Proceeding by contradiction, we suppose $(\lmbd_n)_n$ bounded.
Up to considering a subsequence, we assume that $\p{\lmbd_n}_n$ converges to $\lmbd_\infty\in [1, \infty)$.
Let $u^n_\infty$ be the $\Cscr^0((0, T), H^1(M))\cap\Cscr^1((0, T), L^2(M))$-solution of~\eqref{eq:Wave} with potential $\lmbd_\infty V$ and initial data $(u_0^n, u_1^n)$.
Since~\ref{GCC}$(\omega, T)$ holds,~\eqref{eq:Obs} is satisfied (from \cite{bardos1992sharp}) and thus there is $C = C\p{\lmbd_\infty} > 0$ such that
\begin{equation}\label{eq:6}
\forall n\geq 1, \quad \Escr_{\lmbd_\infty V}\p{u_0^n, u_1^n}\leq C\int_0^T\int_M\abs{b_\omega(x) \partial_t u_\infty^n(t, x)}^2\d x\d t.
\end{equation}
Furthermore, we set
\begin{equation*}
w^n = u^n - u^n_\infty, \quad n\geq 1,
\end{equation*}
which is a solution of
\begin{equation}\label{eq:w}
\eqcases{rcll}{
\partial^2_t w^n- \Delta_g w^n &=& \p{\lmbd_\infty - \lmbd_n}V u^n &\text{in} ~ (0, T)\times M,\\
\p{w^n, \partial_t w^n}_{|t = 0} &=& (0, 0) &\text{on}~M.
}
\end{equation}

Firstly, using~\eqref{eq:6}, we have
\begin{align}
\notag\Escr_{\lmbd_\infty V}\p{u_0^n, u_1^n}
\notag&\leqsim \int_0^T\norm{b_\omega \partial_t u_\infty^n(t)}^2_{L^2(M)}\d t\\
\notag&\leqsim \int_0^T\norm{b_\omega \partial_t u^n(t)}^2_{L^2(M)}\d t + \int_0^T\norm{b_\omega \partial_t w^n(t)}^2_{L^2(M)}\d t\\
&\leqsim \int_0^T\norm{b_\omega \partial_t u^n(t)}^2_{L^2(M)}\d t + \norm{\partial_t w^n}^2_{L^\infty((0, T), L^2(M))}.\label{eq:7}
\end{align}

Secondly, using a (classical) hyperbolic energy estimate (see~\cite[Chapter 23]{Hörmander1985AnalysisOfLinearOpVol3} for more details) on~\eqref{eq:w}, for all $n\geq 1$ we have
\begin{equation*}
\norm{\partial_t w^n}^2_{L^\infty((0, T), L^2 (M))}\leqsim \norm{\p{\lmbd_\infty - \lmbd_n}V u^n}_{L^1((0, T), L^2(M))}^2.
\end{equation*}
Since $\p{u^n}_n$ is bounded in $L^\infty((0, T), L^2(M))$ by Lemma~\ref{lem:SPLITTING}, we conclude that 
\begin{equation}\label{eq:8}
\norm{\partial_t w^n}^2_{L^\infty((0, T), L^2 (M))} = \O{\p{\lmbd_\infty - \lmbd_n}^2}.
\end{equation}
Finally, using~\eqref{eq:8} in~\eqref{eq:7} and~\eqref{eq:non:UO}, we have
\begin{align}
\notag1 = \Escr_{\lmbd_n V}\p{u_0^n, u_1^n}
&= \Escr_{\lmbd_\infty V}\p{u_0^n, u_1^n} + \p{\lmbd_n - \lmbd_\infty}\int_M V(x)\abs{u_0^n}^2\d x\\
\notag&\leqsim \int_0^T\norm{b_\omega \partial_t u^n}^2_{L^2(M)}\d t + \O{\p{\lmbd_\infty - \lmbd_n}^2} + \O{\lmbd_\infty - \lmbd_n}\\
&\leqsim \frac{1}{n} + \O{\lmbd_\infty - \lmbd_n}\xrightarrow[n\to \infty]{}0.\label{eq:9}
\end{align}
Since~\eqref{eq:9} is absurd, this completes the contradiction and hence the proof of Lemma~\ref{lem:lmbd:unbounded}. 
\end{proof}
\subsection{Semiclassical rescaling}\label{subsection:semiclassical:rescaling}
Still proceeding by contradiction, assuming~\eqref{eq:non:UO} holds: by Lemma~\ref{lem:lmbd:unbounded}, up to considering a subsequence, the sequence $(\lmbd_n)_n$ tends to $+\infty$. We set $h_n := \sqrt{1/\lmbd_n}$, which then converges to $0$ as $n$ tends to infinity.
We now introduce some notation for semiclassical rescaling of~\eqref{eq:non:UO}.

We recall that, for all $h \in (0, 1]$, the rescaled energy $\Escr^{h}$ is defined in~\eqref{df:NRJ:rescaled} and $\Escr^{h_n}\p{u^{n}, \partial_t u^{n}}$ is preserved in time (see~\eqref{eq:WP}) for any $n\geq 1$.
We rewrite~\eqref{eq:non:UO} in a semiclassical way as: there is $(u_0^n, u_1^n, h_n)_n$ in $H^1(M)\times L^2(M)\times(0, 1]$ such that
\begin{align}
\notag \lim_{n\to\infty}h_n&=0,\\
\label{eq:non:UO:rescaling:NRJ}
\forall n\geq 1, \quad \Escr^{h_n}(u_0^n, u_1^n) &= 1,\\
\label{eq:non:UO:rescaling:0}
\lim_{n\to\infty}\int_{0}^T \norm{b_\omega h_n\partial_t u^n(t)}_{L^2(M)}^2 \d t &= 0,
\end{align}
where $u^n$ is the $\Cscr^0((0, T), H^1(M))\cap\Cscr^1((0, T), L^2(M))$-solution of~\eqref{eq:Wave} with initial condition $(u_0^n, u_1^n)$ and potential $\mathsf{V} = h_n^{-2} V$.
The rescaled energy equality in~\eqref{eq:non:UO:rescaling:NRJ} together with~\eqref{eq:WP} implies that the mean value of the energy satisfies
\begin{equation*}
1 = \frac{1}{T}\int_0^T \Escr^{h_n}(u^n(t), \partial_t u^n(t))\d t = \frac{1}{2T}\int_{\MM_{(0, T)}}\abs{h_n\nabla_g u^n(t, x)}_g^2 + \abs{h_n \partial_t u^n(t, x)}^2 + V(x)\abs{u^n(t, x)}^2\d t\d x.
\end{equation*}
In the following three lemmas, we give some properties induced by~\eqref{eq:non:UO:rescaling:NRJ} and~\eqref{eq:non:UO:rescaling:0} on sequence $(u^n)_n$.
\begin{lemma}\label{lem:vanishing:omega}
Let us consider $V$ as in~\eqref{df:potential:fixed:V}, $b_\omega \in\Cscr^{0}(M, \RR_+^*)$, $\omega = \set{b_\omega > 0}$, a time $T > 0$, a sequence $(u^n, h_n)_n$ satisfying~\eqref{eq:non:UO:rescaling:NRJ} and~\eqref{eq:non:UO:rescaling:0}, and $(v_n^{\signe}, f^n)_n$ as in~\eqref{df:sequence:indexed} indexed by $(h_n)_n$.
For all $\phi_\omega^h\in S^0_{\comp} (\cotang \hspace{1pt} \MM_{(0, T)})$ supported in $(0, T)\times \omega$, we have
\begin{equation}\label{eq:29}
\lim_{n\to\infty}\sum_{\signe\in\set{-, +}} \pscal{v_{n}^\signe}{\Op_{h_n}^{\MM}\p{\phi_\omega^{h_n}} v_{n}^\signe}  = 0.
\end{equation}
\end{lemma}
\begin{proof}[Proof of Lemma~\ref{lem:vanishing:omega}]
By~\eqref{eq:non:UO:rescaling:NRJ}, any subsequence of $(u^n, v^+_n, v^-_n, f^n, h_n)_n$ satisfies the conditions required to use Lemma~\ref{lem:intermediate:defect:measures}.
Thus, up to extract a subsequence (still denoted the same by a slight abuse of notation), using Lemma~\ref{lem:intermediate:defect:measures} and uniqueness of $\sdm$ and $\mdm$ as defect measures, we have~\eqref{eq:intermediate:defect:measure:support:sdm},~\eqref{eq:intermediate:defect:measure:support:mdm},~\eqref{eq:intermediate:defect:measure:support:22:sdm} and~\eqref{eq:intermediate:defect:measure:support:22:mdm} from Corollary~\ref{cor:intermediate:defect:measures}.
We set $\chi\in \tilde{S_{\phg}^0}(\cotang \hspace{1 pt}\MM_{(0, T)})$, with $\chi(\cdot, \cdot, \mathsf{R}\cdot, \mathsf{R}\cdot)\in S_{\phg}^0(\cotang \hspace{1 pt}\MM_{(0, T)})$ for $\mathsf{R}$ large enough, with
\begin{equation}\label{eq:13}
\supp\p{\chi} \cap \set{(t, x, \tau, \xi)\in\cotang\hspace{1 pt}\MM_{(0, T)}, ~\varepsilon\abs{\tau} > 1~\text{and}~\varepsilon\abs{\tau}\geq\abs{\xi}_g} = \emptyset
\end{equation}
and with
\begin{equation}\label{eq:14}
\supp\p{1 - \chi}\cap\set{(t, x, \tau, \xi)\in\cotang \hspace{1 pt}\MM_{(0, T)}, ~\tau^2 = \sqrtsymbol(x, \xi)^2} = \emptyset
\end{equation}
(see Figure~\ref{fig:hormander} for an illustration of these two sets).

We recall that $\SqrtOp_{h} = \frac{1}{2}\p{\Op_{h, \tang}^{\MM}(\sqrtsymbol) + \Op_{h, \tang}^{\MM}(\sqrtsymbol)^*}$ is defined in~\eqref{df:SqrtOp} and set
\begin{equation*}
\Phi^\omega_{h} := \Op_{h}^{\MM}\p{\chi}^*\Op_{h}^{\MM}\p{\phi_\omega}\Op_{h}^{\MM}\p{\chi}\in \pseudo_{h, \comp}^0(\MM_{(0, T)})    
\end{equation*}
compactly supported in the region $(0, T)\times \omega$.
Thus, by~\eqref{eq:14},~\eqref{eq:intermediate:defect:measure:support:sdm} and~\eqref{eq:intermediate:defect:measure:support:mdm}, we have
\begin{align*}
\pscal{v_{n}^{\signe}}{\Op_{h_n}^{\MM}\p{\phi_\omega} v_{n}^{\signe}}_{L^2(\MM_{(0, T)})} 
&= \pscal{\Op_{h_n}^{\MM}\p{\chi}v_{n}^{\signe}}{\Op_{h_n}^{\MM}\p{\phi_\omega} \Op_{h_n}^{\MM}\p{\chi}v_{n}^{\signe}}_{L^2(\MM_{(0, T)})} + \o{1}_n\\
&= \pscal{v_{n}^{\signe}}{\Phi_{h_n}^\omega v_{n}^{\signe}}_{L^2(\MM_{(0, T)})} + \o{1}_n.
\end{align*} 
Using the identity~\eqref{eq:identity:v:to:u}, we have
\begin{align}\notag
\frac{1}{2}\sum_{\signe\in\set{-, +}}\pscal{v_{n}^\signe}{\Phi_{h_n}^\omega v_{n}^\signe}_{L^2(\MM_{(0, T)})}
=& \pscal{h_n\partial_t u^n}{\Phi_{h_n}^\omega h_n\partial_t u^n}_{L^2(\MM_{(0, T)})}\\
&+ \pscal{\SqrtOp_{h_n} u^n}{\Phi_{h_n}^\omega \SqrtOp_{h_n} u^{n}}_{L^2(\MM_{(0, T)})}.\label{eq:15}
\end{align}
Then, we use selfadjointness of $\SqrtOp_{h}$ and $\SqrtOp_{h}^2 = \SchroOp_{h} + h\rest^{\SqrtOp, 1}_{h}$ from~\eqref{df:remainder:order:1}, to obtain
\begin{align*}
\pscal{\SqrtOp_{h_n} u^n}{\Phi_{h_n}^\omega \SqrtOp_{h_n} u^{n}}_{L^2(\MM_{(0, T)})}
=& \pscal{\SqrtOp_{h_n}^2 u^n}{\Phi_{h_n}^\omega u^{n}}_{H^{-1}_{h_n, \loc}(\MM_{(0, T)}), H^{1}_{h_n, \comp}(\MM_{(0, T)})} \\
&+ \pscal{\SqrtOp_{h_n} u^h}{\classe{\Phi_{h_n}^\omega, \SqrtOp_{h_n}} u^{n}}_{L^2(\MM_{(0, T)})}\\
=& \pscal{\p{-h_n^2\Delta_g + V} u^n}{\Phi_{h_n}^\omega u^{n}}_{H^{-1}_{h_n, \loc}(\MM_{(0, T)}), H^{1}_{h_n, \comp}(\MM_{(0, T)})}\\
&+ h_n \pscal{u^n}{\rest^{1}_{h_n}u^{n}}_{L^2(\MM_{(0, T)})}\\
=& \pscal{-h_n^2 \partial_t^2u^n}{\Phi_{h_n}^\omega u^{n}}_{H^{-1}_{h_n, \loc}(\MM_{(0, T)}), H^{1}_{h_n, \comp}(\MM_{(0, T)})}\\
&+ h_n \pscal{u^n}{\rest^{1}_{h_n}u^{n}}_{L^2(\MM_{(0, T)})}\\
=& \pscal{h_n \partial_tu^n}{\Phi_{h_n}^\omega h_n \partial_t u^{n}}_{L^2(\MM_{(0, T)})} + h_n \pscal{u^n}{\tilde{\rest}^{1}_{h_n}u^{n}}_{L^2(\MM_{(0, T)})}
\end{align*}
where, using Theorem~\ref{Th:hormander} (by~\eqref{eq:13}) and Proposition~\ref{pr:composition}, we have
\begin{align*}
\rest^{1}_{h} &= \frac{1}{h}\classe{\Phi_{h}^\omega, \SqrtOp_{h}} + \p{\rest^{\SqrtOp, 1}_{h}}^*\Phi_{h}^\omega\in \pseudo^1_{h, \comp}\p{\MM_{(0, T)}},\\
\tilde{\rest}^{1}_{h} &= \frac{1}{h}\classe{h\partial_t, \Phi_{h}^\omega}+ \rest^{1}_{h}\in \pseudo^1_{h, \comp}\p{\MM_{(0, T)}}.
\end{align*}
Then, for remainders terms, using boundedness of $\p{u^n}_n$ in $H_{h_n}^1(\MM_{(0, T)})$ from Lemma~\ref{lem:SPLITTING}, we obtain
\begin{equation*}
\frac{1}{2}\sum_{\signe\in\set{-, +}}\pscal{v_{n}^\signe}{\Phi_{h_n}^\omega v_{n}^\signe}_{L^2(\MM_{(0, T)})}
= 2\pscal{h_n\partial_t u^n}{\Phi_{h_n}^{\omega}h_n\partial_t u^n}_{L^2(\MM_{(0, T)})} + \O{h_n}.
\end{equation*}
Moreover, by Proposition~\ref{pr:Calderon:Vaillancourt}, we have
\begin{equation*}
\abs{\pscal{h_n\partial_t u^n}{\Phi^{\omega}_{h_n}h_n\partial_t u^n}}\leq \sup_{\cotang\hspace{1pt} \MM}\abs{\phi_\omega}\norm{h_n\partial_tu^n}_{L^2(\supp(\phi_\omega))}^2 + \O{h_n^{1/2}}.
\end{equation*}
Then, this last inequality and $\min_{\supp(\phi_\omega)}b_\omega^2 > 0$, we obtain the upper bound
\begin{equation*}
\frac{1}{2}\sum_{\signe\in\set{-, +}} \pscal{v_{n}^\signe}{\Phi_{h_n}^\omega v_{n}^\signe}
\leq 2 
\frac{\sup_{\cotang\hspace{1pt} \MM}\abs{\phi_\omega}}
{\min_{\supp(\phi_\omega)}b_\omega^2}
\norm{b_\omega h_n\partial_t u^n}_{L^2((0, T)\times\omega)}^2 + \O{h_n^{1/2}}.
\end{equation*}
Combining this with~\eqref{eq:15}, and then with~\eqref{eq:non:UO:rescaling:0}, gives~\eqref{eq:29}, we conclude the proof of the subsequence.
Since the argument applies to any subsequence of $(u^n, v_n^\signe, f^n, h_n)$, the proof of Lemma~\ref{lem:vanishing:omega} is complete.
\end{proof}
\begin{lemma}\label{lem:sdm:mdm:vanishes}
Assume $V$ satisfies~\eqref{df:potential:fixed:V}, 
$b_\omega \in\Cscr^{0}(M, \RR_+^*)$, 
$\omega = \set{b_\omega > 0}$, $T > 0$,
$(u^n, h_n)_n$ satisfies~\eqref{eq:non:UO:rescaling:NRJ} and~\eqref{eq:non:UO:rescaling:0}, and $(v_n^{\signe})_n$ is defined in~\eqref{df:sequence:indexed}.
Assume, in addition, that:
\begin{itemize}
\item~\ref{GCC+}$(\omega, T)$ holds;
\item There exist continuous (for the weak-* topology of measures) maps:
\begin{equation*}
\sdm^{\signe} :t\in(0, T)\mapsto \sdm_{t}^{\signe}\in\Mc_+(\cotang M)\text{ and }\mdm^{\signe} :t\in(0, T)\mapsto \mdm_{t}^{\signe}\in\Mc_+(\cosphere M)
\end{equation*}
such that $\p{v_n^{\signe}(t), h_n}_n$ admits $\sdm^{\signe}_t$ (\resp $\mdm^{\signe}_t$) as unique semiclassical (\resp 2-microlocal) defect measure in the sense of~\eqref{eq:SDM} (\resp~\eqref{eq:MDM}) for all $t\in (0, T)$;
\item $(v_n^{\signe}, h_n)_n$ admits $\sdm_{(\signe, \signe)}\in\Mc_+(\cotang\hspace{1pt} \MM_{(0, T)})$ (\resp $\mdm_{(\signe, \signe)}\in\Mc_+(\cosphere\hspace{1pt} \MM_{(0, T)})$) as unique semiclassical (\resp 2-microlocal) defect measure in the sense of~\eqref{eq:SDM} (\resp~\eqref{eq:MDM}) and, denoting Hamiltonian flows as in~\eqref{df:Hamiltonian:Flow}, we have the disintegration and propagation formulas
\begin{align}
\sdm_{(\signe, \signe)}(t, x, \tau, \xi) &= \sdm^{\signe}_{t}\p{x, \xi}\otimes \delta_{\signe\sqrtsymbol(x, \xi)}(\tau)\otimes \Lc_1(t) \quad\text{with}\quad\sdm^{\signe}_{t} = \p{\varphi_{t- t_0}^{\signe\sqrtsymbol}}^\ast\sdm^{\signe}_{t_0}, ~ t_0\in (0, T),\label{eq:sdm:disintegration}\\
\mdm_{(\signe, \signe)}(t, x, \tau, \xi) &= \mdm^{\signe}_{t}\p{x, \xi}\otimes \delta_{\signe\abs{\xi}_g}(\tau)\otimes \Lc_1(t)\quad\text{with}\quad\mdm^{\signe}_{t} =\p{\varphi_{t- t_0}^{\signe\abs{\xi}_g}}^\ast\mdm^{\signe}_{t_0}, ~ t_0\in (0, T).\label{eq:mdm:disintegration}
\end{align}
\end{itemize}
Then, for all $t\in (0, T)$ the measure $\sdm_{t}^{\signe}$ (\resp $\mdm_{t}^{\signe}$) vanishes identically on $\cotang M$ (\resp on $\cosphere M$).
\end{lemma}
The idea of the proof for Lemma~\ref{lem:sdm:mdm:vanishes} is first: using~\eqref{eq:non:UO:rescaling:0} to obtain
\begin{equation*}
\sdm_{(\signe, \signe)}\p{\set{(t, x, \tau, \xi)\in \cotang\hspace{1pt}\MM_{(0, T)},~ x\in\omega}} = \mdm_{(\signe, \signe)}\p{\set{(t, x, \tau, \xi)\in \cosphere\hspace{1pt}\MM_{(0, T)},~ x\in\omega}} = 0,
\end{equation*}
and, second, to use~\ref{GCC+}$(\omega, T)$ and the propagation formulas~\eqref{eq:sdm:disintegration} and~\eqref{eq:mdm:disintegration} to obtain that for any point in $\cotang \hspace{1 pt}\MM_{(0, T)}$ (for semiclassical measures) or $\cosphere \hspace{1 pt}\MM_{(0, T)}$ (for 2-microlocal measures) these measures vanish in a neighborhood of this point.
See Figure~\ref{fig;illustration:Omega:0} for an illustration of the strategy of proof.
\begin{proof}[Proof for semiclassical result of Lemma~\ref{lem:sdm:mdm:vanishes}]
First, using Lemma~\ref{lem:vanishing:omega} and~\eqref{eq:SDM}: for all $\phi_\omega\in \test (\cotang \hspace{1pt} \MM_{(0, T)})$ supported in $(0, T)\times \omega$, we have
\begin{equation*}
0 = \lim_{n\to\infty}\sum_{\signe\in\set{-, +}} \pscal{v_{n}^\signe}{\Op_{h_n}^{\MM}(\phi_\omega) v_{n}^\signe}_{L^2\p{\MM_{(0, T)}}} = \pscal{\sdm_{(+, +)} + \sdm_{(-, -)}}{\phi_\omega}_{\Mc_+, \Cscr_c^0\p{\cotang \hspace{1 pt}\MM_{(0, T)}}}.
\end{equation*}
Therefore, $\sdm_{(+, +)} + \sdm_{(-, -)}$ vanishes identically on $\set{(t, x, \tau, \xi)\in \cotang\hspace{1pt}\MM_{(0, T)},~ x\in\omega}$. Since $\sdm_{(+, +)}$ and $\sdm_{(-, -)}$ are nonnegative Radon measures, they both vanish identically on $\set{(t, x, \tau, \xi)\in \cotang\hspace{1pt}\MM_{(0, T)},~ x\in\omega}$.

Second, in this paragraph we prove that $\sdm_{t_0}^{\signe}$ vanishes identically on $\cotang M$ for a fixed $t_0\in (0, T)$ and a fixed sign $\signe\in\set{-, +}$ (see Remark~\ref{rk:sign} for the definition of $-\signe$).
Let us consider $\rho_{t_0}\in \cotang M$.
By~\ref{GCCV}$(\omega, T)$, there is a time $t_{\signe}\in (0, T)$ such that 
\begin{equation*}
\pi_M\circ \varphi^{\sqrtsymbol}_{\frac{T -\signe T}{2} +\signe t_{\signe}}\p*{\varphi_{\signe\frac{T -\signe T}{2} -\signe t_0}^{\sqrtsymbol}(\rho_{t_0})} = \pi_M\circ\varphi_{t_{\signe} - t_0}^{\signe\sqrtsymbol}(\rho_{t_0})\in \omega,
~\text{where}~
\textstyle\frac{T -\signe T}{2} +\signe t_{\signe} =
\eqcases{ll}{
T - t_{\signe}&\text{if}~\signe = -,\\
t_{\signe}&\text{if}~\signe = +.
}
\end{equation*}
Let $\Omega_{\signe t_{\signe}}\subset \set{(x, \xi)\in \cotang M,~x\in\omega}$ be an open  neighborhood of $\varphi_{t_{\signe} - t_0}^{\signe\sqrtsymbol}(\rho_{t_0})$ and, using bicontinuity of the flow $\varphi^{\sqrtsymbol}$, let $\delta > 0$ be small enough such that 
\begin{equation}\label{df:Omega_0}
(t_{\signe}-\delta,t_{\signe}+ \delta)\subset (0, T)
\quad\text{and}\quad
\forall \epsilon \in (-\delta, \delta), \quad \varphi^{\sqrtsymbol}_{\epsilon}(\Omega_{\signe t_{\signe}})\subset \set{(x, \xi)\in \cotang M,~x\in\omega}.
\end{equation}
We define $\Omega_{0} := \varphi^{\signe\sqrtsymbol}_{t_0-t_{\signe}}(\Omega_{\signe t_{\signe}})$, a neighborhood of $\rho_{t_0}$ by bicontinuity of the flow $\varphi^{\sqrtsymbol}$.
Then, for all time cut-off function $\theta\in \test((-\delta, \delta))$ such that $\int_\RR \theta(t)\d t = 1$, all space cut-off function $\chi_{\Omega_{0}}\in\test(\Omega_{0})$, 
and all $t\in (t_{\signe} -\delta, t_{\signe} + \delta)$, by~\eqref{df:Omega_0}, we have
\begin{equation}
\supp(\chi_{\Omega_{0}}\circ\varphi^{-\signe\sqrtsymbol}_{t - t_0})
\subset \varphi^{\signe \sqrtsymbol}_{t - t_0} (\Omega_{0})
= \varphi^{\signe\sqrtsymbol}_{t - t_0}\circ \varphi^{\signe\sqrtsymbol}_{t_0 - t_{\signe}}(\Omega_{\signe t_{\signe}})= \varphi^{\signe\sqrtsymbol}_{t - t_{\signe}}(\Omega_{\signe t_{\signe}})
\subset\set{(x, \xi)\in \cotang M,~x\in\omega}.
\label{eq:support:sdm:proof}
\end{equation}
Since $\int_{\RR}\theta(t)\d t = 1$, we have
\begin{equation*}
\pscal{\sdm_{t_0}^{\signe}}{\chi_{\Omega_{0}}}_{\Mc_+, \Cscr^0_c(\cotang M)} = \int_\RR \theta(t - t_{\signe})\pscal{\sdm_{t_0}^{\signe}}{\chi_{\Omega_{0}}}_{\Mc_+, \Cscr^0_c(\cotang M)}\d t
\end{equation*}
Then, using~\eqref{eq:sdm:disintegration}, we have
\begin{align*}
\pscal{\sdm_{t_0}^{\signe}}{\chi_{\Omega_{0}}}_{\Mc_+, \Cscr^0_c(\cotang M)}
&= \int_\RR\theta(t - t_{\signe})\pscal{\p{\varphi^{\signe\sqrtsymbol}_{t - t_0}}^\ast\sdm_{t_0}^{\signe}}{\chi_{\Omega_{0}}\circ \varphi^{-\signe\sqrtsymbol}_{t - t_0}}_{\Mc_+, \Cscr^0_c(\cotang M)}\d t\\
&= \pscal{\sdm_{t}^{\signe}\otimes \Lc_1(t)}{\theta(t -t_{\signe})\chi_{\Omega_{0}}\circ\varphi^{-\signe\sqrtsymbol}_{t - t_0}}_{\Mc_+, \Cscr^0_c((0, T)\times \cotang M)}\\
&=\pscal{\sdm_{(\signe, \signe)}}{\theta(\cdot  - t_{\signe})\classe{\chi_{\Omega_{0}}\circ\varphi^{-\signe\sqrtsymbol}_{t - t_0}}}_{\Mc_+, \Cscr^0_c(\cotang \hspace{1 pt}\MM_{(0, T)})}.
\end{align*}
Hence, because $\theta(\cdot - t_{\signe})\classe{\chi_{\Omega_{0}}\circ\varphi^{-\signe\sqrtsymbol}_{t - t_0}}$ is supported in $(t_{\signe} - \delta, t_{\signe} + \delta)\times \cotang\omega \subset (0, T)\times\cotang \omega$ by~\eqref{eq:support:sdm:proof} and since $\sdm_{(\signe, \signe)}$ vanishes identically on $\cotang\hspace{1 pt}\p{(0, T)\times \omega}$, we conclude
\begin{equation*}
\pscal{\sdm_{t_0}^{\signe}}{\chi_{\Omega_{0}}}_{\Mc_+, \Cscr^0_c(\cotang M)} = 0.
\end{equation*}
As a consequence, $\Omega_{0}\cap\supp(\sdm_{t_0}^{\signe}) = \emptyset$.
This is true for any point $\rho_{t_0}\in\cotang M$, and we thus conclude that $\sdm_{t_0}^{\signe}$ vanishes identically on $\cotang M$. 

Finally, since~\eqref{eq:sdm:disintegration} gives the propagation result: $\sdm_{t}^{\signe} = \p{\varphi^{\signe\sqrtsymbol}_{t - t_0}}^\ast\sdm_{t_0}^{\signe}$ for all $t\in (0, T)$, we conclude that $\sdm_t^{\signe}$ vanishes identically on $\cotang M$ at any $t\in (0, T)$.
\end{proof}
\begin{figure}
\begin{center}
\tikzset{every picture/.style={line width=0.75pt}} 
\begin{tikzpicture}[x=0.75pt,y=0.75pt,yscale=-1,xscale=1]
\draw  [fill={rgb, 255:red, 233; green, 233; blue, 233 }  ,fill opacity=1 ][dash pattern={on 4.5pt off 4.5pt}] (337,64.5) .. controls (357,54.5) and (461,44.17) .. (471,71.17) .. controls (481,98.17) and (467.95,114.13) .. (461,148.17) .. controls (454.05,182.2) and (475,176.17) .. (432,217.17) .. controls (389,258.17) and (322,219.17) .. (314,169.17) .. controls (306,119.17) and (317,74.5) .. (337,64.5) -- cycle ;
\draw  (55,68.5) .. controls (75,58.5) and (187,102) .. (167,122) .. controls (147,142) and (147,152) .. (167,182) .. controls (187,212) and (97,212) .. (77,182) .. controls (57,152) and (35,78.5) .. (55,68.5) -- cycle ;
\draw  (361,74) .. controls (381,64) and (471,54) .. (451,74) .. controls (431,94) and (431,104) .. (451,134) .. controls (471,164) and (408,256.5) .. (357,218.5) .. controls (306,180.5) and (341,84) .. (361,74) -- cycle ;
\draw (94,143.5) .. controls (135,193) and (253,212.5) .. (381,167.5) ;
\draw [shift={(381,167.5)}, rotate = 25.63] [line width=0.75]    (-5.59,0) -- (5.59,0)(0,5.59) -- (0,-5.59)   ;
\draw [shift={(238.13,192.76)}, rotate = 180.7] [fill=black][line width=0.08]  [draw opacity=0] (10.72,-5.15) -- (0,0) -- (10.72,5.15) -- (7.12,0) -- cycle    ;
\draw [shift={(94,143.5)}, rotate = 95.37] [line width=0.75]    (-5.59,0) -- (5.59,0)(0,5.59) -- (0,-5.59)   ;
\draw (154.97,109.25) .. controls (205,72.84) and (326.58,57.34) .. (355,98.5) ;
\draw [shift={(152,111.5)}, rotate = 321.63] [fill=black][line width=0.08]  [draw opacity=0] (10.72,-5.15) -- (0,0) -- (10.72,5.15) -- (7.12,0) -- cycle    ;
\draw (90,85) node [anchor=north west][inner sep=0.75pt] {$\Omega_{t_0}$};
\draw (320,30) node [anchor=north west][inner sep=0.75pt]    {${\varphi^{\signe\sqrtsymbol}_{\epsilon}\p{\Omega_{\signe t_{\signe}}}}\subset \set{(x, \xi)\in\cotang M, ~x\in\omega}$};
\draw (83,150) node [anchor=north west][inner sep=0.75pt]    {$\rho_{t_0}$};
\draw (362,140) node [anchor=north west][inner sep=0.75pt]    {$\varphi^{\signe\sqrtsymbol}_{(t_{\signe} - t_0) + \epsilon}(\rho_{t_0})$};
\draw (373,78.4) node [anchor=north west][inner sep=0.75pt]  {${\varphi^{\signe\sqrtsymbol}_{\epsilon}\p{\Omega_{\signe t_{\signe}}}}$};
\draw (251,193.4) node [anchor=north west][inner sep=0.75pt] {${\varphi^{\signe\sqrtsymbol}_{(t_{\signe} - t_0) + \epsilon}}$};
\draw (210,50) node [anchor=north west][inner sep=0.75pt]    {${\varphi^{\signe\sqrtsymbol}_{(t_0 - t^{\star}) -\epsilon}}$};
\draw (400,250) node {time $t_{\signe} + \epsilon\in (0, T)$};
\end{tikzpicture}
\end{center}
\caption{Illustration of~\eqref{df:Omega_0}, with $\epsilon \in (-\delta, \delta)$.}\label{fig;illustration:Omega:0}
\end{figure}
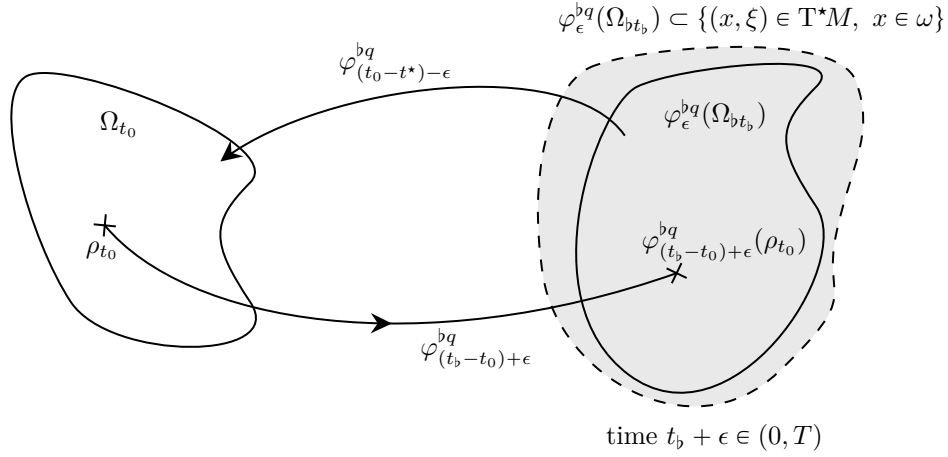
\begin{proof}[Proof for the 2-microlocal result of Lemma~\ref{lem:sdm:mdm:vanishes}]~

First, for all $\phi_\omega\in \test ((0, T)\times \omega)$ we prove that $\phi_\omega (\mdm_{(+,  +)} + \mdm_{(-, -)})$ vanishes identically on $\cosphere\hspace{1pt} \MM$.
Using Lemma~\ref{lem:vanishing:omega} and~\eqref{eq:MDM}, we obtain
\begin{align}
0 &= \lim_{R\to\infty}\lim_{n\to\infty}\sum_{\signe\in\set{-, +}}\pscal{v_n^{\signe}}{\Op_{h_n}^{\MM}\p*{
\phi_\omega\p*{1 - 
\theta\p*{\frac{\sqrt{1+\tau^2 + \abs{\xi}_g^2}}{R}}
}
} v_n^{\signe}}_{L^2(\MM)}\notag \\
&= \pscal{\mdm_{(+,  +)} + \mdm_{(-, -)}}{\phi_\omega}_{\Mc_+, \Cscr_c^0(\MM_{(0, T)})}.\label{eq:vanishes:windows:mdm}
\end{align}
In conclusion,~\eqref{eq:vanishes:windows:mdm} implies $\mdm_{(+,  +)} + \mdm_{(-, -)}$ vanishes identically on $\set{(t, x, \tau, \xi)\in\cosphere\hspace{1pt}\MM_{(0, T)}, ~ x\in \omega}$.
Then, since $\mdm_{(+, +)}$ and $\mdm_{(-, -)}$ are nonnegative Radon measures, they both vanish identically on $\set{(t, x, \tau, \xi)\in\cosphere\hspace{1pt}\MM_{(0, T)}, ~ x\in \omega}$.

Secondly, we follow the strategy of Lebeau as for the semiclassical defect measures case for 2-microlocal defect measure to prove that $\mdm_{t_0}^{\signe}$ vanishes identically on $\cosphere M$ for a fixed $t_0\in (0, T)$.
Since the proof is the same as for the semiclassical case, replacing~\ref{GCCV}$(\omega, T)$ by~\ref{GCC}$(\omega, T)$ and working on $\cosphere M$ instead of $\cotang M$, we do not rewrite it.
Then, we conclude that $\mdm_t^{\signe}$ vanishes identically on $\cosphere M$ for all $t\in(0, T)$.
\end{proof}
\subsection{Conclusion of the contradiction argument}
In this subsection, using Lemma~\ref{lem:lmbd:unbounded} and~\ref{lem:sdm:mdm:vanishes}, we prove that~\eqref{GCC+} implies~\eqref{eq:UO} by contradiction.
\begin{proof}[Proof that~\eqref{GCC+} implies~\eqref{eq:UO} in Theorem~\ref{Th:GCC:UO}]
Proceeding by contradiction, we suppose that the pair $(\omega, T)$ satisfies~\eqref{GCC+} and that~\eqref{eq:non:UO} holds. Then, by Lemma~\ref{lem:lmbd:unbounded},~\eqref{eq:non:UO} implies the existence of a sequence $(u_0^n, u_1^n, u^n, h_n)_n$ satisfying both~\eqref{eq:non:UO:rescaling:NRJ} and~\eqref{eq:non:UO:rescaling:0}.
Then, we set $(v^{\signe}_n, f^n)_n$ as the family defined in~\eqref{df:sequence:indexed} indexed here by $(h_n)_n$.
Then, by~\eqref{eq:non:UO:rescaling:NRJ}, we know that $(u_0^n, h_n u_1^n)_n$ is bounded in $H^1_{h_n}(M)\times L^2(M)$.
Hence, by Lemma~\ref{lem:intermediate:defect:measures}, up to considering a subsequence still denoted $\p{\p{v^{+}_n,v^{-}_n, f^n, u^n}, h_n}_n$:
\begin{itemize}
\item $\p{\p{v^{+}_n,v^{-}_n, f^n, u^n}, h_n}_n$ admits $\sdm$ --- defined in~\eqref{eq:intermediate:defect:measure} from Lemma~\ref{lem:intermediate:defect:measures} --- as unique semiclassical defect measure in the sense of~\eqref{eq:SDM}.
\item $\p{\p{v^{+}_n,v^{-}_n, f^n, u^n}, h_n}_n$ admits $\mdm$ --- defined in~\eqref{eq:intermediate:defect:measure} from Lemma~\ref{lem:intermediate:defect:measures} --- as unique 2-microlocal defect measure in the sense of~\eqref{eq:MDM}.
\end{itemize}
Therefore, up to extracting another subsequence of $\p{u^n, v^{\signe}_n, f^n, h_n}_n$, still denoted $\p{u^n, v^{\signe}_n, f^n, h_n}_n$, by Theorem~\ref{Th:Slices} (the assumptions~\eqref{eq:wave-type-ghost} is satisfied by~\eqref{eq:1:intermediate:defect:measure:support:22:sdm} and~\eqref{eq:1:intermediate:defect:measure:support:22:mdm} from Lemma~\ref{lem:intermediate:defect:measures}) we define continuous (for the weak-* topology of measures) maps:
\begin{equation*}
\sdm^{\signe} :t\in(0, T)\mapsto \sdm_{t}^{\signe}\in\Mc_+(\cotang M)
\quad\text{and}\quad
\mdm^{\signe} :t\in(0, T)\mapsto \mdm_{t}^{\signe}\in\Mc_+(\cosphere M),\quad \signe \in\set{-, +}
\end{equation*}
such that $\p{v_n^{\signe}(t), h_n}_n$ admits $\sdm^{\signe}_t$ (\resp $\mdm^{\signe}_t$) as unique semiclassical (\resp 2-microlocal) defect measure in the sense of~\eqref{eq:SDM} (\resp~\eqref{eq:MDM}) for all $t\in (0, T)$.
Furthermore, using Corollary~\ref{cor:PG} (according to Lemma~\ref{lem:intermediate:defect:measures} and the last point of Corollary~\ref{cor:intermediate:defect:measures}, the hypotheses of Corollary~\ref{cor:PG} are satisfied), the maps $\sdm^{\signe}$ and $\mdm^{\signe}$ satisfy the disintegration and propagation formulas~\eqref{eq:sdm:disintegration} and~\eqref{eq:mdm:disintegration} from Lemma~\ref{lem:sdm:mdm:vanishes}.
Hence by Lemma~\ref{lem:sdm:mdm:vanishes},
\begin{equation}\label{eq:sdm:mdm:vanishes}
\text{for all}~t\in(0, T),~\sdm_{t}^{\signe}~(\text{\resp}\mdm_{t}^{\signe})~\text{vanishes identically on}~ \cotang M~(\text{\resp}\cosphere M).
\end{equation}

For any fixed $t_0\in (0, T)$, any fixed cut-off function $\theta$ as~\eqref{df:theta} and all $R > 0$, using~\eqref{df:remainder:order:1} and boundedness of $(u^n(t_0))_n$ in $L^2(M)$, we have
\begin{align}
\notag \Escr^{h_n}(u^n(t_0), \partial_tu^n(t_0)) 
&= \frac{1}{2}\classe{\norm{h_n\partial_ t u^n(t_0)}^{2}_{L^2(M)} + \pscal{u^n(t_0)}{\p{-h_n^2\Delta_g + V}u^n(t_0)}_{H^1_{h}(M), H^{-1}_{h}(M)}}\\
\notag&= \frac{1}{2}\classe{\norm{h_n\partial_ t u^n(t_0)}^{2}_{L^2(M)} + \norm{\SqrtOp_{h_n}u^n(t_0)}^{2}_{L^2(M)}} + \O{h_n}\\
&=\Ke_{h_n}^R(t_0) + \Le_{h_n}^R(t_0) + \O{h_n},\label{eq:NRJ:cut}
\end{align}
where we set
\begin{align*}
{\Ke}_{h_n}^R(t_0) &:= \frac{1}{2}\pscal{h_n D_t u^n(t_0)}{\textstyle\Op_{h_n}^M\classe{{\theta\p{{\jap{\xi}_g}/{R}}}}h_n D_t u^n(t_0)}_{L^2(M)} \\
&\phantom{:=} + \frac{1}{2}\pscal{\SqrtOp_{h_n} u^n(t_0)}{\textstyle\Op_{h_n}^M\classe{{\theta\p{{\jap{\xi}_g}/{R}}}}\SqrtOp_{h_n} u^n(t_0)}_{L^2(M)},
\end{align*}
and 
\begin{align*}
{\Le}_{h_n}^R(t_0) &:= \frac{1}{2}\pscal{h_n D_t u^n(t_0)}{\textstyle\Op_{h_n}^M\classe{{1 - \theta\p{{\jap{\xi}_g}/{R}}}}h_n D_t u^n(t_0)}_{L^2(M)} \\
&\hspace{10pt} + \frac{1}{2}\pscal{\SqrtOp_{h_n} u^n(t_0)}{\textstyle\Op_{h_n}^M\classe{{1 - \theta\p{{\jap{\xi}_g}/{R}}}}\SqrtOp_{h_n} u^n(t_0)}_{L^2(M)}.
\end{align*}
We recall that~\eqref{eq:non:UO:rescaling:NRJ} together with~\eqref{eq:WP} imply that $(u^n(t_0))_n$ is bounded in $H^1_{h_n}(M)$ and thus that $\p{\SqrtOp_{h_n}u^n(t_0), h_n\partial_t u^n(t_0)}_n$ is bounded in $L^2\p{M}\times L^2(M)$.
Thus, by Proposition~\ref{pr:SDM} and Theorem~\ref{Th:Microlocalization}, up to extracting again a subsequence of $(u^n, v_n^{+}, v_n^{-}, h_n)_n$, $\p{\p{\SqrtOp_{h_n}u^n(t_0), h_n\partial_t u^n(t_0)}, h_n}_n$ admits a unique semiclassical defect measure in the sense of~\eqref{eq:SDM} and a unique 2-microlocal defect measure in the sense of~\eqref{eq:MDM}.
Therefore, we can use Corollary~\ref{cor:use:HWE:for:WE} with~\eqref{eq:sdm:mdm:vanishes} to conclude that
\begin{equation*}
\lim_{n\to\infty} \Ke_{h_n}^R(t_0) = \frac{1}{4}\pscal{\sdm_{t_0}^+ + \sdm_{t_0}^-}{\textstyle\theta\p{{\jap{\xi}_g}/{R}}}_{\Mc_+, \Cscr^0_c(\cotang M)} = 0, \quad \forall R >0,
\end{equation*}
and
\begin{equation*}
\lim_{R\to\infty}\lim_{n\to\infty}{\Le}_{h_n}^R(t_0) = \frac{1}{4}\pscal{\mdm_{t_0}^+ + \mdm_{t_0}^-}{1}_{\Mc_+, \Cscr^0_c(\cosphere M)} = 0.
\end{equation*}
Coming back to~\eqref{eq:NRJ:cut}, we obtained
\begin{align*}
\lim_{n\to\infty}\Escr^{h_n}(u^n(t_0), \partial_tu^n(t_0))
&= \lim_{R\to\infty} \lim_{n\to\infty}\Escr^{h_n}(u^n(t_0), \partial_tu^n(t_0))\\
&= \lim_{R\to\infty} \lim_{n\to\infty} \Ke_{h_n}^R(t_0)+ \lim_{R\to\infty}\lim_{n\to\infty}{\Le}_{h_n}^R(t_0) = 0.
\end{align*}
However, if~\eqref{eq:non:UO} would hold, then~\eqref{eq:non:UO:rescaling:NRJ} would give
\begin{equation*}
1 = \lim_{n\to\infty}\Escr^{h_n}(u^n(t_0), \partial_tu^n(t_0)) = 0.
\end{equation*}
Therefore,~\eqref{eq:non:UO} leads to a contradiction.
This concludes the proof that~\eqref{GCC+} implies uniform observability~\eqref{eq:UO} of Theorem~\ref{Th:GCC:UO}.
\end{proof}
\section{Uniform observability \texorpdfstring{\eqref{eq:UO}}{} implies \texorpdfstring{\eqref{GCC+}}{}}\label{section:CN}
In this section, using the technical Sections~\ref{section:microlocalization} and~\ref{section:slice:defect:measure}, we prove the necessity part of Theorem~\ref{Th:GCC:UO}.
It is known (see~\cite{bardos1992sharp, burq1997condition, dehman2014controllability}) that \emph{observability inequality}~\eqref{eq:Obs} (implied by~\eqref{eq:UO}) implies~\eqref{GCC}.
The goal of this section is then to show that uniform observability~\eqref{eq:UO} also implies~\eqref{GCCV}, and hence~\eqref{GCC+}.
Recall that $(M, g)$ is a smooth, connected and compact Riemannian manifold without boundary.

Proceeding by contraposition, we assume that the pair $(\omega, T)$ does not satisfy~\eqref{GCCV} and prove that uniform observability~\eqref{eq:UO} cannot hold.
Thus, for $\sqrtsymbol$ defined in~\eqref{df:Sqrt:Sym} and $\varphi^\sqrtsymbol$ in~\eqref{df:Hamiltonian:Flow}, there exists a point $\rho_{-T/2}\in \cotang M$, fixed throughout this section, such that there is no $t\in (0, T)$ for which $\pi_M\circ\varphi_t^\sqrtsymbol(\rho_{-T/2})\in\omega$.
More conveniently, setting $\rho_0 = (x_0, \xi_0) := \varphi_{T/2}^\sqrtsymbol(\rho_{-T/2})$, we have
\begin{equation}\label{eq:non:GCC+}
\textstyle\forall t \in (-\frac{T}{2}, \frac{T}{2}), \quad \pi_M\circ\varphi_t^\sqrtsymbol(\rho_{0})\notin\omega,
\end{equation}
and we set
\begin{equation}\label{df:varGamma}
\textstyle\varGamma = \set{\varphi^{\sqrtsymbol}_t(\rho_0), ~ t \in (-\frac{T}{2}, \frac{T}{2})}.
\end{equation}
The idea is to construct a sequence of solutions of~\eqref{eq:Wave} concentrating on $\varGamma$, and then show that it does not satisfy~\eqref{eq:UO}.
\subsection{Concentration around an Hamiltonian curve transgressing~\texorpdfstring{\eqref{GCCV}}{}}

We consider a chart $(\kappa, \Uc_\kappa)$ containing $x_0$ and denote by $\tilde{\kappa}$ the symplectic lift of $\kappa$, namely
\begin{equation*}
\tilde{\kappa}\colon (x, \xi)\in \cotang\hspace{1 pt}\Uc_\kappa\mapsto \big(\kappa(x), \trans(\d_x\kappa)^{-1}\xi\big)\in \cotang\hspace{1pt}\RR^d.
\end{equation*}
\paragraph{Coherent states in the Euclidean case $\cotang\hspace{1pt}\RR^d$.}
Consider $(y_0, \gamma_0) := \tilde{\kappa}(x_0, \xi_0)$, and define the Gaussian \emph{coherent state}
\begin{equation}\label{eq:gaussian:family}
g^h(y) = \frac{C_0}{(\pi h)^{d/4}}e^{\frac{i}{h}\pscal{y - y_0}{\gamma_0}}e^{-\frac{1}{2h}\abs{y - y_0}^2},
\end{equation}
where $C_0 > 0$ is a constant that will be chosen later on. 
The family $\p{g^h}_{h}$ admits $C_0\delta_{(y_0, \gamma_0)}$ as unique semiclassical defect measure (see, for example,~\cite{macia2002observability, dehman2014dependence, zworski2022semiclassical,leautaud2023long}, the following version is proved in~\cite[Lemma A.9]{leautaud2023long}).
\begin{lemma}\label{lem:Gaussian:concentration}
Let $(g^h)_{h}\in H_{h}^\infty(\RR^d)$ be the Gaussian \emph{coherent state} defined in~\eqref{eq:gaussian:family}.
For all $h\in(0, 1]$ we have $\norm{g^h}_{L^2(\RR^d)} = C_0$ and
\begin{equation*}
\forall m\in\RR,~\forall a^h\in S^m(\cotang\hspace{1pt}\RR^d),~\exists C > 0,\quad \abs*{\pscal{g^h}{\Op^{\mathrm{W}}_{h}(a^h)g^h}_{L^2} - C_0a^h(y_0, \gamma_0)}\leq Ch^{\frac{1}{2}}.
\end{equation*}
\end{lemma}
\paragraph{Coherent states on $\cotang M$.}
We now use $\kappa$ to export the Gaussian \emph{coherent state} from $\cotang\hspace{1pt}\RR^d$ to $\cotang M$.
Consider a test function $\psi\in\test(\RR^d_y)$ supported in $\kappa(\Uc_\kappa)$, equal to $1$ near $y_0$ (see, for example,~\cite[Section A.7]{leautaud2023long}).
Then, we set
\begin{equation}\label{eq:coherent:state:family}
u_0^{h} := \kappa^\ast (g^h\psi)\in \test(M),
\end{equation}
and, for convenience, we choose $C_0$ such that $\norm{u_0^{h}}_{L^2(M)} = 1$.
Then, we have an analogous result to Lemma~\ref{lem:Gaussian:concentration} in $\cotang M$ (see~\cite[Corollary A.10]{leautaud2023long} for a proof).
\begin{lemma}\label{lem:diracmass}
Let $(u_0^{h})_{h}$ be the sequence defined in~\eqref{eq:coherent:state:family}. Then, for all $m\in\RR$ and all $A_{h}\in \pseudo_{h}^m(\cotang M)$.
\begin{equation}\label{eq:coherent:state:convergency}
\exists C > 0, ~\exists h_0\in(0, 1),~\forall h\in(0, h_0),\quad \abs*{\pscal{u_0^h}{A_{h}u_0^h}_{L^2} -\sigma_{h}(A_{h})(\rho_0)}\leq Ch^{\frac{1}{2}}.
\end{equation}
\end{lemma}
\begin{corollary}\label{cor:diracmass}
Let $(u_0^{h})_{h}$ be the sequence defined in~\eqref{eq:coherent:state:family}.
For all $A_{h}\in \pseudo_{h}^m(\cotang M)$ with principal symbol $\sigma_h(A_h)$ independent of $h$, the family $(A_{h} u_0^h)_{h}$ admits $\abs{\sigma_{h}\p{A_{h}}\p{\rho_0}}^2\delta_{\rho_0}$ as unique semiclassical measure in the sense of~\eqref{eq:SDM} and $0$ as unique 2-microlocal defect measure in the sense of~\eqref{eq:MDM}.
\end{corollary}
\begin{proof}[Proof of Corollary~\ref{cor:diracmass}]
The proof relies on Lemma~\ref{lem:diracmass}, indeed, using~\eqref{eq:coherent:state:convergency} and Proposition~\ref{pr:composition}: for all $B_{h}\in\pseudo_{h}^{\comp}(M)$, we have
\begin{align*}
\lim_{h\to 0}\pscal{A_{h} u_0^h}{B_{h} A_{h}u_0^h}_{L^2(M)} &= \lim_{h\to 0}\pscal{u_0^h}{A_{h}^*B_{h} A_{h}u_0^h}_{L^2(M)}\\
&= \pscal{\delta_{\rho_0}}{\abs{\sigma_{h}(A_{h})\p{\rho_0}}^2\sigma_{h}(B_{h})\p{\rho_0}}_{\Mc_+, \Cscr_c^0(\cotang M)}.
\end{align*}
And, for all $\theta$ as~\eqref{df:theta}, there is $R_0 > 0$ such that $\rho_0\notin\supp\p{1-\theta\p{\frac{~\jap{\xi}_g}{R}}}$ for any $R\geq R_0$. Using again~\eqref{eq:coherent:state:convergency} and Proposition~\ref{pr:composition}: for all $b \in\tilde{S^0_{\phg}}(\cotang M)$ we have
\begin{equation*}
\lim_{R\to\infty}\lim_{h\to 0}\pscal{A_{h}u_0^h}{\textstyle\Op_{h}^{\MM}\classe{b\p{1-\theta\p{\frac{~\jap{\xi}_g}{R}}}}A_{h}u_0^h}_{L^2(M)} = \lim_{R\to\infty}\classe*{\abs{\sigma_{h}(A_{h})(\rho_0)}^2\textstyle b(\rho_0)\p{1-\theta\p{\frac{~\jap{\xi_0}_g}{R}}}} = 0,
\end{equation*}
This concludes the proof.
\end{proof}
\subsection{Transgressing~\eqref{eq:UO}}
We consider a family of initial conditions $(u_0^h, hu_1^h)_{h}$ bounded in $H^1_h(M)\times L^2(M)$, with $(u_0^h)_{h}$ defined by~\eqref{eq:coherent:state:family} and
\begin{equation}\label{df:coherent:state:family:derive}
u_1^h := \frac{i}{h}\SqrtOp_{h} u_0^h\in \test\p{M},
\end{equation}
where $\SqrtOp_{h}$ is defined in~\eqref{df:SqrtOp} and is bounded as a pseudodifferential operator from $H^1_{h}(M)$ to $L^2(M)$ by Proposition~\ref{pr:boundedness:pseudo}.
\begin{proposition}\label{pr:Gaussian:Dirac:Propagation}
With $\varGamma$ defined in~\eqref{df:varGamma}, $\SqrtOp_{h}$ defined in~\eqref{df:SqrtOp}, and the family $(u_0^h, u_1^h)_{h}$ defined in~\eqref{eq:coherent:state:family} and~\eqref{df:coherent:state:family:derive}.
For all $h\in(0, 1]$, let ${u}^{h}$ be the $\Cscr^2\p{\RR, L^2(M)}\cap \Cscr^1\p{\RR, H^{1}(M)}\cap \Cscr^0\p{\RR, H^{2}(M)}$-solution of~\eqref{eq:Wave} with potential $\mathsf{V} = h^{-2} V$ and initial data $\p{{u}_0^{h}, u_1^{h}}$.
Then, for all $a\in \tilde{S^{0}_{\phg}}(\cotang M)$ (see Definition~\ref{df:pseudo:phg:0}) such that $\supp(a)\cap \varGamma = \emptyset$ and any $t\in (-\frac{T}{2}, \frac{T}{2})$, we have
\begin{equation*}
\lim_{h\to 0}\p*{\pscal{h \partial_t u^h(t)}{\Op_{h}^{M}(a) h \partial_tu^h(t)}_{L^2(M)} + \pscal{\SqrtOp_{h} u^h(t)}{\Op_{h}^{M}(a)\SqrtOp_{h} u^h(t)}_{L^2(M)}} = 0.
\end{equation*}
\end{proposition}
\begin{proof}[Proof of Proposition~\ref{pr:Gaussian:Dirac:Propagation}]
Recall $\sqrtsymbol$ is defined in~\eqref{df:Sqrt:Sym}.
For all $a\in \tilde{S^{0}_{\phg}}(\cotang M)$ such that $\supp(a)\cap \varGamma = \emptyset$ and all $t \in (-\frac{T}{2}, \frac{T}{2})$, we set 
\begin{equation*}
E^a_{h}(t) := \frac{1}{2}\p*{\pscal{h \partial_t u^h(t)}{\Op_{h}^{M}(a) h \partial_tu^h(t)}_{L^2(M)} + \pscal{\SqrtOp_{h} u^h(t)}{\Op_{h}^{M}(a)\SqrtOp_{h} u^h(t)}_{L^2(M)}}.
\end{equation*}
Using Corollary~\ref{cor:diracmass}, $(\SqrtOp_{h} u^h, h\partial_t u^h)_{|t = 0} = (\SqrtOp_{h} u_0^h, i\SqrtOp_{h} u_0^h)_{h}$ admits
\begin{equation*}
\sqrtsymbol(\rho_0)^2\delta_{\rho_0}\begin{pmatrix} 1  & -i\\ i & 1\end{pmatrix}\in\Mc_+(\cotang M,  \CC^{2\times 2}),
\end{equation*}
as unique semiclassical defect measure in the sense of~\eqref{eq:SDM} and $\p{0}_{1\leq j, k\leq 2}\in \Mc_+(\cosphere M, \CC^{2\times 2})$ as unique 2-microlocal defect measure in the sense of~\eqref{eq:MDM}.
Note that $\sqrtsymbol(\rho_0) > 0$.
\begin{enumerate}
\item \label{point:1}
For all $t\in (-\frac{T}{2}, \frac{T}{2})$, $\varphi^{\sqrtsymbol}_{t}(\rho_0)\in\varGamma$ and then $\supp(a)\cap \varGamma = \emptyset$ implies $a(\varphi^{\sqrtsymbol}_{t}(\rho_0)) = 0$. Furthermore, if we assume $a\in \test(\cotang M)$, we use the first point of Corollary~\ref{cor:use:HWE:for:WE} to conclude
\begin{equation*}
\forall t\in {\textstyle(-\frac{T}{2}, \frac{T}{2})}, ~\lim_{h\to 0}E^a_{h}(t) = \pscal{\sqrtsymbol(\rho_0)^2\delta_{\rho_0}}{\sum_{\signe\in\set{-, +}}a\circ \varphi^{\signe\sqrtsymbol}_{t}}_{\Mc_+, \Cscr^0_c(\cotang M)} = \sqrtsymbol(\rho_{0})^2 {\sum_{\signe\in\set{-, +}} a(\varphi^{\sqrtsymbol}_{\signe t}(\rho_0))} = 0.
\end{equation*}
\item\label{point:2} For any $\theta$ as~\eqref{df:theta}, using the second point of Corollary~\ref{cor:use:HWE:for:WE} we conclude
\begin{equation*}
\lim_{R\to\infty}\lim_{h\to 0}\textstyle E^{a\p{1-\theta\p{\frac{~\jap{\xi}_g}{R}}}}_{h}(t) = 0, \quad t \in (-\frac{T}{2}, \frac{T}{2}).
\end{equation*}
\end{enumerate}
Finally, for all $t\in (-\frac{T}{2}, \frac{T}{2})$ and $R > 0$, we write the decomposition
\begin{equation}\label{eq:20}
E^a_{h}(t) = E^{a\theta\p{\jap{\xi}_g/R}}_{h}(t) + E^{a\p{1-\theta\p{\jap{\xi}_g/R}}}_{h}(t),\quad h\in (0, 1].
\end{equation}
On the one hand, by point~\ref{point:1}, we have
\begin{equation*}
\forall R >0, \quad \lim_{h\to 0}E^{a\theta\p{\jap{\xi}_g/R}}_{h}(t) = 0.
\end{equation*}
On the other hand, by point~\ref{point:2}, we have
\begin{equation*}
\lim_{R\to\infty}\lim_{h\to 0}E^{a\p{1-\theta\p{\jap{\xi}_g/R}}}_{h}(t) = 0.
\end{equation*}
Thus, since there is no dependence on $R$ and using decomposition~\eqref{eq:20}, we have 
\begin{equation*}
\lim_{h\to 0}E^a_{h}(t) = \lim_{R\to\infty}\lim_{h\to 0}E^a_{h}(t) =0.
\end{equation*}
This is true for any $a\in\tilde{S^0_{\phg}}(\cotang M)$ compactly supported away from $\varGamma$, which concludes the proof of Proposition~\ref{pr:Gaussian:Dirac:Propagation}.
\end{proof}
\subsection{Conclusion of proof by contradiction}
We now conclude, by contraposition, the proof of the necessity part of Theorem~\ref{Th:GCC:UO}, namely that~\eqref{eq:UO} implies~\eqref{GCC+}.
\begin{proof}[Proof of the necessity in Theorem~\ref{Th:GCC:UO}]
Under the assumption that~\ref{GCCV}$(\omega, T)$ does not hold, there exists $\rho_0\in \cotang M$ such that~\eqref{eq:non:GCC+} holds.
We thus define $\varGamma$, $(u_0^h)_{h}$, $(u_1^h)_{h}$, and $V$ respectively as~\eqref{df:varGamma}, \eqref{eq:coherent:state:family}, \eqref{df:coherent:state:family:derive}, and \eqref{df:potential:fixed:V}.
For all $h\in (0, 1]$, we let $u^{h}$ be the $\Cscr^2\p{\RR, L^2(M)}\cap \Cscr^1\p{\RR, H^{1}(M)}\cap \Cscr^0\p{\RR, H^{2}(M)}$-solution of~\eqref{eq:Wave} with initial condition $\p{{u}_0^{h}, u_1^{h}}$ and potential $\mathsf{V} = h^{-2}V$.

For all $t\in(-\frac{T}{2}, \frac{T}{2})$, we have
\begin{equation*}
0\leq \norm{b_\omega h\partial_t u^{h}(t)}^2_{L^2(M)}\leq \pscal{h\partial_t u^{h}(t)}{b_\omega^2h\partial_ t u^{h}(t)}_{L^2(M)} + \pscal{\SqrtOp_{h} u^{h}(t)}{b_\omega^2\SqrtOp_{h} u^{h}(t)}_{L^2(M)}.
\end{equation*}
Using Proposition~\ref{pr:Gaussian:Dirac:Propagation}, we obtain
\begin{equation*}
\lim_{h\to 0}\norm{b_\omega h\partial_t u^{h}(t)}^2_{L^2(M)}= 0.
\end{equation*}
On the one hand, by Lebesgue Theorem, we have
\begin{equation*}
\lim_{h\to 0}\int_{-\frac{T}{2}}^{\frac{T}{2}}\norm{b_\omega h\partial_t u^{h}(t)}^2_{L^2(M)}\d t = 0.
\end{equation*}
On the other hand, recalling that~\eqref{df:potential:fixed:V} implies $V > 0$: by~\eqref{eq:coherent:state:family}, we have
\begin{equation*}
1 = \norm{u_0^{h}}_{L^2(M)}^2\leq \frac{2}{\min_{M}V}\Escr^{h}\p{u_0^{h}, u_1^{h}},\quad h\in (0, 1]
\end{equation*}
where the energy term $\Escr^{h}$ is defined in~\eqref{df:NRJ:rescaled}.

We thus conclude that if $(\omega, T)$ does not satisfy~\eqref{GCCV}, then~\eqref{eq:UO} cannot hold.
Moreover,~\eqref{eq:UO} implies~\eqref{eq:Obs} and, by~\cite{bardos1992sharp},~\eqref{eq:Obs} implies~\eqref{GCC}.
Therefore, uniform observability~\eqref{eq:UO} implies~\eqref{GCC+}.
\end{proof}
\section{Some examples}\label{section:Some examples}
This section is first devoted to prove Propositions~\ref{pr:1D:circle},~\ref{pr:2D:torus} and~\ref{pr:2D:sphere} and thus provides non-trivial examples satisfying~\eqref{GCC+} for dimension $1$ or $2$.
Second, we prove Lemmas~\ref{lem:GCCV-and-non-GCC} and~\ref{lem:GCC:and:no:GCCV} which allow us to compare~\eqref{GCC} and~\eqref{GCCV}.
Finally, we prove of Proposition~\ref{pr:Blow:Up:Estimates} using Agmon estimates.

We begin with a proposition that gives equivalent reformulations of~\eqref{GCCV}.
\begin{proposition}[Equivalent formulations of~\eqref{GCCV}]\label{pr:other:possible:formulations:for:GCCV}
Let us consider $(M, g)$ a smooth, compact and connected Riemannian manifold without boundary, $\omega$ an open subset of $M$ and a time $T > 0$.
Let $p$ as~\eqref{df:Sqrt:Sym} and use the notation~\eqref{df:Hamiltonian:Flow} for the flow of a symbol.
We have the following two equivalences:
\begin{align}
\text{\ref{GCCV}}(\omega, T)
&\iff \forall \rho_0\in\cotang M, ~\exists t\in \p*{0, \p{2\sqrt{p(\rho_0)}}^{-1}T},\quad \pi_M\circ\varphi^{p}_{t}(\rho_0)\in \omega\label{eq:rescaled:condition}\\
&\iff\forall \varrho_0\in \Char(\timespacesymbol),~\exists t \in\RR,\quad \pi_{\MM}\circ\varphi^{\timespacesymbol}_t(\varrho_0)\in \p{0, T}\times \omega\label{eq:time:times:space:condition}.
\end{align}
Where the symbol $\timespacesymbol\in S^2\p{\cotang\hspace{1pt} \MM_{(0, T)}}$ is defined in~\eqref{df:Sym:time:times:space}, and its characteristic set is given by
\begin{equation*}
\Char\p{\timespacesymbol} := \set{\rho\in\cotang \hspace{1pt}\MM, ~ \timespacesymbol(\rho) = 0}.
\end{equation*}
\end{proposition}
The formulation~\eqref{eq:rescaled:condition} of~\eqref{GCCV} will be used to prove Propositions~\ref{pr:1D:circle},~\ref{pr:2D:torus} and~\ref{pr:2D:sphere}.
The formulation~\eqref{eq:time:times:space:condition} of~\eqref{GCCV} may also be used to prove that~\eqref{GCC+} implies~\eqref{eq:UO} with a similar strategy than the one in~\cite[pages 136 - 137]{dehman2014controllability} which differs from the one proposed in Section~\ref{section:CS}.
\begin{proof}[Proof of Proposition~\ref{pr:other:possible:formulations:for:GCCV}]
Recall that $p$ and $\sqrtsymbol$ are defined in~\eqref{df:Sqrt:Sym}.
Fix $\varrho_0 = (t_0, x_0, \tau_0, \xi_0)\in \Char\p{\timespacesymbol}$, and local coordinates $(t, x_1, \dots, x_d)$ in a neighborhood of $(t_0, x_0)$.
Denote by $(\tau, \xi_1, \dots, \xi_d)$ the associated cotangent variables.
The Hamiltonian vector field of $\timespacesymbol$ at $\varrho_0$ is given by
\begin{align}
H_\timespacesymbol(\varrho_0) &= \frac{\partial \timespacesymbol}{\partial \tau}(\varrho_0)\frac{\partial}{\partial t} - \frac{\partial \timespacesymbol}{\partial t}(\varrho_0)\frac{\partial}{\partial \tau} + \sum_{j = 1}^d\frac{\partial \timespacesymbol}{\partial \xi_j}(\varrho_0)\frac{\partial}{\partial x_j} - \frac{\partial \timespacesymbol}{\partial x_j}(\varrho_0)\frac{\partial}{\partial \xi_j}\notag\\
&= -2\tau_0 \frac{\partial}{\partial t} + \classe*{\sum_{j = 1}^d\frac{\partial p}{\partial \xi_j}(\rho_0)\frac{\partial}{\partial x_j} - \frac{\partial p}{\partial x_j}(\rho_0)\frac{\partial}{\partial \xi_j}}\notag\\
&= -2\tau_0\frac{\partial}{\partial t} + H_p(\rho_0)\label{eq:other:possible:formulation:q:to:p}\\
&= -2\tau_0 \frac{\partial}{\partial t} + 2\sqrtsymbol(\rho_0)\classe*{\sum_{j = 1}^d\frac{\partial \sqrtsymbol}{\partial \xi_j}(\rho_0)\frac{\partial}{\partial x_j} - \frac{\partial \sqrtsymbol}{\partial x_j}(\rho_0)\frac{\partial}{\partial \xi_j}}\notag\\
&= -2\tau_0\partial_t  + 2\sqrtsymbol(\rho_0)H_{\sqrtsymbol}(\rho_0),\label{eq:other:possible:formulation:q:to:lambda}
\end{align}
where $\rho_0 = (x_0, \xi_0)\in\cotang M$.
Then, writing $(t^{\timespacesymbol}_s, x^{\timespacesymbol}_s, \tau^{\timespacesymbol}_s, \xi^{\timespacesymbol}_s) := \varphi^{\timespacesymbol}_s(t_0, x_0, \tau_0, \xi_0)$ for $s\in\RR$ and using Remark~\ref{rk:preserved:along:flow}, we have 
\begin{equation*}
\tau_0 = \signe\sqrtsymbol\p*{x^{\timespacesymbol}_s, \xi^{\timespacesymbol}_s}\quad \text{and}\quad t^{\timespacesymbol}_s = -2\tau_0 s\quad\text{for all}~ s\in\RR.
\end{equation*}
On the one hand, by~\eqref{eq:other:possible:formulation:q:to:p}, we have
\begin{equation*}
\varphi^p_s(x_0, \xi_0) = \p{x^{\timespacesymbol}_s, \xi^{\timespacesymbol}_s}, \quad s\in\RR
\end{equation*}
and thus, we have the equivalence between~\eqref{eq:time:times:space:condition} and~\eqref{eq:rescaled:condition}.
On the other hand, by~\eqref{eq:other:possible:formulation:q:to:p} and~\eqref{eq:other:possible:formulation:q:to:lambda}, we have
\begin{equation*}
\varphi^{\signe\sqrtsymbol}_{t}(x_0, \xi_0) = \varphi^p_{\frac{t}{2\signe\sqrtsymbol\p{\rho_0}}}(x_0, \xi_0) = \varphi^p_{\frac{t}{2\tau_0}}(x_0, \xi_0), \quad t\in\RR
\end{equation*}
and thus, we finally have the equivalence between~\eqref{eq:rescaled:condition} and~\eqref{GCCV}.
This concludes the proof of Proposition~\ref{pr:other:possible:formulations:for:GCCV}.
\end{proof}
\subsection{Examples satisfying~\eqref{GCC+}}
In this subsection, we provide the proofs of Propositions~\ref{pr:1D:circle},~\ref{pr:2D:torus} and~\ref{pr:2D:sphere}.
\subsubsection{One-dimensional example on the circle}
This part is devoted to the proof of Proposition~\ref{pr:1D:circle}.
To do so, we introduce the following crucial lemma.
\begin{lemma}\label{lem:circle}
Let $M$ be $\TT^1$ equipped with the Euclidean metric, let $V$ satisfy~\eqref{df:potential:fixed:V}, and assume $\omega\subset\TT^1$ is an open set containing critical points of $V$.
Let $p$ be defined by~\eqref{df:Sqrt:Sym} and $\varphi^p$ by~\eqref{df:Hamiltonian:Flow}.
Then, for any time 
\begin{equation*}
T \geq \textstyle {4\max\set{2\p*{\min_{\TT^1\setminus \omega}\abs{V'}}^{-1/2}, 1}},
\end{equation*}
the pair $(\omega, T)$ satisfies~\eqref{GCCV}.
\end{lemma}
\begin{proof}[Proof of Lemma~\ref{lem:circle}]
Since there is nothing to prove when $\omega = \TT^1$, we assume $\omega \subsetneq \TT^1$.
For any (temporarily fixed) point $\rho_0 = (x_0, \xi_0)\in \set{(x, \xi)\in \cotang\hspace{1pt}\TT^1,~x\notin \omega}$, suppose that there is a time $t_1 > 0$ such that 
\begin{equation}\label{eq:condition:absurde:1D:circle}
\set{\pi_{\TT^1}\circ\varphi^{p}_{t_1}(\rho_0),~ t\in (0, t_1)} \cap  \omega = \emptyset.
\end{equation}
The proof therefore relies on showing an upper bound on $t_1$.
We recall that, setting the notation $(x_t, \xi_t) := \varphi^p_t\p{\rho_0}$ for $t\in\RR$, by Remark~\ref{rk:preserved:along:flow} we have the invariance law
\begin{equation}\label{eq:invariance:circle}
p\p{\rho_0} = p\p{x_t, \xi_t} = \xi^2_t + V\p{x_t}, \quad t\in \RR.
\end{equation}
By~\eqref{eq:condition:absurde:1D:circle}, $V'$ has constant sign on $[x_0, x_{t_1}]$.
Then, we have
\begin{equation}\label{eq:circle:potential}
\abs{V(x_{t_1}) - V(x_0)} = \int_{[x_0, x_{t_1}]} \abs{V'(x)} \d x\geq \p{\min_{\TT^1\setminus \omega}\abs{V'}}\abs{x_{t_1} - x_0}.
\end{equation}
For technical purposes, fix a small parameter $\epsilon > 0$.
Then, because $t\mapsto x_t$ is solution of $\ddot{x} = -2V'(x)$, there are two possible cases.
\begin{enumerate}
\item Either $\phi_0 \colon t\in (0, t_1)\mapsto x_t\in (x_0, x_{t_1})$ is a $\Cscr^1$-diffeomorphism.
\begin{enumerate}
\item If $\sqrt{p(\rho_0)}\leq 1+\epsilon$, using $\dot{x}  = 2\xi$ from~\eqref{df:Hamiltonian:Flow} and~\eqref{eq:invariance:circle}, we obtain
\begin{equation*}
t_1 = \int_0^{t_1}\d t
= \int_{x_0}^{x_{t_1}}\frac{\d x}{2{\xi_{\phi_0^{-1}(x)}}} 
\leq \int_{\classe{x_0, x_{t_1}}}\frac{\d x}{2\sqrt{p(\rho_0) - V(x)}}
\leq \int_{\classe{x_0, x_{t_1}}}\frac{\d x}{2\sqrt{\abs{V(x_{t_1}) - V(x_0)}}}.
\end{equation*}
By~\eqref{eq:circle:potential}, we have
\begin{equation}\label{eq:iode}
t_1\leq \p{\min_{\TT^1/\omega}\abs{V'}}^{-1/2}\sqrt{\abs{x_0 - x_{t_1}}} \leq \p{\min_{\TT^1\setminus\omega}\abs{V'}}^{-1/2}\leq \frac{2(1 + \epsilon)}{2 \sqrt{p(\rho_0)}}\p{\min_{\TT^1\setminus \omega}\abs{V'}}^{-1/2}.
\end{equation}
\item Else $\sqrt{p(\rho_0)} > 1+\epsilon$, using $\dot{x} = 2\xi$ from~\eqref{df:Hamiltonian:Flow} and~\eqref{eq:invariance:circle}, we obtain
\begin{equation*}
t_1 = \int_0^{t_1}\d t 
= \int_{x_0}^{x_{t_1}}\frac{\d x}{2{\xi_{\phi_0^{-1}(x)}}} 
\leq \int_{\classe{x_0, x_{t_1}}}\frac{\d x}{2\sqrt{p(\rho_0) - V(x)}}
\leq \int_{\classe{x_0, x_{t_1}}}\frac{\d x}{2\sqrt{p(\rho_0) - 1}}
\leq \frac{1}{2\sqrt{p(\rho_0) - 1}}.
\end{equation*}
Then, using 
\begin{equation*}
\sqrt{p(\rho_0) - 1} = \sqrt{\p{\sqrt{p(\rho_0)} - 1}\p{\sqrt{p(\rho_0)} + 1}}\geq \sqrt{p(\rho_0)} - 1
\end{equation*}
and 
\begin{equation*}
\p*{1 - \frac{1}{\sqrt{p(\rho_0)}}}
\geq \p*{1 - \frac{1}{1+\epsilon}}
= \p*{1 + \frac{1}{\epsilon}}^{-1},
\end{equation*}
we finally have
\end{enumerate}
\begin{equation}\label{eq:okpez}
t_1\leq \frac{1}{2(\sqrt{p(\rho_0)} - 1)}\leq \frac{1+\frac{1}{\epsilon}}{2\sqrt{p(\rho_0)}}.
\end{equation}
In conclusion, using~\eqref{eq:iode} and~\eqref{eq:okpez}, we have
\begin{equation*}
t_1\leq \frac{\max\set{2(1+\epsilon)\p{\min_{\TT^1\setminus \omega}\abs{V'}}^{-1/2}, 1+\frac{1}{\epsilon}}}{2\sqrt{p(\rho_0)}}.
\end{equation*}
\item Or there is $t_{1/2}\in (0, t_1)$ such that 
\begin{equation*}
\phi_0\colon t\in (0, t_{1/2})\mapsto x_t\in (x_0, x_{t_{1/2}})
\end{equation*}
and
\begin{equation*}
\phi_{1/2}\colon t\in (t_{1/2}, t_1)\mapsto x_t\in (x_{t_{1/2}}, x_{t_1})
\end{equation*}
are two $\Cscr^1$-diffeomorphisms.
Then, with the same calculation as performed above we obtain
\begin{align*}
t_1 &= \int_0^{t_1}\d t = \int_{x_0}^{x_{t_{1/2}}}\frac{\d x}{2{\xi_{\phi_0^{-1}(x)}}} + \int_{x_{t_{1/2}}}^{x_{t_1}}\frac{\d x}{2{\xi_{\phi_{1/2}^{-1}(x)}}}\\
&\leq 2\frac{\max\set{2(1+\epsilon)\p{\min_{\TT^1\setminus \omega}\abs{V'}}^{-1/2}, 1+\frac{1}{\epsilon}}}{2\sqrt{p(\rho_0)}}.
\end{align*}
\end{enumerate}
Therefore, using the two previous points: we obtain
\begin{equation*}
t_1\leq \frac{1}{2\sqrt{p(\rho_0)}}\p*{2\max\set{2(1+\epsilon)\p{\min_{\TT^1\setminus \omega}\abs{V'}}^{-1/2}, 1+\frac{1}{\epsilon}}}
\end{equation*}
and hence we conclude the proof by taking $\epsilon = 1$.
\end{proof}
\begin{proof}[Proof of Proposition~\ref{pr:1D:circle}]
\begin{enumerate}
\item 
First, since we consider $\TT^1 \simeq [0, 1)$ equipped with the Euclidean metric and $\omega \neq \emptyset$, all geodesics (which are straight lines here) enter $\omega$ at time $t < 1$. Then, for any fixed time $T_1 > 1$, $(\omega, T_1)$ satisfies~\eqref{GCC}.

\item Secondly, let us consider a time $T_2 >4\max\set{2\p{\min_{\TT^1\setminus \omega}\abs{V'}}^{-1/2}, 1}$ given by Lemma~\ref{lem:circle}.
For all $\rho_0 = (x_0, \xi_0)\in \cotang(\TT^1)$, either $x_0\in\omega$, then for all sufficiently small $t > 0$ we have $\pi_{\TT^1}\circ\varphi_t^{\sqrtsymbol}(\rho_0)\in\omega$ ; or, at the opposite $x_0\notin\omega$, then Lemma~\ref{lem:circle} implies (by contradiction) that there is a time $t < \frac{T_2}{2\sqrt{p(\rho_0)}}$ such that $\pi_{\TT^1}\circ\varphi_t(\rho_0)\in\omega$.
Hence, by~\eqref{eq:rescaled:condition}, $(\omega, T_2)$ also satisfies~\eqref{GCCV}.
\end{enumerate}
In conclusion, $(\omega, \max\set{T_1, T_2})$ satisfies~\eqref{GCC+}. 
\end{proof}
\subsubsection{Two-dimensional examples}
In this subsection, we prove both Propositions~\ref{pr:2D:torus} and~\ref{pr:2D:sphere} by the same argument.
The underlying idea is as follows: in both cases, after removing a suitable open subset $\Uc$ from the manifold $M$, the remaining set $\Sc := M\setminus \Uc$ is a compact surface of revolution (diffeomorphic to a cylinder).
On this surface $\Sc$, the dynamical trajectories admit a more explicit description than on the original manifold $M$.

Inspired by~\cite[Section 4]{laurent2021uniform} and~\cite[Section 3.2]{laurent2021observability} in the setting of surfaces of revolution, we consider $(M, g)$ a smooth, compact, connected and two-dimensional Riemannian manifold without boundary, and we assume:
\begin{enumerate}
\item There is an open subset $\Uc\subset M$ such that
\begin{equation}\label{df:compact:subset:of:revolution}
\Sc := M\setminus \Uc
\end{equation}
is diffeomorphic to the cylinder $\Cc = [0, L]\times \SS^1$:
\begin{equation}\label{eq:parametrisation:revolution:surface}
\Phi\colon x \in \Sc\mapsto (s, \theta)\in \Cc,
\end{equation}
\item The metric induced by $g$ on $\Cc$ have the form $(\Phi^{-1})^\star g = \d s^2 + R(s)^2\d\theta^2$ for a profile function 
\begin{equation}\label{df:profil:function}
R\colon [0, L] \to [\mathsf{m}, \mathsf{M}]~\text{with}~0 < \mathsf{m} < \mathsf{M} < +\infty.
\end{equation}
\end{enumerate}
Examples given in Propositions~\ref{pr:2D:torus} and~\ref{pr:2D:sphere} both satisfy these hypotheses.
See~\cite[Remark 3.4]{laurent2021observability} for details of geometric interpretation of profile function $R$.
Let $p$ as~\eqref{df:Sqrt:Sym}, $V$ as~\eqref{df:potential:fixed:V}, and assume that $V$ is rotationally invariant and has no critical points in $\Sc$, \ie there exist a function $\Ve\in\smooth\p{(0, L), \RR_+^*}$ such that
\begin{equation}\label{df:potential:of:revolution}
\forall (s, \theta)\in\Cc,\quad
\classe{\p{\Phi^{-1}}^\star V}(s, \theta) = \Ve(s)~\text{and}~\min_{[0, L]}\abs{\Ve'} > 0.
\end{equation}
Then, we have, with a slight abuse of notations,
\begin{equation*}
p(s, \theta, \xi_s, \xi_\theta) = \classe{(\Phi^{-1})^\star p} (s, \theta, \xi_s, \xi_\theta) = \xi_s^2 + \frac{1}{R(s)^2}\xi_\theta^2 + \Ve(s).
\end{equation*}
Hence, for $\varphi^p$ as~\eqref{df:Hamiltonian:Flow} and any point $\rho_0 = (x_0, \xi_0)\in \cotang \Sc$, setting $\rho_t := (s(t), \theta(t), \xi_s(t), \xi_\theta(t)) := {\varphi_t^p(\rho_0)}$ for each $t\in \RR$, we obtain the Hamiltonian system
\begin{equation}\label{eq:revolution:surface:Hamiltonian}
\frac{\d }{\d t}s = 2\xi_s , \quad \frac{\d }{\d t}\xi_s = -\frac{\d }{\d s}\p*{\frac{1}{R^2}}\xi_\theta^2 - \Ve', \quad
\frac{\d }{\d t}\theta = \frac{2 \xi_\theta}{R(s)^2}\quad \text{and}\quad\frac{\d }{\d t}\xi_\theta = 0
\end{equation}
with initial condition $(s, \theta, \xi_s, \xi_\theta)_{|t = 0}  = \rho_0$.
Furthermore, using flow invariance property from Remark~\ref{rk:preserved:along:flow}, we have the conservation law
\begin{equation}\label{eq:2D:revolution:surface:Hamiltonnian:invariant}
p\p{\rho_t} = \xi_s^2(t) + \frac{1}{R\p{s(t)}^2}\xi_\theta^2(t) + \Ve(s(t)) = p(\rho_0).
\end{equation}
For any small $\epsilon \in (0, \pi/10)$, we set
\begin{equation}\label{df:omega:surface:revolution}
\omega := \set{x\in M, ~\dist_g\p{x, \Uc} < \epsilon}\cup\Phi^{-1}\p{\set{(s,  \theta)\in \Cc,~ \abs{\theta} < \epsilon}},
\end{equation}
as control set.
\begin{proposition}\label{pr:revolution:surface}
Assume $(M, g)$ is a smooth, compact, connected, two-dimensional Riemannian manifold without boundary, and $\Sc$ is as in~\eqref{df:compact:subset:of:revolution}.
Then, for any pair $(\omega, V)$ satisfying~\eqref{df:omega:surface:revolution} and~\eqref{df:potential:of:revolution}, there is a time $T > 0$ such that~\ref{GCCV}$(\omega, T)$ holds.
\end{proposition}
\begin{proof}[Proof of Proposition~\ref{pr:revolution:surface}]
The strategy is to show that, for a (temporarily fixed) point $\rho_0 = (x_0, \xi_0) \in \set{(x, \xi)\in \cotang M,~x\notin\omega}$ and all time $t_1 > 0$ such that
\begin{equation*}
\set{\pi_{M}\circ\varphi^p_t(\rho_0),~ t\in (0, t_1)}\cap \omega = \emptyset,
\end{equation*}
there is the upper bound $t_1 < \frac{T}{2\sqrt{p(\rho_0)}}$, where $T > 0$ is independent of the choice of $\rho_0$.
To simplify the computations, we take $\mathsf{m}$ and $\mathsf{M}$ as~\eqref{df:profil:function} and fix three strictly positive real numbers
\begin{equation}\label{eq:constantes:sphere:proof}
\frac{1}{1+ \frac{\mathsf{m}^2}{\mathsf{M}^2}} < \alpha <1,~ \beta := \sqrt{1 - \alpha\frac{\mathsf{m}^2}{\mathsf{M}^2}} \in (0, 1) \text{ and }\gamma > \max\set{\frac{1}{1-\alpha}, \frac{1}{1-\beta^2}},
\end{equation}
where we chose $\beta$ to have $\p{1-\beta^2}\frac{\mathsf{M}^2}{\mathsf{m}^2} = \alpha$,
and distinguish two cases with respect to $p\p{\rho_0}$.
\begin{enumerate}
\item If  $p(\rho_0)\leq \gamma$, then the studied part of phase space belongs to a compact subset
\begin{equation*}
\Ke = \set{\rho\in \cotang (\Sc \setminus \omega),~ p(\rho)\leq \gamma}.
\end{equation*}
Furthermore, let us consider $\delta \in (0, \epsilon)$ such that the thinner control set 
\begin{equation*}
\omega^\delta := \set{x\in \omega, ~ B_{\SS^2}(x, \delta)\subset \omega} = \p{\Uc\sqcup\set{(s, \theta)\in\Cc,~ \theta = 0}}_{\epsilon - \delta}
\end{equation*}
is not empty and still of the form~\eqref{df:omega:surface:revolution}.
Set a temporarily fixed point $\rho = \tilde{\Phi}(s(0), \theta(0), \xi_s(0), \xi_\theta (0) = 0)\in\cotang \p{\Sc\setminus \omega}$, 
$\pi_{\Sc}\circ\varphi_t(\rho)$ is contained in the line of angle $\theta(0)$ for any $t\geq 0$ and the dynamic is then described by 
\(
\ddot{s}= - 2\Ve'(s).
\)
Therefore, using Proposition~\ref{pr:1D:circle} we obtain the existence of a time $T^\delta_0$ (independent of the choice of $\rho$) such that 
\begin{equation*}
\set{\textstyle\pi_{\Sc}\circ\varphi^p_{t}\p{\rho}, ~ t\in\p{0, T^\delta_0}}\cap\omega^\delta\neq \emptyset.
\end{equation*}
Then, using continuity of the solution of~\eqref{eq:revolution:surface:Hamiltonian} with respect to initial data in the compact set $\Ke$, there is $\varepsilon^\delta > 0$ (independent of the choice of $\rho$), such that for all $\rho_0\in B(\rho, \varepsilon^\delta)$ and all time $t\in (0, T^\delta_0)$ we have
\begin{equation}\label{eq:sphere:continuity:of:solution}
\abs{\pi_{\Sc}\circ\varphi^p_{t}(\rho_0) - \pi_{\Sc}\circ\varphi^p_{t}(\rho)}< \delta.
\end{equation}
By~\eqref{eq:sphere:continuity:of:solution} we have $t_1 < T^\delta_0\leq \frac{2\sqrt{\gamma}}{2\sqrt{p(\rho_0)}}T^\delta_0$.

\item Else we have $p(\rho_0) > \gamma$.
\begin{enumerate}
\item If $\abs{\xi_s}(0) \leq \beta\sqrt{p(\rho_0)}$, using~\eqref{eq:2D:revolution:surface:Hamiltonnian:invariant} we then have
\begin{equation*}
\xi_{\theta}^2\geq R\p{s(0)}^2\p*{\p{1-\beta^2} p(\rho_0) -1}\geq \mathsf{m}^2\p*{\p{1-\beta^2}p(\rho_0) -1}\geq p(\rho_0)\mathsf{m}^2\underbracket{\p*{\p{1-\beta^2} -\frac{1}{\gamma}}}_{> 0 \text{ by~\eqref{eq:constantes:sphere:proof}}}.
\end{equation*}
Thus, we conclude that $t_1 < \frac{T_1}{2\sqrt{p(\rho_0)}}$ with 
\(
T_1 := 2\pi\frac{2}{\mathsf{m}\sqrt{\p*{\p{1-\beta^2} -\frac{1}{\gamma}}}}.
\)
\item Else $\abs{\xi_s}(0) > \beta\sqrt{p(\rho_0)}$, and then, using~\eqref{eq:2D:revolution:surface:Hamiltonnian:invariant}, we have the upper bound
\begin{equation}\label{eq:xi:theta:2:a}
\xi_\theta^2 < R\p{s(0)}^2\p{1-\beta^2}p(\rho_0) < \p{1-\beta^2}\mathsf{M}^2 p(\rho_0).
\end{equation}
Then, re-injecting~\eqref{eq:xi:theta:2:a} in~\eqref{eq:2D:revolution:surface:Hamiltonnian:invariant} we obtain the lower bound
\begin{equation*}
\xi_s^2(t) > \p*{1 - \p{1-\beta^2}\frac{\mathsf{M}^2}{\mathsf{m}^2}}p(\rho_0) - 1 > p(\rho_0)\underbracket{\p*{\p*{1 - \alpha} - \frac{1}{\gamma}}}_{>0\text{ by~\eqref{eq:constantes:sphere:proof}}}.
\end{equation*}
Thus, we conclude that $t_1 < \frac{T_2}{2\sqrt{p(\rho_0)}}$ with
\(
T_2 := 2\frac{L}{\sqrt{1-\alpha -\gamma^{-1}}}.
\)
\end{enumerate}
\end{enumerate}
To conclude the proof: we show, for a pair $(\rho_0, t_1)\in\cotang M\times \RR_+$, that if the flow issued from $\rho_0$ does not intersect $\omega$ at time $t_1$ then we have
\begin{equation*}
t_1\leq \frac{\max\set{2\sqrt{\gamma}T^\delta_0,  T_1, T_2}}{2\sqrt{p(\rho_0)}},
\end{equation*}
and hence we can take $T = \max\set{2\sqrt{\gamma}T^\delta_0, T_1, T_2}$
--- independent of the choice of $\rho_0$ ---
to end the proof.
\end{proof}
\paragraph{Application to the two-dimensional sphere.}
This paragraph is devoted to the proof of the Proposition~\ref{pr:2D:sphere} for the two-dimensional sphere $M = \SS^2$ equipped with its canonical metric $g$ induced by the Euclidean metric of $\RR^3$.
\begin{proof}[Proof of Proposition~\ref{pr:2D:sphere}]
Let $S$ be the south pole and $N$ be the north pole of $M$, set a parameter $\delta \in (0, \frac{\pi}{2})$ and a neighborhood of $\set{S, N}$ as
\begin{equation*}
\Uc = \set{x\in M, ~ \dist_g(x, S) < \delta \text{ or }\dist_g(x, N) < \delta}.
\end{equation*}
The compact subset $\Sc = M\setminus \Uc$ is diffeomorphic to a cylinder $\Cc = [0, \pi - 2\delta]\times \SS^1$ by a parametrization $\Phi$ as~\eqref{eq:parametrisation:revolution:surface}.
The metric $g$ is then given by $\d s^2 + R(s)^2\d \theta^2$ for the profile $R\colon s \in [0, \pi - 2\delta]\mapsto \sin\p{s + \delta}\in [\sin\p{\delta}, 1]$.
Let us consider $\omega$ as~\eqref{df:omega:surface:revolution} with a small parameter $\epsilon \in (0, \frac{\pi}{10})$, note that if $\epsilon \geq \frac{\pi-2\delta}{2}$ then $\omega = M$.

On the one hand, using the canonical metric of the two-dimensional sphere: all geodesics (arcs of great circles here) enter $\omega$ at time $t < 2\pi$. Then, for any time $T_1 > 2\pi$, the pair $(\omega, T_1)$ satisfies~\eqref{GCC}.

On the other hand, since for all $(s, \theta)\in\Cc$ we have $\classe{\p{\Phi^{-1}}^\star V} (s, \theta) = \Ve(\p{s+\delta}^2)$ with $\Ve$ increasing on $[\delta^2, \p{\pi-\delta}^2]$, we can use Proposition~\ref{pr:revolution:surface} to conclude that there is a time $T_2 > 0$ such that the pair $(\omega, T_2)$ also satisfies~\eqref{GCCV}.
\end{proof}
\paragraph{Application to the two-dimensional torus.}
This paragraph is devoted to the proof of Proposition~\ref{pr:2D:torus}, where we consider the two-dimensional torus $M := [-\frac{1}{2}, \frac{1}{2})^2$ equipped with the Euclidean metric.
The strategy of the proof is to consider a compact subset of $M$ compatible with the revolution coordinates~\eqref{eq:parametrisation:revolution:surface} and then to use Proposition~\ref{pr:revolution:surface}.
\begin{proof}[Proof of Proposition~\ref{pr:2D:torus}]
For any two radii $0 < \tilde{R}_1 < \tilde{R}_2 < \frac{1}{2}$, we consider the open subset
\begin{equation*}
\Uc = \set{x\in M, ~ \abs{x} <  \tilde{R}_1 ~\text{or}~ \tilde{R}_2  < \abs{x}}.
\end{equation*}
Then, the compact subset $\Sc := M\setminus \Uc$ is diffeomorphic to a cylinder $\Cc = (0, \tilde{R}_2 - \tilde{R}_1)\times \SS^1$ by the parametrization
\begin{equation*}
\Phi\colon x = \p{\tilde{R}_1 + s}\p{\cos\theta, \sin\theta} \in \Sc\mapsto (s, \theta)\in \Cc.
\end{equation*}
In those coordinates, the metric is then given by $\d s^2 + R(s)^2\d \theta^2$ for the profile function
\begin{equation*}
R\colon s\in (0, \tilde{R}_2-\tilde{R}_1)\mapsto \tilde{R}_1 + s \in (\tilde{R}_1, \tilde{R}_2).
\end{equation*}
Let us consider $\omega$ as~\eqref{df:omega:surface:revolution} with a fixed parameter $\epsilon \in (0, \pi/10)$.

On the one hand, using the Euclidean metric: since $\omega$ is a neighborhood of $\set{x_1 = -\frac{1}{2} \text{ or } x_2 = -\frac{1}{2}}$, all geodesics (straight lines here)
enter $\omega$ at time $t < 1$.
Then, for any time $T_1 > 1$, the pair $(\omega, T_1)$ satisfies~\eqref{GCC}.

On the other hand, since for all $(s, \theta)\in \Cc$ we have $\classe{\p{\Phi^{-1}}^\star V}(s, \theta) = \Ve (\p{s + \tilde{R}_1}^2)$ with $\Ve$ increasing on $[\tilde{R}_1^2,\tilde{R}_2^2]$, we can use Proposition~\ref{pr:revolution:surface} and thus conclude that there is a time $T_2 > 0$ such that the pair $(\omega, T_2)$ also satisfies~\eqref{GCCV}.

Therefore, the pair $(\omega, \max\set{T_1, T_2})$ satisfies~\eqref{GCC+}.
This is true for any choice of $(\tilde{R}_1, \tilde{R}_2, \epsilon)$, and then ends the proof.
\end{proof}
\subsection{Comparing~\eqref{GCCV} and~\eqref{GCC}}\label{subsection:GCCV:GCC}
In this subsection, we provide the proofs of Lemmas~\ref{lem:GCC+:implies:GCC},~\ref{lem:GCCV-and-non-GCC} and~\ref{lem:GCC:and:no:GCCV}, giving differences between~\eqref{GCC} and~\eqref{GCCV}. Recall that $\MM$ is defined in~\eqref{df:MM}.
\begin{proof}[Proof of Lemma~\ref{lem:GCC+:implies:GCC}]
We let $k\colon (x, \xi)\in\cotang M \mapsto \abs{\xi}_g^2$, and recall that the function $\timespacesymbol\in\smooth\p{\cotang \p{\RR\times M}}$ is defined in~\eqref{df:Sym:time:times:space}.
We defined the limit of $\timespacesymbol$ on the cosphere at infinity by
\begin{equation*}
\timespacesymbol_\infty(t, x, \tau, \xi) := \lim_{R\to\infty}\p*{\frac{\timespacesymbol(t, x, R\tau, R\xi)}{R^2}} = k(x, \xi) - \tau^2,\quad (t, x, \tau, \xi)\in\cosphere\p{\RR\times M}.
\end{equation*}

For any $\rho_0 = (t_0, x_0, \tau_0, \xi_0)\in\set{(t, x, \tau, \xi)\in\cotang\hspace{1 pt}\MM_{(0, T)},~\tau^2 = \abs{\xi}_{g}^2+V(x)}$ and any radius $R > 0$ we set: $\rho_{0, R} = (t_0, x_0, R\tau_0, R\xi_0)$.
Then denoting $(t_R(s),x_R(s), \tau_R(s), \xi_R(s)) := \varphi^{\timespacesymbol}_s(\rho_{0, R})$, in local charts we have
\begin{equation*}
\eqcases{ll}{    
\frac{\d t_R}{\d s}(s) = - 2 R\tau_0 ;& \frac{\d \tau_R}{\d s}(s) = 0;\\
\frac{\d x_R}{\d s}(s) = {\nabla_\xi k}(x_R(s), \xi_R(s)) ;& \frac{\d \xi_R}{\d s}(s) = -\nabla_x k(x_R(s), \xi_R(s)) - \p{\nabla_x V}(x_R(s)).
}
\end{equation*}
For all $s\in \RR$, we set $\tilde{t}_R(Rs) := t_R(s)$, $\tilde{\tau}_R(Rs) := R^{-1}\tau_R(s) = \tau_0$, $\tilde{x}_R(Rs) := x_R(s)$ and $\tilde{\xi}_R(Rs) := R^{-1}\xi_R(s)$, giving
\begin{equation*}
\p{\tilde{t}_R, \tilde{x}_R, \tilde{\tau}_R, \tilde{\xi}_R}(0) = \rho_0\in\set{(t, x, \tau, \xi)\in\cotang\hspace{1 pt}\MM_{(0, T)},~\tau^2 = \abs{\xi}_{g}^2+ R^{-2}V(x)}.
\end{equation*}
Remark that $s\mapsto (\tilde{t}_R(s), \tilde{x}_R(s), \tilde{\tau}_R(s), \tilde{\xi}_R(s), R^{-1})$ is the solution of the Cauchy problem
\begin{equation*}
\eqcases{rl}{
\frac{\d}{\d s}(\mathsf{t}, \mathsf{x}, \mathsf{s}, \mathsf{y}, \mathsf{z}) &= \p{-2\tau_0, \nabla_\xi k(\mathsf{x},\mathsf{y}),0, -\nabla_x k(\mathsf{x}, \mathsf{y}) - \mathsf{z}^2\nabla_g V(\mathsf{x}), 0}\\
(\mathsf{t}, \mathsf{x}, \mathsf{s}, \mathsf{y}, \mathsf{z})_{|t = 0} &= (t_0, x_0,\tau_0, \xi_0, R^{-1}).
}
\end{equation*}
Using continuity of solutions with respect to initial data, we can conclude that $(\tilde{t}_R, \tilde{x}_R, \tilde{\tau}_R, \tilde{\xi}_R)$ converges uniformly on all compact subset of $\RR$ to $s\mapsto\varphi^{\timespacesymbol_\infty}_{s}(\rho_0)$ as $R$ tends to infinity.

Hence, by~\ref{GCCV}$(\omega, T)$, we know that for all $R > 0$, there is $s_R\in \p{\frac{-T}{2\tau_0}, 0}$ such that $\tilde{t}_R\p{s_R}\in (0, T)$ and $\tilde{x}_R(s_R)\in\omega$. Then uniform convergence leads us to the existence of $s_\infty\in \classe{\frac{-T}{2\tau_0}, 0}$ such that  
\begin{equation*}
\pi_{\MM}\circ\varphi^{\timespacesymbol_\infty}_{s_\infty}(\rho_0)\subset [0, T]\times \barre{\omega}.
\end{equation*}
Then, in conclusion
\begin{equation*}
\begin{array}{rrrl}
&\forall\rho_0\in\set{(t,x, \tau, \xi)\in\cotang\hspace{1 pt}\MM_{(0, T)}, ~\tau^2 = \abs{\xi}_g^2+V(x)}, ~&\exists s \in \RR,~ &\pi_{\MM}\circ \varphi_s^{\timespacesymbol}(\rho_0)\in(0, T)\times\omega\\
\implies&\forall \rho_0\in\set{(t,x, \tau, \xi)\in\cosphere\hspace{1 pt}\MM_{(0, T)}, ~\tau^2 = \abs{\xi}_g^2}, ~&\exists s \in \RR,~ &\pi_{\MM}\circ \varphi_s^{\timespacesymbol_\infty}(\rho_0)\in\classe{0, T}\times\barre{\omega}\\
\implies& \forall \rho_0\in\cosphere M, ~&\exists t \in \classe{0, T}, ~&\pi_M\circ \varphi_t^{\abs{\xi}_g}(\rho_0)\in \barre{\omega}\\
\implies& \forall \delta, \epsilon > 0,~ \forall \rho_0\in\cosphere M, ~&\exists t \in \p{-\frac{\varepsilon}{2}, T + \frac{\epsilon}{2}}, ~&\pi_M\circ \varphi_t^{\abs{\xi}_g}(\rho_0)\in \omega_\delta.
\end{array}
\end{equation*}
Applying this statement to $\rho_0 = \varphi_{\frac{\varepsilon}{2}}^{\abs{\xi}_g}(\rho)$ with $\rho\in\cosphere M$, we obtain
\begin{equation*}
    \forall \delta, \epsilon > 0,~ \forall \rho_0\in\cosphere M, ~\exists t \in \p{0, T + \epsilon}, \quad \pi_M\circ \varphi_t^{\abs{\xi}_g}(\rho_0)\in \omega_\delta,
\end{equation*}
which proves that $(\omega_\delta, T + \varepsilon)$ satisfies~\eqref{GCC}.
\end{proof}
\begin{proof}[Proof of Lemma~\ref{lem:GCCV-and-non-GCC}]
On the one hand, for any $T>0$, the pair $(\omega, T)$ fails to satisfy~\eqref{GCC} since all trajectories starting from $\set{(x_1, \xi_1) = (c, 0)}$ remain trapped in $\set{x_1 = c} = \TT^2 \setminus \omega$.

On the other hand, for all $\rho_0\in\set{(x_1, \xi_1) = (c, 0)}$, writing $E := \sqrtsymbol\p{\rho_0} > 0$ we recall that
\begin{equation*}
\forall t\in\RR, \quad \sqrtsymbol\p{\varphi^\sqrtsymbol_t\p{\rho_0}} = E,
\end{equation*}
by Remark~\ref{rk:preserved:along:flow}. 
The Hamiltonian flow $\varphi^{\sqrtsymbol}$, given by~\eqref{df:Hamiltonian:Flow}, satisfies
\begin{equation*}
\frac{\d \xi_1}{\d t}(t) = - \frac{\partial_{x_1} V(x_1(t), x_2(t))}{2\sqrt{\xi_1^2(t) + \xi_2^2(t) + V(x_1(t), x_2(t))}} = -\frac{1}{2 E}\partial_{x_1} V(x_1(t), x_2(t)).
\end{equation*}
According to~\eqref{eq:3}, this last equality implies $\frac{\d \xi_1}{\d t}(0) > 0$.
As a consequence, there is $\varepsilon_0 >0$ such that for all $t\in (0, \varepsilon_0)$ we have $\xi_1(t) > 0$.
Since $\frac{\d x}{\d t} = 2\xi_1  >0$ (from~\eqref{df:Hamiltonian:Flow}) on $(0, \varepsilon_0)$, we have
\begin{equation*}
    x_1(\varepsilon) = x_1(0) + \int_{0}^{\varepsilon} 2\xi_1(t)\d t > x_1(0) = c, \quad \varepsilon \in (0, \varepsilon_0)
\end{equation*}
Therefore, every trajectory enters $\omega$, thus~\ref{GCCV}$(\omega, \varepsilon)$ holds for arbitrarily small $\varepsilon > 0$, hence $(\omega, T)$ satisfy~\eqref{GCCV}.
\end{proof}
\begin{proof}[Proof of Lemma~\ref{lem:GCC:and:no:GCCV}]
Let $x_0\in M$ be such that $\nabla_g V(x_0) = 0$, and recall $\sqrtsymbol(x, \xi) = \sqrt{\abs{\xi}_g^2 + V(x)}$.
For all $\xi_0\in \mathrm{T}^\star_{x_0}M$, since $H_{\sqrtsymbol}(x_0, \xi_0) = 0$, we have $\varphi^{\sqrtsymbol}_t (x_0, \xi_0) = (x_0, \xi_0)$ for any $t\geq 0$.
As a consequence, if $\omega\cap\set{x_0} = \emptyset$, there is no time $T > 0$ such that $(\omega, T)$ can satisfy~\eqref{GCCV}.
\end{proof}
\subsection{Blow-up rate of the \emph{observability cost} when~\eqref{GCCV} does not hold}\label{subsection:blow-up}
The purpose of this subsection is to prove Proposition~\ref{pr:Blow:Up:Estimates}, which gives the blow-up rate~\eqref{eq:blow:up:estimate} of the \emph{optimal observability cost} defined in~\eqref{df:optimal:Cobs} if $(\omega, T)$ does not satisfy~\eqref{GCCV}.
We prove Proposition~\ref{pr:Blow:Up:Estimates} relying on Lemma~\ref{lem:Eigenmode} and Theorem~\ref{Th:Agmon:Estimate}.
The proofs of these two results are given in~\cite[Lemma 3.2 and Theorem 3.4]{laurent2021uniform} and rely on Agmon estimates, following~\cite[Proposition 3.3.1]{Helffer_1988} (see also~\cite{Helffer84Multiples, allibert1998control, diamssi99Spectral, laurent2023uniform} for classical results on decay rates of eigenfunction in the classically forbidden region).
\begin{lemma}\label{lem:Eigenmode}
For all $x_0\in M$ there is a family $(E_{h}, \psi^h)_{h}$ of $\RR\times H^2(M)$ such that
\begin{equation*}
\p{- h^2\Delta_g  + V}\psi^h = E_{h}\psi^h,\quad  E_{h} = V(x_0) + \o{h^{2/3}} \quad\text{and}\quad \norm{\psi^h}_{L^2(M)} = 1, \quad h\in (0, 1).
\end{equation*}
\end{lemma}
\begin{theorem}\label{Th:Agmon:Estimate}
Let $\omega$ be an open subset of $M$ with $\barre{\omega}\neq M$.
Assume $V\in \smooth(M, \RR_+^*)$ has a global minimum  at $x_0\in M$, such that $\barre{\omega}\cap \set{x_0}\neq \emptyset$.
Then, there are $\delta > 0$ and $h_{0}\in (0, 1)$ such that for all $h\in (0, h_{0})$ and all $(E_{h}, \psi^h)\in\RR\times H^2(M)$ satisfying
\begin{equation*}
(- h^2\Delta_g + V)\psi^h = E_{h}\psi^h,\quad E_{h} = V(x_0) + \o{1}_{h}\quad \text{and}\quad \norm{\psi^h}_{L^2(M)} = 1,
\end{equation*}
we have the following estimate
\begin{equation}\label{eq:Agmon:Estimate}\tag{AE}
\exists C > 0,~\forall h \in (0, h_{0}), \quad \norm{\psi^h}_{L^2(\omega)}\leq Ce^{-\delta /h}.
\end{equation}
\end{theorem}
\begin{proof}[Proof of Proposition~\ref{pr:Blow:Up:Estimates}]
Lemma~\ref{lem:Eigenmode} ensures the existence of $(E_{h}, \psi^h)_{h}$ in $\RR\times H^2(M)$ such that
\begin{equation*}
\p{- h^2\Delta_g + V}\psi^h = E_{h}\psi^h\text{ with } E_{h} = V(x_0) + \o{h^{2/3}}\text{ and } \norm{\psi^h}_{L^2(M)} = 1, \quad h\in (0, 1).
\end{equation*}
The function
\begin{equation*}
u^h\colon (t, x)\in \MM\mapsto e^{\frac{it\sqrt{E_{h}}}{h}}\psi^h(x)
\end{equation*}
is solution to~\eqref{eq:Wave} with initial data $(\psi^h, \frac{i\sqrt{E_{h}}}{h}\psi^h)$ and potential $\mathsf{V} = h^{-2} V$.
On the one hand, for $h$ small enough we have $E_{h}\geq \frac{1}{2}E_0$ and hence
\begin{equation*}
\Escr^h\p{u^h(0), \partial_tu^h(0)} = \frac{1}{2}\int_M\abs{h\nabla \psi^h}_g^2 + V(x)\abs{\psi^h}^2 + E_{h}\abs{\psi^h}^2\d x \geq \frac{1}{2}E_{h}\norm{\psi^h}^2_{L^2(M)} \geq \frac{1}{4}E_0 > 0.
\end{equation*}
On the other hand, using~\eqref{eq:Agmon:Estimate} from Theorem~\ref{Th:Agmon:Estimate}, there is a pair $(\delta, h_0)$ such that
\begin{equation*}
\exists C > 0,~\forall h\in (0, h_0), \quad \int_{\MM_{(0, T)}} \abs{b_\omega h\partial_t u^h}^2\d x\d t = \int_{0}^{T} \norm{b_\omega\sqrt{E_{h}}\psi^h}^2_{L^2\p{M}}\d t = T\norm{b_\omega\sqrt{E_{h}}\psi^h}^2_{L^2\p{M}}\leq TC e^{-2\delta/h}.
\end{equation*}
Using these two inequalities above, we deduce that
\begin{equation*}
\textstyle\exists \delta, h_0 > 0,~ \forall T > 0,~ \exists C > 0,~ \forall h < h_{0}, \quad \Cobs^{\mathrm{opt}}\p*{h^{-2}, V, T}\geq \frac{C}{T} e^{2\delta/h},
\end{equation*}
where $\Cobs^{\mathrm{opt}}$ is the \emph{optimal observability cost} defined in~\eqref{df:optimal:Cobs}.
This concludes the proof of the proposition with $\lmbd = h^{-2}$.
\end{proof}
\appendix
\section{Hilbert Uniqueness Method for \texorpdfstring{\eqref{eq:Control:Wave}}{}}\label{section:HUM}
The aim of this section is to provide a proof of the Hilbert Uniqueness Method due to~\cite{Dolecki77observation, lions1988controlabilite} in the context of Proposition~\ref{pr:dual:HUM} below.
See also~\cite{Coron07Control} for more details.
We will consider the space $H^1_{\mathsf{V}}(M)=H^1(M)$ endowed with the norm $\norm{\cdot}_{H^1_{\mathsf{V}}(M)}$ given in~\eqref{df:H^1_V} --- under the assumption $V\in\smooth\p{M, \RR_+^*}$ --- and $H^{-1}_{\mathsf{V}}(M)$ its dual space.
For simplicity, we set $\He_{\mathsf{V}}(M) = H^{1}_{\mathsf{V}}(M)\times L^2(M)$ and $\He_{\mathsf{V}}'(M) = L^2(M)\times H^{-1}_{\mathsf{V}}(M)$, equipped with the duality product
\begin{equation*}
\pscal{\begin{pmatrix}u_0\\ u_1\end{pmatrix}}{\begin{pmatrix} v_0\\ v_1\end{pmatrix}}_{\He_{\mathsf{V}}(M), \He'_{\mathsf{V}}(M)}
= \pscal{u_0}{v_1}_{H^{1}_{\mathsf{V}}(M), H^{-1}_{\mathsf{V}}(M)} - \pscal{u_1}{v_0}_{L^2\p{M}}.
\end{equation*}
We rewrite System~\eqref{eq:Control:Wave} as a first order system
\begin{equation}\label{eq:1st:order:wave}
\eqcases{rcll}{
\frac{\partial}{\partial t} \Ve + \Ae \Ve &=&  \Be f &\text{in}~(0, T)\times M,
\\ \Ve_{|t = 0} &=& \Ve_0&\text{in}~M,
}
\end{equation}
where 
\begin{equation}\label{eq:1st:order:notation}
\Ve := \begin{pmatrix} v \\ \partial_t v\end{pmatrix},~\Ae  := \begin{pmatrix} 0 & -1 \\ -\Delta_g + \mathsf{V}& 0\end{pmatrix},~\Be := \begin{pmatrix} 0 \\ b_\omega \end{pmatrix}~\text{and}~\Ve_0 \in \He_{\mathsf{V}}(M).
\end{equation}
We set the domain of $\Ae$ as $D(\Ae) = H^2(M)\times H^1(M)$.
Note that
\begin{equation}\label{rk:isometry}
\Ae\colon \p{\He_{\mathsf{V}}(M), \norm{\cdot}_{\He_{\mathsf{V}}}}\to \p{\He_{\mathsf{V}}'(M), \norm{\cdot}_{\He_{\mathsf{V}}'}}~\text{is an isometry.}
\end{equation}
\begin{lemma}\label{lem:dual:HUM}
Using notations~\eqref{eq:1st:order:notation} and setting  $\Be^{*} = \begin{pmatrix} b_\omega & 0\end{pmatrix}$, for a fixed time $T > 0$, the following two operators 
\begin{equation}\label{eq:R:HUM}
\Re_{\HUM}\colon f\in L^2((0, T)\times M)\mapsto -\int_0^Te^{s\Ae}\p{\Be f(s)}\d s\in \He_{\mathsf{V}}(M)
\end{equation}
and
\begin{equation}\label{eq:S:HUM}
\Se_{\HUM}\colon \Ue_0 \in \He_{\mathsf{V}}'(M) \mapsto  \p*{t\mapsto \Be^{*} e^{-t\Ae}\Ue_0}\in \Cscr^0\p{\p{0, T}, L^2(M)}
\end{equation}
are continuous and satisfy the dual identity
\begin{equation}\label{eq:HUM:dual}
\forall (f, \Uc_0)\in L^2((0, T)\times M)\times \He_\mathsf{V}'(M),\quad \pscal{\Re_{\HUM}(f)}{\Ue_0}_{\He_{\mathsf{V}}(M), \He'_{\mathsf{V}}(M)} = \pscal{f}{\Se_{\HUM}\p{\Ue_0}}_{L^{2}\p{\p{0, T}\times M}}.
\end{equation}
\end{lemma}
\begin{proof}[Proof of Lemma~\ref{lem:dual:HUM}]
On the one hand, by Hille--Yosida theorem, $\p{e^{-t\Ae}}_{t \in \RR}$ is a group of evolution on $\He_{\mathsf{V}}(M)$, and then the operators $\Re_{\HUM}$ and $\Se_{\HUM}$ are continuous.

On the other hand, using Duhamel formula with initial time $T$:
\begin{equation*}
    \Ve(t) = e^{-(t-T)\Ae}\Ve(T) + \int_{T}^{t}e^{-(t - s)\Ae}(\Be f(s))\d s,\quad t\in\RR,
\end{equation*}
for each \emph{control function} $f\in L^2\p{\p{0, T}\times M}$, the initial condition $\Ve_0^f := \Re_{\HUM}(f)$ is such that the $\Cscr^0\p{\p{0, T}, \He_{\mathsf{V}}(M)}$-solution of~\eqref{eq:1st:order:wave}, with initial data $\Ve_0^f\in\He_{\mathsf{V}}(M)$ and \emph{control function} $f$, vanishes in time $T$.
For all $\Ue_0\in \He_{\mathsf{V}}(M)$ and all $f\in L^2\p{\p{0, T}\times M}$, setting 
\begin{itemize}
\item $u$ the $\Cscr^0\p{\p{0, T}, H^1(M)}\cap\Cscr^1\p{\p{0, T}, L^2(M)}$-solution of~\eqref{eq:Wave} with initial data $\Ue_0 = (u_0, u_1)$,
\item $v$ the $\Cscr^0\p{\p{0, T}, H^1(M)}\cap\Cscr^1\p{\p{0, T}, L^2(M)}$-solution of~\eqref{eq:Control:Wave} with initial data $(v_0^f, v_1^f) := \Re_{\HUM}(f)$ and vanishing in time $T$,
\end{itemize}
by selfadjoint nature of $-\Delta_g + \mathsf{V}$ we have
\begin{align*}
\pscal{f}{\Se_{\HUM}\p{\Ue_0}}_{L^2\p{\p{0, T}\times M}}
&= \pscal{b_\omega f}{u}_{L^2((0, T)\times M)}\\
&= \pscal{\partial_t^2 v - \p{\Delta_g  + \mathsf{V}}v}{u}_{L^2((0, T)\times M)}\\
&= \int_{0}^{T}\pscal{\partial_t^2 v(t)}{u(t)}_{H^{-1}_{\mathsf{V}}, H^1_{\mathsf{V}}} + \pscal{v(t)}{ \p{- \Delta_g + \mathsf{V}}u(t)}_{H^{1}_{\mathsf{V}}, H^{-1}_{\mathsf{V}}}\d t\\
&= \int_0^T \pscal{\partial_t^2 v(t)}{u(t)}_{H^{-1}_{\mathsf{V}}, {H^1_{\mathsf{V}}}} - \pscal{v(t)}{ \partial_t^2u(t)}_{H^{1}_{\mathsf{V}}, H^{-1}_{\mathsf{V}}}\d t.
\end{align*}
Then, using $\partial_{t}\p{\pscal{\partial_tv}{u}_{L^2} - \pscal{v}{\partial_t u}_{H^{1}_{\mathsf{V}},H^{-1}_{\mathsf{V}}}} = \pscal{\partial_t^2 v}{u}_{{H^{-1}_{\mathsf{V}}, H^{1}_{\mathsf{V}}}} - \pscal{v}{\partial_t^2u}_{{H^{1}_{\mathsf{V}}, H^{-1}_{\mathsf{V}}}}$, we have
\begin{align*}
\pscal{f}{\Se_{\HUM}\p{\Ue_0}}_{L^2\p{\p{0, T}\times M}}
&= -\p*{\pscal{v_1^f}{u_0}_{L^2(M)} - \pscal{v_0^f}{u_1}_{H^{1}_{\mathsf{V}}(M),H^{-1}_{\mathsf{V}}(M)}}\\
&= \pscal{\begin{pmatrix} v_0^f \\ v_1^f \end{pmatrix}}{\begin{pmatrix} u_0 \\ u_1 \end{pmatrix}}_{\He_{\mathsf{V}}(M), \He'_{\mathsf{V}}(M)}\\
&= \pscal{\Re_{\HUM}(f)}{\Ue_0}_{\He_{\mathsf{V}}(M), \He'_{\mathsf{V}}(M)}.
\end{align*}
Density of $H^{1}_{\mathsf{V}}(M)\times L^2(M)$ in $L^2(M)\times H^{-1}_{\mathsf{V}}(M)$ concludes the proof of Lemma~\ref{lem:dual:HUM}.
\end{proof}
\begin{proposition}[Hilbert Uniqueness Method]\label{pr:dual:HUM}
For each time $T > 0$, the \emph{null-controllability of~\eqref{eq:Control:Wave} in time $T$  from $\omega$} is satisfied if and only if there is $\Cobs(\mathsf{V})\in (0, \infty)$ such that 
\begin{equation}\label{eq:dual:HUM}
\forall (u_0, u_1)\in L^2(M)\times H^{-1}(M),\quad 
\norm{u_0}_{L^2(M)}^{2} + \norm{u_1}_{H^{-1}_{\mathsf{V}}(M)}^{2}\leq \Cobs\p{\mathsf{V}}\norm{b_\omega u}_{L^2\p{\p{0, T}\times M}}^{2},
\end{equation}
where $u$ is the $\Cscr^0\p{\p{0, T}, L^2(M)}\cap\Cscr^1\p{\p{0, T}, H^{-1}(M)}$-solution of~\eqref{eq:Wave} with initial data $(u_0, u_1)$.
Furthermore, if~\eqref{eq:dual:HUM} holds, then for all initial data $(v_0, v_1)\in H^1_{\mathsf{V}}(M)\times L^2(M)$ there is a \emph{control function} $f_{(v_0, v_1)}\in L^2\p{\p{0, T}\times M}$ with
\begin{equation*}
\norm{f_{(v_0, v_1)}}_{L^2\p{\p{0, T}\times M}}\leq \sqrt{\Cobs\p{\mathsf{V}}} \norm{(v_0, v_1)}_{H^1_{\mathsf{V}}(M)\times L^2(M)},
\end{equation*}
such that \emph{null-controllability of~\eqref{eq:Control:Wave} in time $T$ from $\omega$} holds.
\end{proposition}
\begin{proof}[Proof of Proposition~\ref{pr:dual:HUM}] Let us consider  a time $T > 0$, $\Re_{\HUM}$ as~\eqref{eq:R:HUM} and $\Se_{\HUM}$ as~\eqref{eq:S:HUM}.
Under notations~\eqref{eq:1st:order:notation}, we rewrite~\eqref{eq:dual:HUM} as
\begin{equation}\label{eq:dual:HUM:1st}
\forall \Ue_0\in \He'_{\mathsf{V}}(M),\quad 
\norm{\Ue_0}_{\He'_{\mathsf{V}}(M)}^{2}\leq \Cobs\p{\mathsf{V}} \norm{\Se_{\HUM}\p{\Ue_0}}_{L^2\p{\p{0, T}\times M}}^{2}.
\end{equation}

First, suppose that the \emph{null-controllability of~\eqref{eq:Control:Wave} in time $T$ from $\omega$} is satisfied, then
\begin{equation*}
\mathrm{Ran}\p{\Re_{\HUM}} = \He_{\mathsf{V}}(M).
\end{equation*}
By open mapping theorem,
\begin{equation}\label{df:radius:open:mapping}
\exists r > 0, \quad 
B_{\He_{\mathsf{V}}(M)}\p{0, 1}\subset \Re_{\HUM}\p{B_{L^2\p{\p{0, T}\times M}}(0, r)}.
\end{equation}
For all $\Ue_0\in \He'_{\mathrm{V}}(M)$, there is $\Ve_0\in \He_{\mathsf{V}}(M)$, such that
\begin{equation*}
\norm{\Ve_0}_{\He_{\mathsf{V}}(M)}^2 = 1\quad\text{and}\quad \pscal{\Ve_0}{\Ue_0}_{\He_{\mathsf{V}}'(M), H^1_{\mathsf{V}}(M) \times L^2(M)} = \norm{\Ue_0}_{\He_{\mathsf{V}}'(M)}.
\end{equation*}
Then, by~\eqref{df:radius:open:mapping} there is $f\in B_{L^2\p{\p{0, T}, M}}(0, r)$ such that $\Ve_0 = \Re_{\HUM}\p{f}$ and then, by~\eqref{eq:HUM:dual}, we have
\begin{equation*}
\norm{\Ue_0}_{\He'_{\mathsf{V}}(M)} = \pscal{\Ve_0}{\Ue_0}_{\He_{\mathsf{V}}(M), \He'_{\mathsf{V}}(M)} = \pscal{f}{\Se_{\HUM}\p{\Ue_0}}_{L^{2}\p{(0, T)\times M}},
\end{equation*}
hence, by Cauchy-Schwarz inequality, we have
\begin{equation*}
\norm{\Ue_0}_{\He'_{\mathsf{V}}(M)} \leq r\norm{\Se_{\HUM}\p{\Ue_0}}_{L^2\p{\p{0, T}\times M}}
\end{equation*}
Setting $\Cobs\p{\mathsf{V}} := r^2$, this last inequality gives the upper bound~\eqref{eq:dual:HUM}.

\medskip Second, suppose now that the observability inequality~\eqref{eq:HUM:dual} holds.
By~\eqref{eq:dual:HUM:1st}, the sesquilinear continuous map
\begin{equation*}
\alpha\colon (\Ve, \Ue)\in \He'_{\mathsf{V}}(M)^2 \mapsto \pscal{\Se_{\HUM}\p{\Ve}}{\Se_{\HUM}\p{\Ue}}_{L^2\p{\p{0, T}\times M}}\in \CC
\end{equation*}
is coercive.
Consider a temporarily fixed $\Ve_0\in \He_{\mathsf{V}}(M)$ and define
\begin{equation*}
\Le_{\Ve_0}\colon \Ue\in \He'_{\mathsf{V}}(M) \mapsto \pscal{\Ve_0}{\Ue}_{\He_{\mathsf{V}}(M), \He'_{\mathsf{V}}(M)}.
\end{equation*}
By Lax-Milgram Theorem, there exists a unique $F_{\Ve_0}\in \He'_{\mathsf{V}}(M)$ such that
\begin{align*}
\forall \Ue \in \He'_{\mathsf{V}}(M), \quad\Le_{\Ve_0}(\Ue) &=  \alpha(F_{\Ve_0}, \Ue)\\
&=\pscal{\Re_{\HUM}\circ\Se_{\HUM}\p{F_{\Ve_0}}}{\Ue}_{\He_{\mathsf{V}}(M), \He'_{\mathsf{V}}(M)}\tag*{by~\eqref{eq:HUM:dual}.}
\end{align*}
Then, by uniqueness, we conclude that $\Ve_0 = \Re_{\HUM}(f_{\Ve_0})$ with $f_{\Ve_0} := \Se_{\HUM}\p{F_{\Ve_0}}$.
Furthermore, by~\eqref{eq:dual:HUM:1st}, we have
\begin{equation*}
\norm{f_{\Ve_0}}_{L^2\p{(0, T)\times M}}^2 = \abs{\alpha\p{F_{\Ve_0}, F_{\Ve_0}}} = \abs{\Le_{\Ve_0}\p{F_{\Ve_0}}} \leq \norm{\Ve_0}_{\He_{\mathsf{V}}(M)}\norm{F_{\Ve_0}}_{\He'_{\mathsf{V}}(M)}
\end{equation*}
hence, using $f_{\Ve_0} = \Se_{\HUM}\p{F_{\Ve_0}}$ and~\eqref{eq:dual:HUM:1st}, we conclude
\begin{equation}
\norm{f_{\Ve_0}}_{L^2\p{(0, T)\times M}}^2 \leq \sqrt{\Cobs(\mathsf{V})}\norm{\Ve_0}_{\He_{\mathsf{V}}(M)}\norm{f_{\Ve_0}}_{L^2\p{\p{0, T}\times M}}.\label{eq:astuce:lax:milgram}
\end{equation}
Therefore, by~\eqref{eq:astuce:lax:milgram}, we obtain $\norm{f_{\Ve_0}}_{L^2\p{(0, T)\times M}}\leq \sqrt{\Cobs(\mathsf{V})}\norm{\Ve_0}_{\He_{\mathsf{V}}(M)}$.
This concludes the proof of Proposition~\ref{pr:dual:HUM}.
\end{proof}
\begin{corollary}\label{Cor:HUM}
For each time $T > 0$, the \emph{null-controllability of~\eqref{eq:Control:Wave} in time $T$ from $\omega$} is satisfied if and only if there is an \emph{observability cost} $\Cobs(\mathsf{V})\in (0, \infty)$ such that~\eqref{eq:Obs} is satisfied for each initial data $(u_0, u_1)\in H^1(M)\times L^2(M)$.
\end{corollary}
\begin{proof}[Proof of Corollary~\ref{Cor:HUM}]
Firstly, using~\eqref{rk:isometry} and Proposition~\ref{pr:dual:HUM}, we have that \emph{null-controllability of~\eqref{eq:Control:Wave} in time $T > 0$ from $\omega$} is satisfied if and only if there is an \emph{observability cost} $\Cobs(\mathsf{V})\in (0, \infty)$ such that for all initial data $(u_0, u_1)\in H^1(M)\times L^2(M)$,
\begin{align*}
\norm{u_0}_{H^{1}_{\mathsf{V}}(M)}^{2} + \norm{u_1}_{L^2(M)}^{2} &= 
\norm{u_1}_{L^2(M)}^2 + \norm{\p{-\Delta_g + \mathsf{V}}u_0}_{H^{-1}_{\mathsf{V}}(M)}^2 \leq \Cobs\p{\mathsf{V}}\norm{b_\omega  \partial_t u}_{L^2\p{\p{0, T}\times M}}^2
\end{align*}
where $u$ is the $\Cscr^0\p{\p{0, T}, H^1(M)}\cap\Cscr^1\p{\p{0, T}, L^2(M)}$-solution of~\eqref{eq:Wave} with initial data $(u_0, u_1)$ and hence $\partial_t u$ is the $\Cscr^0\p{\p{0, T}, L^2(M)}\cap\Cscr^1\p{\p{0, T}, H^{-1}(M)}$-solution of~\eqref{eq:Wave} with initial data $\p{u_1, \p{-\Delta_g + \mathsf{V}}u_0}\in L^2(M)\times H^{-1}(M)$.
In conclusion~\eqref{eq:Obs} is equivalent to \emph{null-controllability of~\eqref{eq:Control:Wave} in time $T$ from $\omega$}.
\end{proof}
\section{Proof of Theorem~\ref{Th:Slices}}\label{proof:Th:Slices}
In this section, we provide the proof of Theorem~\ref{Th:Slices} (statement and proof are inspired from~\cite{LecturePG}).
We begin by recalling some technical facts concerning $(v_n^{\signe}, f^n, h_n)$, then we prove the result for the semiclassical defect measure in Subsection~\ref{Subsection:SDM:proof}, and we end with the proof of the result for the 2-microlocal defect measure in Subsection~\ref{subsection:MDM:proof}.
We recall that, for any sign $\signe\in\set{-, +}$, $-\signe$ is defined in Remark~\ref{rk:sign}. 

\paragraph{Sketch of the proof.}
The semiclassical and the 2-microlocal statements are proved in parallel, and their proofs follow the same scheme.
\begin{enumerate}
    \item We first give a commutator identity (Lemma~\ref{lem:L:crochet}) that rewrites $\pscal{v_n^{\signe}}{\frac{1}{h_n}[\HWOp^{-\signe}_{h_n},A_{h_n}]v_n^{\signe}}$ into pairings involving the source term $f^n$. Then, we prove an equicontinuity result (Lemma~\ref{lem:equicontinuous:family} together with Lemma~\ref{lem:equicontinuous:satisfied} in the 2-microlocal case) in order to apply the Ascoli Theorem. Finally, we end this technical part with Lemma~\ref{lem:Symbols:Tangentiels} (\resp Lemma~\ref{lem:Symbols:Tangentiels:micro}) which identifies the limit of quantizations of tangential symbols.
    For technical purposes, we also state a Faà di Bruno type result (Lemma~\ref{lem:uniform:symbol:substitution}) for the 2-microlocal case.
    \item We then prove Theorem~\ref{Th:Slices} as follows.
    \begin{enumerate}\renewcommand{\labelenumi}{\alph{enumi}.}
    \item Using the equicontinuity lemma and the Ascoli Theorem, we extract a subsequence for which the slice defect measures exist, and give its regularity with respect to the time parameter.
    \item For a general test symbol $a$, a Taylor expansion in the $\tau$-variable yields the decomposition
    \begin{equation*}
        a(t,x,\tau,\xi)=\underbracket{\tilde a_t(x,\xi)}_{\text{tangential part}}+\underbracket{\ell^{-\signe}(t,x,\tau,\xi)\,\alpha_{\signe}(t,x,\tau,\xi)}_{\text{remainder}},\quad (t,x,\tau,\xi)\in\cotang\hspace{1pt}\MM_I.
    \end{equation*}
    Combining this with Lemma~\ref{lem:Symbols:Tangentiels} or~\ref{lem:Symbols:Tangentiels:micro} gives the measure decomposition~\eqref{eq:PG:semiclassique:point2} (\resp\eqref{eq:PG:microlocal:point2}).
    \item We apply Lemma~\ref{lem:L:crochet} with a suitable family of operators (with or without the cut-off parameter $R$) and identify $\frac{i}{h_n}[\HWOp^{-\signe}_{h_n},A_{h_n}]$ with the corresponding Poisson bracket term $\set{\ell^{-\signe}, a}$, and then pass to the limits $n\to\infty$ (and $R\to\infty$ in the 2-microlocal case) to obtain the desired result.
\end{enumerate}
\end{enumerate}
\paragraph{First technical remarks and results.}
Recall that~\eqref{eq:unitary:operators} holds.
Then, using the Duhamel formula, since $\p{v_n^{\signe}(0)}_n$ is \apriori~ only bounded in $L^2(M)$, we have
\begin{equation*}
v_n^{\signe}(t)  = e^{\signe i t \frac{\SqrtOp_{h_n}}{h_n}}v_n^{\signe}(0) + \int^{t}_{0} e^{\signe i (t - \tau) \frac{\SqrtOp_{h_n}}{h_n}}h_n f^n(\tau)\d\tau
\end{equation*}
with $v^{\signe}_n\in\Cscr^0(I, L^2(M))$ only.
Then, for a temporarily fixed index $n\geq 1$, using the density of $H^1(M)$ in $L^2(M)$, there is a sequence $(v_{n, k}^{\signe}(0), f^{n, k})_k$ in $H^1(M)\times L^{\infty}(I, H^{1}(M))$ converging to $(v_{n}^{\signe}(0), f^{n})$ in $L^2(M)\times L^{\infty}(I, L^2(M))$.
For all $k\geq 1$, define the $\Cscr^{1}\p{I, L^2(M)}\cap \Cscr^0\p{I, H^1(M)}$ function
\begin{equation}\label{df:v:n:k}
v_{n, k}^{\signe}\colon t\in I \mapsto  e^{\signe i t \frac{\SqrtOp_{h_n}}{h_n}}v_{n, k}^{\signe}(0) + \int^{t}_{0} e^{\signe i (t - \tau) \frac{\SqrtOp_{h_n}}{h_n}}h_n f^{n, k}(\tau)\d\tau.
\end{equation}
Then, since $\p{e^{\signe i t \frac{\SqrtOp_{h}}{h}}}_{t\in I}$ is a family of unitary operators on $L^2(M)$, we have
\begin{equation}\label{eq:v:f:n:k:tends}
f^{n, k} \xrightarrow[k\to\infty]{L^\infty(I, L^2(M))} f^n
~\text{and}~
v_{n, k}^{\signe} \xrightarrow[k\to\infty]{L^\infty(I, L^2(M))} v^{\signe}_n
.
\end{equation}
\begin{lemma}\label{lem:L:crochet}
Let $(h_n)_n$ be a sequence of positive real numbers tending to $0$, $(v^{\signe}_{n})_n$ and $(f^n)_n$ be two bounded sequences of $L^\infty(I, L^2(M))$ such that for all $n\geq 1$ the function $v^{\signe}_{n}$ is $\Cscr^{0}(I, L^2(M))$ and \(\HWOp^{-\signe}_{h_n}v^{\signe}_{n} = h_n f^n\).
For all $A_{h}\in \pseudo^{0}_{h, \comp}\p{\MM_I}$ and all $n\geq 1$, we have
\begin{equation}\label{eq:L:crochet}
\pscal{v_n^{\signe}}{\frac{1}{h_n}\classe{\HWOp^{-\signe}_{h_n}, A_{h_n}}v_n^{\signe}}_{L^2\p{\MM_I}} = \pscal{f^n}{A_{h_n} v_n^{\signe}}_{L^2\p{\MM_I}} - \pscal{v_n^{\signe}}{A_{h_n} f^n}_{L^2\p{\MM_I}}.
\end{equation}
\end{lemma}
\begin{proof}[Proof of Lemma~\ref{lem:L:crochet}]
Let us consider $A_{h}\in \pseudo^{0}_{h, \comp}\p{\MM_I}$, we look at
\begin{equation*}
\classe{\HWOp^{-\signe}_{h}, A_{h}} = \classe{h D_t, A_{h}} - \classe{\signe\SqrtOp_{h}, A_{h}}.
\end{equation*}
First, if $v^{\signe}_n\in \Cscr^1\p{I, L^2(M)}\cap \Cscr^0\p{I, H^1(M)}$.
On the one hand,
\begin{align*}
\pscal{v_n^{\signe}}{\classe{D_t, A_{h_n}}v_n^{\signe}}_{L^2\p{\MM_I}}
& = \pscal{v_n^{\signe}}{D_tA_{h_n} v_n^{\signe}}_{L^2\p{\MM_I}} - \pscal{v_n^{\signe}}{A_{h_n} D_tv_n^{\signe}}_{L^2\p{\MM_I}}\\
& = \pscal{D_t v_n^{\signe}}{A_{h_n} v_n^{\signe}}_{L^2\p{\MM_I}} - \pscal{v_n^{\signe}}{A_{h_n} D_tv_n^{\signe}}_{L^2\p{\MM_I}}.
\end{align*}
On the other hand,
\begin{align*}
\pscal{v_n^{\signe}}{\classe{\SqrtOp_{h_n}, A_{h_n}}v_n^{\signe}}_{L^2\p{\MM_I}}
& = \pscal{v_n^{\signe}}{\SqrtOp_{h_n}A_{h_n} v_n^{\signe}}_{L^2\p{\MM_I}} - \pscal{v_n^{\signe}}{A_{h_n} \SqrtOp_{h_n}v_n^{\signe}}_{L^2\p{\MM_I}}\\
& = \pscal{\SqrtOp_{h_n}v_n^{\signe}}{A_{h_n} v_n^{\signe}}_{L^2\p{\MM_I}} - \pscal{v_n^{\signe}}{A_{h_n} \SqrtOp_{h_n}v_n^{\signe}}_{L^2\p{\MM_I}}.
\end{align*}
Adding these two lines and recalling that $\HWOp^{-\signe}_{h_n}v^{\signe}_n = h_nf^n$ we deduce that~\eqref{eq:L:crochet} is true for $v^{\signe}_n\in \Cscr^1\p{I, L^2(M)}\cap \Cscr^0\p{I, H^1(M)}$.
Second, using density of $H^1_{h}(M)$ in $L^2(M)$ and Lebesgue's dominated convergence theorem, we conclude that~\eqref{eq:L:crochet} still holds.
\end{proof}

We then prove boundedness and equicontinuity of $t\mapsto \pscal{v^{\signe}_{n}(t)}{A_{h_n}v^{\signe}_{n}(t)}_{L^2(M)}$, where $A_h$ may or may not depend on the parameter $R$.
This will allow us to apply Ascoli theorem in the first part of the proof of the semiclassical and 2-microlocal result of Theorem~\ref{Th:Slices}.
\begin{lemma}[Equicontinuous family]\label{lem:equicontinuous:family}
Let $(h_n)_n$ be a sequence of positive real numbers tending to $0$, $(v^{\signe}_{n})_n$ and $(f^n)_n$ be two bounded sequences of $L^\infty(I, L^2(M))$ such that for all $n\geq 1$ the function $v^{\signe}_{n}$ is $\Cscr^{0}(I, L^2(M))$ and $\HWOp^{-\signe}_{h_n}v^{\signe}_{n} = h_n f^n$, and $q$ be as defined in~\eqref{df:Sqrt:Sym}.
For any $\mathfrak{m}_1, \mathfrak{m}_2 > 0$, we consider the (nonempty) pseudodifferential operator family
\begin{equation}\label{df:family:Ak}
\Ak \subset \set*{\Op_{h}^{M}(a^h)\in\pseudo_{h, \comp}^{0}(M)
~\middle| ~
\begin{array}{l}
    a^h\in S^{0}_{\comp}(\cotang M),\\\sup_{h}\norm{a^h}_{L^\infty(\cotang M)}\leq \mathfrak{m}_1,\\
    \sup_{h}\norm{\set{a^h,\sqrtsymbol}}_{L^\infty(\cotang M)}\leq \mathfrak{m}_2
\end{array}
}.
\end{equation}
Then, for all $A_h\in\Ak$, the family of functions
\begin{equation*}
    \Fk^{A} := \set*{F^{A}_n\colon t\in I\mapsto \pscal{v^{\signe}_{n}(t)}{A_{h_n}v^{\signe}_{n}(t)}_{L^2(M)},~ n\geq 1}
\end{equation*}
is uniformly bounded and equicontinuous on $I$.
If we furthermore assume 
that
\begin{equation*}
    \forall A_h\in\Ak,~ \forall t\in I,~\exists F_{\infty}^{A}(t) \in\CC, \quad \lim_{n\to\infty}\pscal{v^{\signe}_{n}(t)}{A_{h_n}v^{\signe}_{n}(t)}_{L^2(M)} = F_{\infty}^{A}(t),
\end{equation*}
then, the family
$\Fk^{\Ak} := \set*{t\in I\mapsto F^{A}_{\infty}(t),~ A_{h}\in\Ak}$
also is uniformly bounded and equicontinuous on $I$.
\end{lemma}
\begin{proof}[Proof of Lemma~\ref{lem:equicontinuous:family}]
Fix temporarily $A_h\in\Ak$.
    
First, by Proposition~\ref{pr:Calderon:Vaillancourt} the family of continuous functions $\Fk^{A}$ is bounded on the interval $I$ in the following sense: there exists $C_A > 0$ such that for all $n\geq 1$ we have
\begin{equation}\label{eq:semiclassic:bounded:family}
\sup_{t\in I}\abs{F^A_n(t)}\leq \p{\mathfrak{m}_1 + C_A h_n^{1/2}}\sup_{m\geq 1}\norm{v^{\signe}_m}_{L^\infty(I, L^2(M))}^2.
\end{equation}
Thus, $\Fk^{A}$ is uniformly bounded by~\eqref{eq:semiclassic:bounded:family}.

Second, we prove that $\Fk^{A}$ is equicontinuous: since both $A_{h}$ and $h^{-1}[A_{h},\signe\SqrtOp_{h}]$ belong to $\Psi^{0}_{h, \comp}(M)$, Proposition~\ref{pr:Calderon:Vaillancourt} yields constants $C_A', C_A'' > 0$ such that for all $n\geq 1$ we have
\begin{equation}\label{eq:diag:bound}
\abs{\pscal{v^{\signe}_{n}(t)}{h_n^{-1}[A_{h_n},\signe\SqrtOp_{h_n}]v^{\signe}_{n}(t)}_{L^2(M)}} \leq \p{\mathfrak{m}_2 + C_A'h_n^{1/2}} \sup_{m\geq 1}\norm{v^{\signe}_m}^2_{L^\infty(I, L^2(M))}
\end{equation}
and
\begin{equation}\label{eq:antidiag:bound}
\abs{\pscal{{f}^n(t)}{A_{h_n}v^{\signe}_{n}(t)}_{L^2}}\leq \p{\mathfrak{m}_1 + C_A''h_n^{1/2}} \sup_{m\geq 1}\p{\norm{f^m}_{L^\infty(I, L^2(M))}\norm{v^{\signe}_m}_{L^\infty(I, L^2(M))}}.
\end{equation}
Hence, using~\eqref{eq:v:f:n:k:tends} and the notation $(v_{n,k}^{\signe}, f^{n, k})_k$ from~\eqref{df:v:n:k}, for each temporarily fixed $n\geq 1$ and all $t_1, t_2\in I$ we have
\begin{equation}\label{eq:30}
\abs{F^A_{n, k}(t_1) - F^A_{n,k}(t_2)}\leq \abs{t_1 - t_2}\,
\sup_{t\in I}
\abs{D_t \pscal{v^{\signe}_{n, k}(t)}{A_{h_n}v^{\signe}_{n, k}(t)}_{L^2}}, \quad k\geq 1.
\end{equation}
Moreover, for each $k\in\NN$, since $v^{\signe}_{n,k}\in \Cscr^{1}(I,L^2(M))$, we have for every $t\in I$:
\begin{align}
\notag h_n D_t\pscal{v^{\signe}_{n, k}(t)}{A_{h_n}v^{\signe}_{n, k}(t)}_{L^2} &= -\pscal{h_n D_t v^{\signe}_{n, k}(t)}{A_{h_n}v^{\signe}_{n, k}(t)}_{L^2} + \pscal{v^{\signe}_{n, k}(t)}{A_{h_n}(h_n D_t v^{\signe}_{n, k}(t))}_{L^2}\\
&\notag \hspace{-4cm} = \pscal{ v^{\signe}_{n, k}(t)}{[A_{h_n},\signe\SqrtOp_{h_n}]v^{\signe}_{n, k}(t)}_{L^2} - \pscal{h_n{f}^{n, k}(t)}{A_{h_n}v^{\signe}_{n, k}(t)}_{L^2} + \pscal{v^{\signe}_{n, k}(t)}{h_nA_{h_n}{f}^{n, k}(t)}_{L^2}\\
&\hspace{-4cm} = \pscal{v^{\signe}_{n, k}(t)}{[A_{h_n},\signe\SqrtOp_{h_n}]v^{\signe}_{n, k}(t) + h_nA_{h_n}{f}^{n, k}(t)}_{L^2} - \pscal{h_n{f}^{n, k}(t)}{A_{h_n}v^{\signe}_{n, k}(t)}_{L^2}.\label{eq:31}
\end{align}
Therefore, using~\eqref{eq:31} and then passing to the limit $k\to\infty$ in~\eqref{eq:30}, we obtain
\begin{equation}\label{eq:bounded:derivative}
\frac{\abs{F^A_{n}(t_1) - F^A_{n}(t_2)}}{\abs{t_1 - t_2}}\leq
\sup_{t\in I}
\abs{\pscal{v^{\signe}_{n}(t)}{h_n^{-1}[A_{h_n},\signe\SqrtOp_{h_n}]v^{\signe}_{n}(t) + A_{h_n}{f}^n(t)}_{L^2} - \pscal{{f}^n(t)}{A_{h_n}v^{\signe}_{n}(t)}_{L^2}}.
\end{equation}
Thus,~\eqref{eq:bounded:derivative} together with~\eqref{eq:diag:bound} and~\eqref{eq:antidiag:bound} implies that there exists $C > 0$ such that $\Fk^{A}$ is a family of $C$-Lipschitz functions. In conclusion, by~\eqref{eq:semiclassic:bounded:family} and~\eqref{eq:bounded:derivative}, the family $\Fk^{A}$ is uniformly bounded and uniformly equicontinuous on $I$.

Finally, combining~\eqref{eq:semiclassic:bounded:family},~\eqref{eq:bounded:derivative},~\eqref{eq:diag:bound}, and~\eqref{eq:antidiag:bound} and letting $n\to\infty$, we deduce that the family $\Fk^{\Ak}$ is uniformly bounded and consists of $C'$-Lipschitz functions for some constant $C'>0$.
This concludes the proof of Lemma~\ref{lem:equicontinuous:family}.
\end{proof}
\subsection{Slice semiclassical defect measure}\label{Subsection:SDM:proof}
In this subsection, we provide the proof of the first part (the result for semiclassical defect measures) of Theorem~\ref{Th:Slices}.
We begin by a lemma to compute the limit of quantized tangential symbols and then provide the proof for Theorem~\ref{Th:Slices}.
In the following lemma, we assume that the sequence admits a unique semiclassical defect measure that depends continuously on time.
\begin{lemma}\label{lem:Symbols:Tangentiels}
Let us consider $(w^n)_n$ a sequence of $L^\infty(I, L^2(M))$ such that for all $t\in I$, $(w^n(t), h_n)_n$ admits a semiclassical defect measure $\sdm_{t}\in\Mc_+(\cotang M)$ in the sense of~\eqref{eq:SDM}.
Assume furthermore that $t\in I \mapsto \sdm_{t}$ is continuous for the weak-* topology of measures.
Then, for all $a\in\test(I\times \cotang M)$, we have
\begin{equation*}
\lim_{n\to \infty}\pscal{w^n}{\Op_{h_n, \tang}^{M}(a)w^n}_{L^2(\MM_I)} = \int_{I\times \cotang M}a(t, x, \xi) \d\sdm_{t}(x, \xi)\d t,
\end{equation*}
where the tangential quantization $\Op_{h, \tang}^{M}$ is defined in Definition~\ref{df:tangential:quantization}.
\end{lemma}
\begin{proof}[Proof of Lemma~\ref{lem:Symbols:Tangentiels}]
For all $(\psi, \phi)\in \test(I)\times \test(\cotang M)$, we have 
\begin{equation*}
\lim_{n\to \infty}\pscal{w^n}{\Op_{h_n, \tang}^{M}(\psi\phi)w^n}_{L^2} = \lim_{n\to \infty}\int_I\psi(t)\pscal{w^n(t)}{\Op_{h_n}^{M}(\phi)w^n(t)}_{L^2(M)}\d t \\
\end{equation*}
Thus, by Lebesgue's dominated convergence theorem and~\eqref{eq:SDM}, we have
\begin{equation*}
\lim_{n\to \infty}\pscal{w^n}{\Op_{h_n, \tang}^{M}(\psi\phi)w^n}_{L^2}= \int_{I\times \cotang M}\psi(t)\phi(x, \xi) \d\sdm_{t}(x, \xi)\d t.
\end{equation*}
Then for any $a \in \test(I)\otimes \test(\cotang M)$, we obtain 
\begin{align*}
\lim_{n\to \infty}\pscal{w^n}{\Op_{h_n, \tang}^{M}(a)w^n}_{L^2}  = \int_{I\times \cotang M}a(t, x, \xi) \d\sdm_{t}(x, \xi)\d t.
\end{align*}
Therefore, using the density of $\test(I)\otimes \test(\cotang M)$ in $\test(I\times \cotang M)$ for the topology of {uniform
convergence} (see for example~\cite[Chapter 11, Remark 11.12]{duistermaat2010distributions}), we can conclude the proof for functions of $\test(I\times \cotang M)$.
\end{proof}
\begin{proof}[Proof of first point (semiclassical result) of Theorem~\ref{Th:Slices}] We use the same enumeration as in the statement of Theorem~\ref{Th:Slices}, and prove them in the same order as stated.
\begin{enumerate}\renewcommand{\labelenumi}{\alph{enumi}.}
\item Let $a\in\test(\cotang M)$ and set $A_{h} := \Op_{h}^{M}(a)\in{\pseudo}^{\comp}_{h}(M)$.
For all $n\geq 1$ we define
\begin{equation*}
\Fk^{a} := \set{F_{n}^{a}\colon t\in I\mapsto \pscal{v^{\signe}_{n}(t)}{A_{h_n}v^{\signe}_{n}(t)}_{L^2(M)}, n\geq 1}.
\end{equation*}
Then, the family $\Fk^{a}$ is uniformly bounded and uniformly equicontinuous on $I$ by Lemma~\ref{lem:equicontinuous:family} and therefore Ascoli Theorem ensures that, for any compact interval $J\subset I$, $\Fk^{a}$ has compact closure for the topology of $\Cscr^0_{b}(J, \CC)$.
As a consequence, for all compact interval $J\subset I$ and all $a\in \test(\cotang M)$ there is a subsequence of $(v^{\signe}_{n})_n$, still denoted $(v^{\signe}_{n})_n$, such that $\p{F_{n}^{a}}_n$ converges uniformly on $J$.
Finally, fix an increasing sequence of compact intervals $(J_k)_k$ with $I=\bigcup_k J_k$ and a countable dense subset of $\test(\cotang M)$.
By diagonal extraction argument, there is a subsequence of $(v^{\signe}_{n})_n$, still denoted $(v^{\signe}_{n})_n$, and a family $(\sdm^{\signe}_{t})_{t\in I}$ of semiclassical defect measures in $\Mc_+(\cotang M)$ such that: for all compact interval $J\subset I$, all $a\in\test(\cotang M)$ and all $\varepsilon>0$, there is $N(J,a)\ge 1$ such that
\begin{equation*}\forall n\geq N(J, a),\quad \sup_{t\in J} \abs*{\pscal{v^{\signe}_{n}(t)}{A_{h_n}v^{\signe}_{n}(t)}_{L^2(M)} - \int_{\cotang M}a(x, \xi)\d\sdm^{\signe}_{t}(x, \xi)}\leq \varepsilon.\end{equation*}
As a limit of a sequence of continuous functions with uniform convergence on all compact subset of $I$,
\begin{equation*}
t\in I\mapsto \int_{\cotang M}a\d\sdm^{\signe}_{t}
\end{equation*}
is continuous.
That is to say $t\in I\mapsto \sdm^{\signe}_{t}$ is continuous for the weak-* topology of measures.
\item 
We use the function $\chi$ from~\eqref{eq:wave-type-ghost}.
By Taylor formula, for all $a\in\test(\cotang\hspace{1pt} \MM_I)$ we have
\begin{equation*}
a(t, x, \tau, \xi) = \tilde{a}_{t}(x, \xi) + \ell^{-\signe}(t, x, \tau, \xi)\alpha_{\signe}(t, x, \tau,\xi), \quad \quad (t, x, \tau, \xi)\in\cotang\hspace{1pt} \MM_I,
\end{equation*}
where we set $\tilde{a}\in\test\p{I\times \cotang M}$ as
\begin{equation*}
\tilde{a}_{t}(x, \xi) := a\p{t, x, \signe\sqrtsymbol(x, \xi), \xi}, \quad (t, (x, \xi))\in I\times \cotang M,
\end{equation*}
and
\begin{equation*}
\alpha_{\signe}(t, x, \tau,\xi) := \int_{0}^{1}\frac{\partial a}{\partial\tau}\p{t, x, s\tau + (1-s)\signe\sqrtsymbol(x, \xi), \xi}\d s, \quad (t, x, \tau, \xi)\in\cotang\hspace{1pt} \MM_I.
\end{equation*}
Since $a$ is compactly supported, so are $\tilde{a}$ (in $I\times \cotang M$) and $\frac{\partial a}{\partial\tau}$, and hence $\alpha_{\signe}$.
By Theorem~\ref{Th:hormander}, we obtain
\begin{equation}\label{eq:hormander:semiclassic:calculus}
\Op_{h}^{\MM}(\chi a) = \Op_{h,\tang}^{M}(\tilde{a})\Op_{h}^{\MM}(\chi) + \underbracket{\Op_{h}^{\MM}(\chi\alpha_{\signe})}_{\in\pseudo_{h}^{\comp}(\MM)}\HWOp^{-\signe}_{h} + h\rest^{\comp}_{h},
\quad\rest^{\comp}_{h}\in \pseudo^{\comp}_{h}(\MM_I).
\end{equation}
As a consequence of~\eqref{eq:hormander:semiclassic:calculus} we have
\begin{align}
\notag\pscal{v^{\signe}_{n}}{\Op_{h_n}^{\MM}(\chi a)v^{\signe}_{n}}_{L^2(\MM_I)} = &\pscal{v^{\signe}_{n}}{\Op_{h_n,\tang}^{M}(\tilde{a})\Op_{h_n}^{\MM}(\chi)v^{\signe}_{n}}_{L^2(\MM_I)} \\
&+ h_n\pscal{v^{\signe}_{n}}{\Op^{\MM}_{h_n}(\chi \alpha_{\signe}) {f}^n + \rest_{h_n}^{\comp}v^{\signe}_n}_{L^2(\MM_I)}.\label{eq:sc:hormander:calculus:pscal}
\end{align}
By Proposition~\ref{pr:Calderon:Vaillancourt} and the Cauchy-Schwarz inequality, 
\begin{equation}\label{eq:sc:bounded:rest}
\pscal{v^{\signe}_{n}}{\Op^{\MM}_{h_n}(\chi \alpha_{\signe}) {f}^n + \rest_{h_n}^{\comp}v^{\signe}_n}_{L^2(\MM_I)} = \O{1}_n. 
\end{equation}
By~\eqref{eq:wave-type-ghost}, we also have
\begin{equation}\label{eq:ghost:argument}
    \pscal{v^{\signe}_{n}}{\Op_{h_n,\tang}^{M}(\tilde{a})\Op_{h_n}^{\MM}(\chi)v^{\signe}_{n}}_{L^2(\MM_I)} = \pscal{v^{\signe}_{n}}{\Op_{h_n,\tang}^{M}(\tilde{a})v^{\signe}_{n}}_{L^2(\MM_I)} + \o{1}_n.
\end{equation}
So, using~\eqref{eq:sc:bounded:rest} and~\eqref{eq:ghost:argument} in~\eqref{eq:sc:hormander:calculus:pscal}, Lemma~\ref{lem:Symbols:Tangentiels} and~\eqref{eq:SDM} imply 
\begin{equation*}
\int_{I\times\cotang M}\tilde{a}_t(x, \xi)\d\sdm_t^{\signe}(x, \xi)\d t = \lim_{n\to\infty}\pscal{v^{\signe}_{n}}{\Op_{h_n}^{\MM}(\chi a)v^{\signe}_{n}}_{L^2(\MM_I)} = \int_{\cotang\hspace{1pt} \MM_I} a(t, x, \tau, \xi) \d\sdm_{(\signe, \signe)}(t, x, \tau, \xi).
\end{equation*}
Thus,~\eqref{eq:PG:semiclassique:point2} holds.
\item 
We use the function $\chi$ from~\eqref{eq:wave-type-ghost}.
From~\eqref{eq:PG:semiclassique:point2}, we have 
\begin{equation}\label{eq:semiclassic:support}
\supp\p{\sdm_{(\signe, \signe)}}\subseteq\set{(t, x, \tau, \xi)\in\cotang \hspace{1 pt}\MM_I,~ \tau = \signe\sqrtsymbol(x, \xi)}.
\end{equation}
Furthermore, using Proposition~\ref{pr:inter:supp:k}, $\sdm_{(\signe, f)}$ (a finite-valued complex measure) is absolutely continuous with respect to $\sdm_{(\signe, \signe)}$ (a nonnegative $\sigma$-finite measure), hence, by Radon--Nikodym Theorem and~\eqref{eq:PG:semiclassique:point2} there is a complex measure $\tilde{\sdm}_{(\signe, f)}^t\Mc_+(\cotang M)$ such that
\begin{equation*}
\sdm_{(\signe, f)}(t, x, \tau, \xi) = \tilde{\sdm}_{(\signe, f)}^t(x, \xi)\otimes \delta_{\signe\sqrtsymbol(x, \xi)}\p{\tau}\otimes\Lc_1(t).
\end{equation*}
Take a smooth function $a\in\test(I\times \cotang M)$.
We have $\chi\ell^{-\signe}\in S^{1}(\cotang\hspace{1pt} \MM_I)$ by Theorem~\ref{Th:hormander} and $\chi a\in \test(\cotang\hspace{1pt} \MM_I)$.
On the one hand, since $\chi = 1$ near $\set{(t, x, \tau, \xi)\in\cotang \hspace{1 pt}\MM_I,~ \tau = \signe\sqrtsymbol(x, \xi)}$, using~\eqref{eq:semiclassic:support} we have 
\begin{equation*}
\set{\ell^{-\signe}, \chi} =  0\text{ near }\supp\p{\sdm_{(\signe, \signe)}}.
\end{equation*}
Thus, we have
\begin{equation*}
\notag\set{\ell^{-\signe}, \chi a} = r + \set{\ell^{-\signe}, a}\chi = r + \p{\partial_t a + \set{-\signe\sqrtsymbol,  a}}\chi
\end{equation*}
where $r = \set{\ell^{-\signe}, \chi}a$ is equal to $0$ near $\set{(t, x, \tau, \xi)\in\cotang \hspace{1 pt}\MM_I,~ \tau = \signe\sqrtsymbol(x, \xi)}$ and compactly supported.
Thus, 
\begin{equation*}
\pscal{v^{\signe}_{n}}{\Op_{h_n}^{\MM}\p{\set{\ell^{-\signe}, \chi a}}v^{\signe}_{n}}_{L^2(\MM_I)} = \pscal{v^{\signe}_{n}}{\Op_{h_n}^{\MM}\classe{\p{\partial_t a + \set{-\signe\sqrtsymbol,  a}}\chi}v^{\signe}_{n}}_{L^2} + \o{1}_n.
\end{equation*}
In conclusion, since $\p{\partial_t a + \set{-\signe\sqrtsymbol,  a}}\chi\in\test(\cotang\hspace{1pt} \MM_I)$,~\eqref{eq:PG:semiclassique:point2} implies 
\begin{equation}\label{eq:sc:bla}
\lim_{n\to\infty}\pscal{v^{\signe}_{n}}{\Op_{h_n}^{\MM}\p{\set{\ell^{-\signe}, \chi a}}v^{\signe}_{n}}_{L^2(\MM_I)} = \int_{I\times \cotang M}\partial_t a - \set{\signe\sqrtsymbol,  a}\d\sdm^{\signe}_{t}\d t
\end{equation}
by~\eqref{eq:SDM}.
On the other hand, by Theorem~\ref{Th:hormander}, we have $\set{\ell^{-\signe}, \chi a}\in S^1\p{\cotang\hspace{1pt} \MM_I}$ and, using furthermore Proposition~\ref{pr:composition}, we have
\begin{equation}\label{eq:hormander:lie:bracket}
\Op_{h}^{\MM}\p*{\set{\ell^{-\signe}, \chi a}} = \frac{i}{h}\classe{\HWOp^{-\signe}_{h_n}, \Op_{h}^{\MM}\p{\chi a}} + h\rest^{\comp}_{h}, \quad \rest_{h}^{\comp}\in \pseudo^{\comp}_{h}\p{\MM_I}.
\end{equation}
Then, by~\eqref{eq:hormander:lie:bracket}, Lemma~\ref{lem:L:crochet} and Proposition~\ref{pr:boundedness:pseudo}, we have
\begin{align}
\notag\pscal{v^{\signe}_{n}}{\Op_{h_n}^{\MM}\p{\set{\ell^{-\signe}, \chi a}}v^{\signe}_{n}}_{L^2(\MM_I)}
&= \frac{i}{h_n}\pscal{v^{\signe}_{n}}{[\HWOp^{-\signe}_{h_n}, \Op_{h_n}^{\MM}(\chi a)]v^{\signe}_{n}}_{L^2(\MM_I)} + \O{h_n}\\
&\hspace{-2cm}= i\p*{\pscal{f^{n}}{\Op_{h_n}^{\MM}(\chi a)v^{\signe}_{n}}_{L^2(\MM_I)} - \pscal{v^{\signe}_{n}}{\Op_{h_n}^{\MM}\p{\chi a}f^{n}}_{L^2(\MM_I)}}+ \O{h_n}.\label{eq:no:se:sc}
\end{align}
Hence, by~\eqref{eq:SDM} and~\eqref{eq:no:se:sc}, we have
\begin{equation}\label{eq:sc:blabla}
\lim_{n\to\infty}\pscal{v^{\signe}_{n}}{\Op_{h_n}^{\MM}\p{\set{\ell^{-\signe}, \chi a}}v^{\signe}_{n}}_{L^2(\MM_I)} = -2\int_{\MM_I}a(t, x, \xi)\d\Imaginary\sdm_{(\signe, f)}(t, x, \tau, \xi).
\end{equation}
In conclusion,~\eqref{eq:sc:bla} and~\eqref{eq:sc:blabla} prove~\eqref{eq:PG:semiclassic:point3}.
\end{enumerate}
Then, we proved all points of the theorem.
\end{proof}
\subsection{Slice 2-microlocal defect measure}\label{subsection:MDM:proof}
In the present subsection, we provide the proof of the first part of the result concerning 2-microlocal defect measures of Theorem~\ref{Th:Slices}.
The scheme of the proof is the same as in Appendix~\ref{Subsection:SDM:proof} except double limit and homogeneity.
\paragraph{Technical lemmas.}
We begin by a lemma to compute limit of quantified tangential symbols, then we prove three technical lemmas (the proof of Lemma~\ref{lem:uniform:symbol:substitution} is provided at the end of the subsection) and we finally provide the proof for Theorem~\ref{Th:Slices}.
\begin{lemma}\label{lem:Symbols:Tangentiels:micro}
Let $(h_n)_n$ be a sequence of positive real numbers tending to $0$ and $(w^n)_n$ a sequence of $L^\infty(I, L^2(M))$ such that, for all $t\in I$, $(w^n(t), h_n)_n$ admits a unique 2-microlocal defect measure $\mdm_{t}\in\Mc_+(\cosphere M)$ in the sense of~\eqref{eq:MDM}.
Assume furthermore that $t\in I \mapsto \mdm_{t}$ is continuous for weak-* topology of measures.
Then, for any cut-off function $\theta$ as~\eqref{df:theta} and for all $a\in S^{0}_{\phg, \tang}(I\times \cotang M)$, we have
\begin{equation*}
\lim_{R\to\infty}\lim_{n\to +\infty}
{\textstyle
\pscal{w^n}{\Op_{h_n, \tang}^{M}\p{a\p{1-\theta\p{\frac{\jap{\xi}_g}{R}}}}w^n}_{L^2(\MM_I)} 
}
= \int_{I\times \cosphere M}a(t, x, \xi) \d\mdm_{t}(x, \xi)\d t,
\end{equation*}
where the tangential quantization $\Op_{h, \tang}^{M}$ is defined in Definition~\ref{df:tangential:quantization}.
\end{lemma}
\begin{proof}[Proof of Lemma~\ref{lem:Symbols:Tangentiels:micro}]
Let us consider the function
\begin{equation}\label{df:theta:xi}
\holeperp{R}(x, \xi) := 1-\theta\p*{\frac{\jap{\xi}_g}{R}}, \quad (x, \xi)\in\cotang M,~ R > 0,
\end{equation}
following the notation of~\cite{Anantharaman14Semiclassical}.
Then, for all $(\psi, \phi)\in \test(I)\times S^{0}_{\phg}(\cotang M)$, we have
\begin{equation*}
\lim_{R\to\infty}\lim_{n\to \infty}\pscal{w^n}{\Op_{h_n, \tang}^{M}(\psi\phi\holeperp{R})w^n}_{L^2(\MM_I)} 
= \lim_{R\to\infty}\lim_{n\to \infty}\int_I\psi(t)\pscal{w^n(t)}{\Op_{h_n}^{M}(\phi\holeperp{R})w^n(t)}_{L^2(M)}\d t.
\end{equation*}
Then, by Lebesgue's dominated convergence theorem, we have
\begin{equation*}
\lim_{R\to\infty}\lim_{n\to \infty}\pscal{w^n}{\Op_{h_n, \tang}^{M}(\psi\phi\holeperp{R})w^n}_{L^2(\MM_I)} = \int_I\psi(t)\lim_{R\to\infty}\lim_{n\to \infty}\pscal{w^n(t)}{\Op_{h_n}^{M}(\phi\holeperp{R})w^n(t)}_{L^2(M)}\d t.
\end{equation*}
Finally, by~\eqref{eq:MDM} and assumption,
\begin{equation*}
\lim_{R\to\infty}\lim_{n\to \infty}\pscal{w^n}{\Op_{h_n, \tang}^{M}(\psi\phi\holeperp{R})w^n}_{L^2(\MM_I)} = \int_{I}\psi(t)\pscal{\mdm_{t}}{\phi}_{\Mc_+, \Cscr^0_c\p{\cosphere M}}\d t.
\end{equation*}
We conclude that, for any $\mathsf{a} \in \test(I)\otimes \smooth(\cosphere M)$, we have 
\begin{equation}\label{eq:tensorial:product}
\lim_{R\to\infty}\lim_{n\to \infty}\pscal{w^n}{\Op_{h_n, \tang}^{M}\p{\textstyle \mathsf{a}\p{t, x, \frac{\xi}{~\abs{\xi}_g}}\holeperp{R}}w^n}_{L^2(\MM_I)}  = \int_{I\times \cosphere M}\mathsf{a}(t, x, \xi) \d\mdm_{t}(x, \xi)\d t.
\end{equation}
Then, using the density of $\test(I)\otimes \smooth(\cosphere M)$ in $\test(I\times \cosphere M)$ for the topology of uniform convergence (see for example~\cite[Chapter 11, Remark 11.12]{duistermaat2010distributions}), we can conclude that~\eqref{eq:tensorial:product} also holds for $S^{0}_{\phg, \tang}(I \times\cotang M)$.
\end{proof}
\begin{lemma}\label{lem:equicontinuous:satisfied}
For $\Theta_{\xi}^R$ as~\eqref{df:theta:xi} and all $a\in S^{0}_{\phg}(\cotang \hspace{1 pt} \MM_I)$, there are $\mathfrak{m}_1, \mathfrak{m}_2 > 0$ such that the condition~\eqref{df:family:Ak} holds for the family $\Ak = \set{\Op^{\MM}_{h}\p{\holeperp{R}a},~R >0}$.
\end{lemma}
\begin{proof}[Proof of Lemma~\ref{lem:equicontinuous:satisfied}]
We first note that,
\begin{equation*}
\forall R >0, \quad \norm{\holeperp{R} a}_{L^{\infty}(\cotang \hspace{1 pt}\MM_I)}\leq \norm{a}_{L^{\infty}(\cotang \hspace{1 pt}\MM_I)} =: \mathfrak{m}_1.
\end{equation*}
Second, for any $\gamma\in \NN_0^{d}$ with $\abs{\gamma} = 1$ and $\set{\bullet_1, \bullet_2} = \set{x, \xi}$, we have 
\begin{align*}
\exists C_1, C_2 > 0,~\forall R > 0,\quad
\norm{\partial_{\bullet_1}^{\gamma}\jap{\xi}_g\partial_{\bullet_2}^{\gamma}\sqrtsymbol}_{L^\infty\p{\supp\p{\theta'\p{\jap{\xi}_g/R}}}}
&\leq C_1 \norm{\jap{\xi}_g}_{L^\infty\p{\supp\p{\theta'\p{\jap{\xi}_g/R}}}}\\
&\leq C_2 R.
\end{align*}
Then, using the chain rule, we infer
\begin{align*}
\exists C > 0, ~ \forall R > 0,\quad
\norm{\partial_{\bullet_1}^{\gamma}\p{\theta\p{\jap{\xi}_g/R}}\partial_{\bullet_2}^{\gamma}\sqrtsymbol}_{L^\infty(\cotang \hspace{1 pt}\MM_I)} &\leq R^{-1}\norm{\partial_{\bullet_1}^{\gamma}\jap{\xi}_g\partial_{\bullet_2}^{\gamma}\sqrtsymbol}_{L^\infty\p{\supp\theta'\p{\jap{\xi}_g/R}}}\norm{\theta'}_{L^{\infty}(\RR_+)}\\
&\leq C \norm{\theta'}_{L^{\infty}(\RR_+)}.
\end{align*}
In conclusion, there is $\mathfrak{m}_2 > 0$, such that for all $R >0$, $\norm{\set{\holeperp{R} a, \signe\sqrtsymbol}}_{L^\infty(\cotang \hspace{1pt}\MM_I)}\leq \mathfrak{m}_2$.
This ends the proof of Lemma~\ref{lem:equicontinuous:satisfied}.
\end{proof}
\begin{lemma}\label{lem:uniform:symbol:substitution}
Let us consider $\sqrtsymbol$ defined in~\eqref{df:Sqrt:Sym}, $m < 0$, $s\in [0, 1]$, $k_{m, s}\in S^{m}_{\comp}(\cotang\hspace{1pt} \MM_I)$ uniformly with respect to $s$ --- in the sense that $k_{m, s}$ is supported in a compact subset independent of $s$ in the $(t, x)$ variables and all its seminorms in $S^{m}_{\comp}(\cotang\hspace{1pt}\MM_I)$ are uniformly bounded with respect to $s\in [0, 1]$ --- and set
\begin{equation*}
e_{m, s}(t, x, \tau, \xi) := k_{m, s}(t, x, s \tau + (1-s)\sqrtsymbol(x, \xi), \xi), \quad (t,x, \tau, \xi)\in\cotang \hspace{1 pt}\MM_{I}.
\end{equation*}
Then, for all relatively compact local chart $K$ of $\MM_{I}$, and all $\alpha, \beta\in\NN^d$ we have
\begin{equation}\label{eq:19}
\sup_{s\in [0,1]}\sup_{
\begin{subarray}{c}
(t, x)\in K,\\ (\tau, \xi)\in \RR\times\RR^d
\end{subarray}
}
\jap{\xi}^{\abs{\beta}-m}
\abs{\partial^{\alpha}_{t, x}\partial^{\beta}_{\tau, \xi}e_{m, s}} < \infty.
\end{equation}
\end{lemma}\noindent
Since the proof of Lemma~\ref{lem:uniform:symbol:substitution} is not helpful to understand the proof of Theorem~\ref{Th:Slices} and relies on a technical induction, we provide it at the end of the section.
\begin{lemma}\label{lem:Taylor:microlocal}
For all $a\in S^0_{\phg}(\cotang\hspace{1pt} \MM_I)$, we have
\begin{equation}\label{df:taylor:microlocal:decomposition}
a(t, x, \tau, \xi) = \tilde{a}_{t}(x, \xi) + \big(\tau -\signe \sqrtsymbol(x, \xi)\big)\alpha_{\signe}(t, x, \tau,\xi),\quad (t, x, \tau, \xi)\in\cotang\hspace{1pt}\MM_I
\end{equation}
with:
\begin{enumerate}
\item $\chi\alpha_{\signe}\in S^{-1}_{\comp}(\cotang\hspace{1.5pt}\MM_I)$ with $\chi$ from~\eqref{eq:wave-type-ghost} and $\alpha_{\signe}$ as 
\begin{equation}\label{df:taylor:microlocal:alpha}
\alpha_{\signe}(t, x, \tau,\xi) := \int_{0}^{1}\frac{\partial a}{\partial\tau}\p{t, x, s\tau + (1-s)\signe\sqrtsymbol(x, \xi), \xi}\d s, \quad (t, x, \tau, \xi)\in\cotang\hspace{1pt} \MM_I.
\end{equation}
\item $\tilde{a}_{t}\in S_{\comp, \tang}^{0}\p{I\times \cotang M}$ with
\begin{equation}\label{eq:taylor:tilde:a:microlocal}
    \tilde{a}_{t}(x, \xi) := a\p{t, x, \signe\sqrtsymbol(x, \xi), \xi}, \quad (t, (x, \xi))\in I\times \cotang M.
\end{equation}
\item If we assume furthermore that $\supp\p{a}\cap\set{(t, x, \tau, \xi)\in\cotang\hspace{1pt}\MM_I, ~\tau = 0~\text{or}~\xi = 0} = \emptyset$, we have the following decomposition~\eqref{eq:characterization:S^0_phg}
\begin{equation}\label{eq:taylor:tilde:a:microlocal:phg}
\forall (t, (x, \xi))\in I\times \cotang M,\quad \tilde{a}_{t}(x, \xi) = \underbracket{\check{a}^t_{\phg}(x, \xi)}_{:=a\p{t, x, \signe\abs{\xi}_g, \xi}} + b(t, x, \xi),
\end{equation}
where $\check{a}_{\phg}\in S_{\phg, \tang}^{0}(I\times \cotang M)$ and $b\in S_{\comp, \tang}^{0}(I\times\cotang M)$.
Moreover, there is $C > 0$ such that
\begin{equation}\label{eq:taylor:b}
\sup_{(t, x, \xi)\in I\times \cotang M,~\abs{\xi}_{g}\geq R}\abs{b(t, x, \xi)} \leq CR^{-2},\quad  R \geq 1.
\end{equation}
\end{enumerate}
\end{lemma}
\begin{proof}[Proof of Lemma~\ref{lem:Taylor:microlocal}]
By Taylor formula, we have~\eqref{df:taylor:microlocal:decomposition} with $\tilde{a}$ in~\eqref{eq:taylor:tilde:a:microlocal} and $\alpha_{\signe}$ in~\eqref{df:taylor:microlocal:alpha}.
\begin{enumerate}
\item In this point we prove that $\chi\alpha_{\signe}\in S^{-1}_{\comp}(\cotang\hspace{1 pt}\MM_I)$.
To do so, we use Lemma~\ref{lem:uniform:symbol:substitution} with $m = -1$ and $k_{m} = \frac{\partial a}{\partial_\tau} \in S^{-1}_{\phg}\p{\cotang \hspace{1pt} \MM_{I}}$, and thus, $\alpha_{\signe}$ satisfies~\eqref{eq:19}.
Thus, $\chi\alpha_{\signe}\in S^{-1}_{\comp}(\cotang\hspace{1 pt}\MM_I)$
\item In this point we prove that $\tilde{a}\in S^{0}_{\comp, \tang}(I\times \cotang M)$.
To do so, we use Lemma~\ref{lem:uniform:symbol:substitution} with $(m, k_m, 0)=(0, a, 0)$, thus $\tilde{a}$ satisfies~\eqref{eq:19}, so is $\tilde{a}\in S^{0}_{\comp, \tang}(I\times \cotang M)$ 
\item In this point, we prove~\eqref{eq:taylor:tilde:a:microlocal:phg}. Since $a\in S^{0}_{\phg}(\cotang\hspace{1 pt}\MM_I)$ is supported away from $\set{\tau = 0~\text{or}~\xi = 0}$, we have 
\begin{equation*}
    \check{a}_{\phg} = \classe{(t, x, \tau, \xi)\mapsto a(t, x, \signe\abs{\xi}_g, \xi)}\in S^{0}_{\phg, \tang}(I\times \cotang M).
\end{equation*}
Thus, $b = \tilde{a}- \check{a}_{\phg} \in S^0_{\comp, \tang}\p{I\times \cotang M}$ as sum of two such functions.
Remark that $\tilde{a}$, $\check{a}_{\phg}$ and $b$ are supported away from $\set{\xi = 0}$.

Moreover, for all $(t, x, \tau, \xi)\in\cotang\hspace{1pt} \MM_I$, using $\partial_\tau a \in S^{-1}_{\phg}(\cotang \hspace{1 pt}\MM_I)$ and the mean value inequality, we have
\begin{align*}
\abs{b(t, x, \xi)} = \abs{\tilde{a}_t(x, \xi) - \check{a}^t_{\phg}(x, \xi)} &= \abs{a(t, x, \signe\sqrtsymbol(x, \xi),\xi) - a(t, x, \signe\abs{\xi}_g, \xi)}\\
&\textstyle\leq \abs{\sqrtsymbol(x, \xi) - \abs{\xi}_g}\norm{\frac{\partial a}{\partial \tau}}_{L^\infty\p{\cotang\hspace{1pt} \MM_I}}\\
&\leqsim \abs{\sqrt{\abs{\xi}_g^2 + V(x)} - \abs{\xi}_g}\jap{\xi}_g^{-1}\\
&\textstyle\leqsim \abs{\sqrt{1 + \frac{V(x)}{\abs{\xi}_g^2}} - 1}\abs{\xi}_g\jap{\xi}_g^{-1}.
\end{align*}
Then, by a second mean value inequality on $r\in\RR_+\mapsto \sqrt{1 + r}$, there is $C' \geq C$ --- independent of the choice of $(t, x, \tau, \xi)\in\cotang\hspace{1pt} \MM_I$ --- such that we have
\begin{equation*}
\abs{\tilde{a}_t(x, \xi) - a(t, x, \abs{\xi}_g, \xi)} \leq C \frac{V(x)}{\abs{\xi}_g^2}\abs{\xi}_g\jap{\xi}_g^{-1} \leq C'\abs{\xi}_g^{-2}.
\end{equation*}
Thus, this last inequality proves that~\eqref{eq:taylor:tilde:a:microlocal:phg} holds.
\end{enumerate}
This last point ends the proof.
\end{proof}
\paragraph{Proof of the 2-microlocal result.}
Equipped with Lemmas~\ref{lem:Symbols:Tangentiels:micro},~\ref{lem:equicontinuous:satisfied},~\ref{lem:uniform:symbol:substitution} and~\ref{lem:Taylor:microlocal} we are now prepared to give a proof of the 2-microlocal version of Theorem~\ref{Th:Slices}.
\begin{proof}[Proof of the second point (2-microlocal result) of Theorem~\ref{Th:Slices}]
The proof of this theorem is similar to the proof of the first point (semiclassical result) of Theorem~\ref{Th:Slices}.
Consider a cut-off function $\theta$ as~\eqref{df:theta} and set the notation
\begin{equation}\label{df:1-Theta::R:space}
\holeperp{R}(x, \xi) := 1-\theta\p*{\frac{\jap{\xi}_g}{R}} \quad\text{and}\quad\holepara{R}(\tau) := 1-\theta\p*{\frac{\tau}{R}},\quad (t, x, \tau, \xi)\in\cotang \hspace{1pt} \MM_I,~ R > 0.
\end{equation}
By hypothesis~\eqref{eq:wave-type-ghost}, Lemma~\ref{lem:Wave:time:space:reduction} allows us to use cut-off function $\holeperp{R}$ or $\holepara{R}$ as we used $\Theta^R$ in~\eqref{eq:MDM}.

Before the proof, to simplify the following arguments, let us exhibit a polyhomogeneous decomposition of $\sqrtsymbol$ (define in~\eqref{df:Sqrt:Sym}) away from $0$. In other words, for all points $(x, \xi)\in \cotang M$ with $\abs{\xi}_g > 1$, we have 
\begin{equation*}
\notag\sqrtsymbol(x, \xi) = \abs{\xi}_g\sum_{n = 0}^{\infty} \binom{1/2}{n}\p*{\frac{V(x)}{\abs{\xi}_g^2}}^n,
\end{equation*}
where, for $r\in\RR$, we recall the (classical) notation $\binom{r}{n} := \prod_{j = 1}^{n}\frac{r - (j-1)}{j}$.
Thus, for $\chi$ from in~\eqref{eq:wave-type-ghost} and $R$ large enough, we have 
\begin{equation}
\chi\holepara{R}\sqrtsymbol = \chi\holepara{R}\abs{\xi}_g + \chi\sum_{n = 1}^{\infty} \holepara{R}a_{1-2n}.\label{eq:18}
\end{equation}
where $a_{1 - 2n}\in S^{1 - 2n}_{\phg}\p{\cotang\hspace{1 pt}\MM_I}$ for all $n\geq 1$.
\begin{enumerate}\renewcommand{\labelenumi}{\alph{enumi}.}
\item
Let $a\in S^0_{\phg}(\cotang M)$ and set $A_{h}^R := \Op_{h}^{M}(a\holeperp{R})\in{\Psi}^{0}_{h, \comp}(M)$.
Now, for all integer $n\geq 1$ and all radius $R > 0$, we set
\begin{equation*}
\Fk^{a}_{R} := \set{F_{R, n}^{a}\colon t\in I\mapsto \pscal{v^{\signe}_{n}(t)}{A_{h_n}^Rv^{\signe}_{n}(t)}_{L^2(M)},~n \geq 1}.
\end{equation*}

First, we consider a temporarily fixed radius $R > 0$ and let $n$ tends to $\infty$.
The family $\Fk^{a}_R$ is uniformly bounded and uniformly equicontinuous on $I$ (both with respect to $R$) by Lemma~\ref{lem:equicontinuous:family}.
Therefore, Ascoli Theorem ensures that, for any compact interval $J\subset I$, $\Fk^{a}_{R}$ has compact closure for the topology of $\Cscr^{0}_{b}(J, \CC)$.
As a consequence, for all $J\subset I$ compact interval, all $a\in S^0_{\phg}(\cotang M)$ and all radius $R > 0$: there is a subsequence of $(v^{\signe}_{n})_n$, still denoted $(v^{\signe}_{n})_n$, such that $(F_{R,n}^{a})_n$ converges uniformly on $J$.
Fix an increasing sequence of compact intervals $(J_k)_k$ with $I = \bigcup_{k} J_k$, a countable dense (for the topology of uniform convergence) subset $\De\subset S_{\hom}^{0}\p{\cotang M}$ and consider radii in $\QQ$. 
By a first diagonal extraction, there is a subsequence of $(v^{\signe}_{n})_n$, still denoted $(v^{\signe}_{n})_n$, such that:
for all compact interval $J\subset I$, all $a\in S^0_{\phg}(\cotang M)$, all radius $R>0$ and all $\varepsilon >0$ there is $N(J, a, R)\geq 1$ such that 
\begin{equation*}
\forall n\geq N(J, a, R),\quad \sup_{t\in J}\abs{F^a_{R, n}(t) - \lim_{m\to + \infty}F^a_{R, m}(t)}\leq \varepsilon.
\end{equation*}
Therefore, as limit of a sequence of equicontinuous (with respect to $R$) functions on $I$ with uniform convergence on all compact subsets of $I$, the family
\begin{equation*}
\set{\lim_m F_{R, m}^{a},~R> 0}    
\end{equation*}
is continuous on $I$.
    
Second, we let $R$ tends to $\infty$.
By the second point of Lemma~\ref{lem:equicontinuous:family} (the condition is satisfied by Lemma~\ref{lem:equicontinuous:satisfied} and the previous paragraph), the family $\set{\lim_m F_{R, m}^{a},~R> 0}$ is equicontinuous on $I$.
Thus, Ascoli Theorem ensures that, for any compact interval $J\subset I$, the family $\set{\lim_m F_{R, m}^{a},~R > 0}$ has compact closure for the topology of $\Cscr^{0}_{b}(J,\CC)$.
As a consequence, for all compact interval $J\subset I$ and all $a\in S^0_{\phg}(\cotang M)$, there is a sequence of $(R^{J, a}_k)_k$ such that $\p{\lim_mF_{R^{J, a}_k,m}^{a}}_k$ converges uniformly on $J$.
Fix an increasing sequence of compact intervals $(J_k)_k$ with $I = \bigcup_{k} J_k$ and a countable dense (for the topology of uniform convergence) subset $\De\subset S_{\hom}^{0}\p{\cotang M}$.
By a second diagonal extraction, there is a sequence of $(R_k)_k$ such that: for all compact interval $J\subset I$, all $a\in\De$ and all $\varepsilon >0$, there is $K(J, a)\geq 1$ such that 
\begin{equation*}
\forall k\geq K(J, a),\quad \sup_{t\in J}\abs*{\lim_{m\to + \infty}\p{F^a_{R_k, m}(t)} - \int_{\cosphere M}a(x, \xi)\d\mdm^{\signe}_t(x, \xi)}\leq \varepsilon.
\end{equation*}
Therefore, as limit of a sequence of continuous functions on $I$ with uniform convergence on all compact subsets of $I$, the function
\begin{equation*}
t\in I\mapsto \int_{\cosphere M}a\d\mdm^{\signe}_{t}
\end{equation*}
is continuous on $I$.

Finally, we obtain that $t\in I\mapsto \mdm^{\signe}_{t}$ is continuous for the weak-* topology of measures.
\item 
We use the function $\chi$ from~\eqref{eq:wave-type-ghost}, the cut-off function $\holeperp{R}$ from~\eqref{df:1-Theta::R:space}, and we consider $a\in S^{0}_{\phg}(\cotang\hspace{1pt}\MM_I)$.
Since $\supp\p{\chi\holeperp{R}a}\cap \set{\tau = 0~\text{or}~\xi = 0} = \emptyset$ for any $R > 0$, we can consider
\begin{equation*}
    \supp(a)\cap \set{(t, x, \tau, \xi)\in\cotang \hspace{1 pt}\MM_I,\quad \tau = 0~\text{or}~\xi = 0}.
\end{equation*}
Lemma~\ref{lem:Taylor:microlocal} then yields the Taylor decomposition~\eqref{df:taylor:microlocal:decomposition}: 
\begin{equation*}
    a(t, x, \tau, \xi) = \underbracket{\check{a}^{t}_{\phg}(x, \xi) + b(t, x, \xi)}_{\text{given in~\eqref{eq:taylor:tilde:a:microlocal}}} + \ell^{-\signe}(t, x, \tau, \xi)\underbracket{\alpha_{\signe}(t, x, \tau, \xi)}_{\text{given in~\eqref{df:taylor:microlocal:alpha}}},\quad (t, x, \tau, \xi)\in\cotang \hspace{1 pt}\MM_I.
\end{equation*}
Moreover, using Theorem~\ref{Th:hormander}, we have
\begin{equation*}
    \chi\holeperp{R}\check{a}_{\phg}\in {S^0_{\comp}}\p{\cotang\hspace{1pt} \MM_I}~\text{and}~\chi \holeperp{R}b,~\chi\holeperp{R}\ell^{-\signe}\alpha_{\signe}\in S^{0}_{\comp}\p{\cotang\hspace{1pt} \MM_I},
\end{equation*}
and there is $\rest^{-1}_{R, h}\in{\Psi^{-1}_{h, \comp}(\MM_I)}$ compactly supported in $\pi_{\MM_I}(\supp(a))$ such that
\begin{equation}\label{eq:Taylor:decomposition:Microlocalization}
\Op_{h}^{\MM}(\chi\holeperp{R}a) = \Op_{h, \tang}^{\MM}(\holeperp{R}\check{a}_{\hom})\Op_{h}^{\MM}(\chi) + \Op_{h}^{\MM}(\chi \holeperp{R}b) + \Op_{h}^{\MM}(\chi\holeperp{R}\alpha_{\signe})\HWOp^{-\signe}_{h} + h \rest^{-1}_{R, h}.
\end{equation}
Then, by Proposition~\ref{pr:Calderon:Vaillancourt} together with $\chi \holeperp{R}b\in S^{0}_{\comp}\p{\cotang\hspace{1pt} \MM_I}$ and  $\chi\alpha_{\signe}\in S^{-1}_{\comp}(\cotang\hspace{1pt}\MM_I)$ (from the first point of Lemma~\ref{lem:Taylor:microlocal}), we have the two upper bounds
\begin{align}
\notag\norm{\Op_{h_n}^{\MM}(\chi \holeperp{R}\alpha_{\signe})}_{\linear\p{L^2\p{\MM_I}}} &\leq \norm{\holeperp{R}\p{\chi \alpha_{\signe}}}_{L^\infty\p{\MM_I}} + C_Rh_n^{1/2}\\
\notag& \leqsim \norm{\holeperp{R}(\xi)\jap{\xi}_g^{-1}}_{L^\infty\p{\cotang\hspace{1pt} \MM_I}} + C_Rh_n^{1/2}\\
\label{eq:operator:norm:TH:PG:Microlocalization}&\leqsim R^{-1} + C_Rh_n^{1/2},
\end{align}
and, by~\eqref{eq:taylor:b},
\begin{equation}\label{eq:operator:norm:TH:PG:Microlocalization:chi}
\norm{\Op_{h_n}^{\MM}(\chi \holeperp{R}b)}_{\linear\p{L^2\p{\MM_I}}} \leqsim R^{-2} + C_Rh_n^{1/2}.
\end{equation}

Thus, using the decomposition~\eqref{eq:Taylor:decomposition:Microlocalization} and $\HWOp^{-\signe}v^{\signe}_n = h_nf^n$, we obtain
\begin{align*}
\pscal{v^{\signe}_{n}}{\Op_{h_n}^{\MM}(\chi \holeperp{R}a)v^{\signe}_{n}}_{L^2(\MM_I)} 
=&\textstyle 
\pscal{v^{\signe}_{n}}{\Op_{h_n, \tang}^{\MM}(\holeperp{R}\check{a}_{\phg})\Op_{h_n}^{\MM}(\chi)v^{\signe}_{n}}_{L^2(\MM_I)}\\
&+ \pscal{v^{\signe}_{n}}{\Op_{h_n}^{\MM}(\chi \holeperp{R}b)v^{\signe}_{n}}_{L^2(\MM_I)}\\
&+ h_n\pscal{v^{\signe}_{n}}{\Op^{\MM}_{h_n}(\chi \holeperp{R}\alpha_{\signe}) {f}^n + \rest^{-1}_{R, h_n}v^{\signe}_{n}}_{L^2(\MM_I)},
\end{align*}
then, by the upper bounds~\eqref{eq:operator:norm:TH:PG:Microlocalization},~\eqref{eq:operator:norm:TH:PG:Microlocalization:chi} and Proposition~\ref{pr:Calderon:Vaillancourt}, we have
\begin{align}
\notag \pscal{v^{\signe}_{n}}{\Op_{h_n}^{\MM}(\chi \holeperp{R}a)v^{\signe}_{n}}_{L^2(\MM_I)}
=&
\pscal{v^{\signe}_{n}}{\Op_{h_n, \tang}^{\MM}(\holeperp{R}\check{a}_{\phg})\Op_{h_n}^{\MM}(\chi)v^{\signe}_{n}}_{L^2(\MM_I)}\\
&+ \O{R^{-1}} +\O{h_n^{1/2}}.\label{eq:Right:Hand:2:point:microlocalization}
\end{align}
Therefore, on the one hand, using Lemma~\ref{lem:Symbols:Tangentiels:micro} (up to consider $R > 0$ large enough, we may consider $\chi a$ as a $S^0_{\phg}(\cotang\hspace{1 pt} \MM_I)$ function) and~\eqref{eq:Right:Hand:2:point:microlocalization}, we obtain
\begin{align}
\notag\lim_{R\to\infty}\lim_{n\to\infty}\pscal{v^{\signe}_{n}}{\Op_{h_n}^{\MM}(\chi\holeperp{R} a)v^{\signe}_{n}}_{L^2(\MM_I)} &= \int_{I\times \cosphere M}a(t, x, +\signe\abs{\xi}_g, \xi)\d\mdm^{\signe}_{t}(x, \xi)\d t\\
&=\int_{\cosphere\hspace{1pt} \MM_I}a(t, x, \tau, \xi)~\d\p{\mdm^{\signe}_{t}(x, \xi)\otimes \delta_{\signe\abs{\xi}_g}(\tau)\otimes \Lc(t)}.\label{eq:microlocal:point:2:PG:proof:decomposition}
\end{align}
On the other hand, by~\eqref{eq:wave-type-ghost} and~\eqref{eq:MDM} together with Lemma~\ref{lem:Wave:time:space:reduction} (whose hypotheses are satisfied by~\eqref{eq:wave-type-ghost}), we also have
\begin{equation}
\lim_{R\to\infty}\lim_{n\to\infty}\pscal{v^{\signe}_{n}}{\Op^{\MM}_{h_n}\p{\chi \holeperp{R} a}v^{\signe}_{n}}_{L^2(\MM_I)} =\int_{\cosphere\hspace{1 pt}\MM_I}a\d\mdm_{(\signe, \signe)}.\label{eq:microlocal:point:2:PG:proof:MDM}
\end{equation}
Then, by~\eqref{eq:microlocal:point:2:PG:proof:MDM} and~\eqref{eq:microlocal:point:2:PG:proof:decomposition}, we conclude~\eqref{eq:PG:microlocal:point2}.
\item From~\eqref{eq:PG:microlocal:point2}, we have 
\begin{equation*}
\supp\p{\mdm_{(\signe, \signe)}}\subseteq\set{(t, x, \tau, \xi)\in\cosphere \hspace{1 pt}\MM_I,~ \tau = \signe\abs{\xi}_{g}}.
\end{equation*}
Furthermore, using Proposition~\ref{pr:inter:supp:k:2mdm}, $\mdm_{(\signe, f)}$ (a finite-valued complex measure) is absolutely continuous with respect to $\mdm_{(\signe, \signe)}$ (a nonnegative $\sigma$-finite measure), hence, by Radon--Nikodym Theorem there is a complex measure $\tilde{\mdm}_{(\signe, f)}\in\Mc_{+}\p{\cosphere M}$ such that
\begin{equation*}
\mdm_{(\signe, f)}(t, x, \tau, \xi) = \tilde{\mdm}_{(\signe, f)}^t( x, \xi)\otimes \delta_{\signe\abs{\xi}_g}\p{\tau}\otimes\Lc_1(t).
\end{equation*}
Take $a\in S^0_{\phg, \tang}\p{I\times \cotang M}$, $\chi$ from~\eqref{eq:wave-type-ghost} and $\holepara{R}$ as defined in~\eqref{df:1-Theta::R:space}.
By Theorem~\ref{Th:hormander}, we have $\chi\ell^{-\signe}\in S^1(\cotang \hspace{1 pt}\MM_I)$, $(\chi a,\set{\ell^{-\signe}, \chi\holepara{R}a}) \in S^0_{\comp}(\cotang \hspace{1 pt}\MM_I)^2$, and
\begin{equation}\label{eq:microlocal:corchet:poisson}
\classe{\HWOp^{-\signe}_{h}, \Op_{h}^{\MM}\p{\chi\holepara{R}a}}= \frac{h}{i}\Op_{h}^{\MM}\p{\set{\ell^{-\signe}, \chi\holepara{R}a}} + h^2\rest^{-1}_{R, h},\quad  \rest^{-1}_{R, h}\in \Psi^{-1}_{h, \comp}\p{\MM_I}.
\end{equation}
On the one hand, using boundedness of $\p{v^{\signe}_n}_n$ in $L^2(\MM_I)$ and Proposition~\ref{pr:Calderon:Vaillancourt}, for all temporarily fixed $R > 0$ we have
\begin{equation}\label{eq:remainder:microlocal:c}
\pscal{v^{\signe}_n}{h_n\rest^{-1}_{R, h_n} v^{\signe}_n}_{L^2(\MM_I)} = \O{h_n}.
\end{equation}
Furthermore, using Lemma~\ref{lem:L:crochet}, we have 
\begin{align}
\notag\pscal{v^{-\signe}_n}{\frac{1}{h_n}\classe{\HWOp^{-\signe}_{h_n}, \Op_{h_n}^{\MM}\p{\chi\holepara{R}a}} v^{\signe}_n}_{L^2(\MM_I)}
=&
\pscal{f^n}{\Op_{h_n}^{\MM}\p{\chi\holepara{R}a} v^{\signe}_n}_{L^2(\MM_I)}\\ 
&- \pscal{v^{\signe}_n}{\Op_{h_n}^ {\MM}\p{\chi\holepara{R}a} f^n}_{L^2(\MM_I)}.\label{eq:microlocal:corchet}
\end{align}
Thus, by~\eqref{eq:microlocal:corchet:poisson},~\eqref{eq:remainder:microlocal:c},~\eqref{eq:microlocal:corchet},~\eqref{eq:wave-type-ghost} and~\eqref{eq:MDM} together with Lemma~\ref{lem:Wave:time:space:reduction} (whose hypotheses are satisfied by~\eqref{eq:wave-type-ghost}), we conclude
\begin{equation}\label{eq:proof:A}
\lim_{R\to\infty}\lim_{n\to\infty}\pscal{v^{\signe}_n}{\Op_{h_n}^{\MM}\p{\set{\ell^{-\signe}, \chi\holepara{R}a}}v^{\signe}_n}_{L^2(\MM_I)} = -2\int_{I\times \cosphere M} a(t, x, \xi)\d\p{\Imaginary\tilde{\mdm}^t_{(\signe,f)}(x, \xi)}\d t.
\end{equation}
On the other hand, since $\holepara{R}$ is independent of $(t, x, \xi)$ and $\ell^{-\signe}$ is independent of $t$, we have
\begin{equation*}
\set{\ell^{-\signe}, \holepara{R}} = - \underbracket{\partial_t(\ell^{-\signe})}_{=0}\partial_\tau\holepara{R} = 0.
\end{equation*}
Thus, we have
\begin{align}
\notag\set{\ell^{-\signe}, \chi\holepara{R}a} &= \set{\ell^{-\signe}, \holepara{R}}\chi a + \set{\ell^{-\signe}, \chi a}\holepara{R}\\
\notag&= \set{\ell^{-\signe}, \chi a}\holepara{R}\\
\label{eq:24}&= \p{r + \set{\ell^{-\signe}, a}\chi}\holepara{R}
\end{align}
where $r =  \set{\ell^{-\signe}, \chi}a \in\tilde{S^{0}_{\phg}}\p{\cotang \hspace{1pt} \MM_I}$ is a remainder term in the sense that, using point~\ref{eq:chi:homogeneous} of Definition~\ref{df:WGF}, $r$ vanishes near 
\begin{equation}\label{eq:23}
\set{(t, x, \tau, \xi)\in\cosphere \hspace{1pt} \MM_I, ~ \chi(t, x, \mathsf{R}(\chi)\tau, \mathsf{R}(\chi)\xi) = 1}\supset \supp\p{\mdm^{+}}\sqcup \supp\p{\mdm^{-}}. 
\end{equation}
Then, using~\eqref{eq:MDM} together with Lemma~\ref{lem:Wave:time:space:reduction} (hypothesis are satisfied by~\eqref{eq:wave-type-ghost}),~\eqref{eq:23} yields 
\begin{equation*}
\lim_{R\to\infty}\lim_{n\to\infty}\pscal{v^{\signe}_n}{\Op_{h}^{\MM}\p{r\holepara{R}}v^{\signe}_n}_{L^2(\MM_I)} = 0.
\end{equation*}
Therefore, by~\eqref{eq:24} together with~\eqref{eq:23}, we have
\begin{align}\label{eq:proof:B}
\notag \pscal{v^{\signe}_{n}}{\Op_{h_n}^{\MM}\p{\set{\ell^{-\signe}, \chi \holepara{R}a}}v^{\signe}_{n}}
&= \pscal{v^{\signe}_{n}}{\Op_{h_n}^{\MM}\p{\set{\ell^{-\signe}, a}\chi\holepara{R}}v^{\signe}_{n}} + \O{\o{1}_{R}}_n\\
&= \pscal{v^{\signe}_{n}}{\Op_{h_n}^{\MM}\p{\p{\partial_t a + \set{-\signe\sqrtsymbol(x, \xi),  a}}\chi\holepara{R}}v^{\signe}_{n}} + \O{\o{1}_{R}}_n.
\end{align}
So, since $\p{\partial_t a + \set{-\signe\sqrtsymbol(x, \xi),  a}}\chi\in \tilde{S^0_{\phg}}(\cotang\hspace{1pt} \MM_I)$ with homogeneous principal part $\partial_t a + \set{-\signe\abs{\xi}_g,  a}$ (see polyhomogeneous decomposition of $\sqrtsymbol$ given in~\eqref{eq:18}),~\eqref{eq:proof:A} in addition with both~\eqref{eq:proof:B} and~\eqref{eq:PG:microlocal:point2} implies 
\begin{equation*}
\int_{I\times \cosphere M}\partial_t a + \set{-\signe\abs{\xi}_g,  a}\d\mdm^{\signe}_{t}\d t = 2\int_{I\times\cosphere M}a(t, x, \xi) \d(\Imaginary\tilde{\mdm}^{t}_{(\signe, f)}(x, \xi))\d t.
\end{equation*}
\end{enumerate}
That ends the proof.
\end{proof}
\begin{proof}[Proof of Lemma~\ref{lem:uniform:symbol:substitution}]
Remark firstly that $\sqrtsymbol\in S^1\p{\cotang M}\subset S^1_{\tang}\p{I\times \cotang M}$ and secondly that, for all $m' < 0 $, 
\begin{equation}\label{eq:proof:decreasing:polynomial}
x\in\RR^{*}_{+}\mapsto x^{m'/2}~\text{and}~x\in\RR^{*}_{+}\mapsto x^{(m'-1)/2} ~\text{are decreasing.}
\end{equation}
For all relatively compact local chart $K$ of $\MM_{I}$, since $k_{m, s}\in S^{m}_{\comp}(\cotang\hspace{1pt} \MM_I)$ uniformly with respect to $s\in [0, 1]$, there is $C > 0$ such that for all $s\in [0, 1]$, all point $(t, x)\in K$ and all $(\tau, \xi)\in \mathrm{T}^{\star}_{(t, x)}\MM$, using~\eqref{eq:proof:decreasing:polynomial}, we have the upper bound
\begin{equation*}
\abs{e_{m, s}(t, x, \tau, \xi)} \leq C \p*{1 + \p{s \tau + (1-s)\sqrtsymbol(x, \xi)}^2 + \abs{\xi}_g^{2}}^{m/2} \leq C \jap{\xi}_{g}^{m}.
\end{equation*}
We now prove by induction on $n$ that, for all $n\in\NN_0$, all relatively compact local chart $K$ of $\MM_{I}$, and all $\beta_1, \beta_2\in\NN_0^{d+1}$ with $\abs{\beta_1}+\abs{\beta_2}\leq n$ we have
\begin{equation}\label{eq:fdb}
\sup_{s\in [0,1]}\sup_{
\begin{subarray}{c}
(t, x)\in K,\\ (\tau, \xi)\in \RR\times\RR^d
\end{subarray}
}
\jap{\xi}^{\abs{\beta_2}-m}
\abs{\partial^{\beta_1}_{t, x}\partial^{\beta_2}_{\tau, \xi}e_{m, s}} < \infty
\end{equation}
and there is $k^{\beta_1, \beta_2}_{m, s}\in S^{m-\abs{\beta_2}}_{\comp}(\cotang\hspace{1pt} \MM_I)$ uniformly with respect to $s\in [0, 1]$, such that
\begin{equation*}
e_{m, s}^{\beta_1, \beta_2}(\rho) := \partial^{\beta_1}_{t, x}\partial^{\beta_2}_{\tau, \xi}e_{m, s}(\rho) = k_{m, s}^{\beta_1, \beta_2}(t, x, s \tau + (1-s)\sqrtsymbol(x, \xi), \xi), \quad \rho = (t,x, \tau, \xi)\in\cotang \hspace{1 pt}\MM_{I}.
\end{equation*}
This holds for $n=0$.
Assume it holds for some $n\in\NN_0$, we show it also holds for rank $n+1$.
Consider any $\beta_1, \beta_2\in\NN_0^{d+1}$ with $\abs{\beta_1}+\abs{\beta_2}\leq n$.
\begin{enumerate}
\item Set $\beta\in\NN_0^{d+1}$ with $\abs{\beta} = 1$.
For $s\in[0, 1]$ and all point $(t, x, \tau, \xi)\in \cotang\hspace{1pt} \MM_{I}$, we have
\begin{align*}
\partial_{t, x}^\beta e_{m, s}^{\beta_1, \beta_2}(t, x, \tau, \xi)
=& \partial_{t, x}^\beta k_{m, s}^{\beta_1, \beta_2}(t, x, s \tau + (1-s)\sqrtsymbol(x, \xi), \xi)\\
&+ (1-s) \partial_{t, x}^\beta \sqrtsymbol(x, \xi)\partial_{\tau} k_{m, s}^{\beta_1, \beta_2}(t, x, s \tau + (1-s)\sqrtsymbol(x, \xi), \xi).
\end{align*}
Thus, because $\partial_{t, x}^{\beta} \sqrtsymbol\in S^{1}(\cotang M)\subset S_{\tang}^{1}(I\times \cotang M)$, $k_{m, s}^{\beta_1, \beta_2}, \partial_{t, x}^\beta k_{m, s}^{\beta_1, \beta_2} \in S^{m - \abs{\beta_2}}_{\comp}(\cotang\hspace{1pt} \MM_I)$ and $ \partial_{\tau}k_{m, s}^{\beta_1, \beta_2}\in S^{m - \abs{\beta_2}-1}_{\phg}(\cotang\hspace{1pt} \MM_I)$ uniformly with respect to $s\in [0, 1]$, we set
\begin{equation*}
k_{m, s}^{\beta_1 + \beta, \beta_2} := \partial_{t, x}^\beta k_{m, s}^{\beta_1, \beta_2}
+ (1-s) \partial_{t, x}^\beta \sqrtsymbol\partial_{\tau} k_{m, s}^{\beta_1, \beta_2}\in S^{m - \abs{\beta_2}}_{\comp}(\cotang\hspace{1pt} \MM_I)
\end{equation*}
uniformly with respect to $s\in [0, 1]$.
Then, by~\eqref{eq:proof:decreasing:polynomial}, for all relatively compact local chart $K$ of $\MM_{I}$ there is $C_1, C_2 > 0$ such that for all $s\in [0, 1]$, all point $(t, x)\in K$ and all $(\tau, \xi)\in \mathrm{T}^{\star}_{(t, x)}\MM$ we have
\begin{equation*}
\abs{\partial_{t, x}^\beta e_{m, s}^{\beta_1, \beta_2}(t, x, \tau, \xi)}\leq C_1\jap{\xi}_{g}^{m - \abs{\beta_2}} + (1-s) C_2 \jap{\xi}_g\jap{\xi}_g^{m - \abs{\beta_2}-1}\leq(C_1 + C_2)\jap{\xi}_g^{m}.
\end{equation*}
\item For $s\in [0, 1]$ and all point $(t, x, \tau, \xi)\in \cotang\hspace{1pt} \MM_{I}$, we have
\begin{equation*}
\partial_{\tau} e_{m, s}^{\beta_1, \beta_2}(t, x, \tau, \xi) = s\partial_{\tau} k_{m, s}^{\beta_1, \beta_2}(t, x, s \tau + (1-s)\sqrtsymbol(x, \xi), \xi).
\end{equation*}
Thus, because $\partial_{\tau} k_{m, s}^{\beta_1, \beta_2} \in S^{m-\abs{\beta_2} - 1}_{\comp}(\cotang\hspace{1pt} \MM_I)$ uniformly with respect to $s\in [0, 1]$, we set
\begin{equation*}
k_{m, s}^{\beta_1, \beta_2 + \beta} := s\partial_\tau k_{m, s}^{\beta_1, \beta_2}\in S^{m - \abs{\beta_2} - 1}_{\comp}(\cotang\hspace{1pt} \MM_I)
\end{equation*}
uniformly with respect to $s\in [0, 1]$, where $\beta = (1, 0, \dots, 0)\in\NN_0^{d+1}$.
Then, by 
~\eqref{eq:proof:decreasing:polynomial}, for all relatively compact local chart $K$ of $\MM_{I}$ there is $C > 0$ such that for all point $(t, x)\in K$ and all $(\tau, \xi)\in \mathrm{T}^{\star}_{(t, x)}\MM$ we have
\begin{equation*}
\sup_{s\in [0,1]}\abs{\partial_{\tau}e_{m, s}^{\beta_1, \beta_2}(t, x, \tau, \xi)}\leq C\jap{\xi}_{g}^{m}.
\end{equation*}
\item Set $\beta := (0, \tilde{\beta})\in\NN_0^{d}$ with $\abs{\beta} = 1$.
For $s\in [0, 1]$ and all point $(t, x, \tau, \xi)\in \cotang\hspace{1pt} \MM_{I}$, we have
\begin{align*}
\partial_{\xi}^{\tilde{\beta}} e_{m, s}^{\beta_1, \beta_2}(t, x, \tau, \xi)
=& \partial_{\xi}^{\tilde{\beta}} k_{m, s}^{\beta_1, \beta_2}(t, x, s \tau + (1-s)\sqrtsymbol(x, \xi), \xi)\\
&+ (1-s) \partial_{\xi}^{\tilde{\beta}} \sqrtsymbol(x, \xi)\partial_{\tau} k_{m, s}^{\beta_1, \beta_2}(t, x, s \tau + (1-s)\sqrtsymbol(x, \xi), \xi).
\end{align*}
Thus, because $\partial_{\xi}^{\beta} k_{m, s}^{\beta_1, \beta_2}, \partial_{\tau}k_{m, s}^{\beta_1, \beta_2} \in S^{m-1}_{\comp}(\cotang\hspace{1pt} \MM_I)$ uniformly with respect to $s\in [0, 1]$ and  $\partial_{\xi}^\beta \sqrtsymbol\in S^0(\cotang M)\subset S^{-1}_{\tang}\p{ I \times \cotang M}$, we set
\begin{equation*}
k_{m, s}^{\beta_1, \beta_2 + \beta} := \partial_{\xi}^{\tilde{\beta}} k_{m, s}^{\beta_1, \beta_2} + (1-s) \partial_{\xi}^{\tilde{\beta}} \sqrtsymbol\partial_{\tau} k_{m, s}^{\beta_1, \beta_2}\in S^{m - \abs{\beta_2} - 1}_{\comp}(\cotang\hspace{1pt} \MM_I)
\end{equation*}
uniformly with respect to $s\in [0, 1]$.
Then, by~\eqref{eq:proof:decreasing:polynomial}, for all relatively compact local chart $K$ of $\MM_{I}$ there is $C_1, C_2 > 0$ such that for all $s\in [0, 1]$, all point $(t, x)\in K$ and all $(\tau, \xi)\in \mathrm{T}^{\star}_{(t, x)}\MM$ we have
\begin{equation*}
\sup_{s\in [0,1]}\abs{\partial_{\xi}^{\tilde{\beta}} e_{m, s}^{\beta_1, \beta_2}(t, x, \tau, \xi)}\leq C_1\jap{\xi}_g^{m-1} + C_2\jap{\xi}_{g}^{m-1} = (C_1 + C_2)\jap{\xi}_{g}^{m-1}.
\end{equation*}
\end{enumerate}
Therefore, by~\eqref{eq:fdb} and the three previous points, we showed that for all relatively compact local chart $K$ of $\MM_{I}$, and all $\beta_1, \beta_2\in\NN_0^d$ with $\abs{\beta_1}+\abs{\beta_2}\leq n+1$ we have
\begin{equation*}
\sup_{s\in [0,1]}\sup_{
\begin{subarray}{c}
(t, x)\in K,\\ (\tau, \xi)\in \RR\times\RR^d
\end{subarray}
}
\jap{\xi}^{\abs{\beta_2}-m}
\abs{\partial^{\beta_1}_{t, x}\partial^{\beta_2}_{\tau, \xi}e_{m, s}^{\beta_1, \beta_2}} < \infty,
\end{equation*}
and there is $k^{\beta_1, \beta_2}_{m, s}\in S^{m-\abs{\beta_2}}_{\comp}(\cotang\hspace{1pt} \MM_I)$ uniformly with respect to $s\in [0, 1]$, such that
\begin{equation*}
e_{m, s}^{\beta_1, \beta_2}(\rho) := \partial^{\beta_1}_{t, x}\partial^{\beta_2}_{\tau, \xi}e_{m, s}(\rho) = k_{m, s}^{\beta_1, \beta_2}(t, x, s \tau + (1-s)\sqrtsymbol(x, \xi), \xi), \quad \rho = (t,x, \tau, \xi)\in\cotang \hspace{1 pt}\MM_{I}.
\end{equation*}
Hence, by induction on $\abs{\beta_1} + \abs{\beta_2}$, we have that the upper bound~\eqref{eq:19}.
\end{proof}
\begingroup
\small
\bibliographystyle{alpha}
\bibliography{Bibliographie.bib}
\endgroup
\end{document}